\documentclass[11pt,reqno]{amsart}

\usepackage{enumerate}
\usepackage{amsmath, amssymb, amsthm}
\usepackage{mathrsfs}
\usepackage{esint}
\usepackage{xcolor}
\usepackage{mathtools}
\usepackage{hyperref}
\usepackage{bm}

\DeclareMathOperator\supp{supp}

\newtheorem{thm}{Theorem}[section]
\newtheorem{corollary}[thm]{Corollary}
\newtheorem{lem}[thm]{Lemma}
\newtheorem{prop}[thm]{Proposition}
\theoremstyle{definition}
\newtheorem{defn}[thm]{Definition}
\newtheorem{example}[thm]{Example}
\newtheorem{assumption}[thm]{Assumption}
\theoremstyle{remark}
\newtheorem{rem}[thm]{Remark}

\newcommand\bC{\mathbb{C}}

\newcommand\bE{\mathbb{E}}

\newcommand\bL{\mathbb{L}}

\newcommand\fB{\mathbf{B}}

\newcommand\fR{\mathbf{R}}

\newcommand\fC{\mathbf{C}}

\newcommand\fZ{\mathbf{Z}}

\newcommand\fN{\mathbf{N}}

\newcommand\cB{\mathcal{B}}
\newcommand\cC{\mathcal{C}}
\newcommand\cD{\mathcal{D}}
\newcommand\cF{\mathcal{F}}

\newcommand\cH{\mathcal{H}}

\newcommand\cM{\mathcal{M}}
\newcommand\cP{\mathcal{P}}

\newcommand\cS{\mathcal{S}}
\newcommand\cT{\mathcal{T}}
\newcommand\cU{\mathcal{U}}

\newcommand\rF{\mathscr{F}}
\newcommand\rG{\mathscr{G}}
\newcommand\rH{\mathscr{H}}
\newcommand\rT{\mathscr{T}}

\DeclareMathOperator*{\esssup}{ess\,sup}

\newcommand\cbrk{\text{$]$\kern-.15em$]$}}
\newcommand\clbrk{\text{$[$\kern-.15em$[$}}
\newcommand\opar{\text{\,\raise.2ex\hbox{${\scriptstyle|}$}\kern-.34em$($}}
\newcommand\cpar{\text{$)$\kern-.34em\raise.2ex\hbox{${\scriptstyle |}$}}\,}

\newcommand{\mysection}[1]{\section{#1}
\setcounter{equation}{0}}

\begin{document}

\title[SPDEs with sign-changing time-measurable P-D operators]
{An existence and uniqueness theory to stochastic partial differential equations involving pseudo-differential operators driven by space-time white noise}

\author[J.-H. Choi]{Jae-Hwan Choi}
\address[J.-H. Choi]{Research Institute of Mathematics, Seoul National University, 1 Gwanak-ro, Gwanak-gu, Seoul 08826, Republic of Korea}
\email{choijh1223@snu.ac.kr}

\author[I. Kim]{Ildoo Kim}
\address[I. Kim]{Department of Mathematics, Korea University, 145 Anam-ro, Seongbuk-gu, Seoul, 02841, Republic of Korea}
\email{waldoo@korea.ac.kr}

\thanks{The authors were supported by the National Research Foundation of Korea(NRF) grant funded by the Korea government(MSIT) (No.2020R1A2C1A01003959)}



\maketitle
\begin{abstract}
In this paper, we aim to develop a new weak formulation that ensures well-posedness for a broad range of stochastic partial differential equations with pseudo-differential operators whose symbols depend only on time and spatial frequencies. The main focus of this paper is to relax the conditions on the symbols of pseudo-differential operators and data while still ensuring that the stochastic partial differential equations remain well-posed in a weak sense. Specifically, we allow symbols to be random and remove all regularity and ellipticity conditions on them. As a result, our main operators include many interesting rough operators that cannot generate any regularity gain or integrability improvement from the equations.
In addition, our data do not need to be regular or possess finite stochastic moments.
\end{abstract}

\tableofcontents

\mysection{Introduction}

Research on stochastic partial differential equations (SPDEs) is now recognized as a significant field in mathematics, given its numerous intriguing applications across various sciences and the challenging mathematical problems it presents.
Much of the research in this area has focused on solving specific equations originating from scientific disciplines. 
For instance, extensive studies have been dedicated to filtering equations, KPZ equations, Anderson models, stochastic fluid dynamics, population dynamics, financial mathematics, and other related topics (\textit{cf}. \cite{LR 2017, Par 2021} and references therein).

The irregular effects of random noises typically hinder the derivation of these theories from (deterministic) partial differential equations (PDEs). As a result, many SPDEs are considered ill-posed within the traditional PDE framework. 
To overcome this challenge, new concepts have been developed to provide precise mathematical meanings for solutions of various SPDEs. 
Consequently, a broad array of mathematical tools are now available for studying these equations.
For instance, these tools include variational methods with monotone operators (\cite{Par 1975, KR 1981, LR 2015}), random field approach (\cite{Walsh 1986, khosh 2014}), semi-group methods (\cite{DZ 2014, NVW 2008}), $L_p$-analytic approach (\cite{Krylov 1999}), 
$H^\infty$-calculus (\cite{NVW 2012, NVW 2012-2}), regularity structures (\cite{FH 2020, Hairer 2014}), paracontrolled calculus (\cite{GIP 2015}), stochastic Calderón-Zygmund theory (\cite{KK 2020}), scaling limit approach (\cite{CSZ 2023}), and vector-valued harmonic analysis methods (\cite{Lori 2021, LV 2021}).

Most research focuses on solving SPDEs that are relevant to scientific applications and most mathematical models are simplified with nice operators and data. Consequently, these studies strive to understand solutions and equations in a practical sense rather than merely exploring possible mathematical generalizations. 
As a result, the main operators used in SPDEs, such as the Laplacian operator, the fractional Laplacian operator, and diffusion operators, are highly regular and well-behaved. 
These operators enhance the regularity of the solutions, enabling us to comprehend the equations and solutions in a robust manner, even when the input data is relatively rough due to randomness.

However, historically, there had been attempts to solve PDEs with general operators that do not generate regularity gains in weak formulations, particularly after Schwartz's distribution theory became prominent. Mathematicians discovered long ago that any linear partial differential equation with constant coefficients, such as
\begin{align}
\label{20240714 01}
\sum_{\alpha} c_\alpha D^\alpha u(x) = f(x) \quad x \in \fR^d,
\end{align}
is solvable for any distribution $f$ on $\fR^d$. 
Here, $\alpha$ is a multi-index, $c_\alpha$ is a complex number, and the summation is finite.
 For more details, see \cite{hor 1990, Komech 1994} and references therein.

We provide additional details with some heuristics from the aforementioned references. Formally, by taking the Fourier transform of both sides of \eqref{20240714 01}, we get
\begin{align}
										\label{20240714 02}
\psi(\xi) \cF[u](\xi) = \cF[f](\xi) \quad  x \in \fR^d,
\end{align}
where the symbol $\psi(\xi)$ is given by $\sum_{\alpha} c_\alpha \mathrm{i}^{|\alpha|} \xi^\alpha$.

Due to the Fourier inversion theorem, it suffices to solve \eqref{20240714 02} instead of \eqref{20240714 01}.
Since $\psi(\xi)$ is a polynomial, both sides of \eqref{20240714 02} can be understood as  weak functionals on a class of test functions according to the multiplication action, \textit{i.e.}
\begin{align}
								\label{20240715 10}
\langle \psi  \cF[u] , \varphi  \rangle 
:=\langle  \cF[u] , \psi  \varphi  \rangle
= \langle \cF[f], \varphi \rangle,
\end{align}
where $\varphi$ is a test function and $\langle \cdot, \cdot \rangle$ denotes a duality paring.

It is important to note that \eqref{20240714 02} loses the uniqueness property of a solution even as linear functionals unless the symbol $\psi$ satisfies additional conditions such as ellipticity and polynomial growth.
It is naturally expected since if both the symbol $\psi$ and data $\cF[f]$ are zero on a domain, then any value of $\cF[u]$ on this domain is enough to be a solution.
The breakdown of the uniqueness also can be comprehended through the following simple example. 
For any real number $c$, consider two functions $u_1(x)=e^{cx}$ and $u_2(x)=e^{-cx}$.
Then both $u_1$ and $u_2$ become solutions to
the equation 
\begin{align*}
u''(x)-c^2u(x)=0 \quad x \in \fR.
\end{align*}

At a glimpse, it also looks impossible that there exists a weak solution to \eqref{20240714 02} since 
if the symbol $\psi$ is zero but data $\cF[f]$ is non-zero in a domain, then any value of $\cF[u]$ in this domain cannot satisfy the equation. However, if $\psi$ is not identically zero, then the Lebesgue measure of the set of all zeros of $\psi$, $\psi^{-1}(0):= \{\xi \in \fR^d: \psi(\xi)=0\}$, becomes $0$ since $\psi$ is a non-zero polynomial. Thus, \eqref{20240714 02} is solvable in a weak sense and numerous methods had been developed to find solutions to \eqref{20240714 02} in a weak sense, thanks to the contributions of Ehrenpreis, Malgrange, Bernshtejn, Gel'fand, H\"ormander, Lojasiewicz, and others.
This problem is called ``the division problem" and was successfully solved a long time ago even though the uniqueness fails to hold.

This issue seems to extend to pseudo-differential operators by considering various symbols rather than just polynomials.
In other words, one can consider the equation
\begin{align}
										\label{20240715 01}
\cF^{-1}\left[\psi(x,\xi) \cF[u](\xi)\right] (x)  = f(x) \quad  x \in \fR^d,
\end{align}
 which is derived from \eqref{20240714 02} and the inverse Fourier transform.
Perturbing these types of operators with respect to $x$ without a very strong assumption is nearly impossible.
Therefore, we specifically consider the special case where the symbol $\psi$ is independent of $x$, since we aim to find a weaker condition on $\psi$ as possible.

Nonetheless, the equation 
\begin{align}
										\label{20240715 02}
\psi(-\mathrm{i}\nabla)u(x):=\cF^{-1}\left[\psi(\xi) \cF[u](\xi)\right] (x)  = f(x) \quad  x \in \fR^d,
\end{align}
is still impossible to solve even in a weak sense unless there exists a strong mathematical assumption on $\psi$.
Recall that all these classical theories were feasible since our symbol $\psi$ is a polynomial, \textit{i.e.} the main operator in \eqref{20240714 01} is a merely (partial) differential operator with constant coefficients. 
In particular, the multiplication action $\varphi \to \psi \varphi$ in \eqref{20240715 10} is well-defined since $\psi$ is a polynomial.
Additionally, the uniqueness of a solution does not even hold for polynomial symbols. 
Therefore, it is impossible to find a unique weak solution to \eqref{20240715 02} with a non-polynomial symbol unless strong conditions such as smoothness, polynomial growth, and uniform ellipticity appear (\textit{cf}. \cite[Chapters 12 and 13]{Krylov 2008}).
In other words, \eqref{20240715 02} is generally ill-posed.

However, if we reconsider the equation in the form of \eqref{20240714 02} by ignoring the inverse Fourier transform in \eqref{20240715 02}, it might seem solvable without smoothness on symbols if the set of all zeros of $\psi$ is a null set. 
Nonetheless, uniqueness is still unattainable without additional strong conditions. 
This vague idea might be developed into a rigorous mathematical theory, but to the best of our knowledge, no suitable references address this with non-regular symbols.

As a broader extension, one can also consider an evolutionary generalization of \eqref{20240715 02} with a family of pseudo differential operators $\psi(t,-\mathrm{i}\nabla)$ as follows:
\begin{align}
										\label{20240715 10}
u_t(t,x)=  \psi(t,-\mathrm{i}\nabla)u(t,x)+ f(t,x) \quad  (t,x) \in (0,T) \times \fR^d,
\end{align}
where $T \in (0,\infty]$,
\begin{align}
												\label{20240715 20}
\psi(t,-\mathrm{i}\nabla)u(t,x):=\cF_\xi^{-1}\left[\psi(t,\xi) \cF[u(t,\cdot)](\xi)\right] (x)
\end{align}
and both $d$-dimensional Fourier and inverse Fourier transforms $\cF$ and  $\cF_\xi^{-1}$ are taken only with respect to the space variable.

It might seem that our evolutionary equations are more complicate than stationary equations such as \eqref{20240715 02}.
However, if our symbol $\psi(t,\xi)$ is merely given by a polynomial only with respect to $\xi$ (uniformly for all $t$), then the equation \eqref{20240715 10} can be viewed as a special case of \eqref{20240715 02} (and thus a simple particular case of \eqref{20240715 01}) on $\fR^{d+1}$ by extending the symbol for all $(t,\xi) \in (-\infty, \infty) \times \fR^d =\fR^{d+1}$.

Consequently, there exists a distribution solution $u$ to \eqref{20240715 10} by applying classical theories on the $(d+1)$-dimensional Euclidean space if $\psi(t,\xi)= \sum_{\alpha} c^\alpha \mathrm{i}^{|\alpha|}\xi^\alpha$ and $f$ is a distribution on $\fR^{d+1}$, where the summation is finite. 

Moreover, we can even address SPDEs driven by space-time white noise, since this random noise can be interpreted as inhomogeneous data $f$ given by tempered distributions. However, classical theories do not guarantee the uniqueness of a weak solution, even for this uniformly well-behaved polynomial symbol. Additionally, there is no comprehensive theory for \eqref{20240715 01} if $\psi$ lacks smoothness.

In conclusion, it might be impossible to show the existence and uniqueness of a weak solution to \eqref{20240715 10} if the symbol $\psi(t,\xi)$ varies with respect to the time and lacks regularity and ellipticity.

Surprisingly, we recently discovered a unique solution to \eqref{20240715 10} in a weak sense, despite the symbols being non-elliptic and non-smooth (\cite{Choi Kim 2024}). 
Notably, the symbols can take on any values in the complex number system, which causes certain functions to display exponential growth. 
Specifically, the real parts of these symbols can simultaneously include zeros, positive, and negative values. 
This violates the conditions for any ellipticity and allows the kernels of \eqref{20240715 10} to exhibit exponential growth at infinity.
We succinctly describe these properties by saying that our symbols (or pseudo-differential operators) are {\bf sign-changing}.

The key idea behind this result involves using a special class of test functions that contain an approximation of the identity and extending the domain of the Fourier and inverse Fourier transforms to a larger class than that of all tempered distributions on $\fR^d$. 
Interestingly, this new class is neither the class of all distributions in $\fR^d$ nor does it have an inclusion relation with it. However, the new class and the class of all distributions are homeomorphic due to the newly extended Fourier transform. More details will be provided in the following section.

This new finding has prompted us to pursue further research. Specifically, our new primary objective is to investigate whether these novel concepts can be applied to SPDEs involving pseudo-differential operators driven by space-time white noise. We are writing this paper because we realized that all these ideas can indeed be extended to such SPDEs. A more precise mathematical formulation will be presented in the next section.

We conclude this introduction with a few thoughts. We are confident that our theory offers numerous advantages, even from an application standpoint, as many generators of two-parameter semi-groups can be expressed in the form of the operator in \eqref{20240715 20}. 
Specifically, many generators of Markov processes can be represented by these operators.

However, beyond practical applications, we view our research as a mathematical challenge in its own right. Our goal is to solve as general equations as possible, driven by the following question:
$$
\text{\emph{What is the most general operator that maintains well-posedness for linear SPDEs in a weak sense?}}
$$

Although we cannot claim that our results provide the ultimate answer to this question, we believe they are quite remarkable. 
Our findings show that all SPDEs involving pseudo-differential operators are well-posed in a certain weak sense, even though the symbols are irregular, exhibit exponential growth at infinity, lack any elliptic condition, and even are stochastic.

\vspace{2mm}
At last, we specify the notation used in the article.

\begin{itemize}
\item 
Let $\fN$, $\fZ$, $\fR$, $\fC$ denote the natural number system, the integer number system, the real number system, and the complex number system, respectively. For $d\in \fN$, $\fR^d$ denotes the $d$-dimensional Euclidean space. 
\item 
For $i=1,...,d$, a multi-index $\alpha=(\alpha_{1},...,\alpha_{d})$ with
$\alpha_{i}\in\{0,1,2,...\}$, and a function $g$, we set
$$
\frac{\partial g}{\partial x^{i}}=D_{x^i}g,\quad
D^{\alpha}g=D_{x^1}^{\alpha_{1}}\cdot...\cdot D^{\alpha_{d}}_{x^d}g,\quad |\alpha|:=\sum_{i=1}^d\alpha_i.
$$

\item 
We use $C_c^{\infty}(\fR^d)$ or $\cD(\fR^d)$ to denote the space of all infinitely differentiable (complex-valued) functions with compact supports.
$\cS(\fR^d)$ represents the Schwartz space in $\fR^d$ and the topology on $\cS(\fR^d)$ is generated by the Schwartz semi-norms $\sup_{x \in \fR^d} |x^\alpha (D^\beta f)(x)|$ for all multi-indexes $\alpha$ and $\beta$.  
$\cS'(\fR^d)$ is used to denote the dual space of $\cS(\fR^d)$, \textit{i.e.} $\cS'(\fR^d)$ is the space of all tempered distributions on $\fR^d$. 
Additionally, we assume that $\cS'(\fR^d)$ is a topological space equipped with the weak*-topology if there is no special remark about the topology on $\cS'(\fR^d)$.

\item 
Let $(X,\mathcal{M},\mu)$ be a measure space.
Assume that $f(t)$ and $g(t)$ are complex-valued $\cM$-measurable  functions on $X$.
We write
\begin{align}
									\label{20240806 50}
f(x)=g(x) \quad (a.e.)~ x \in X
\end{align}
or say that for almost every $x \in X$, $f$ and $g$ are equal 
iff there exists a measurable subset $\cT \subset X$ such that the measure $\mu(X \setminus \cT)$ is zero and $f(x)=g(x)$ for all $t \in \cT$.
Moreover, we say that a function $f(t)$ is defined $(a.e.)$ on a $X$ if there exists a measurable subset $\cT \subset X$ such that the measure $\mu(X \setminus \cT)=0$ and $f(t)$ is defined for all $t \in \cT$. 
Especially, if the measure $\mu$ is a probability measure, \textit{i.e.} $\mu(X)=1$, then we use $(a.s.)$ instead of $(a.e.)$. 
Furthermore, we also say that $f(x)=g(x)$ with probability one instead of \eqref{20240806 50} if $\mu$ is a probability measure.

\item
For $R>0$,
$$
B_R(x):=\{y\in\fR^d:|x-y|< R\}.
$$

\item 
For a measurable function $f$ on $\fR^d$, we denote the $d$-dimensional Fourier transform of $f$ by 
\[
\cF[f](\xi) := \frac{1}{(2\pi)^{d/2}}\int_{\fR^{d}} \mathrm{e}^{-\mathrm{i}\xi \cdot x} f(x) \mathrm{d}x
\]
and the $d$-dimensional inverse Fourier transform of $f$ by 
\[
\cF^{-1}[f](x) := \frac{1}{(2\pi)^{d/2}}\int_{\fR^{d}} \mathrm{e}^{ \mathrm{i} x \cdot \xi} f(\xi) \mathrm{d}\xi.
\]
We refer to $\xi$ as the frequency of $f$.

Moreover, for a function $f(t,x)$ defined on $(0,T)\times \fR^d$, we use the notation
\begin{align*}
\cF_x[f](\xi)
:=\cF[f(t,x)](\xi)
:=\cF[f(t,\cdot)](\xi)
=\frac{1}{(2\pi)^{d/2}}\int_{\fR^{d}} \mathrm{e}^{ -\mathrm{i} \xi \cdot x} f(t,x) \mathrm{d}x
\end{align*}
and it is called the Fourier transform of $f$ with respect to the space variable.
Additionally, we say that $\xi$ is a frequency with respect to the space variable.
In particular, we sometimes use the terminology ``frequency function" instead of the Fourier transform.

On the other hand, for the inverse Fourier transform of $f(t,\xi)$ with respect to the space variable, we use the notation
\begin{align*}
\cF_\xi^{-1}[f](x)
:=\cF^{-1}[f(t,\xi)](x)
:=\cF^{-1}[f(t,\cdot)](x)
=\frac{1}{(2\pi)^{d/2}}\int_{\fR^{d}} e^{  \mathrm{i} x \cdot \xi} f(t,\xi) \mathrm{d}\xi.
\end{align*}
More general Fourier and inverse Fourier transforms acting on linear functionals are discussed in Section \ref{main section}.
To simplify notation, we typically omit the subscripts indicated by $x$ and $\xi$.

\item 
We write $\alpha \lesssim \beta$ if there is a positive constant $N$ such that $ \alpha \leq N \beta$.

\item 
For $z\in\fC$, $\Re[z]$ denotes the real part of $z$, $\Im[z]$ is the imaginary part of $z$, and $\bar{z}$ is the complex conjugate of $z$.
\end{itemize}

\mysection{Settings and function classes}
											\label{main section}

We fix a finite stopping time $\tau$ and $d \in \fN$ throughout the paper. 
Here, $\tau$ and $d$ denote the (random) terminal time of the equation and the dimension of the space-variable, respectively.

Let $(\Omega,\rF,P)$ be a complete probability space,
 $\rG$ be a sub-$\sigma$-field of $\rF$, and $\rH$ be a sub-$\sigma$-algebra of $\rF \times \cB([0,\infty))$, where
 where $\cB([0,\infty))$ denote the Borel sets on $[0,\infty)$ (generated by the Euclidean norms).
 $\rG$ and $\rH$ provide measurability information for an initial data $u_0$ and a deterministic inhomogeneous data $f$, respectively.
Note that $\rG$ and $\rH$ can be arbitrary sub-$\sigma$-algebras of $\rF$ and $\rF \times \cB([0,\infty))$, respectively.
For notational convenience, we fix 
$\rG$ and $\rF$ throughout the paper, but it is important to note that all results can be applied to different sub-$\sigma$-algebras due to their generality.

Additionally, let $\{\rF_{t},t\geq0\}$ be an increasing filtration of
$\sigma$-fields $\rF_{t}\subset\rF$, each of which contains all
$(\rF,P)$-null sets. By  $\cP$ we denote the predictable
$\sigma$-algebra generated by $\{\rF_{t},t\geq0\}$ and  we assume that
on $\Omega$ there exist  independent one-dimensional Wiener
processes (Brownian motions) $B^{1}_{t},B^{2}_{t},...$, each of which is a Wiener
process relative to $\{\rF_{t},t\geq0\}$.

Now we consider a complex-valued $\cP \times \cB(\fR^d)$-measurable function  $\psi(t,\xi)$  defined on $\Omega \times [0,\infty) \times \fR^d$ and denote
\begin{align}
									\label{psi defn}
\psi(t,-\mathrm{i}\nabla)u(t,x):=\cF^{-1}\left[\psi(t,\cdot)\cF[u](t,\cdot)\right](x)
\end{align}
in the whole paper, where $\cB(\fR^d)$ denotes the Borel sets on $\fR^d$ generated by the standard Euclidean norm.
The operator $\psi(t,-\mathrm{i}\nabla)$ is called a (time-measurable) pseudo-differential operator
and the function $\psi(t,\xi)$ is called {\bf the symbol} of the operator $\psi(t,-\mathrm{i}\nabla)$.

We will discuss the domain of the operator $f \mapsto \psi(t,-\mathrm{i}\nabla)f$ and properties of the symbol $\psi(t,\xi)$ to ensure the operator is well-defined later. 
In simple terms, the operator can be defined based on the identity from Plancherel's theorem, which is provided in Definition \ref{defn psi operator} below.
It is important to note that our symbol can be random, sign-changing, and does not require any regularity condition.

In this paper, we study the existence and uniqueness of a weak solution $u$ to the SPDEs with the pseudo-differential operator $\psi(t,-\mathrm{i} \nabla)$ driven by space-time white noise as follows:
\begin{align}
								\notag
&\mathrm{d}u=\left(\psi(t,-\mathrm{i}\nabla)u(t,x) + f(t,x)\right) \mathrm{d}t + h(t,x) \mathrm{d}\fB_t,\quad &(t,x) \in (0,\tau] \times \fR^d,\\
&u(0,x)=u_0(x),\quad & x\in \fR^d,
								\label{main eqn}
\end{align}
where $\mathrm{d}\fB_t$ denotes a space-time white noise, \textit{i.e.} $\fB_t$ is a cylindrical Wiener process on $L_2(\fR^d)$.

The equation conceals the sample point $\omega \in \Omega$ as it typically does.
Here, our solutions are weak with respect to an analytic point of view, \textit{i.e.} a solution $u$ satisfies \eqref{main eqn} 
as a linear functional (complex-valued linear function) defined on a class of test functions.  
This class of test functions is selected as a dense subspace of the Schwartz class, as defined in Definition \ref{defn fourier distribution}, rather than the usual $C_c^\infty(\fR^d)$ or the Schwartz class itself, to manage the positive part of the symbol $\psi$. The precise definition of our solution is provided in Definition \ref{space weak solution}. 
Furthermore, it is important to note that our solutions are stochastically strong in the sense that we solve \eqref{main eqn} after fixing a probability space where noises $ \mathrm{d}\fB_t$ can be constructed. 

It is well-known that the space-white noise $\mathrm{d}\fB_t$ has a representation so that 
\begin{align}
										\label{20240525 01}
\mathrm{d}\fB_t = \eta_k \mathrm{d}B_t^k,
\end{align}
where $\{\eta_k : k \in \fN\}$ is an orthonomal basis of $L_2(\fR^d)$ and $\{B^k_t: k \in \fN\}$ is a sequence of independent one-dimensional  (real-valued) Brownian motions.
In particular, we fix an orthonomal basis  $\{\eta_k : k \in \fN\}$  so that $\eta_k \in \cS(\fR^d)$ for all $k \in \fN$;
This type of orthonomal basis can be easily constructed based on the Hermite polynomials. 
We consider the space-time white noise constructed from \eqref{20240525 01} throughout the paper. 
Due to this construction, at least formally, the term $h(t,x) \mathrm{d}\fB_t$ is equal to
$h(t,x) \eta_k(x) \mathrm{d}B_t^k$, where Einstein's summation convention is used, \textit{i.e.} the summation notation is hidden for the repeated index, which implies
\begin{align*}
h(t,x) \eta_k(x) \mathrm{d}B_t^k = h(t,x)\sum_{k=1}^\infty \eta_k(x) \mathrm{d}B_t^k.
\end{align*}
This connection naturally raises the necessity of handling sequence-valued functions and stochastic processes.
Specifically, $l_2$-valued stochastic processes are importantly treated to classify stochastic inhomogeneous data, where
$l_2$ denotes the class of all (complex-valued) sequences $a=\left(a^1,a^2,a^3,\ldots  \right)$ such that
\begin{align*}
|a|_{l_2} := \left(\sum_{k=1}^\infty |a^k|^2\right)^{1/2} < \infty.
\end{align*}

Recall that our main goal is to find the weakest possible conditions on the symbol $\psi$ and data $u_0$, $f$, and $h$ to ensure the existence of a unique weak solution to \eqref{main eqn}. 
Especially, our goal is to eliminate all regularity conditions on the symbol $\psi$ and the data $u_0$, $f$, and $h$.
More specifically, we seek appropriate local integrability conditions on the symbol $\psi$ and data $u_0$, $f$, and $h$ with respect to the spatial frequencies to make \eqref{main eqn} well-posed. 
All functions in the paper are complex-valued if there is no special mention about ranges of functions.

We present our main assumptions on the symbol $\psi(t,\xi)$ in the next section to ensure our weak well-posedness. 
Initially, we discuss defining our main operator $\psi(t,\xi)$ without imposing strong assumptions on $\psi$.
The operator $\psi(t,-\mathrm{i}\nabla)$ is defined using the Fourier and inverse Fourier transforms, as shown in \eqref{psi defn}.
Thus, it is necessary to consider the broadest possible domain for these transforms if we aim to relax the conditions on the symbol $\psi$ ensuring that the operator $\psi(t,-\mathrm{i}\nabla)$ is well-defined.

It is well-known that the Fourier and inverse Fourier transforms are automorphisms on the class of all tempered distributions, where a tempered distribution is a (complex-valued) continuous linear functional defined on the 
$d$-dimensional Schwartz class $\cS(\fR^d)$.
 Additionally, the domain of the Fourier transform can be extended to include all distributions, aided by the Paley-Wiener theorem (\textit{cf.} \cite[Proposition 4.1]{Komech 1994}). This extension relies on analytic continuations and is therefore only suitable for smooth symbols $\psi$. 
 Consequently, this extension is not useful for defining the operator $\psi(t,-\mathrm{i}\nabla)$ since it requires the smoothness of the symbol $\psi$, which contradicts our main objective.

Fortunately, we discovered an alternative method to extend the domain of the Fourier and inverse Fourier transforms to include all distributions, as suggested by the authors in \cite{Choi Kim 2024}. We provide more precise definitions of the extended Fourier and inverse Fourier transforms below.

\begin{defn}[Fourier transforms of distributions]
								\label{defn fourier distribution}
We use $\cF^{-1}\cD(\fR^d)$ to denote the subclass of the Schwartz class whose Fourier transform is in $\cD(\fR^d)$, \textit{i.e.}
\begin{align*}
\cF^{-1}\cD(\fR^d) 
:= \{ u \in \cS(\fR^d) : \cF[u] \in \cD(\fR^d)\},
\end{align*}
where $\cS(\fR^d)$ represents the (complex-valued) Schwartz class on $\fR^d$ and $\cD(\fR^d)$ is the class of all complex-valued infinitely differentiable functions with compact supports defined on $\fR^d$.
The Schwartz class $\cS(\fR^d)$ is equipped with the strong topology generated by Schwartz's semi-norms. Since $\cF^{-1}\cD(\fR^d)$ is a subset of $\cS(\fR^d)$, it inherits the subspace topology. We can then consider the dual space of $\cF^{-1}\cD(\fR^d)$, denoted as $\cF^{-1}\cD'(\fR^d)$, \textit{i.e.} $u \in \cF^{-1}\cD' (\fR^d) $
if and only if $u$ is a continuous linear functional defined on $\cF^{-1}\cD(\fR^d)$.
For $u \in \cF^{-1}\cD'(\fR^d)$ and $\varphi \in \cF^{-1}\cD(\fR^d)$, we write
\begin{align*}
\langle u, \varphi \rangle := u(\varphi).
\end{align*}
In other words, $\langle u, \varphi \rangle$ denotes the image of $\varphi$ under $u$.
We adopt the weak*-topology for $\cF^{-1}\cD'(\fR^d)$ and consider the Borel sets generated by this topology.
We call $\cF^{-1}\cD'(\fR^d)$ the space of all \textbf{the Fourier transforms of distributions}. The rationale behind this naming and the notation will be explained in Definition \ref{fourier defn 2} below.

Similarly, we define the space of $l_2$-valued continuous linear functions on $\cF^{-1}\cD(\fR^d)$.
This space is denoted by $\cF^{-1}\cD'(\fR^d; l_2)$ and consists of all continuous linear functions from $\cF^{-1}\cD(\fR^d)$ to $l_2$. For $u \in \cF^{-1}\cD'(\fR^d; l_2)$ and $\varphi \in \cF^{-1}\cD(\fR^d)$, we denote the $k$-th component of the sequence $u(\varphi)$ by $\langle u^k, \varphi \rangle$ for all $k \in \fN$.
\end{defn}

We recall the definition of distributions (generalized functions). The space $\cD'(\fR^d)$ denotes all linear continuous functionals on $\cD(\fR^d)$. Similarly, $\cD'(\fR^d; l_2)$ denotes the space of all $l_2$-valued continuous linear functions on $\cD(\fR^d)$. We consider the weak*-topologies on $\cD'(\fR^d)$ and $\cD'(\fR^d; l_2)$. Consequently, the Borel sets generated by these topologies can be used to discuss the measurability of functions whose ranges lie in $\cD'(\fR^d)$ or $\cD'(\fR^d; l_2)$.

It is easy to observe that both $\cF^{-1}\cD'(\fR^d)$ and $\cD'(\fR^d)$ are larger than the class of all tempered distributions $\cS'(\fR^d)$. Additionally, it is obvious that $\cF^{-1}\cD'(\fR^d) \not\subset \cD'(\fR^d)$ and $ \cD'(\fR^d) \not\subset \cF^{-1}\cD'(\fR^d)$. However, there is an interesting relation that $\cF^{-1}\cD'(\fR^d) \cap \cD'(\fR^d)=\cS'(\fR^d)$.
Another connection between the spaces $\cF^{-1}\cD'(\fR^d)$ and $\cD'(\fR^d)$ can be established through the Fourier and inverse Fourier transforms. Specifically, the Fourier and inverse Fourier transforms can be defined on both $\cF^{-1}\cD'(\fR^d)$ and $\cD'(\fR^d)$ in such a way that they can be interchanged. Here are definitions of the Fourier and inverse Fourier transforms for elements in $\cF^{-1}\cD'(\fR^d)$ and $\cD'(\fR^d)$, which follows the conventional method to define these transforms in a weak sense.

\begin{defn}[Fourier and inverse Fourier transforms of elements in $\cF^{-1}\cD'(\fR^d)$ and $\cF^{-1}\cD'(\fR^d;l_2)$]
							\label{fourier defn 1}
 For any $u \in \cF^{-1}\cD'(\fR^d)$, the Fourier and inverse Fourier transforms of $u$ are defined in the following canonical way:
\begin{align*}
\langle \cF[u], \varphi \rangle
:=\langle u, \cF^{-1}[\varphi] \rangle
\end{align*}
and
\begin{align*}
\langle \cF^{-1}[u], \varphi \rangle
:=\langle u, \cF[\varphi] \rangle
\end{align*}
for all $\varphi \in \cD(\fR^d)$.
Thus for any $u \in \cF^{-1}\cD'(\fR^d) $, we have
$\cF[u] \in \cD'(\fR^d)$ and $ \cF^{-1}[u] \in \cD'(\fR^d)$
since it is evident that both $\cF^{-1}[\varphi]$ and $\cF[\varphi] $ are in $\cF^{-1}\cD(\fR^d)$ for all $\varphi \in \cD(\fR^d)$ due to well-known properties of the Fourier and inverse Fourier transforms of a Schwartz function.
The continuity of $\cF[u]$ and $\cF^{-1}[u]$ on $\cD(\fR^d)$ to be distributions is also easily obtained from properties of the Fourier transform, inverse Fourier transform, and the Schwartz functions.
Similarly, for $u \in \cF^{-1}\cD'(\fR^d;l_2)$, we define
\begin{align*}
\langle \cF[u]^k, \varphi \rangle
:=\langle u^k, \cF^{-1}[\varphi] \rangle
\end{align*}
and
\begin{align*}
\langle \cF^{-1}[u]^k, \varphi \rangle
:=\langle u^k, \cF[\varphi] \rangle
\end{align*}
for all $\varphi \in \cD(\fR^d)$ and $k \in \fN$. 
Simply, by $\cF[u]$ and $\cF^{-1}[u]$, we denote the $l_2$-valued continuous linear functions on $\cD(\fR^d)$ so that for any $k \in \fN$ and $\varphi \in \cD(\fR^d)$,
the $k$-th terms of $\cF[u](\varphi)$ and $\cF^{-1}[u](\varphi)$ are $\langle \cF[u]^k, \varphi \rangle$ and $\langle \cF^{-1}[u]^k, \varphi \rangle$, respectively.
\end{defn}

\begin{defn}[Fourier and inverse Fourier transforms of elements in $\cD' (\fR^d)$ and $\cD' (\fR^d;l_2)$]
								\label{fourier defn 2}
Let $u$ be a distribution on $\fR^d$, \textit{i.e.} $u \in \cD'(\fR^d)$.
The Fourier and inverse Fourier transforms of $u$ are defined as follows:
\begin{align*}
\langle \cF[u], \varphi \rangle 
:=\langle u,  \cF^{-1}[\varphi] \rangle \quad \forall \varphi \in \cF^{-1}\cD(\fR^d) 
\end{align*}
and
\begin{align*}
\langle \cF^{-1}[u], \varphi \rangle 
:=\langle u,  \cF[\varphi] \rangle  \quad \forall \varphi \in \cF^{-1}\cD(\fR^d).
\end{align*}
Then both $\cF[u]$ and $\cF^{-1}[u]$ belong to $\cF^{-1}\cD'(\fR^d)$.
Similarly, for $u \in \cD'(\fR^d;l_2)$, we define
\begin{align*}
\langle \cF[u]^k, \varphi \rangle
:=\langle u^k, \cF^{-1}[\varphi] \rangle
\end{align*}
and
\begin{align*}
\langle \cF^{-1}[u]^k, \varphi \rangle
:=\langle u^k, \cF[\varphi] \rangle
\end{align*}
for all $ \varphi \in \cF^{-1}\cD(\fR^d)$ and $k \in \fN$. 
Here, $\cF[u]$ and $\cF^{-1}[u]$ denote the $l_2$-valued continuous linear functions on $\cF^{-1}\cD(\fR^d)$ so that for any $k \in \fN$ and $\varphi \in \cF^{-1}\cD(\fR^d)$,
the $k$-th terms of $\cF[u](\varphi)$ and $\cF^{-1}[u](\varphi)$ are $\langle \cF[u]^k, \varphi \rangle$ and $\langle \cF^{-1}[u]^k, \varphi \rangle$, respectively.
\end{defn}

It can be easily demonstrated that both $\cF$ and $\cF^{-1}$ are homeomorphisms between $\cD'(\fR^d)$ and $\cF^{-1}\cD'(\fR^d)$. We omit a detailed proof as it follows the classical proof that the Fourier transform is an automorphism on the class of all tempered distributions (\textit{cf.} \cite[Theorem 2.4]{Choi Kim 2024} and references therein).
 
\begin{thm}
									\label{homeo thm}
Both Fourier and inverse Fourier transforms are homeomorphisms from $\cD'(\fR^d)$ onto $\cF^{-1}\cD'(\fR^d)$ with respect to the weak*-topologies. 
The inverse mappings of the homeomorphisms are given by the inverse Fourier transform and the Fourier transform, respectively.
\end{thm}

It is well-known that any locally integrable function on $\fR^d$  is a distribution on $\fR^d$.
Specifically, a locally integrable function with exponential growth at infinity is included in the class of all distributions but not in the class of tempered distributions. Controlling the exponential growth with respect to the frequencies of solutions and data is a major task in achieving our results since this naturally occurs when allowing symbols $\psi$ to be sign-changing. This necessity to handle exponential growth is one of the main reasons we need to extend the Fourier transform to include all distributions. Here is a more detailed mathematical explanation of distributions and their realizations as locally integrable functions.

\begin{defn}[Locally integrable functions and realizations of distributions]
										\label{local inte class}
Let $\cU$ be an open subset of $\fR^d$.
We say that  $v$ is locally integrable on $\cU$ if and only if
\begin{align*}
\int_{K}|v(\xi)| \mathrm{d} \xi < \infty \quad \text{for all compact subsets $K$ of $\cU$}.
\end{align*}
In particular, it is easy to show that  $v$ is locally integrable on $\fR^d$ if and only if
\begin{align*}
\int_{B_R}|v(\xi)| \mathrm{d} \xi < \infty \quad \forall R \in (0,\infty),
\end{align*}
where $B_R$ denotes the Euclidean ball whose center is zero and radius is $R$, \textit{i.e.}
\begin{align*}
B_R:= \{y \in \fR^d: |y| < R\}.
\end{align*}
We write $ v \in L_{1,\ell oc}(\cU)$ if $v$ is locally integrable on $\cU$.
Recall that any locally integrable function $v$ on $\fR^d$ is a distribution on $\fR^d$ due to the identification $v$ with the mapping
\begin{align}
										\label{integral identification}
\varphi \in \cD(\fR^d) \mapsto \int_{\fR^d} v(x) \overline{\varphi(x)} \mathrm{d}x,
\end{align}
where $\overline{\varphi(x)}$ denotes the complex conjugate of $\varphi(x)$.
In this sense, $L_{1,loc}(\fR^d) \subset \cD'(\fR^d)$ and there is the subspace topology on $L_{1,loc}(\fR^d)$.
Additionally, we say that $u \in \cD'(\fR^d)$ has a realization $v$ on $\cU$ if there exists a locally integrable function $v$ on $\cU$ so that 
\begin{align*}
\langle u , \varphi \rangle = \int_{\cU} v(x) \overline{\varphi(x)} \mathrm{d}x \quad \forall \varphi \in \cD(\cU),
\end{align*}
where $\cD \left(\cU\right)$ denotes the set of all infinitely differentiable functions on $\fR^d$ with compact supports in $\cU$.
We simply say that $u \in \cD'(\fR^d)$ is locally integrable on $\cU$ if it has a realization $v$ on $\cU$.

Similarly, we use the notation  $L_{1,\ell oc}(\cU ; l_2)$ to denote the space of all $l_2$-valued locally integrable functions $v$ on $\cU$ such that 
\begin{align*}
\int_{K}|v(\xi)|_{l_2} \mathrm{d} \xi < \infty \quad \text{for all compact subsets $K$ of $\cU$ }.
\end{align*}
In particular, $L_{1,\ell oc}(\fR^d ; l_2)$ becomes a subspace of $\cD'(\fR^d ; l_2)$ by considering the canonical identification
\begin{align*}
\varphi \in \cD(\fR^d) \mapsto   \left( \int_{\fR^d} v^1(x) \overline{\varphi(x)} \mathrm{d}x, \int_{\fR^d} v^2(x) \overline{\varphi(x)} \mathrm{d}x, \cdots  \right).
\end{align*}
We say that $u=(u^1,u^2,\cdots) \in \cD'(\fR^d;l_2)$ has a realization $v=(v^1,v^2, \cdots)$ on $\cU$ if
$v \in L_{1,\ell oc}(\cU ; l_2)$ and 
\begin{align*}
\langle u^k , \varphi \rangle = \int_{\cU} v^k(x) \overline{\varphi(x)} \mathrm{d}x \quad \forall \varphi \in \cD(\cU)~\text{and}~\forall k \in \fN.
\end{align*}
We simply say that $u \in \cD'(\fR^d;l_2)$ is ($l_2$-valued) locally integrable on $\cU$ if it has a realization $v$ on $\cU$.
\end{defn}

Here, we only considered the restrictions of $u \in \cD'(\fR^d)$ and $u \in \cD'(\fR^d;l_2)$ on $\cU$ rather than defining local distributions on $\cU$ to circumvent difficulties in combining or extending them.
This seems to be reasonable since our solutions and data are all defined on $\fR^d$ (in a generalized sense).

Furthermore, we temporarily equip $L_{1, \ell oc}(\fR^d)$ with the weak*-topology, which is derived from the subspace topology of $\mathcal{D}'(\fR^d)$, instead of the strong topology generated by the semi-norms $\|\cdot\|_{L_1(B_R)}$. This choice is because the coarse topology is more suitable for handling the Fourier transform and the inverse Fourier transform due to Theorem \ref{homeo thm}. 
We will revisit $L_{1, \ell oc}(\fR^d)$ with the strong topology when we examine $L_{1, \ell oc}(\fR^d)$-valued continuous functions later on.

\begin{defn}[The space of the inverse Fourier transforms of locally integrable functions]
Let $\cU$ be an open subset of $\fR^d$.
Due to Theorem \ref{homeo thm}, 
$$
\cF^{-1}L_{1,loc}(\cU) := \{ \cF^{-1}[u] : u \in \cD'(\fR^d) \cap L_{1,loc}(\cU)\}
$$
is a subspace of $\cF^{-1}\cD'(\fR^d)$. 
In other words, $\cF^{-1}L_{1,loc}(\cU)$ is the space of all the inverse Fourier transforms of distributions on $\fR^d$ which has a realization on $\cU$.
In a simpler way, we say that $u \in \cF^{-1}L_{1,loc}(\cU)$ has the Fourier transform which is locally integrable on $\cU$.
By applying Definition \ref{fourier defn 1} and the identification from \eqref{integral identification},
for any $v \in \cF^{-1}L_{1,loc}(\cU)$ and $\varphi \in \cD(\cU)$, we have
\begin{align*}
\langle v,  \cF^{-1}[\varphi] \rangle 
=\langle \cF[v], \varphi \rangle 
= \int_{\cU} \cF[v](\xi) \overline{\varphi(\xi)} \mathrm{d}\xi.
\end{align*}
Here, the obvious inclusion $\cD(\cU) \subset \cD(\fR^d)$ is used
to ensure $\cF^{-1}[\varphi]  \in \cS(\fR^d)$ by extending $\varphi$ trivially with zeros outside of $\cU$ (or simply by the definition of $\cD(\cU)$).
In particular, it is obvious that for any $v \in \cF^{-1}L_{1,loc}(\cU)$, $\cF[v]$ has a realization on $\cU$.
For this case, we say that $v$ has a {\bf realizable frequency function} on $\cU$. 

Similarly, we use the notation $\cF^{-1}L_{1,loc}(\cU;l_2)$ to denote the subspace of $\cF^{-1}\cD'(\fR^d;l_2)$ consisting of all $\cF^{-1}[u]$ with $u \in L_{1,loc}(\cU ; l_2) \cap \cD'(\fR^d;l_2)$.
In particular, for all $v=(v^1, v^2, \cdots ) \in \cF^{-1}L_{1,loc}(\cU;l_2)$, $\varphi \in \cD(\cU)$, and $k \in \fN$, we have
\begin{align*}
\langle v^k,  \cF^{-1}[\varphi] \rangle 
=\langle \cF[v^k], \varphi \rangle 
= \int_{\cU} \cF[v^k](\xi) \overline{\varphi(\xi)} \mathrm{d}\xi.
\end{align*}
\end{defn}

Based on these realizations of the elements in $\cF^{-1}\cD'(\fR^d)$, we are now ready to define the operator $\psi(t,-\mathrm{i} \nabla)$. To define the operator universally for all $t$, we slightly abuse the notation by considering $\psi(-\mathrm{i}\nabla)$ with a function $\psi(\xi)$ in $\fR^d$ as described below.
\begin{defn}
							\label{defn psi operator}
Let $g \in \cF^{-1}\cD'(\fR^d)$ so that  it has a realizable frequency function on an open set $\cU$ covering the support of $\psi$ and
$\xi \in \fR^d \mapsto \psi(\xi)\cF[g](\xi) \in \fC$ be a locally integrable function on $\fR^d$.
Here the function $\psi(\xi)\cF[g](\xi)$ is extended to be complex-valued over the entire $\fR^d$ by assigning it zero outside of $\cU$.
Then we define $\psi(-\mathrm{i}\nabla)g$ as an element in $\cF^{-1}\cD'(\fR^d)$ so that
\begin{align}
									\label{2024061901}
\langle \psi(-\mathrm{i} \nabla)g, \varphi \rangle
:=\langle  \cF^{-1}\left[\psi \cF[g] \right], \varphi \rangle
=\langle  \psi \cF[g], \cF[\varphi] \rangle
=\int_{\fR^d} \psi(\xi)\cF[g](\xi) \overline{\cF[\varphi](\xi)} \mathrm{d}\xi \quad \forall \varphi \in \cF^{-1}\cD(\fR^d).
\end{align}
It is clear that the open set $\cU$ is not uniquely determined according to the $g$.
However, it does not affect the value of the last right-hand side of \eqref{2024061901}, which implies that our operator $\psi(-\mathrm{i}\nabla)$ is well-defined.
In particular, if $f$  is a locally integrable function on $\fR^d$, then 
\begin{align*}
\langle \psi(-\mathrm{i} \nabla)(\cF^{-1}[f]), \varphi \rangle
=\int_{\fR^d} \psi(\xi) f(\xi) \overline{\cF[\varphi](\xi)} \mathrm{d}\xi \quad \forall \varphi \in \cF^{-1}\cD'(\fR^d).
\end{align*}
due to Theorem \ref{homeo thm}.
\end{defn}

Now it is time to define classes used to treat solutions and data.
We recognize that constructing these necessary classes is inherently complicated as they encompass not only random functions but also various types of linear functions such as $\cD'(\fR^d)$, $\cF^{-1} \cD'(\fR^d)$, $\cD'(\fR^d;l_2)$, and $\cF^{-1}\cD'(\fR^d;l_2)$.
Although defining all these classes may be tedious, it is a crucial foundational step to ensure the rigor of our theories.
Before specifically defining all spaces, recall that the weak*-topologies are given on $\mathcal{D}'(\fR^d)$, $\mathcal{F}^{-1} \mathcal{D}'(\fR^d)$, $\mathcal{D}'(\fR^d;l_2)$, and $\mathcal{F}^{-1} \mathcal{D}'(\fR^d;l_2)$. 
As a result, there are corresponding Borel sets on these spaces, which allows us to consider measurable functions that take values in $\mathcal{D}'(\fR^d)$, $\mathcal{F}^{-1} \mathcal{D}'(\fR^d)$, $\mathcal{D}'(\fR^d;l_2)$, or $\mathcal{F}^{-1} \mathcal{D}'(\fR^d;l_2)$.

\begin{defn}[$\cD'(\fR^d)$-valued random variable]
								\label{d valued rv}
\begin{enumerate}[(i)]
\item By $\bL_{0}\left( \Omega, \rG;\cD'(\fR^d)\right)$, we denote the class of all $\cD'(\fR^d)$-valued $\rG$-measurable random variables $u$ such that for each $\varphi \in \cD(\fR^d)$,
\begin{align*}
|\langle u, \varphi \rangle |< \infty \quad (a.s.).
\end{align*}

\item $\bL_{1}\left( \Omega, \rG;\cD'(\fR^d)\right)$ denotes the subspace of $\bL_{0}\left( \Omega, \rG;\cD'(\fR^d)\right)$ 
consisting of elements $u$ such that for each $\varphi \in \cD(\fR^d)$,
\begin{align*}
\bE\left[|\langle u, \varphi \rangle| \right] < \infty.
\end{align*}
\end{enumerate}
\end{defn}

\begin{defn}[$\cD'(\fR^d)$-valued processes]
									\label{d valued process}
Let $\rH|_{\opar 0,\tau \cbrk}$ be the restriction of $\rH$ on $\opar 0,\tau \cbrk$, which is defined by the sub-$\sigma$-algebra of $\rF \times \cB([0,\infty)) |_{\opar 0,\tau \cbrk}$ so that 
\begin{align*}
\rH|_{\opar 0,\tau \cbrk}
=\{ A \cap \opar 0,\tau \cbrk : A  \in \rH \},
\end{align*}
 where 
\begin{align*}
\opar 0,\tau \cbrk
:=\{ (\omega,t) \in \Omega \times (0,\infty) : 0<t \leq \tau(\omega) \}
\end{align*}
and
 $\rF \times \cB([0,\infty)) |_{\opar 0,\tau \cbrk}$ denotes the restriction of $\rF \times \cB([0,\infty))$ on $\opar 0,\tau \cbrk$, \textit{i.e.}
$$
\rF \times \cB([0,\infty)) |_{\opar 0,\tau \cbrk}
=\{ A \cap \opar 0,\tau \cbrk : A  \in \rF \times \cB([0,\infty))\}.
$$
\begin{enumerate}[(i)]
\item By $\bL_{0}\left( \opar 0,\tau \cbrk, \rH ;\cD'(\fR^d)\right)$, we denote the space of 
all $\cD'(\fR^d)$-valued functions $u$ on $\opar 0,\tau \cbrk$ such that
for each $\varphi \in \cD(\fR^d)$,
\begin{align*}
 |\langle u(t), \varphi \rangle|
 :=|\langle u(\omega,t), \varphi \rangle|
< \infty \quad (a.e.)~ (\omega,t) \in \opar 0,\tau \cbrk,
\end{align*}
and the mapping $(\omega,t) \in \opar 0, \tau \cbrk \mapsto \langle u(t), \varphi \rangle$
is $\rH|_{\opar 0,\tau \cbrk}$-measurable,
equivalently, the mapping $(\omega,t) \in \Omega \times [0,\infty) \mapsto  1_{\opar 0,\tau \cbrk }(t)\langle u(t), \varphi \rangle$
is $\rH$-measurable, where
\begin{align*}
1_{\opar 0,\tau \cbrk}(t)
:=1_{\opar 0,\tau \cbrk}(\omega,t)
:=
\begin{cases}
&1 \quad  \text{if $(\omega,t) \in \opar 0,\tau \cbrk$} \\
&0 \quad  \text{if $(\omega,t) \notin \opar 0,\tau \cbrk$}.
\end{cases}
\end{align*}
Here the trivial extension by zero is used to ensure the term $\langle u(t), \varphi \rangle$ is well-defined on $\Omega \times [0,\infty)$.
We also introduce an important subspace of $\bL_{0}\left( \opar 0,\tau \cbrk, \rH ;\cD'(\fR^d)\right)$ that possesses certain integrability properties.

\item We use the notation $\bL_{0,1}\left( \opar 0,\tau \cbrk, \rH ;\cD'(\fR^d)\right)$
to denote the subspace of   $\bL_{0}\left( \opar 0,\tau \cbrk, \rH ;\cD'(\fR^d)\right)$ consisting of all elements $u \in \bL_{0}\left( \opar 0,\tau \cbrk, \rH ;\cD'(\fR^d)\right)$  such that for each $\varphi \in \cD(\fR^d)$,
\begin{align*}
\int_0^{\tau} |\langle u(t), \varphi \rangle| \mathrm{d}t  
:= \int_0^{\tau(\omega)} |\langle u(\omega,t), \varphi \rangle| \mathrm{d}t
< \infty \quad (a.s.).
\end{align*}
\item $\bL_{1,1,loc}\left( \opar 0,\tau \cbrk, \rH ;\cD'(\fR^d)\right)$
denotes the subspace of $\bL_{0,1}\left( \opar 0,\tau \cbrk, \rH ;\cD'(\fR^d)\right)$ consisting of all elements $u \in \bL_{0,1}\left( \opar 0,\tau \cbrk, \rH ;\cD'(\fR^d)\right)$  such that for all $T \in (0,\infty)$ and $\varphi \in \cD(\fR^d)$,
\begin{align*}
\bE\left[\int_0^{\tau \wedge T} |\langle u(t), \varphi \rangle| \mathrm{d}t \right]
< \infty,
\end{align*}
where $a \wedge b := \min\{a,b\}$.
\item $\bL_{1,1}\left( \opar 0,\tau \cbrk, \rH ;\cD'(\fR^d)\right)$
denotes the subspace of $\bL_{1,1,loc}\left( \opar 0,\tau \cbrk, \rH ;\cD'(\fR^d)\right)$ consisting of all elements $u \in \bL_{1,1,loc}\left( \opar 0,\tau \cbrk, \rH ;\cD'(\fR^d)\right)$  such that for each $\varphi \in \cD(\fR^d)$,
\begin{align*}
\bE\left[\int_0^{\tau} |\langle u(t), \varphi \rangle| \mathrm{d}t \right]
< \infty.
\end{align*}
\end{enumerate}
\end{defn}
Similarly, we can readily define $\cD'(\fR^d;l_2)$-valued processes.
These processes are typically associated with integrands of Itô integrals, which are well-known in Krylov's analytic theories for SPDEs. 
Therefore, we restrict our attention to predictable processes.

\begin{defn}[$\cD'(\fR^d;l_2)$-valued processes]
									\label{d l2 valued process}
Recall that $\cP$ is the predictable $\sigma$-algebra.
\begin{enumerate}[(i)]
\item $\bL_{0}\left( \opar 0,\tau \cbrk, \cP;\cF^{-1}\cD'(\fR^d;l_2)\right)$ denotes the space of all $\cD'(\fR^d;l_2)$-valued function $u$ on $\opar 0,\tau \cbrk$ such that for each $\varphi \in \cD(\fR^d)$, 
\begin{align} 
										\label{20240322 01}
\left|\langle u(t), \varphi \rangle \right|_{l_2}
:= \left(\sum_{k=1}^\infty \left|\langle u^k(t), \varphi \rangle \right|^2\right)^{1/2} 
:= \left(\sum_{k=1}^\infty \left|\langle u^k(\omega,t), \varphi \rangle \right|^2\right)^{1/2} 
< \infty \quad (a.e.)~ (\omega,t) \in \opar 0,\tau \cbrk
\end{align}
and $(\omega,t) \mapsto 1_{\opar 0,\tau \cbrk }(t) \langle u(t,\cdot), \varphi \rangle$ is $\cP$-measurable.
\item $\bL_{0,2}\left( \opar 0,\tau \cbrk, \cP;\cD'(\fR^d;l_2)\right)$ denotes  
the subspace of $\bL_{0}\left( \opar 0,\tau \cbrk, \cP;\cD'(\fR^d;l_2)\right)$ consisting of all $u$ such that
for each $\varphi \in \cD(\fR^d)$,  
\begin{align*}
\int_0^{\tau} \left|\langle u(t,\cdot), \varphi \rangle\right|^2_{l_2} \mathrm{d}t  
:= \int_0^{\tau(\omega)} \sum_{k=1}^\infty \left|\langle u(\omega,t,\cdot)^k, \varphi \rangle\right|^2 \mathrm{d}t 
< \infty \quad (a.s.).
\end{align*}
\item $\bL_{1,2,loc}\left( \opar 0,\tau \cbrk, \cP;\cD'(\fR^d;l_2)\right)$ denotes the subspace of 
$\bL_{0,2}\left( \opar 0,\tau \cbrk, \cP;\cD'(\fR^d;l_2)\right)$ consisting of all $u$ such that
for all $T \in (0,\infty)$ and $\varphi \in \cD(\fR^d)$
\begin{align*}
\bE\left[\int_0^{\tau \wedge T} \left|\langle u(t,\cdot), \varphi \rangle\right|^2_{l_2} \mathrm{d}t\right]  < \infty.
\end{align*}
\item $\bL_{1,2}\left( \opar 0,\tau \cbrk, \cP;\cD'(\fR^d;l_2)\right)$ denotes the subspace of 
$\bL_{1,2,loc}\left( \opar 0,\tau \cbrk, \cP;\cD'(\fR^d;l_2)\right)$ consisting of all $u$ such that
\begin{align*}
\bE\left[\int_0^{\tau} \left|\langle u(t,\cdot), \varphi \rangle\right|^2_{l_2} \mathrm{d}t\right]  < \infty.
\end{align*}
\end{enumerate}
\end{defn}
\begin{rem}

The predictable measurability $\cP$ could be relaxed to progressive measurability. However, we chose predictable measurability for two reasons. First, predictable measurability is sufficient for considering classes of integrands for various stochastic integrals derived from (càdlàg) local martingales. Therefore, we believe that this formulation is more suitable for future research. Second, we adopt many terminologies from Krylov's theories which utilize predictable $\sigma$-algebras instead of progressive $\sigma$-algebras.

\end{rem}

The class of locally integrable functions is a key subset of distributions, as they realize distributions through integral actions, as detailed in Definition \ref{local inte class}. 
Since $L_{1,\ell oc}(\fR^d) \subset \cD'(\fR^d)$, there exists the natural $\sigma$-algebra on $L_{1,\ell oc}(\fR^d)$ inherited from the Borel sets on $\cD'(\fR^d)$. Alternatively, one can directly construct the Borel sets on $L_{1,\ell oc}(\fR^d)$ using the weak*-topology on $L_{1,\ell oc}(\fR^d)$. 
This allows for the consideration of measurable functions that take values in $L_{1, \ell oc}(\fR^d)$. 
Additionally, subclasses such as locally integrable function-valued random variables and stochastic processes are intimately connected with Fourier transforms of solutions and data. Below, we provide the notations relevant to these subclasses. 
Detailed explanations are omitted because they are straightforward modifications of the previous definitions.

\begin{defn}[locally integrable function-valued processes]
										\label{loc valued}
All spaces 
$$
\bL_{0}\left( \Omega, \rG; L_{1,\ell oc}(\fR^d)\right),
~\bL_{1}\left( \Omega, \rG; L_{1,\ell oc}(\fR^d)\right),
~\bL_{0}\left( \opar 0,\tau \cbrk, \rH ; L_{1,\ell oc}(\fR^d)\right), 
$$
$$
\bL_{0,1}\left( \opar 0,\tau \cbrk, \rH ; L_{1,\ell oc}(\fR^d)\right), 
~\bL_{1,1,loc}\left( \opar 0,\tau \cbrk, \rH ; L_{1,\ell oc}(\fR^d)\right),
~\bL_{1,1}\left( \opar 0,\tau \cbrk, \rH ; L_{1,\ell oc}(\fR^d)\right) 
$$
$$
\bL_{0}\left( \opar 0,\tau \cbrk, \cP;L_{1,\ell oc}(\fR^d;l_2)\right), 
~\bL_{0,2}\left( \opar 0,\tau \cbrk, \cP; L_{1,\ell oc}(\fR^d;l_2)\right),
~\bL_{1,2,loc}\left( \opar 0,\tau \cbrk, \cP; L_{1,\ell oc}(\fR^d;l_2)\right), 
$$
and $\bL_{1,2}\left( \opar 0,\tau \cbrk, \cP; L_{1,\ell oc}(\fR^d;l_2)\right)$ 
 are defined similarly as in Definitions \ref{d valued rv}, \ref{d valued process}, and \ref{d l2 valued process} by considering $L_{1,\ell oc}(\fR^d)$ and $L_{1,\ell oc}(\fR^d;l_2)$ instead of $\cD'(\fR^d)$ and $\cD(\fR^d;l_2)$.
\end{defn}

The extended Fourier and inverse Fourier transforms are essential in establishing our primary operators as defined in Definition \ref{defn psi operator}. 
Additionally, they act as homeomorphisms between $\cD'(\fR^d)$ and $\cF^{-1}\cD'(\fR^d)$, as demonstrated in Theorem \ref{homeo thm}. 
This connection allows us to define the classes that manage our solutions and data in a general context.

\begin{defn}[$\cF^{-1}\cD'(\fR^d;l_2)$-valued processes]
All spaces 
$$
\bL_{0}\left( \Omega, \rG; \cF^{-1}\cD'(\fR^d)\right),
~\bL_{1}\left( \Omega, \rG; \cF^{-1}\cD'(\fR^d)\right),
~\bL_{0}\left( \opar 0,\tau \cbrk, \rH ; \cF^{-1}\cD'(\fR^d)\right), 
$$
$$
\bL_{0,1}\left( \opar 0,\tau \cbrk, \rH ; \cF^{-1}\cD'(\fR^d)\right), 
~\bL_{1,1,loc}\left( \opar 0,\tau \cbrk, \rH ; \cF^{-1}\cD'(\fR^d)\right),
~\bL_{1,1}\left( \opar 0,\tau \cbrk, \rH ; \cF^{-1}\cD'(\fR^d)\right) 
$$
$$
\bL_{0}\left( \opar 0,\tau \cbrk, \cP;\cF^{-1}\cD'(\fR^d;l_2)\right), 
~\bL_{0,2}\left( \opar 0,\tau \cbrk, \cP; \cF^{-1}\cD'(\fR^d;l_2)\right),
~\bL_{1,2,loc}\left( \opar 0,\tau \cbrk, \cP; \cF^{-1}\cD'(\fR^d;l_2)\right), 
$$
and $\bL_{1,2}\left( \opar 0,\tau \cbrk, \cP; \cF^{-1}\cD'(\fR^d;l_2)\right)$
 are defined similarly as in Definitions \ref{d valued rv}, \ref{d valued process}, and \ref{d l2 valued process} by considering $\cF^{-1}\cD'(\fR^d)$ and $\cF^{-1}\cD'(\fR^d;l_2)$ instead of $\cD'(\fR^d)$ and $\cD(\fR^d;l_2)$.
\end{defn}

By Theorem \ref{homeo thm}, it is obvious that
$$
u \in \bL_{0}\left( \opar 0,\tau \cbrk, \rH ;\cF^{-1}\cD'(\fR^d)\right) ~\text{if and only if}~
\cF[u] \in \bL_{0}\left( \opar 0,\tau \cbrk, \rH ;\cD'(\fR^d)\right).
$$
Similarly,
$u \in \bL_{0}\left( \opar 0,\tau \cbrk, \cP;\cF^{-1}\cD'(\fR^d;l_2)\right)$ if and only if
$\cF[u] \in \bL_{0}\left( \opar 0,\tau \cbrk, \cP;\cD'(\fR^d;l_2)\right)$.
All other spaces have the same equivalences.

By Theorem \ref{homeo thm} again, if $u \in L_{1, \ell oc}(\fR^d) \subset \cD'(\fR^d)$, then $\cF^{-1}[u] \in \cF^{-1}\cD'(\fR^d)$. 
Additionally, $\cF^{-1}L_{1, \ell oc}(\fR^d)$ denotes the class of all $u \in \cF^{-1}\cD'(\fR^d)$ such that $\cF[u] \in L_{1, \ell oc}(\fR^d)$, \textit{i.e.}
\begin{align*}
u \in \cF^{-1}L_{1, \ell oc}(\fR^d) \iff \cF[u] \in L_{1, \ell oc}(\fR^d).
\end{align*}
In particular, it is evident that for any $u \in \cF^{-1}L_{1, \ell oc}(\fR^d)$, $u$ has a realizable frequency function on $\fR^d$, \textit{i.e.}
\begin{align*}
\langle \cF[u], \varphi \rangle = \int_{\fR^d} \cF[u](\xi) \overline{\varphi(\xi)} \mathrm{d}\xi \quad \forall \varphi \in \cD(\fR^d).
\end{align*}
Similarly, $\cF^{-1}L_{1,\ell oc}(\fR^d;l_2)$ is defined as the class of all the inverse Fourier transforms of $u$ in $L_{1,\ell oc}(\fR^d ; l_2)$.

We are now ready to define $\cF^{-1}L_{1,\ell oc}(\fR^d;l_2)$-valued processes, which form essential subclasses of $\cF^{-1}\cD'(\fR^d;l_2)$-valued processes. Here $\cF^{-1}L_{1,\ell oc}(\fR^d)$ can be considered as a subspace of $\cF^{-1}L_{1,\ell oc}(\fR^d;l_2)$.  These subclasses are crucial for understanding solutions and data characterized by frequencies derived from complex or $l_2$-valued stochastic processes. 
Additionally, we consider the Borel sets in $\mathcal{F}^{-1}L_{1, \ell oc}(\fR^d; l_2)$ to discuss the measurability of functions that take values in $\mathcal{F}^{-1}L_{1, \ell oc}(\fR^d; l_2)$.

\begin{defn}[$\cF^{-1}L_{1,\ell oc}(\fR^d;l_2)$-valued processes]
All spaces 
$$
\bL_{0}\left( \Omega, \rG; \cF^{-1}L_{1,\ell oc}(\fR^d)\right),
~\bL_{1}\left( \Omega, \rG; \cF^{-1} L_{1,\ell oc}(\fR^d)\right),
~\bL_{0}\left( \opar 0,\tau \cbrk, \rH ; \cF^{-1}L_{1,\ell oc}(\fR^d)\right), 
$$
$$
\bL_{0,1}\left( \opar 0,\tau \cbrk, \rH ; \cF^{-1}L_{1,\ell oc}(\fR^d)\right), 
~\bL_{1,1,loc}\left( \opar 0,\tau \cbrk, \rH ; \cF^{-1}L_{1,\ell oc}(\fR^d)\right),
~\bL_{1,1}\left( \opar 0,\tau \cbrk, \rH ; \cF^{-1}L_{1,\ell oc}(\fR^d)\right) 
$$
$$
\bL_{0}\left( \opar 0,\tau \cbrk, \cP; \cF^{-1} L_{1,\ell oc}(\fR^d;l_2)\right), 
~\bL_{0,2}\left( \opar 0,\tau \cbrk, \cP; \cF^{-1} L_{1,\ell oc}(\fR^d;l_2)\right),
~\bL_{1,2,loc}\left( \opar 0,\tau \cbrk, \cP; \cF^{-1} L_{1,\ell oc}(\fR^d;l_2)\right), 
$$
and
$\bL_{1,2}\left( \opar 0,\tau \cbrk, \cP; \cF^{-1} L_{1,\ell oc}(\fR^d;l_2)\right)$
 are defined similarly as in Definition \ref{loc valued} by considering $\cF^{-1} L_{1,\ell oc}(\fR^d)$ and $\cF^{-1} L_{1,\ell oc}(\fR^d;l_2)$ instead of $L_{1,\ell oc}(\fR^d)$ and $ L_{1,\ell oc}(\fR^d;l_2)$.
\end{defn}

Let $u \in \bL_{0}\left( \opar 0,\tau \cbrk, \rH ; L_{1,\ell oc}(\fR^d)\right)$.
Then it is possible to consider $u$ as a complex-valued function defined on $\opar 0,\tau \cbrk \times \fR^d$ by putting
\begin{align*}
u(t,\xi):= u(t)(\xi)
\end{align*}
In addition, it is obvious that for each $(\omega, t) \in \opar 0,\tau \cbrk$, the mapping 
$$
\xi \in \fR^d \mapsto u(t,\xi)
$$
is $\cB(\fR^d)$-measurable. 
However, the joint measurability with respect to $\rH \times \cB(\fR^d)$ is generally not guaranteed.

The joint measurability can be restored if $u$ has integrability with respect to the product measures.
Specifically, if $u$ belongs to a better space $\bL_{1,1}\left( \opar 0,\tau \cbrk, \rH ; L_{1,\ell oc}(\fR^d)\right)$, then a $\rH \times \cB(\fR^d)$-measurable modification of $u$ can be found, leveraging an approximation due to the separability of the space $L_{1,\ell oc}(\fR^d)$.

Moreover, it is possible to find a jointly measurable modification even without a strong integrability condition. This procedure will be described in detail after introducing some classes of complex-valued (or $l_2$ -valued) functions.

\begin{defn}[(finite) measurable functions]
\begin{enumerate}[(i)]
\item $\bL_0\left(\Omega \times \fR^d, \rG \times \cB(\fR^d) \right)$ denotes the space of all $\rG \times \cB(\fR^d)$-measurable functions on $\Omega \times \fR^d$.

\item $\bL_0\left(\opar 0,\tau \cbrk \times \fR^d, \rH \times \cB(\fR^d)  \right)$ denotes the space of all $\rH|_{\opar 0,\tau \cbrk}\times \cB(\fR^d)$-measurable functions $u$ on $\opar 0,\tau \cbrk \times \fR^d$.

\item $\bL_0\left(\opar 0,\tau\cbrk \times \fR^d,\cP \times \cB(\fR^d)  ; l_2 \right)$ denotes the space of all $l_2$-valued $\cP|_{\opar 0,\tau \cbrk} \times \cB(\fR^d)$-measurable functions $u$ on $\opar 0,\tau \cbrk \times \fR^d$.
\end{enumerate}
\end{defn}
Now we are prepared to identify our complex-valued (or $l_2$-valued) classes as subclasses of $\cD'(\fR^d)$-valued (or $\cD'(\fR^d;l_2)$-valued) processes.
Here is the core idea.
Let $u \in \bL_0\left(\Omega \times \fR^d, \rG \times \cB(\fR^d) \right)$ and assume that
\begin{align}
										\label{20240323 01}
\int_{B_R} |u( \xi)| \mathrm{d}\xi  < \infty \quad (a.s.) \quad \forall R \in (0,\infty).
\end{align}
Then for almost every $\omega \in \Omega$, the mapping 
$$
\varphi \in \cD(\fR^d) \mapsto \int_{\fR^d} u(\xi) \overline{\varphi(\xi)} \mathrm{d} \xi
$$ 
is in $\cD'(\fR^d)$.
Additionally, for each $\varphi \in \cD(\fR^d)$, the mapping
$$
\omega \in \Omega \mapsto \int_{\fR^d} u(\xi) \overline{\varphi(\xi)} \mathrm{d} \xi:=\int_{\fR^d} u(\omega,\xi) \overline{\varphi(\xi)} \mathrm{d} \xi
$$
is $\rG$-measurable by Fubini's theorem since the probability space is complete.
Thus, the mapping
$$
\omega \in \Omega \mapsto \left(\varphi \in \cD(\fR^d) \mapsto \int_{\fR^d} u(\omega,\xi) \varphi(\xi) \mathrm{d} \xi\right)
$$
is a $\cD'(\fR^d)$-valued $\rG$-measurable function on $\Omega$ since the weak*-topology and its Borel sets are given on $\cD'(\fR^d)$.
More precisely, $u$ is a $L_{1, \ell oc}(\fR^d)$-valued $\rG$-measurable function on $\Omega$.
Therefore $u$ becomes an element in $\bL_{0}\left( \Omega, \rG; L_{1,\ell oc}(\fR^d) \right)$ if \eqref{20240323 01} holds.
Moreover, one can consider the Fourier transform or the inverse Fourier transform of $u$ (for each fixed $\omega$) and specifically, 
$$
\cF^{-1}[u] \in \bL_{0}\left( \Omega, \rG; \cF^{-1}L_{1,\ell oc}(\fR^d)\right).
$$

Next, let $v \in \bL_{0}\left( \Omega, \rG; \cF^{-1}L_{1,\ell oc}(\fR^d)\right)$.
Then for each $\omega \in \Omega$, the mapping $\xi \in \fR^d \mapsto \cF[v](\xi) \in \fC$  becomes a complex-valued $\cB(\fR^d)$-measurable function defined $(a.e.)$ in $\fR^d$. 
Thus, we can regard $\cF[v](\xi):=\cF[v](\omega,\xi):=\cF[v(\omega)](\xi)$ as a complex-valued function defined on $\Omega \times \fR^d$ such that for each $\omega \in \Omega$, the mapping $\xi \in \fR^d \mapsto \cF[v](\xi)$ is $\cB(\fR^d)$-measurable.
Furthermore, for any $R \in (0,\infty)$,
\begin{align*}
\int_{B_R} |\cF[v](\xi)| \mathrm{d}\xi < \infty \quad (a.s.).
\end{align*}
However, there is no guarantee that 
$(\omega,\xi) \in \Omega \times \fR^d \mapsto \cF[v](\omega,\xi)$ is $\rG \times \cB(\fR^d)$-measurable.
Without joint measurability, it is impossible to apply the Fubini theorem. 
Thus, it needs to introduce a subspace of $\bL_{0}\left( \Omega, \rG; \cF^{-1}L_{1,\ell oc}(\fR^d)\right)$ with joint measurability.

On the other hand, if $\cF[v]$ has better integrability so that $\bE[\int_{\fR^d}|\cF[v](\xi)| \mathrm{d}\xi ] < \infty$, then one can easily find a $\rG \times \cB(\fR^d)$-measurable modification of $\cF[u]$ based on an approximation of nice functions. Here, the nice functions consist of finite sums of linear combinations of $(X \otimes f) (\omega,x) := X(\omega) f(x)$, where $X \in L_1(\Omega,\rG)$ and $f \in \cD(\fR^d)$. 

Furthermore, it is also possible to find a joint-measurable modification of $\cF[v]$
without integrability by using the approximation $ 1_{|\xi| < n} \cdot \left(-n \vee \cF[v](\xi) \wedge n\right)$ to $\cF[v]$ as $n \to \infty$.

Nonetheless, we prefer to consider classes of jointly measurable functions to avoid the complexity of finding modifications.
Thus we provide more precise definitions of subspaces to address joint measurability. 

\begin{defn}
\begin{enumerate}[(i)]
\item $\bL_{0,1,\ell oc}\left(\Omega \times \fR^d, \rG \times \cB(\fR^d)\right)$ denotes the subspace of $\bL_{0}\left( \Omega, \rG;L_{1,\ell oc}(\fR^d)\right)$ consisting of all $\rG \times \cB(\fR^d)$-measurable complex-valued functions $u$ on $\Omega \times \fR^d$ such that 
\begin{align*}
\int_{B_R} |u(\xi)| \mathrm{d}\xi  < \infty \quad (a.s.) \quad \forall R \in (0,\infty).
\end{align*}

\item $\bL_{1,1,\ell oc}\left(\Omega \times \fR^d, \rG \times \cB(\fR^d)\right)$ denotes the subspace of $\bL_{0,1,\ell oc}\left(\Omega \times \fR^d, \rG \times \cB(\fR^d)\right)$ consisting of all $\rG \times \cB(\fR^d)$-measurable complex-valued functions $u$ on $\Omega \times \fR^d$ such that 
\begin{align*}
\bE\left[\int_{B_R} |u(\xi)| \mathrm{d}\xi\right]  < \infty \quad \forall R \in (0,\infty).
\end{align*}

\item $\cF^{-1}\bL_{0,1,\ell oc}\left(\Omega \times \fR^d, \rG \times \cB(\fR^d)\right)$ denotes the subspace of $\bL_{0}\left( \Omega, \rG;\cF^{-1}L_{1,\ell oc}(\fR^d)\right)$ consisting of all $u \in \bL_{0}\left( \Omega, \rG;\cF^{-1}L_{1,\ell oc}(\fR^d)\right)$ such that the mapping 
\begin{align*}
(\omega,\xi) \in \Omega \times \fR^d \mapsto  \cF[u](\xi):=\cF[u](\omega, \xi):=\cF[u(\omega)](\xi) \in \fC
\end{align*}
is $\rG \times \cB(\fR^d)$-measurable and 
\begin{align*}
\int_{B_R} |\cF[u](\xi)| \mathrm{d}\xi  < \infty \quad (a.s.) \quad \forall R \in (0,\infty).
\end{align*}

\item $\cF^{-1}\bL_{1,1,\ell oc}\left(\Omega \times \fR^d, \rG \times \cB(\fR^d)\right)$ denotes the subspace of $\cF^{-1}\bL_{0,1,\ell oc}\left(\Omega \times \fR^d, \rG \times \cB(\fR^d)\right)$ consisting of all $u \in \cF^{-1}\bL_{0,1,\ell oc}\left(\Omega \times \fR^d, \rG \times \cB(\fR^d)\right)$ such that 
\begin{align*}
\bE\left[\int_{B_R} |\cF[u](\xi)| \mathrm{d}\xi \right]  < \infty  \quad \forall R \in (0,\infty).
\end{align*}

\end{enumerate}
\end{defn}

If we consider a function defined on $\opar 0,\tau \cbrk \times \fR^d$, then all the measurability issues become much more complicated. 
Recall that $\rH$ is a sub-$\sigma$-algebra of $\rF \times \cB([0,\infty))$ and
$\bL_{0,1}\left( \opar 0,\tau \cbrk, \rH;\cD'(\fR^d)\right)$ denotes the class of all $\cD'(\fR^d)$-valued $\rH|_{\opar 0,\tau \cbrk}$-measurable functions defined in
$\opar 0,\tau \cbrk$.
Let $u \in \bL_{0,1}\left( \opar 0,\tau \cbrk,\rH;\cD'(\fR^d)\right)$ and suppose that for each $(\omega,t) \in \opar 0,\tau \cbrk$, $u(\omega,t)$ has a realization on $\cU$, \textit{i.e.} for any compact subset $K \subset \cU$,
\begin{align}
									\label{20240324 01}
\int_{K}|u(t,\xi)| \mathrm{d}\xi
:=\int_{K}|u(\omega,t,\xi)| \mathrm{d}\xi
:=\int_{K}|u(\omega,t)(\xi)| \mathrm{d}\xi < \infty \quad (a.e.)~ (\omega,t) \in \opar 0,\tau \cbrk.
\end{align}
By $\bL_{0,1}\left( \opar 0,\tau \cbrk, \rH;L_{1,\ell oc}(\cU)\right)$, we denote the subspace of $\bL_{0,1}\left( \opar 0,\tau \cbrk, \rH;\cD'(\fR^d)\right)$ consisting of all $u \in \bL_{0,1}\left( \opar 0,\tau \cbrk, \rH;\cD'(\fR^d)\right)$ so that \eqref{20240324 01} holds.

In particular, $u \in \bL_{0,1}\left( \opar 0,\tau \cbrk, \rH;L_{1,\ell oc}(\fR^d)\right)$ if $\cU = \fR^d$ and $u$ can be identified with a complex-valued function defined on $\opar 0,\tau \cbrk \times \fR^d$ by putting
\begin{align*}
u(t,\xi):=u(\omega,t,\xi) := u(\omega,t)(\xi).
\end{align*}

However, unless $\cU=\fR^d$, then this identification does not work properly for $u \in \bL_{0,1}\left( \opar 0,\tau \cbrk, \rH;L_{1,\ell oc}(\cU)\right)$ 
with a complex-valued function defined on  $\opar 0,\tau \cbrk \times \fR$ according to the trivial extension 
$$
u(t,\xi) 1_{\opar 0,\tau \cbrk \times \cU}(t,\xi)
:=u(\omega,t,\xi) 1_{\opar 0,\tau \cbrk \times \cU}(\omega,t,\xi)
:=u(\omega,t)(\xi) 1_{\opar 0,\tau \cbrk \times \cU}(\omega,t,\xi)
$$ 
since it is not guaranteed to be an element in  $\bL_{0,1}\left( \opar 0,\tau \cbrk, \rH;\cD'(\fR^d)\right)$.

In particular, this identification is easily possible for special $u \in \bL_{0,1}\left( \opar 0,\tau \cbrk;L_{1,\ell oc}(\cU)\right)$  which has a compact support with respect to $\xi$ as a distribution uniformly for all $(\omega,t) \in \opar 0,\tau \cbrk$ by considering the canonical extension.
Additionally, for any $u \in \bL_{0,1}\left( \opar 0,\tau \cbrk;L_{1,\ell oc}(\cU)\right)$, it is obvious that for almost every $(\omega,t) \in \opar 0,\tau \cbrk$, 
the mapping
$$
\xi \in \fR^d \mapsto u(t,\xi) 1_{\opar 0,\tau \cbrk \times \cU}(t,\xi) \in \fC
$$
is $\cB(\fR^d)$-measurable.

However, \eqref{20240324 01} does not generally ensure the joint measurability of $u(\omega,t,\xi)$.
Thus we consider a subclasses of $\bL_{0,1}\left( \opar 0,\tau \cbrk, \rH;\cD'(\fR^d)\right)$ with enhanced measurability.
By 
$$
\bL_{0,1,1,\ell oc}\left( \opar 0,\tau \cbrk \times \fR^d,  \rH \times \cB(\fR^d) \right),
$$
we denote  the class consisting of  $u \in \bL_{0,1}\left( \opar 0,\tau \cbrk, \rH;\cD'(\fR^d)\right)$ such that \eqref{20240324 01} holds with $\cU=\fR^d$ and the mapping
$$
(\omega,t,\xi) \mapsto 
1_{\opar 0,\tau \cbrk}(t)u(t,\xi)
:=1_{\opar 0,\tau \cbrk}(\omega,t)u(\omega,t,\xi)
$$
is $\rH \times \cB(\fR^d)$-measurable. 
Moreover, we need to define a subclass of $\bL_{0,1}\left( \opar 0,\tau \cbrk, \rH; \cF^{-1}\cD'(\fR^d)\right)$ consisting of all elements having (spatial) realizable frequency functions on open sets in $\fR^d$. 
Formally, it is given by
\begin{align*}
\{ \cF^{-1}u : u(\omega,t,\cdot) \in L_{1,\ell oc}(\cU) \},
\end{align*}
where $\cU$ is an open set in $\fR^d$.
However, it is difficult to define it rigorously since 
all we know is that for each $(\omega,t)$, $u(\omega,t,x)$ becomes a complex-valued function only on $\cU$. 
Thus we redefine $L_{1,\ell oc}(\cU)$ as the subspace of $\cD'(\fR^d)$ consisting of $u$ having a realization on $\cU$.
Then $u \in \bL_{0,1}\left( \opar 0,\tau \cbrk, \rH;L_{1,\ell oc}(\cU)\right)$ implies that $u$ is a locally integrable function on $\cU$ but also a distribution on $\fR^d$ for each $(\omega,t) \in \opar 0, \tau \cbrk$ due to the definition of the class.
We provide rigorous definitions below with the additional consideration that $\cU$ varies depending on the time variable.

\begin{defn}
								\label{major function class}
We consider a family of open subsets of $\fR^d$ indexed by $(0,\infty)$ and  denoted by $\{ \cU_{t} \subset \fR^d :   t \in (0,\infty)\}$. 
We put
$$
\opar 0,\tau \cbrk \times \cU_{t}
:= \left \{ (\omega,t,\xi) \in \Omega \times (0, \infty) \times \fR^d :  0 < t \leq \tau(\omega)~\text{and}~\xi \in \cU_{t} \right\}.
$$
Here, we employ the notation $\opar 0,\tau \cbrk \times \cU_{t}$ somewhat loosely since $\cU_t$ is determined by each $ t \in (0, \tau(\omega)]$. 
In other words, for each $\omega$, $(0,\tau(\omega)] \times \cU_t:= \{ (t,\xi) \in (0,\infty) \times \fR^d :  t \in (0,\infty)~ \text{and}~ \xi \in \cU_t \} $ is actually a  subset of $(0,\infty) \times \fR^d$ but it is not a strict Cartesian product.
More precisely, 
\begin{align*}
(\omega,t,\xi) \in \opar 0,\tau \cbrk \times \cU_{t} \quad \iff \quad
(\omega,t) \in \opar 0,\tau \cbrk ~\text{and}~ \xi \in \cU_t
\quad \iff \omega \in \Omega,~t \in (0,\tau(\omega)],~ \text{and}~ \xi \in \cU_t.
\end{align*}

\begin{enumerate}[(i)]

\item For any $u \in \bL_{0}\left( \opar 0,\tau \cbrk, \rH ; \cF^{-1}\cD'(\fR^d)\right)$,
we write $u \in \cF^{-1}\bL_{0,0,1,\ell oc}\left( \opar 0,\tau \cbrk \times \cU_{t}, \rH \times \cB(\fR^d)\right)$
if there exists a complex-valued function $f(\omega,t,\xi)$ on $\opar 0,\tau \cbrk \times \cU_{t}$
such that
the mapping $(\omega,t,\xi) \mapsto 1_{\opar 0,\tau \cbrk \times \cU_{t}}(\omega,t,\xi)f(\omega,t,\xi)$ is $\rH \times \cB(\fR^d)$-measurable, and for almost every $(\omega,t) \in \opar 0,\tau \cbrk$,
\begin{align*}
 \int_{K}  \left|f(\omega,t,\xi) \right| \mathrm{d}t \mathrm{d}\xi 
< \infty  \quad \text{for all compact subsets $K$ of $\cU_{t}$},
\end{align*}
and 
\begin{align*}
\langle \cF[u(t,\cdot)], \varphi \rangle 
:=\langle \cF[u(\omega,t,\cdot)], \varphi \rangle 
:=\langle \cF[u(\omega,t)], \varphi \rangle 
=  \int_{\cU_{t}} f(\omega,t,\xi) \overline{\varphi(\xi)} \mathrm{d} \xi \quad \forall \varphi \in \cD(\cU_{t}).
\end{align*}
In particular, we write $\cF[u(t,\cdot)](\xi)=f(\omega,t,\xi)$ by identifying two mappings 
$$
(\omega,t) \mapsto \left(\varphi \in \cD(\cU_{t}) \mapsto \langle \cF[u(t,\cdot)], \varphi \rangle \right)
$$ 
and 
$$
(\omega,t) \mapsto \left(\varphi \in \cD(\cU_{t}) \mapsto \int_{\fR^d} f(\omega,t,\xi) \overline{\varphi(\xi)} \mathrm{d} \xi\right).
$$
Thus for any $u \in \cF^{-1}\bL_{0,0,1,\ell oc}\left( \opar 0,\tau \cbrk \times \cU_{t}, \rH \times \cB(\fR^d)\right)$, the Fourier transform of $u$ (with respect to the space variable) is a function defined  on $\opar 0,\tau \cbrk \times \cU_{t}$ so that the mapping 
$$
(\omega,t,\xi) \mapsto 1_{\opar 0,\tau \cbrk \times \cU_{t}}(\omega,t,\xi)\cF[u(\omega,t,\cdot)](\xi)
$$
is $\rH \times \cB(\fR^d)$-measurable, and for almost every $(\omega,t) \in \opar 0,\tau \cbrk$,
\begin{align}
										\label{20240325 10}
 \int_{K}  \left| \cF[u(t,\cdot)](\xi) \right|  \mathrm{d}\xi 
< \infty  \quad \text{for all compact subsets $K \subset \cU_{t}$}.
\end{align}
Here \eqref{20240325 10} means that there exists a subset $\cT \subset \opar 0,\tau \cbrk$  so that
$\opar 0,\tau \cbrk \setminus \cT$ is a null set and for each $(\omega,t) \in \cT$,  we have
\begin{align*}
 \int_{K}  \left| \cF[u(\omega,t,\cdot)](\xi) \right| \mathrm{d}t \mathrm{d}\xi < \infty \quad \text{for all compact subsets $K \subset \cU_{t}$}.
\end{align*}

\item 
We use the notation $\cF^{-1}\bL_{0,1,1,\ell oc}\left( \opar 0,\tau \cbrk \times \cU_{t}, \rH \times \cB(\fR^d)\right)$ to denote the subspace consisting of all $u \in \cF^{-1}\bL_{0,0,1,\ell oc}\left( \opar 0,\tau \cbrk \times \cU_{t}, \rH \times \cB(\fR^d)\right)$ such that 
\begin{align*}
\int_0^{\tau} \int_{K}  1_{\opar 0,\tau \cbrk \times \cU_{t}}(t,\xi) \left| \cF[u(t,\cdot)](\xi) \right| \mathrm{d}\xi \mathrm{d}t  
< \infty \quad (a.s.) \quad \text{for all compact subsets $K$ of $\fR^d$}, 
\end{align*}
which means that there exists a $\Omega' \subset \Omega$ so that $P(\Omega')=1$ and for all $\omega \in \Omega'$ and compact subset $K$ of $\fR^d$, we have
\begin{align*}
\int_0^{\tau(\omega)} \int_{K}  1_{\opar 0,\tau \cbrk \times \cU_{t}}(\omega,t,\xi) \left| \cF[u(\omega,t,\cdot)](\xi) \right| \mathrm{d}\xi  \mathrm{d}t
< \infty.
\end{align*}

\item 
We use the notation $\cF^{-1}\bL_{1,1,1,loc, \ell oc}\left( \opar 0,\tau \cbrk \times \cU_{t}, \rH \times \cB(\fR^d)\right)$ to denote the subspace consisting of all $u \in \cF^{-1}\bL_{0,1,1,\ell oc}\left( \opar 0,\tau \cbrk \times \cU_{t}, \rH \times \cB(\fR^d)\right)$ such that for all $T \in (0,\infty)$,
\begin{align*}
\bE\left[\int_0^{\tau \wedge T} \int_{K}  1_{\opar 0,\tau \cbrk \times \cU_{t}}(t,\xi) \left| \cF[u(t,\cdot)](\xi) \right| \mathrm{d}\xi  \mathrm{d}t\right] 
< \infty \quad \text{for all compact subsets $K$ of $\fR^d$}.
\end{align*}

\item 
We use the notation $\cF^{-1}\bL_{1,1,1,\ell oc}\left( \opar 0,\tau \cbrk \times \cU_{t}, \rH \times \cB(\fR^d)\right)$ to denote the subspace consisting of all $u \in \cF^{-1}\bL_{1,1,1,loc,\ell oc}\left( \opar 0,\tau \cbrk \times \cU_{t}, \rH \times \cB(\fR^d)\right)$ such that
\begin{align*}
\bE\left[\int_0^{\tau} \int_{K}  1_{\opar 0,\tau \cbrk \times \cU_{t}}(t,\xi) \left| \cF[u(t,\cdot)](\xi) \right|  \mathrm{d}\xi \mathrm{d}t \right] 
< \infty \quad \text{for all compact subsets $K$ of $\fR^d$}.
\end{align*}
\end{enumerate}
Here, $K$ does not vary with respect to the time variable in (ii), (iii), and (iv).
\end{defn}

All the spaces mentioned above appear quite complex due to the presence of the family of open subsets $\cU_t$. 
General definitions involving $\cU_t$ are provided to demonstrate that complicate identifications can work for general domains to some extent, potentially aiding future research. However, we usually set $\cU_t=\fR^d$  uniformly for all $t$. 
In fact, in this paper, we only consider the case where $\cU_t=\fR^d$  uniformly for all $t$.
This simplifies all definitions considerably. Moreover, it is clear that ``for all compact subsets $K$ of $\fR^d$ " can be replaced by ``for all $B_R:=\{\xi \in \fR^d : |\xi| < R\}$ with $R \in (0,\infty)$'' due to the Heine-Borel theorem.

Next, we extend the spaces defined in Definition \ref{major function class} by incorporating non-negative weights. In particular, we aim to permit weights to be zero on sets that have positive measures. Under these conditions, elements cease to be unique as complex-valued functions because they can assume any value in regions where the weights are zero. We resolve this issue by considering the weighted spaces as subspaces of the distribution-valued spaces previously introduced.

\begin{defn}
										\label{major function class 2}
Let $W$ be a non-negative $\rH \times \cB(\fR^d)$-measurable function on $\Omega \times (0,\infty) \times \fR^d$, and $\{\cU_{t}: t \in (0,\infty)\}$ be a family of open subsets of $\fR^d$ so that $\supp W(t,\cdot) \subset \cU_{t}$ with probability one, where
\begin{align*}
\supp W(t,\cdot) := \overline{ \{\xi \in \fR^d : W(t,\xi):=W(\omega,t,\xi) \neq 0 \} }
\end{align*}
and $\overline{ \{\xi \in \fR^d : W(t,\xi) \neq 0 \} }$ denotes the closure of the set $\{\xi \in \fR^d : W(t,\xi) \neq 0\}$ with respect to the topology generated by the Euclidean norm on $\fR^d$.

\begin{enumerate}[(i)]

\item For any $u \in \bL_{0,1}\left( \opar 0,\tau \cbrk, \rH;\cD'(\fR^d)\right)$,
we write $u \in \cF^{-1}\bL_{0,1,1,\ell oc}\left( \opar 0,\tau \cbrk \times \fR^d, \rH \times \cB(\fR^d), W(t,\xi)\mathrm{d}t \mathrm{d}\xi\right)$
if there exists a complex-valued function $f(\omega,t,\xi)$ on $\opar 0,\tau \cbrk \times \cU_{t}$ such that
for almost every $(\omega,t) \in \opar 0,\tau \cbrk$,
\begin{align}
										\label{20240331 20}
\langle \cF[u(\omega,t)], \varphi \rangle =  \int_{\cU_{t}} f(\omega,t,\xi) \varphi(\xi) \mathrm{d} \xi \quad \forall \varphi \in \cD(\cU_{t}),
\end{align}
the mapping $(\omega,t,\xi) \mapsto 1_{\opar 0,\tau \cbrk \times \cU_{t}}(\omega,t,\xi) f(\omega,t,\xi)$ is $\rH \times \cB(\fR^d)$-measurable, and
\begin{align}
									\label{l1 norms}
 \int_0^{\tau(\omega)}  \int_{B_R} 1_{\opar 0,\tau \cbrk \times \cU_{t}}(\omega,t,\xi) \left|f(\omega,t,\xi) \right| W(t,\xi) \mathrm{d}\xi \mathrm{d}t
< \infty \quad (a.s.) \quad \forall R \in (0,\infty).
\end{align}
In particular, we write $\cF[u(t,\cdot)](\xi)=\cF[u(\omega,t,\cdot)](\xi)=f(\omega,t,\xi)$ by identifying two mappings 
$$
(\omega,t) \mapsto \left(\varphi \in \cD(\cU_{t}) \mapsto \langle \cF[u(\omega,t)], \varphi \rangle \right)
$$ and 
$$
(\omega,t) \mapsto \left(\varphi \in \cD(\cU_{t}) \mapsto \int_{\fR^d} f(\omega,t,\xi) \overline{\varphi(\xi)} \mathrm{d} \xi\right).
$$
In other words, for almost every $(\omega,t) \in \opar 0,\tau \cbrk$, the Fourier transform of $u$,
$\cF[u(t,\cdot)]:=\cF[u(\omega,t)]$, has a realization on $\cU_{t}$ so that
\begin{align}
									\notag
&\int_0^{\tau} \int_{B_R}  \left|\cF[u(t,\cdot)](\xi) \right| W(t,\xi) \mathrm{d}\xi \mathrm{d}t  \\
									\label{2024062401}
&:=\int_0^{\tau(\omega)} \int_{B_R}  \left|\cF[u(\omega,t,\cdot)](\xi) \right| W(\omega,t,\xi) \mathrm{d}\xi \mathrm{d}t 
< \infty \quad (a.s.) \quad \forall R \in (0,\infty).
\end{align}
Thus for any 
$$
u \in \cF^{-1}\bL_{0,1,1,\ell oc}\left( \opar 0,\tau \cbrk \times \fR^d, \rH \times \cB(\fR^d), W(t,\xi)\mathrm{d}t \mathrm{d}\xi\right),
$$
the Fourier transform of $u$ (with respect to the space variable) is a function defined on $\opar 0,\tau \cbrk \times \cU_{t}$ such that the mapping $(\omega,t,\xi) \mapsto 1_{\opar 0,\tau \cbrk \times \cU_{t}}(\omega,t,\xi) \cF[u(\omega,t,\cdot)](\xi)$ is $\rH \times \cB(\fR^d)$-measurable, and \eqref{2024062401} holds.
In other words, any $u \in \cF^{-1}\bL_{0,1,1,\ell oc}\left( \opar 0,\tau \cbrk \times \fR^d, \rH \times \cB(\fR^d), W(t,\xi)\mathrm{d}t \mathrm{d}\xi\right)$ (for almost every $(\omega,t) \in \opar 0,\tau \cbrk$) has a realizable frequency function on $\cU_t$.

\item We use 
$$
\cF^{-1}\bL_{1,1,1,loc,\ell oc}\left( \opar 0,\tau \cbrk \times \fR^d, \rH \times \cB(\fR^d),W(t,\xi)\mathrm{d}t \mathrm{d}\xi\right)
$$
to denote the subspace of 
$$
\cF^{-1}\bL_{0,1,1,\ell oc}\left( \opar 0,\tau \cbrk \times \fR^d,\rH \times \cB(\fR^d), W(t,\xi)\mathrm{d}t \mathrm{d}\xi\right)
$$
consisting of all elements 
$$
u \in \cF^{-1}\bL_{0,1,1,\ell oc}\left( \opar 0,\tau \cbrk \times \fR^d,\rH \times \cB(\fR^d), W(t,\xi)\mathrm{d}t \mathrm{d}\xi\right)
$$
so that for each $T \in (0,\infty)$,
\begin{align*}
 \bE\left[\int_{B_R} \int_0^{\tau \wedge T} 1_{\opar 0,\tau \cbrk \times \cU_{t}}(t,\xi) \left|\cF[u(t,\cdot)](\xi) \right| W(t,\xi)\mathrm{d}t \mathrm{d}\xi \right]
< \infty \quad \forall R \in (0,\infty).
\end{align*}
Additionally,
$$
\cF^{-1}\bL_{1,1,1, \ell oc}\left( \opar 0,\tau \cbrk \times \fR^d, \rH \times \cB(\fR^d),W(t,\xi)\mathrm{d}t \mathrm{d}\xi\right)
$$
denotes the subspace of 
$$
\cF^{-1}\bL_{1,1,1,loc,\ell oc}\left( \opar 0,\tau \cbrk \times \fR^d,\rH \times \cB(\fR^d), W(t,\xi)\mathrm{d}t \mathrm{d}\xi\right)
$$
consisting of all elements whose expectation is finite, \textit{i.e.}
$$
u \in \cF^{-1}\bL_{1,1,1,\ell oc}\left( \opar 0,\tau \cbrk \times \fR^d,\rH \times \cB(\fR^d), W(t,\xi)\mathrm{d}t \mathrm{d}\xi\right)
$$
if and only if
$$
u \in \cF^{-1}\bL_{1,1,1,loc,\ell oc}\left( \opar 0,\tau \cbrk \times \fR^d,\rH \times \cB(\fR^d), W(t,\xi)\mathrm{d}t \mathrm{d}\xi\right)
$$
and 
\begin{align*}
 \bE\left[\int_{B_R} \int_0^{\tau} 1_{\opar 0,\tau \cbrk \times \cU_{t}}(t,\xi) \left|\cF[u(t,\cdot)](\xi) \right| W(t,\xi)\mathrm{d}t \mathrm{d}\xi \right]
< \infty \quad \forall R \in (0,\infty).
\end{align*}

\end{enumerate}
\end{defn}

\begin{rem}
Let $u \in \cF^{-1}\bL_{0,1,1,\ell oc}\left( \opar 0,\tau \cbrk \times \fR^d, \rH \times \cB(\fR^d), W(t,\xi)\mathrm{d}t \mathrm{d}\xi\right)$.
Then obviously, 
\begin{align*}
\int_0^{\tau} \int_{B_R}  \left|\cF[u(t,\cdot)](\xi) \right| W(t,\xi) \mathrm{d}\xi \mathrm{d}t  < \infty \quad (a.s.) \quad \forall R \in (0,\infty).
\end{align*}
It might seem that $\cF[u]$ could increase very rapidly within a bounded domain if $W$ strongly diminishes there.
However, the extent of the blow-up is restricted and cannot disrupt local integrability because $u$ must satisfy the additional condition that
\begin{align*}
 \left| \int_{\cU_{t}} \cF[u(\omega,t, \cdot)](\xi) \overline{\varphi(\xi)} \mathrm{d} \xi  \right| <\infty \quad \forall \varphi \in \cD(\cU_{t}).
\end{align*}
Due to this condition, it is also obvious that
$u \in \cF^{-1}\bL_{0,0,1,\ell oc}\left( \opar 0,\tau \cbrk \times \cU_{t}, \rH \times \cB(\fR^d)\right)$, which clearly implies
\begin{align*}
\cF^{-1}\bL_{0,1,1,\ell oc}\left( \opar 0,\tau \cbrk \times \fR^d, \rH \times \cB(\fR^d), W(t,\xi)\mathrm{d}t \mathrm{d}\xi\right)
\subset \cF^{-1}\bL_{0,0,1,\ell oc}\left( \opar 0,\tau \cbrk \times \cU_{t}, \rH \times \cB(\fR^d)\right).
\end{align*}

\end{rem}
\begin{rem}
It appears that all function spaces in Definition \ref{major function class 2} depend on the choice of a family of open subsets 
$\{\cU_t\}$. 
However, the (local) weighted $L_1$-norms in \eqref{l1 norms} are unaffected by this choice because of the condition 
$\supp W(t,\cdot) \subset \cU_{t}$ with probability one.
Furthermore, there is always the trivial option of setting $\cU_{t}= \fR^d$ for all $t$. 
This trivial case is considered throughout the paper. 
In other words, local distributions are not a primary focus since our equation is solved in the entire space with respect to the spatial variable. Nonetheless, each $\cU_t$ can be chosen to be as close to 
$\supp W(t,\cdot)$ as desired, provided the support is uniformly given for the sample points $\omega \in \Omega$.
\end{rem}

\begin{rem}
We do consider random weights in Definition \ref{major function class 2} since our symbol $\psi(t,\xi)$ may involve randomness, which leads to the emergence of certain random weighted spaces related to the symbol. We had restricted the definition of weighted spaces to suit our specific needs. Additionally, these classes are used to characterize our solutions according to the influence of the symbols. However, we do not address general weighted estimates with randomness, which, to the best of our knowledge, remains an unexplored area in theories of stochastic partial differential equations. Nonetheless, our approach enables the definitions of more general weighted spaces, such as 
$\cF^{-1}\bL_{p,q,r,loc,\ell oc}\left( \opar 0,\tau \cbrk \times \fR^d, W(t,\xi)\mathrm{d}t \mathrm{d}\xi ; l_2\right)$ for all $p,q,r \in [1,\infty)$, although these are not utilized in the paper.
\end{rem}

\begin{rem}
											\label{unique rem}
Recall that $\cF^{-1}\bL_{1,1,1,\ell oc}\left( \opar 0,\tau \cbrk \times \fR^d, \rH \times \cB(\fR^d), W(t,\xi)\mathrm{d}t \mathrm{d}\xi\right)$ is a subspace of 
$$
\bL_{0,1}\left(\opar 0,\tau \cbrk, \rH;\cF^{-1}\cD'(\fR^d)\right).
$$
However, it is not a metric space generated by the weighted $L_1$-norms in \eqref{l1 norms} if the weight $W(t,\xi)=0$ on a set having a positive measure.
In other words, there may exist two different $u_1$ and $u_2$ in $\bL_{0,1}\left( \opar 0,\tau \cbrk, \rH;\cF^{-1}\cD'(\fR^d)\right)$
such that
\begin{align}
										\label{2024062410}
 \bE\left[\int_{B_R} \int_0^{\tau(\omega)} 1_{\opar 0,\tau \cbrk \times \cU_{t}}(\omega,t,\xi) \left|\cF[u_1(\omega,t,\xi)] -\cF[u_2(\omega,t,\xi)] \right| W(t,\xi)\mathrm{d}t \mathrm{d}\xi \right]
=0 \quad \forall R \in (0,\infty)
\end{align}
if the weight $W$ is degenerate on a set with a positive measure.
Therefore the space 
$$
\cF^{-1}\bL_{1,1,1,\ell oc}\left( \opar 0,\tau \cbrk \times \fR^d, \rH \times \cB(\fR^d), W(t,\xi)\mathrm{d}t \mathrm{d}\xi\right)
$$ 
itself does not seem to be appropriate for discussing the uniqueness of a solution 
based solely on estimates like \eqref{2024062410}.
\end{rem}

Next, we examine $l_2$-valued correspondents, focusing solely on the predictable $\sigma$-algebra. Elements in these spaces act as stochastic inhomogeneous data and serve as integrands for Itô stochastic integrals, as previously mentioned.
Additionally, we consider only unweighted spaces, since we do not address weighted stochastic inhomogeneous data. It should be noted that some classes with different orders of integration are included to apply the recently developed stochastic Fubini theorems, which operate under weaker assumptions and are introduced in Section \ref{section stochastic Fubini}.
For simplicity, the trivial case $\cU_t= \fR^d$ for all $t$ is considered in the following definition and some notational details from the previous definitions are omitted.

\begin{defn}
										\label{sto integrand class}
\begin{enumerate}[(i)]

\item For any $u \in \bL_{0}\left( \opar 0,\tau \cbrk, \cP ; \cF^{-1}\cD'(\fR^d ; l_2)\right)$,
we write 
$$
u \in \cF^{-1}\bL_{0,0,1,\ell oc}\left( \opar 0,\tau \cbrk \times \fR^d, \cP \times \cB(\fR^d) ; l_2\right)
$$
if the mapping $(\omega,t,\xi) \mapsto 1_{\opar 0,\tau \cbrk}(\omega,t)\cF[u(\omega,t,\cdot)](\xi)$ is a $l_2$-valued $\cP \times \cB(\fR^d)$-measurable function and for almost every $(\omega,t) \in \opar 0, \tau \cbrk$,
\begin{align*}
 \int_{K}  \left|\cF[u(\omega,t,\cdot)](\xi) \right|_{l_2} \mathrm{d}\xi 
< \infty   \quad \text{for all compact subsets $K \subset \fR^d$}.
\end{align*}

\item 
We use the notation $\cF^{-1}\bL^{\omega,\xi,t}_{0,1,2,\ell oc}\left( \opar 0,\tau \cbrk \times \fR^d, \cP \times \cB(\fR^d) ; l_2 \right)$ to denote the subspace consisting of all $u \in \cF^{-1}\bL_{0,0,1,\ell oc}\left( \opar 0,\tau \cbrk \times \fR^d, \cP \times \cB(\fR^d) ; l_2\right)$ such that 
\begin{align*}
\int_{K} \left( \int_0^{\tau(\omega)}  \left| \cF[u(\omega,t,\cdot)](\xi) \right|^2_{l_2} \mathrm{d}t  \right)^{1/2} \mathrm{d}\xi 
< \infty \quad (a.s.) \quad \text{for all compact subsets $K \subset \fR^d$}. 
\end{align*}

\item 
The notation $\cF^{-1}\bL^{\omega,\xi,t}_{1,1,2,loc, \ell oc}\left( \opar 0,\tau \cbrk \times \fR^d, \cP \times \cB(\fR^d); l_2 \right)$ describes  the subspace consisting of all 
$$
u \in \cF^{-1}\bL^{\omega,\xi,t}_{0,1,2,\ell oc}\left( \opar 0,\tau \cbrk \times \fR^d, \cP \times \cB(\fR^d) ; l_2 \right)
$$
such that for each $T \in (0,\infty)$,
\begin{align*}
\bE\left[ \int_{K} \left(\int_0^{\tau(\omega) \wedge T} \left| \cF[u(\omega,t,\cdot)](\xi) \right|^2_{l_2}   \mathrm{d}t \right)^{1/2}\mathrm{d}\xi \right] 
< \infty \quad \text{for all compact subsets $K \subset \fR^d$}.
\end{align*}

\item 
Finally, $\cF^{-1}\bL^{\omega,\xi,t}_{1,1,2,\ell oc}\left( \opar 0,\tau \cbrk \times \fR^d, \cP \times \cB(\fR^d);l_2 \right)$ denotes the subspace consisting of all 
$$
u \in \cF^{-1}\bL^{\omega,\xi,t}_{1,1,2,loc, \ell oc}\left( \opar 0,\tau \cbrk \times \fR^d, \cP \times \cB(\fR^d); l_2 \right)
$$
such that
\begin{align*}
\bE\left[ \int_{K} \left(\int_0^{\tau(\omega)} \left| \cF[u(\omega,t,\cdot)](\xi) \right|^2_{l_2}   \mathrm{d}t \right)^{1/2}\mathrm{d}\xi \right] 
< \infty \quad \text{for all compact subsets $K \subset \fR^d$}.
\end{align*}
\end{enumerate}

\end{defn}

\begin{rem}
The superscript in the notation $\cF^{-1}\bL^{\omega,\xi,t}_{1,1,2,loc, \ell oc}\left( \opar 0,\tau \cbrk \times \fR^d, \cP \times \cB(\fR^d); l_2 \right)$
indicates the order of the integration. 
Specifically,
$$
 u \in \cF^{-1}\bL^{\omega,\xi,t}_{1,1,2,loc, \ell oc}\left( \opar 0,\tau \cbrk \times \fR^d, \cP \times \cB(\fR^d); l_2 \right)
$$
if and only if
for each $T \in (0,\infty)$,
\begin{align*}
\bE\left[ \int_{K} \left(\int_0^{\tau(\omega) \wedge T} \left| \cF[u(\omega,t,\cdot)](\xi) \right|^2_{l_2}   \mathrm{d}t \right)^{1/2}\mathrm{d}\xi \right] 
< \infty \quad \text{for all compact subsets $K \subset \fR^d$}.
\end{align*}
Similarly, the notation
$$
 u \in \bL^{\omega,x,t}_{1,1,2,loc, \ell oc}\left( \opar 0,\tau \cbrk \times \fR^d, \cP \times \cB(\fR^d); l_2 \right)
$$
implies that for each $T \in (0,\infty)$,
\begin{align*}
\bE\left[ \int_{K} \left(\int_0^{\tau(\omega) \wedge T} \left| u(\omega,t,x) \right|^2_{l_2}   \mathrm{d}t \right)^{1/2}\mathrm{d}x \right] 
< \infty \quad \text{for all compact subsets $K \subset \fR^d$}.
\end{align*}
If no superscript is provided, the integration order follows the sequence of the product.
For instance, if
$$
u \in \bL_{1,1,2,\ell oc}\left( \opar 0,\tau \cbrk \times \fR^d, \rF \times \cB([0,\infty)) \times \cB(\fR^d)\right),
$$ 
then 
\begin{align*}
\bE \int_0^\tau \left(\int_{K} |u(t,x)|^2 \mathrm{d}x\right)^{1/2} \mathrm{d}t  < \infty 
 \quad \text{for all compact subsets $K \subset \fR^d$}.
\end{align*}
Additionally, all complex-valued function spaces are defined without superscripts, which means that the order of the integrations is always in the sequence of the random variable $\omega$, time variable $t$, and space variable $x$.
\end{rem}

Typically, it is anticipated that solutions to SPDEs driven by space-time white noise exhibit continuous paths as all paths of  Brownian motions are continuous. However, our operator is quite irregular, meaning its symbol lacks regularity. Thus it is difficult to expect that the paths of our solutions will remain continuous. Nonetheless, we can still assert a certain form of path continuity in a weak sense using test functions. Specifically, the continuity of paths of solutions to \eqref{main eqn} can be maintained as a $\cF^{-1}\cD'(\fR^d)$-valued functions. We establish these concepts by first defining distribution-valued continuous processes.

\begin{defn}[Spaces of continuous paths]
											\label{space conti}
We say that  a function $u : \clbrk 0,\tau \cbrk \to \cD'(\fR^d)$ is a $\cD'(\fR^d)$-valued $\rH$-measurable process 
if  for each $\varphi \in \cD(\fR^d)$, (noting that $\cD'(\fR^d)$ is equipped with the the weak*-topology)
\begin{align*}
(\omega,t) \mapsto 
1_{\clbrk 0,\tau \cbrk}(t) \langle u(t), \varphi \rangle
:=1_{\clbrk 0,\tau \cbrk}(\omega,t) \langle u(\omega,t), \varphi \rangle
\end{align*}
is $\rH$-measurable, where
\begin{align*}
\clbrk 0,\tau \cbrk=\{(\omega,t) \in \Omega \times [0,\infty) :  0 \leq t \leq \tau(\omega) \}
\end{align*}
and
\begin{align*}
 1_{\clbrk 0,\tau \cbrk}(\omega,t)
 := 
\begin{cases}
& 1 \quad \text{if}~  (\omega,t) \in \clbrk 0,\tau \cbrk \\
& 0 \quad \text{if}~  (\omega,t) \in  \Omega\times [0,\infty) \setminus \clbrk 0,\tau \cbrk.
\end{cases}
\end{align*}
By $\bC\left( \clbrk 0,\tau \cbrk, \rH;\cD'(\fR^d)\right)$, we denote the space of all $\cD'(\fR^d)$-valued $\rH$-measurable processes such that for any $\varphi \in \cD(\fR^d)$, the paths
$$
t \mapsto  \langle u(t) , \varphi \rangle 
$$
 are continuous on $[0,\tau]$ with probability one, \textit{i.e.} there exists a $\Omega' \subset \Omega$ so that $P(\Omega')=1$ and for all $\omega \in \Omega'$ and $\varphi \in \cD(\fR^d)$,
the mapping 
\begin{align*}
t \in [0,\tau(\omega)] \mapsto  \langle u(\omega,t), \varphi \rangle \in \fC
\end{align*}
is continuous. 
 
The notation $\bL_1\bC\left( \clbrk 0,\tau \cbrk, \rH;\cD'(\fR^d)\right)$ denotes the subspace of $\bC\left( \clbrk 0,\tau \cbrk, \rH;\cD'(\fR^d)\right)$ consisting of elements with finite expectations, \textit{i.e.}
$u \in \bL_1\bC\left( \clbrk 0,\tau \cbrk, \rH;\cD'(\fR^d)\right)$ if and only if
$u \in \bC\left( \clbrk 0,\tau \cbrk, \rH;\cD'(\fR^d)\right)$ and
\begin{align*}
\bE \left[  \left|\langle u(t), \varphi \rangle \right|\right] <\infty \quad \forall \varphi \in  \cD(\fR^d).
\end{align*}
The spaces $\bC\left( \clbrk 0,\tau \cbrk, \rH;\cF^{-1}\cD'(\fR^d)\right)$ and $\bL_1\bC\left( \clbrk 0,\tau \cbrk,\rH;\cF^{-1}\cD'(\fR^d)\right)$
are defined in a similar way by substituting $\cD'(\fR^d)$ for $\cF^{-1}\cD'(\fR^d)$.
Then, it is obvious that 
$u \in \bC\left( \clbrk 0,\tau \cbrk,\rH;\cF^{-1}\cD'(\fR^d)\right)$ 
if and only if
$\cF[u] \in \bC\left( \clbrk 0,\tau \cbrk,\rH;\cD'(\fR^d)\right)$.

Finally, we also define important subspaces 
$$
\bC\left( \clbrk 0,\tau \cbrk, \rH;L_{1, \ell oc}(\fR^d)\right), ~
\bL_1\bC\left( \clbrk 0,\tau \cbrk, \rH;L_{1, \ell oc}(\fR^d)\right), ~
\bC\left( \clbrk 0,\tau \cbrk, \rH;\cF^{-1}L_{1, \ell oc}(\fR^d)\right), 
$$
and $\bL_1\bC\left( \clbrk 0,\tau \cbrk,\rH;\cF^{-1}L_{1, \ell oc}(\fR^d)\right)$ by considering the subspaces $L_{1, \ell oc}(\fR^d)$ and $\cF^{-1}L_{1, \ell oc}(\fR^d)$ instead of $\cD'(\fR^d)$ and $\cF^{-1}\cD'(\fR^d)$, respectively.
\end{defn}

\begin{rem}

There are typically two ways to define topologies on $L_{1,loc}(\fR^d)$.
The first method is to equip $L_{1,\ell oc}(\fR^d)$ with the subspace topology of $\cD'(\fR^d)$ so called the weak*-topology, which we denote by $\rT_w$. This is the topology we have used on $L_{1,\ell oc}(\fR^d)$  up to this point.
The second method involves using the strong topology generated by the local semi-norms
$$
\|u\|_{L_1(B_R)} = \int_{B_R} |u(x)| \mathrm{d}x \quad R \in (0,\infty)
$$
and use the notation $\rT_s$ to denote this topology.
Then it is obvious that $\rT_w \subset \rT_s$.
Additionally, the space 
$\bC\left( \clbrk 0,\tau \cbrk, \rH; L_{1,loc}(\fR^d) \right)$
can be equipped with either of these topologies. 
By $\bC\left( \clbrk 0,\tau \cbrk, \rH; L^{weak}_{1,\ell oc}(\fR^d) \right)$ and $\bC\left( \clbrk 0,\tau \cbrk, \rH; L^{strong}_{1,\ell oc}(\fR^d) \right)$, we denote the space $\bC\left( \clbrk 0,\tau \cbrk, \rH; L_{1,loc}(\fR^d) \right)$ with topologies $\rT_w$ and $\rT_s$, respectively.
Here we have not defined $\bC\left( \clbrk 0,\tau \cbrk, \rH; L^{strong}_{1,\ell oc}(\fR^d) \right)$ rigorously yet, but it can be understood as space of all continuous functions from $[0,\tau]$ to $L^{strong}_{1,\ell oc}(\fR^d)$ with probability one.

Now first consider the space $\bC\left( \clbrk 0,\tau \cbrk, \rH;L^{weak}_{1, \ell oc}(\fR^d)\right)$. 
Then this topology has an advantage because it can be easily translated to the topology on $\cF^{-1}L_{1,\ell oc}(\fR^d)$ by taking the inverse Fourier transform, due to the homeomorphism in Theorem \ref{homeo thm}. 
Additionally, a coarser topology is preferable when it suffices for constructing results. 
In fact, the weak*-topology is adequate for developing our theories.

However, we can naturally relate the semi-norms to the space $\bC\left( \clbrk 0,\tau \cbrk, \rH;L^{strong}_{1,\ell oc}(\fR^d)\right)$ to measure sizes of elements. 
Additionally, there is a trivial continuous embedding that
$$
\bC\left( \clbrk 0,\tau \cbrk, \rH;L^{strong}_{1,\ell oc}(\fR^d)\right)
\subset \bC\left( \clbrk 0,\tau \cbrk, \rH;L^{weak}_{1, \ell oc}(\fR^d)\right)
$$
according to the identity mapping  $ x \in L^{strong}_{1,\ell oc}(\fR^d) \mapsto x \in L^{weak}_{1, \ell oc}(\fR^d)$.
Thus, the space 
$$
\bC\left( \clbrk 0,\tau \cbrk, \rH;L^{strong}_{1,\ell oc}(\fR^d)\right)
$$
is preferred for considering a solution because it inherently provides more information and is more optimized.

Lastly, any $u \in \bC\left( \clbrk 0,\tau \cbrk, \rH;L^{strong}_{1,\ell oc}(\fR^d)\right)$ can be regarded as a complex-valued function defined on $\clbrk 0,\tau \cbrk \times \fR^d$ so that for any $R \in (0,\infty)$
\begin{align*}
t \in [0, \tau] \mapsto \int_{B_R} |u(\omega,t,x)|\mathrm{d}x
\end{align*}
is continuous and thus
\begin{align}
										\label{20240807 21}
\sup_{t \in  [0,\tau] }\int_{B_R} |u(\omega,t,x)| \mathrm{d}x < \infty
\end{align}
 with probability one by putting $u(\omega,t,x)= u(\omega,t)(x)$.
Conversely, for a nice complex-valued function $u(\omega,t,x)$ defined on  $\clbrk 0,\tau \cbrk \times \fR^d$ can be recognized as an element of
$\bC\left( \clbrk 0,\tau \cbrk, \rH;L^{strong}_{1,\ell oc}(\fR^d)\right)$ by considering the mappings
\begin{align*}
(\omega,t) \in \clbrk 0,\tau \cbrk  \mapsto \int_{B_R} |u(\omega,t,x)|\mathrm{d}x \quad  R \in (0,\infty)
\end{align*}
so that
\begin{align*}
t \in [0, \tau] \mapsto \int_{B_R} |u(\omega,t,x)|\mathrm{d}x
\end{align*}
is continuous with probability one and the condition 
\begin{align}
										\label{20240807 20}
\int_{B_R} \sup_{t \in [0,\tau]}|u(\omega,t,x)| \mathrm{d}x < \infty \quad (a.s.).
\end{align}
holds. 
Therefore, one can consider a subspace of $\bC\left( \clbrk 0,\tau \cbrk, \rH;L^{strong}_{1,\ell oc}(\fR^d)\right)$ consisting of complex-valued joint measurable functions  defined on  $\clbrk 0,\tau \cbrk \times \fR^d$ without losing much information.
Here \eqref{20240807 20} becomes stronger than \eqref{20240807 21} to derive the $\cD'(\fR^d)$-valued continuity from the continuity of $t \mapsto u(\omega,t,x)$ for almost every $\omega$ and $x$.
The detail for the definition is as follows.
\end{rem}
\begin{defn}[Spaces of realizable continuous paths]
Let $u$ be a (complex-valued) $\rH \times \cB(\fR^d)$-measurable function defined on $\clbrk 0,\tau \cbrk \times \fR^d$.
\begin{enumerate}[(i)]
\item We write $u \in \bC L_{1,\ell oc}\left( \clbrk 0,\tau \cbrk \times \fR^d, \rH \times \cB(\fR^d) \right)$ if for any $R$,
\begin{align*}
 \int_{B_R}\sup_{t \in [0,\tau]}|u(t,x)|\mathrm{d}x < \infty \quad (a.s.),
\end{align*}
and the mapping $t \in [0, \tau] \mapsto u(t,x)$  is continuous for almost every $x \in \fR^d$ with probability one, \textit{i.e.} there exists a $\Omega' \subset \Omega$ such that $P(\Omega')=1$, for any $\omega \in \Omega'$ and $R \in (0,\infty)$,
\begin{align*}
 \int_{B_R}\sup_{t \in [0,\tau(\omega)]}|u(\omega,t,x)|\mathrm{d}x < \infty
\end{align*}
and $t \in [0, \tau(\omega)] \mapsto u(\omega,t,x)$ is continuous for almost every $x \in \fR^d$.
\item $\bL_1\bC L_{1,\ell oc}\left( \clbrk 0,\tau \cbrk \times \fR^d, \rH \times \cB(\fR^d) \right)$ denotes the subspace consisting of $u$ such that 
$$
u \in \bC L_{1,\ell oc}\left( \clbrk 0,\tau \cbrk \times \fR^d, \rH \times \cB(\fR^d) \right)
$$
and for any $R$,
\begin{align*}
\bE\left[\int_{B_R}\sup_{t \in [0,\tau]} |u(t,x)|\mathrm{d}x \right] < \infty.
\end{align*}
\end{enumerate}
\end{defn}
\begin{rem}
Let $u \in \bC L_{1,\ell oc}\left( \clbrk 0,\tau \cbrk \times \fR^d, \rH \times \cB(\fR^d) \right)$ and $\varphi \in \cD(\fR^d)$.
Then by applying the dominate convergence theorem,
$t \in [0, \tau] \mapsto \int_{\fR^d}u(t,x) \varphi(x) \mathrm{d}x$ is continuous with probability one.
Additionally, it is obvious that for almost every $(\omega,t) \in \clbrk 0,\tau \cbrk$, the function $x \in \fR^d \mapsto u(t,x)$ is in $L_{1,\ell oc}(\fR^d)$. 
Thus
\begin{align}
									\label{20240801 50}
\bC L_{1,\ell oc}\left( \clbrk 0,\tau \cbrk \times \fR^d, \rH \times \cB(\fR^d) \right) 
\subset \bC\left( \clbrk 0,\tau \cbrk, \rH;L_{1, \ell oc}(\fR^d)\right)
\subset \bC\left( \clbrk 0,\tau \cbrk, \rH ;\cD'(\fR^d)\right).
\end{align}
Similarly, 
\begin{align*}
\bL_1\bC L_{1,\ell oc}\left( \clbrk 0,\tau \cbrk \times \fR^d, \rH \times \cB(\fR^d) \right) 
\subset \bL_1\bC\left( \clbrk 0,\tau \cbrk, \rH;L_{1, \ell oc}(\fR^d)\right)
\subset \bL_1\bC\left( \clbrk 0,\tau \cbrk, \rH;\cD'(\fR^d)\right).
\end{align*}
Due to \eqref{20240801 50}, one can consider the inverse Fourier transform of $u$ so that
\begin{align*}
\cF^{-1}u \in \bC\left( \clbrk 0,\tau \cbrk, \rH;\cF^{-1}L_{1, \ell oc}(\fR^d)\right)
\subset \bC\left( \clbrk 0,\tau \cbrk, \rH;\cF^{-1}\cD'(\fR^d)\right).
\end{align*}
More precisely, for almost every $(\omega,t) \in \clbrk 0,\tau \cbrk$,
\begin{align*}
\cF^{-1}\left[u(\omega,t,\cdot)\right] \in \cF^{-1}L_{1, \ell oc}(\fR^d).
\end{align*}
Therefore, we can consider the spaces consisting of the inverse Fourier transforms of 
$$
u \in \bC L_{1,\ell oc}\left( \clbrk 0,\tau \cbrk \times \fR^d, \rH \times \cB(\fR^d) \right).
$$
\end{rem}

\begin{defn}[Spaces of inverse Fourier transforms of realizable continuous paths]
\begin{enumerate}[(i)]
\item 
$$
\cF^{-1}\bC L_{1,\ell oc}\left( \clbrk 0,\tau \cbrk \times \fR^d, \rH \times \cB(\fR^d) \right)
$$ denotes the subspace of $\bC\left( \clbrk 0,\tau \cbrk, \rH;\cF^{-1}L_{1, \ell oc}(\fR^d)\right)$ consisting of $u$ such that  
$$
\cF[u] \in \bC L_{1,\ell oc}\left( \clbrk 0,\tau \cbrk \times \fR^d, \rH \times \cB(\fR^d) \right).
$$

\item $\cF^{-1}\bL_1\bC L_{1,\ell oc}\left( \clbrk 0,\tau \cbrk \times \fR^d, \rH \times \cB(\fR^d) \right)$ denotes the subspace of $\bL_1\bC\left( \clbrk 0,\tau \cbrk, \rH;\cF^{-1}L_{1, \ell oc}(\fR^d)\right)$ consisting of $u$ such that 
\begin{align*}
\cF[u] \in \bL_1\bC L_{1,\ell oc}\left( \clbrk 0,\tau \cbrk \times \fR^d, \rH \times \cB(\fR^d) \right).
\end{align*}
\end{enumerate}
\end{defn}
\begin{rem}
										\label{20240803 rem 1}
Let
$$
u \in \cF^{-1}\bC L_{1,\ell oc}\left( \clbrk 0,\tau \cbrk \times \fR^d, \rH \times \cB(\fR^d) \right).
$$
Then
\begin{align*}
\int_0^\tau \int_{B_R}|\cF[u(t,\cdot)](\xi)|\mathrm{d}\xi \mathrm{d}t
\leq \tau \cdot \int_{B_R} \sup_{t \in [0,\tau]}|\cF[u(t,\cdot)](\xi)|\mathrm{d}\xi  \quad (a.s.).
\end{align*}
Thus 
\begin{align*}
\cF^{-1}\bC L_{1,\ell oc}\left( \clbrk 0,\tau \cbrk \times \fR^d, \rH \times \cB(\fR^d) \right)
\subset \cF^{-1}\bL_{0,1,1,\ell oc}\left( \opar 0,\tau \cbrk \times \fR^d, \rH \times \cB(\fR^d)\right).
\end{align*}
Similarly,
\begin{align*}
\cF^{-1}\bL_{1,loc}\bC L_{1,\ell oc}\left( \clbrk 0,\tau \cbrk \times \fR^d, \rH \times \cB(\fR^d) \right)
\subset \cF^{-1}\bL_{1,1,1,loc,\ell oc}\left( \opar 0,\tau \cbrk \times \fR^d, \rH \times \cB(\fR^d)\right).
\end{align*}
\end{rem}

\begin{rem}
We treat all the spaces mentioned above as the quotient spaces generated by the equivalence relation which is ``equal almost everywhere''.
For instance,  a $\cF^{-1}\cD'(\fR^d;l_2)$-valued $\overline{\cP}$-measurable function defined $v$ on $\opar 0,\tau \cbrk$ is in $\bL_{0}\left( \opar 0,\tau \cbrk, \cP;\cF^{-1}\cD'(\fR^d;l_2)\right)$ if there exists a $\cF^{-1}\cD'(\fR^d;l_2)$-valued $\cP$-measurable functions $u$ on $\opar 0,\tau \cbrk$ such that
for each $\varphi \in \cF^{-1}\cD(\fR^d)$,
\begin{align}
								\label{20240227}
\langle v(t,\cdot), \varphi \rangle=\langle u(t,\cdot), \varphi \rangle < \infty \quad (a.e.)~(\omega,t) \in \opar 0,\tau \cbrk,
\end{align}
where $\overline{\cP}$ denotes the completion of the $\sigma$-algebra $\cP$ with respect to the measure $dP \times dt$.
Note that \eqref{20240227} holds $(a.e.)$ for each $\varphi \in \cF^{-1}\cD(\fR^d)$ and it is sufficient to identify $u$ and $v$ since the space $\cF^{-1}\cD(\fR^d)$ is separable.
In particular, if for each $\varphi \in \cF^{-1}\cD(\fR^d)$,
$\langle v(t,\cdot), \varphi \rangle$ is a modification $\langle u(t,\cdot), \varphi \rangle$ 
(so sufficiently two complex-valued processes $\langle v(t,\cdot), \varphi \rangle$ and $\langle u(t,\cdot), \varphi \rangle$ are indistinguishable),
then we identify $v$ and $u$.
\end{rem}

\begin{defn}[Temporal local spaces]
We define temporal local spaces based on the typical cutoffs of stopping times.
For instance, we write
\begin{align*}
u \in \cF^{-1}\bL_{1,1,1,loc, \ell oc}\left( \opar 0,\tau \cbrk \times \fR^d, \rH \times \cB(\fR^d), W(t,\xi)\mathrm{d}t \mathrm{d}\xi\right)
\end{align*}
if
\begin{align*}
u \in \cF^{-1}\bL_{1,1,1,\ell oc}\left( \opar 0,\tau \wedge T \cbrk \times \fR^d, \rH \times \cB(\fR^d),W(t,\xi)\mathrm{d}t \mathrm{d}\xi\right) \quad \forall T \in (0,\infty).
\end{align*}
Similarly,
\begin{align*}
u \in \bL_{0,1,loc}\left( \opar 0,\tau \cbrk, \rH;\cF^{-1}\cD'(\fR^d)\right)
\end{align*}
if
\begin{align*}
u \in \bL_{0,1}\left( \opar 0,\tau \wedge T \cbrk, \rH;\cF^{-1}\cD'(\fR^d)\right) \quad \forall T \in (0,\infty).
\end{align*}
All other temporal local spaces can be defined in a similar manner.
Some of these temporal local spaces have been already mentioned in Definition \ref{major function class} and Definition \ref{sto integrand class}.
It is easy to verify that these new temporal local spaces are equivalent to the previously defined ones.
Moreover, some spaces remain unchanged under temporal localization.
For example, it is straightforward to confirm that 
\begin{align}
								\label{20240504 10}
\bL_{0,loc}\left( \opar 0,\tau \cbrk, \rH;\cF^{-1}\cD'(\fR^d)\right)
=\bL_{0}\left( \opar 0,\tau \cbrk, \rH;\cF^{-1}\cD'(\fR^d)\right)
\end{align}
and
\begin{align}
								\label{20240504 11}
\bC_{loc}\left( \clbrk 0,\tau \cbrk, \rH;\cD'(\fR^d)\right)=\bC\left( \clbrk 0,\tau \cbrk, \rH;\cD'(\fR^d)\right)
\end{align}
since $\tau$ is a finite stopping time.
\end{defn}

It is clear that a temporally localized space is larger than the original one. 
Additionally, there are many interesting relationships between them. Specifically, the temporal local property does not affect processes without finite expectations, as our fixed stopping time $\tau$ is finite. Simple examples are provided in \eqref{20240504 10} and \eqref{20240504 11}. We present a straightforward proposition to list some of these properties, which are not as obvious as those in \eqref{20240504 10} and \eqref{20240504 11}. Nonetheless, the proof is not difficult, so we omit the details. 
These properties help justify the appropriateness of our data classes in the main theorems in the next section.

\begin{prop}
									\label{temporal local prop}
\begin{enumerate}[(i)]
\item
\begin{align*}
\cF^{-1}\bL_{0,1,1,loc, \ell oc}\left( \opar 0,\tau \cbrk \times \fR^d, \rH \times \cB(\fR^d)\right)
=\cF^{-1}\bL_{0,1,1, \ell oc}\left( \opar 0,\tau \cbrk \times \fR^d, \rH \times \cB(\fR^d)\right)
\end{align*}
\item
\begin{align*}
\cF^{-1}\bL^{\omega,\xi,t}_{0,1,2,loc,\ell oc}\left( \opar 0,\tau \cbrk \times \fR^d, \cP \times \cB(\fR^d) ; l_2 \right)
=\cF^{-1}\bL^{\omega,\xi,t}_{0,1,2,\ell oc}\left( \opar 0,\tau \cbrk \times \fR^d, \cP \times \cB(\fR^d) ; l_2 \right)
\end{align*}
\end{enumerate}
\end{prop}
As previously noted, all the properties mentioned above are easily derived from the fact that $\tau$ is a finite stopping time. 
A similar detail for the proof can be observed in Lemma \ref{restriction lem} later.
However, if we consider a space consisting of elements with finite expectations, then a temporally localized space becomes strictly larger than the original one.
This is because it is easy to find a stochastic process $X_t$ so that
\begin{align*}
\bE\left[ \int_0^{\tau} |X_t| dt\right] = \infty
\end{align*}
but
\begin{align*}
\bE\left[ \int_0^{\tau \wedge T} |X_t| dt\right] < \infty \quad \forall T \in (0,\infty)
\end{align*}
unless $\tau$ is bounded.

\mysection{Main results}

In this section, we present our main results. We discovered that our operators $\psi(t,-\mathrm{i}\nabla)$ can be defined without regularity conditions as described in Definition \ref{defn psi operator}. Our next step is to determine the conditions on the symbol $\psi(t,\xi)$ to ensure that \eqref{main eqn} is well-posed in a certain sense.

This point has been mentioned several times, but we repeat it for emphasis: the symbol $\psi$ can be random, sign-changing, and irregular, yet the solvability of \eqref{main eqn} remains preserved. 
In essence, our primary assumptions impose only integrability conditions on the symbol $\psi$.

Firstly, we require local boundedness of the symbol uniformly with respect to the sample points $\omega$. 
However, it is sufficient to assume local boundedness only on the real part of the symbol.

\begin{assumption}[Local boundedness on the real part of the symbol]
									\label{main as}
For all  $R,T \in (0,\infty)$,
\begin{align}
										\label{20240725 50}
\esssup_{ (t, \xi) \in  (0,T) \times B_R}\left[\int_0^t \esssup_{\omega \in \Omega}\left|\Re[\psi(r,\xi)]\right|\mathrm{d}r\right]  
< \infty.
\end{align}
\end{assumption}
We also introduce some constants related to Assumption \ref{main as}, which appear in our main estimates of solutions.
Put
\begin{align*}
C^{\mathrm{e}\int\sup|\Re[\psi]|}_{R,T}
=\esssup_{ (t,\xi) \in  (0,T) \times B_R}  
\left[\exp\left( \int_0^t \esssup_{\omega \in \Omega}\left|\Re[\psi(r,\xi)] \right| \mathrm{d}r  \right)\right].
\end{align*}
Then it is obvious that \eqref{20240725 50} holds if and only if
$$
C^{\mathrm{e}\int\sup|\Re[\psi]|}_{R,T} <\infty.
$$ 

Note that our requirement for the local boundedness pertains only to the space variable $\xi$. 
In other words, $\Re[\psi(t,\xi)]$ can be unbounded with respect to the time variable $t$ as long as it does not ruin the integrability.

Next, we impose an additional condition on $\psi(t,\xi)$ to ensure a control over the product $|\psi(t,\xi)|$ and $\exp\left( \int_s^t\left| \Re[\psi(r,\xi)]\right| \mathrm{d}r \right)$.
 This relationship is crucial for defining $\psi(t,-\mathrm{i}\nabla)u$ with a candidate $u$ for a solution as per Definition \ref{defn psi operator} and is necessary for applying the fundamental theorem of calculus in the proofs of the main theorems.

\begin{assumption}
									\label{main as 2}

For all $R,T \in (0,\infty)$,
\begin{align}
										\label{20240725 51}
C^{\sup|\psi|}_{R,T}:=\esssup_{\xi \in  B_R} \left[ \int_0^T \left(\esssup_{\omega \in \Omega}|\psi(t,\xi)| \right)\sup_{ 0\leq s \leq t} \left[\exp\left( \int_s^t \esssup_{\omega \in \Omega} \left| \Re[\psi(r,\xi)]\right| \mathrm{d}r \right) \right]    \mathrm{d}t  \right]
< \infty.
\end{align}
\end{assumption}
In particular, Assumption \ref{main as 2} implies that for almost every $(\omega,\xi) \in \Omega \times \fR^d$,
\begin{align}
									\label{20240801 01}
\int_0^T |\psi(t,\xi)| \mathrm{d}t < \infty \quad  \forall T \in (0,\infty)
\end{align}
since the term $\sup_{ 0\leq s \leq t} \left[\exp\left( \int_s^t \esssup_{\omega \in \Omega} \left| \Re[\psi(r,\xi)]\right| \mathrm{d}r \right) \right]$ is always greater than or equal to one.
Note that $C^{\sup|\psi|}_{R,T}$ could be small if the value of $|\psi(t,\xi)|$ diminishes.
In other words, we could not generally determine if $C^{|\psi|}_{R,T} \geq 1$.

\begin{rem}
All constants and conditions can be simplified if we remove the randomness on $\psi$.
Assume that the symbol $\psi$ is non-random (or deterministic), \textit{i.e.} 
\begin{align*}
\psi(\omega_1,t,\xi) = \psi(\omega_2,t,\xi) \quad \forall \omega_1, \omega_2 \in \Omega.
\end{align*}
Then \eqref{20240725 50} and \eqref{20240725 51} become
\begin{align*}
\esssup_{ (t, \xi) \in  (0,T) \times B_R}\left[\int_0^t\left|\Re[\psi(r,\xi)]\right|\mathrm{d}r\right]  
=\esssup_{ (t, \xi) \in  (0,T) \times B_R}\left[\int_0^t \esssup_{\omega \in \Omega}\left|\Re[\psi(r,\xi)]\right|\mathrm{d}r\right] < \infty
\end{align*}
and
\begin{align*}
&\esssup_{\xi \in  B_R} \left[ \int_0^T |\psi(t,\xi)|\sup_{ 0\leq s \leq t} \left[\exp\left( \int_s^t  \left| \Re[\psi(r,\xi)]\right| \mathrm{d}r \right) \right]    \mathrm{d}t  \right] \\
&=\esssup_{\xi \in  B_R} \left[ \int_0^T \left(\esssup_{\omega \in \Omega}|\psi(t,\xi)| \right)\sup_{ 0\leq s \leq t} \left[\exp\left( \int_s^t \esssup_{\omega \in \Omega} \left| \Re[\psi(r,\xi)]\right| \mathrm{d}r \right) \right]    \mathrm{d}t  \right] < \infty,
\end{align*}
respectively.
Moreover, these conditions can be replaced by the weaker conditions
\begin{align*}
\esssup_{ (t, \xi) \in  (0,T) \times B_R}\left|\int_0^t\Re[\psi(r,\xi)]\mathrm{d}r\right| 
< \infty
\end{align*}
and
\begin{align*}
\esssup_{\xi \in  B_R} \left[ \int_0^T|\psi(t,\xi)|\sup_{ 0\leq s \leq t} \exp\left( \left| \int_s^t\Re[\psi(r,\xi)] \mathrm{d}r \right|  \right)    \mathrm{d}t  \right]
< \infty.
\end{align*}
This will be detailed in Section \ref{sto fubini deter} and Section \ref{exist deter}.
\end{rem}

\begin{rem}
									\label{simple constant}
By the fundamental theorem of calculus,
\begin{align*}
\exp\left( \int_0^t\left|\Re[\psi(r,\xi)] \right|\mathrm{d}r \right)
=1+ \int_0^t |\Re[\psi(s,\xi)]|  \exp\left(\int_0^s \left|\Re[\psi(r,\xi)] \right|\mathrm{d}r \right) \mathrm{d}s.
\end{align*}
Therefore, there is the relation between the constants $C^{\mathrm{e}\int\sup|\Re[\psi]|}_{R,T}$ and $C^{\sup|\psi|}_{R,T}$ such that
\begin{align*}
C^{\mathrm{e}\int\sup|\Re[\psi]|}_{R,T} \leq 1+C^{\sup|\psi|}_{R,T}. 
\end{align*}

\end{rem}

\begin{rem}
$\psi(t,\xi)$ does not have to be defined at $t=0$ as can be observed in the main equation \eqref{main eqn}.
However, our solution $u(t,x)$ satisfies $u(0,x)=u_0$, which implies the solution $u$ is defined at $t=0$ in a certain sense.
Furthermore, the measurability of $u$ definitely depends on the measurability of the symbol $\psi$.
Thus, to explain this measurability relation, one must consider restrictions or extensions of the $\sigma$-algebras related to both $u$ and $\psi$.
Instead, we choose to make $\psi$ defined on the whole $\Omega \times [0,\infty) \times \fR^d$ to avoid all these complications.

Therefore, the value $\psi(0,\xi)$ can be arbitrary and it is provided only to simplify the expression of measurability relations which will be detailed in the proofs of our main theorems.
Specifically,  all estimates with $\psi$ do not depend on $\psi(0,\xi)$ at all. 
This is why the main assumptions regarding $\psi$ do not involve the values of $\psi(0,\xi)$.

\end{rem}

In addition to removing regularity conditions and allowing randomness on $\psi$, another novel aspect of our assumptions on $\psi$ is that $\Re[\psi(t,\xi)]$ is permitted to be sign-changing.
In other words, the range of $\Re[\psi(t,\xi)]$  can include both positive and negative values.
It is crucial to maintain the sign of $\Re[\psi(t,\xi)]$ to be negative in order to ensure even simpler equations with pseudo-differential operators than  \eqref{main eqn}  are well-posed in Sobolev's spaces.
This finding is supported by the results in \cite{CJH KID 2023,CJH LJB KID 2023,CJH 2024,KID SBL KHK 2015, KID SBL KHK 2016,KID KHK 2016,KID 2018}. 
All the results rely on the condition called ``an ellipticity condition" that there exist positive constants $\kappa$ and $\gamma$ so that
\begin{align*}
\Re[\psi(t,\xi)] \leq -\kappa |\xi|^\gamma,
\end{align*}
where $\gamma$ is understood as an order of the operator $\psi(t,-\mathrm{i}\nabla)$.
However, this ellipticity condition does not need to be met under our main assumptions, as we might consider analytically weak solutions rather than strong solutions in Sobolev spaces. 
In particular, it is impossible to assign a specific order to our operators, which partially demonstrates their generality.

It is surprising that uniqueness holds in our framework, given that classical theories for the differential equations with constant coefficients on $\fR^d$  suggest that uniqueness cannot be maintained without ellipticity, as mentioned in the introduction.

The relationship between the sign and ellipticity of symbols becomes very clear when examining the second-order case as a concrete example.
Here even the existence of a weak solution cannot be obtained from classical theories (by considering an equation in $\fR^{d+1}$) since our coefficients in the following example vary with respect to the time variable.

\begin{example}[Second-order operators with time-dependent coefficients]
For $i,j \in \{1,\ldots,d\}$, let $a^{ij}(t)$ be complex-valued functions such that
\begin{align}
									\label{20240626 03}
\int_0^T |a^{ij}(t)| \mathrm{d}t < \infty \quad \forall T \in (0,\infty).
\end{align}
Put
\begin{align*}
\psi(t,\xi) = -a^{ij}(t)\xi^i\xi^j.
\end{align*}
Then it is easy to verify that $\psi$ satisfies Assumptions \ref{main as} and \ref{main as 2}.
Additionally, it is obvious that 
\begin{align*}
\psi(t,-\mathrm{i} \nabla) \varphi = a^{ij}(t) \varphi_{x^ix^j}
\end{align*}
for a nice function $\varphi$ on $\fR^d$ due to some basic properties of the Fourier and inverse Fourier transforms.

We typically say that the coefficients $a^{ij}(t)$ satisfy an (weak) ellipticity condition if 
\begin{align}
										\label{20240626 01}
a^{ij}(t) \xi^i \xi^j  \geq 0 \quad \forall t \in (0,\infty).
\end{align}
In particular, the coefficients $a^{ij}(t)$ are said to be {\bf degenerate} if $a^{ij}(t)=0$ for some $t$.
Recently, there were researches showing some strong $L_p$-estimates are still possible even though the coefficients $a^{ij}(t)$ are degenerate (\cite{KID KHK 2018, KID KHK 2023, KID 2024}). 
Our weak solutions encompass all solutions from these results for degenerate second-order equations.

Moreover, 
we say that the coefficients $a^{ij}(t)$ satisfy a uniform ellipticity condition if there exists a positive constant $\kappa$ so that
\begin{align}
										\label{20240626 02}
a^{ij}(t) \xi^i \xi^j \geq \kappa |\xi|^2 \quad \forall t \in (0,\infty).
\end{align}
Obviously, \eqref{20240626 01}  implies that the symbol $\psi$ is always non-positive.
Especially, the condition in \eqref{20240626 02} makes the symbol $\psi$ strictly negative.

Note that even the condition in \eqref{20240626 01} is not required to satisfy our main assumptions.
The local integrability from \eqref{20240626 03} is sufficient to show that $\psi$ satisfies Assumptions \ref{main as} and \ref{main as 2}.
In conclusion, our operators $\psi(t,-\mathrm{i} \nabla)$ include second-order operators even without weak ellipticity conditions.

\end{example}

Next, we present another intriguing example that satisfies Assumptions \ref{main as} and \ref{main as 2}, involving sign-changing properties and fractional Laplacian operators. 

The fractional Laplacian operators usually denoted by $\Delta^{\alpha/2}$ have been fascinating subjects of study in mathematics. 
Typically, the restriction $\alpha \in (0,2]$ is imposed and then the operator $\Delta^{\alpha/2}$ becomes a generator of a L\'evy process. For this special cases, $\alpha \in (0,2]$, many interesting properties of the operators and their generalizations can be derived from theories of Markov processes (\textit{cf.} \cite{CKS2010,CKK2011,CKS2016,KSV2019}).
Additionally, numerous analytic methods exist to study these operators (\textit{cf.} \cite{CS2007,DJK2023,K2017,RS2014}).
There are also tons of researches addressing these operators in equations and can be found, for instance, in \cite{CK 2024,CKP2024,CKR2023,DK2012,DL2023,DR2024,KK 2012}.
Even for the range $\alpha \in (2,\infty)$, some analytic methods are still available to find out the properties of fractional Laplacian operators and the solvability of related equations. 
However, there is a lack of results considering fractional Laplacian operators with complex exponents despite the existence of natural analytic continuations.

To the best of our knowledge, our theory is the first to tackle the solvability of stochastic partial differential equations involving fractional Laplacian operators with complex exponents. Furthermore, we have not found any existing PDE results that address equations with fractional Laplacian operators featuring complex exponents. 
However, we found a numerical simulation result for this operator as reported by \cite{BB 2023}. 
It is unexpectedly revealed that the fractional Laplacian operator with a complex exponent is not only a mathematical generalization but also has scientific significance, as shown in \cite{BB 2023} and the references therein.

\begin{example}[Fractional Laplacian operators with complex exponents]
										\label{frac exam}
Let $\alpha= \Re[\alpha]+i \Im[\alpha]$ be a complex number and put
\begin{align*}
\psi(t,\xi) = -|\xi|^\alpha.
\end{align*}
Then for any nonzero $\xi \in \fR^d$, 
\begin{align*}
\psi(t,\xi) 
= -\exp\left(\alpha \log |\xi| \right)
&= -\exp\left(\Re[\alpha] \log |\xi| + \mathrm{i} \Im[\alpha] \log|\xi|\right) \\
&= -\exp\left(\Re[\alpha] \log |\xi|\right) \cdot \exp \left( \mathrm{i} \Im[\alpha] \log|\xi|\right)\\
&=- |\xi|^{\Re[\alpha]} \left( \cos (\Im[\alpha] \log|\xi|) + \mathrm{i} \sin (\Im[\alpha]\log|\xi|)  \right).
\end{align*}
For this symbol $\psi$, we adopt the simpler notation
\begin{align*}
\Delta^{\alpha/2}u(t,x) := \psi(t,-\mathrm{i} \nabla)u(t,x)
\end{align*}
and this operator is known as the {\bf fractional Laplacian operator}.
Assume that $\Re[\alpha] \geq 0$, then it is easy to check that $\psi$ satisfies Assumptions \ref{main as} and \ref{main as 2}.
\end{example}

Now we turn our attention to an important inequality in probability theories.

Let $g$ be an $l_2$-valued square integrable predictable process.
Then for all $p \in (0,\infty)$ and $T \in (0,\infty)$, there exist positive constants $c_p$ and $C_p$ (depending only on $p$) such that
\begin{align}
								\label{20240509 10}
c_p\bE\left[\left(\int_0^T |g(s)|^{2p}_{l_2} \mathrm{d}s \right)^{1/(2p)}\right]
\leq \bE\left[\sup_{t \in [0,T]} \left|\int_0^t g^k(s) \mathrm{d} B^k_s \right|^p \right] 
\leq C_p \bE\left[\left(\int_0^T |g(s)|^{2p}_{l_2} \mathrm{d}s \right)^{1/(2p)} \right].
\end{align}
This inequality is well-known as the BDG (Burkholder-Davis-Gundy) inequality.
A proof can be found, for instance, in \cite{Burk 1973,Krylov 1995,RY 1999}.
Especially, the constant $C_1$ plays an important role in our main theorem.
\begin{defn}[BDG constants]
We call $C_1$ in \eqref{20240509 10} {\bf the BDG constant} and use the special notation $C_{BDG}$ instead of $C_1$.
The exact value of the smallest  $C_{BDG}$ is not known to the best of our knowledge. 
However, taking $p=1$ and $g=(1,0,0,\ldots)$ in \eqref{20240509 10},  we have
\begin{align*}
\bE\left[\sup_{t \in [0,T]}|B^1_t|\right] \leq C_{BDG} T^{1/2} .
\end{align*}
Additionally, by Bachelier's theorem (\textit{cf}. \cite[Theorem 2.2.3]{Krylov 2002}),
\begin{align*}
\bE\left[\sup_{t \in [0,T]}|B^1_t|\right]
=2\int_0^\infty x \exp(-x^2) \mathrm{d}x \cdot T^{1/2} 
=T^{1/2}.
\end{align*}
Therefore,
\begin{align}
										\label{bdg lower}
 C_{BDG} \geq 1.
\end{align}
Additionally, it is well-known that $C_{BDG} \leq 3$ (\textit{cf}. \cite{Burk 1973} and \cite[Section 4 in Chapter IV]{Krylov 1995,RY 1999}).

This constant is considered to try to provide explicit constants (at least with respect to the constants appearing in the BDG inequality and assumptions) in our main theorem.
\end{defn}
Before stating the definition of our solution to \eqref{main eqn}, we recall an important property of the space-time white noise $\mathrm{d}\fB_t$ so that
$$
\mathrm{d}\fB_t = \eta_k \mathrm{d}B_t^k,
$$
where $\{\eta_k : k \in \fN\}$ is the orthonomal basis of $L_2(\fR^d)$ from \eqref{20240525 01}.
This relation is easily derived since $L_2(\fR^d)$-cylindrical Brownian motions can be constructed from a sequence of independent one-dimensional Brownian motions. 
Thus recognizing a SPDE driven by space-time white noise as an equation driven by a sequence of independent one-dimensional white noises has been a common approach in studying the properties of their solutions. For example, see \cite{CH2021,HK2020,Krylov 1999,KK 2020,KKL2019,KPR2022}. 
 In this paper, we also employ this method to provide a precise mathematical meaning to solutions of SPDEs in the form of \eqref{main eqn}. This will be detailed in the next section.
For now, we provide the exact definition of our solution.

\begin{defn}[Fourier-space weak solution to \eqref{main eqn}]
									\label{space weak solution}

Let $u_0 \in \bL_0\left(\Omega,\rF;\cF^{-1}\cD'(\fR^d) \right)$, 
\begin{align*}
f \in \bL_{0,1,loc}\left( \opar 0,\tau \cbrk, \rF \times \cB([0,\infty));\cF^{-1}\cD'(\fR^d)\right), 
\quad h \in \bL_{0,2}\left(\opar 0,\tau \cbrk , \cP ; \cF^{-1}\cD'(\fR^d)   \right),
\end{align*}
and 
$$
u \in  \bL_{0}\left( \opar 0,\tau \cbrk, \rF \times \cB\left( [0,\infty) \right);\cF^{-1}\cD'(\fR^d)\right).
$$
Then we say that $u$ is a {\bf Fourier-space weak solution} to equation \eqref{main eqn} if for all $\varphi \in \cF^{-1}\cD(\fR^d)$,
\begin{align}
									\notag
\left\langle u(t,\cdot),\varphi \right\rangle 
&= \langle u_0, \varphi \rangle +  \int_0^t  \left\langle  \psi(s,-\mathrm{i} \nabla)u(s,\cdot) , \varphi \right\rangle\mathrm{d}s 
+ \int_0^t \left\langle f(s,\cdot),\varphi\right\rangle \mathrm{d}s \\
									\label{solution meaning}
& \quad +\int_0^t\left\langle h (s,\cdot) \eta^k(\cdot),\varphi\right\rangle \mathrm{d}B^k_s
\quad  (a.e.) \quad (\omega,t) \in \clbrk 0,\tau \cbrk.
\end{align}
\end{defn}
\begin{rem}
The stochastic term 
$$
\int_0^t\left\langle h (s,\cdot) \eta^k(\cdot),\varphi\right\rangle \mathrm{d}B^k_s
$$
cannot be defined based on a canonical multiplication of a test function.
Specifically, the condition 
$$
h \in \bL_{0,2}\left(\opar 0,\tau \cbrk , \cP ; \cF^{-1}\cD'(\fR^d)   \right),
$$
is insufficient to render the term
$$
\int_0^t\left\langle h (s,\cdot) ,\eta^k \varphi\right\rangle \mathrm{d}B^k_s
$$
well-defined  since $\eta^k \varphi \in \cS(\fR^d)$ but $\cF[\eta^k \varphi]$ generally does not have a compact support.
Thus we need to consider a subclass of $\bL_{0,2}\left(\opar 0,\tau \cbrk , \cP ; \cF^{-1}\cD'(\fR^d)   \right)$, which ensures that the stochastic term in \eqref{solution meaning} is well-defined.
In particular, we will show that the stochastic term 
$$
\int_0^t\left\langle h (s,\cdot) \eta^k(\cdot),\varphi\right\rangle \mathrm{d}B^k_s
:=\int_0^t 1_{\opar 0,\tau \cbrk}(s)\left\langle h (s,\cdot) \eta^k(\cdot),\varphi\right\rangle \mathrm{d}B^k_s
$$
is well-defined if 
$$
h \in \bL_{0,2,2}\left(\opar 0,\tau \cbrk \times \fR^d, \cP \times \cB(\fR^d)   \right),
$$ in Corollary \ref{well define sto}.
\end{rem}
Finally, here is our main result.
\begin{thm}
							\label{main thm}
Let $u_0 \in \cF^{-1}\bL_{0,1,\ell oc}\left( \Omega \times \fR^d, \rG \times \cB(\fR^d)\right)$, $f \in \cF^{-1}\bL_{0,1,1,\ell oc}\left( \opar 0,\tau \cbrk \times \fR^d, \rH \times \cB(\fR^d)\right)$, and
$$
h \in \bL_{0,2,2}\left(\opar 0,\tau \cbrk \times \fR^d, \cP \times \cB(\fR^d)   \right).
$$
Suppose that Assumptions \ref{main as} and \ref{main as 2} hold.
Then there exists a unique Fourier-space weak solution $u$ to \eqref{main eqn} in the intersection of the classes
$$
\cF^{-1}\bC L_{1,\ell oc}\left( \clbrk 0,\tau \cbrk,\sigma\left(   \rG \times \cB([0,\infty))  \cup \rH\cup \cP   \right)\times \cB(\fR^d)\right)
$$
and
\begin{align*}
 \cF^{-1}\bL_{0,1,1, \ell oc}\left( \opar 0,\tau \cbrk \times \fR^d, \sigma\left(  \rG \times \cB([0,\infty))  \cup \rH\cup \cP   \right) \times \cB(\fR^d),  |\psi(t,\xi)|\mathrm{d}t \mathrm{d}\xi\right),
\end{align*}
where $\sigma\left(\rG \times \cB([0,\infty))  \cup \rH\cup \cP   \right)$ denotes the smallest $\sigma$-algebras including all elements in $\rG\times \cB([0,\infty)) \cup \rH  \cup \cP$.
\end{thm}
The proof of Theorem \ref{main thm} will be given in Section \ref{pf main thm}.

\begin{rem}
One might consider our data $u_0$  and $f$ in the theorem to be somewhat regular since they possess (spatial) realizable frequency functions that are locally integrable. However, the local integrability of the frequency function does not impose any restrictions on the behaviors at large frequencies. Consequently, the data could still be irregular due to the unrestricted nature at the large frequencies.
\end{rem}

\begin{rem}
						\label{formal solution}
Formally, the solution $u$ in Theorem \ref{main thm} is expected to be
\begin{align*}
u(t,x)=
&\cF^{-1}\Bigg[\exp\left(\int_0^t\psi(r,\xi)\mathrm{d}r \right) \cF[u_0](\xi)
+\int_0^t  \exp\left(\int_s^t\psi(r,\xi)\mathrm{d}r \right) 1_{\opar 0,\tau \cbrk}(s)\cF[f(s,\cdot)](\xi)\mathrm{d}s  \\
&\quad + \int_0^t  \exp\left(\int_s^t\psi(r,\xi)\mathrm{d}r \right) 1_{\opar 0,\tau \cbrk}(s)\cF[h(s,\cdot)\eta^k(\cdot)](\xi)\mathrm{d}B^k_s \Bigg].
\end{align*}
Then this type of solution could be considered as a mild solution. 
However, this representation loses a mathematical rigor for a random symbol $\psi$, even though we developed new Fourier and inverse Fourier transforms. 

For a deterministic symbol $\psi$, however, this mild formulation has a certain meaning due to the new Fourier and inverse Fourier transforms.
Nevertheless, even with the deterministic symbol $\psi$, the solution is generally not a complex-valued function.
Instead, it is a $\cF^{-1}\cD'(\fR^d)$-valued stochastic process on $\clbrk 0,\tau \cbrk$, which is defined  according to the action on $\cF^{-1}\cD(\fR^d)$ as follows: for almost every $(\omega,t) \in \clbrk 0,\tau \cbrk$,
\begin{align*}
\langle u(t,\cdot), \varphi \rangle
&= \langle \cF[u(t,\cdot)], \cF[\varphi] \rangle \\
&= \int_{\fR^d} \left(\exp\left(\int_0^t\psi(r,\xi)\mathrm{d}r \right) \cF[u_0](\xi) \right) \overline{\cF[\varphi](\xi)} \mathrm{d}\xi \\
&\quad + \int_{\fR^d} \left( \int_0^t  \exp\left(\int_s^t\psi(r,\xi)\mathrm{d}r \right) 1_{\opar 0,\tau \cbrk}(s)\cF[f(s,\cdot)](\xi)\mathrm{d}s \right) \overline{\cF[\varphi](\xi)} \mathrm{d}\xi \\
&\quad + \int_{\fR^d} \left( \int_0^t  \exp\left(\int_s^t\psi(r,\xi)\mathrm{d}r \right) 1_{\opar 0,\tau \cbrk}(s)\cF[h(s,\cdot)\eta^k(s,\cdot)](\xi)\mathrm{d}B^k_s \right)  \overline{\cF[\varphi](\xi)} \mathrm{d}\xi \quad \forall \varphi \in \cF^{-1}\cD(\fR^d).
\end{align*}
On the other hand, the Fourier transform of the solution $u$  (with respect to the space variable) is a complex-valued function.
In this sense, we may say that our solution $u$ is strong with respect to the time and frequencies of the space variable.
\end{rem}

\begin{rem}
A similar theorem to Theorem \ref{main thm} can also be obtained even if the stopping time 
$\tau$ is allowed to be infinite at some points. 
At first glance, it seems sufficient to consider the open interval $\opar 0, \tau \cpar$
 instead of the half-interval $\opar 0,\tau\cbrk$ in \eqref{main eqn}. However, we did not assume that our filtration 
$\rF_t$ is right-continuous, and additionally, the measure of the interval 
$\opar 0,\tau \cpar$ could become infinite. Therefore, to develop this theory, we would need to find many appropriate versions and perform additional localization steps with respect to the time variable. We chose not to present this theorem in detail because these steps are typically tedious, and our settings are already sufficiently complicate.
\end{rem}

If our Fourier transforms of data $u_0$, $f$, and $h$ have certain finite (stochastic) moments, then we can derive certain stability results.
However, Assumptions \ref{main as} and \ref{main as 2} are not enough to gain the stability results since our symbol is random. 
Thus we introduce a stronger assumption.
\begin{assumption}[Local boundedness on the symbol]
									\label{main as sta}
For all  $R,T \in (0,\infty)$,
\begin{align}
\|\psi\|_{L_\infty\left( \Omega \times (0,T) \times B_R \right)}:=\esssup_{ (\omega,t, \xi) \in  \Omega \times (0,T) \times B_R}\left|\psi(t,\xi)\right|  < \infty.
\end{align}
\end{assumption}
\begin{rem}
It is clear that Assumption \ref{main as sta} encompasses both Assumptions \ref{main as} and \ref{main as 2}. This strong assumption is introduced to account for random symbols. Fortunately, the symbol for the fractional Laplacian operator with a complex exponent, as seen in Example \ref{frac exam}, continues to satisfy Assumption \ref{main as sta}. Moreover, we can even consider the fractional Laplacian operator with a random complex exponent, provided that the real part of the symbol is non-negative and uniformly bounded with respect to the sample points. Lastly, if we only consider non-random symbols, then Assumptions \ref{main as} and \ref{main as 2} are sufficient to establish the stability result, which will be demonstrated in Section \ref{exist deter}.
\end{rem}

\begin{corollary}
									\label{main cor}
Let $u_0 \in \cF^{-1}\bL_{1,1,\ell oc}\left( \Omega \times \fR^d, \rG \times \cB(\fR^d)\right)$, $f \in \cF^{-1}\bL_{1,1,1,loc,\ell oc}\left( \opar 0,\tau \cbrk \times \fR^d, \rH \times \cB(\fR^d)\right)$, and
$$
h \in \bL_{1,2,2,loc}\left(\opar 0,\tau \cbrk \times \fR^d, \cP \times \cB(\fR^d)   \right).
$$
Suppose that $\psi$ satisfies Assumption \ref{main as sta}.
Then there exists a unique Fourier-space weak solution  $u$ to \eqref{main eqn} belonging to the intersection of the two classes 
$$
\cF^{-1}\bL_{1,loc}\bC L_{1,\ell oc}\left( \clbrk 0,\tau \cbrk,\sigma\left(   \rG \times \cB([0,\infty))  \cup \rH\cup \cP   \right)\times \cB(\fR^d)\right)
$$
and
\begin{align*}
 \cF^{-1}\bL_{1,1,1,loc, \ell oc}\Big( \opar 0,\tau \cbrk \times \fR^d, \sigma\left( \rG \times \cB([0,\infty))  \cup \rH  \cup \cP \right) \times \cB(\fR^d), |\psi(t,\xi)|\mathrm{d}t \mathrm{d}\xi\Big).
\end{align*}
Additionally, the solution $u$ satisfies
\begin{align}
										\notag
\bE\left[ \int_0^{\tau \wedge T} \int_{B_R}|\psi(t,\xi)| |\cF[u(t,\cdot)](\xi)| \mathrm{d}\xi  \mathrm{d}t \right]  
										\notag
&\leq  \cC^1_{R,T,\psi} \bE \left[\int_{B_R} \left|\cF[u_0](\xi)\right|    \mathrm{d}\xi 
+\int_{B_R} \int_0^{ \tau \wedge T}\left|\cF[f(s,\cdot)](\xi)\right|  \mathrm{d}s  \mathrm{d}\xi \right] \\
										\label{main a priori}
&\quad+ R \cdot C_{BDG} \cdot \cC^1_{R,T,\psi}
 \bE\left[ \left(\int_0^{\tau \wedge T} \|h(t,\cdot)\|^2_{L_2(\fR^d)} \mathrm{d}t\right)^{1/2} \right]
\end{align}
and
\begin{align}
\bE\left[ \int_{B_R} \sup_{t \in [0,\tau \wedge T]} |\cF[u(t,\cdot)](\xi)|\mathrm{d}\xi\right] 
										\notag
&\leq \left(\cC^2_{R,T,\psi} \wedge \cC^3_{R,T,\psi}\right)\bE\left[\int_{B_R} |\cF[u_0](\xi)|\mathrm{d}\xi + \int_{B_R} \int_0^{\tau \wedge T}\left|\cF[f(s,\cdot)](\xi)\right|  \mathrm{d}s  \mathrm{d}\xi \right] \\
										\label{main a priori 2}
&\quad+ R \cdot C_{BDG} \cdot \left(\cC^2_{R,T,\psi} \wedge \cC^3_{R,T,\psi}\right)  \bE\left[ \left(\int_0^{\tau \wedge T} \|h(t,\cdot)\|^2_{L_2(\fR^d)} \mathrm{d}t\right)^{1/2} \right]
\end{align}
for all $R, T \in (0,\infty)$, where $\cC^i_{R,T,\psi}$ $(i=1,2,3)$ are positive constants which are explicitly given by
\begin{align*}
\cC^1_{R,T,\psi}=\left[ C_{R,T}^{\sup|\psi|}+2T\|\psi\|_{L_\infty\left( \Omega \times (0,T) \times B_R \right)} \cdot C_{R,T}^{\mathrm{e}\int\sup|\Re[\psi]|}  \right],
\end{align*}
\begin{align*}
\cC^2_{R,T,\psi}=C_{R,T}^{\mathrm{e}\int\sup|\Re[\psi]|}  \left[ 1+ 2T\|\psi\|_{L_\infty\left( \Omega \times (0,T) \times B_R \right)}C_{R,T}^{\mathrm{e}\int\sup|\Re[\psi]|}  \right],
\end{align*}
and
\begin{align*}
\cC^3_{R,T,\psi} =  1+\cC^1_{R,T,\psi}.
\end{align*}
\end{corollary}
\begin{rem}
\eqref{main a priori 2} clearly implies
\begin{align*}
\bE\left[ \sup_{t \in [0,\tau \wedge T]}  \int_{B_R} |\cF[u(t,\cdot)](\xi)|\mathrm{d}\xi\right] 
&\leq \left(\cC^2_{R,T,\psi} \wedge \cC^3_{R,T,\psi}\right)\bE\left[\int_{B_R} |\cF[u_0](\xi)|\mathrm{d}\xi + \int_{B_R} \int_0^{\tau \wedge T}\left|\cF[f(s,\cdot)](\xi)\right|  \mathrm{d}s  \mathrm{d}\xi \right] \\
&\quad+ R \cdot C_{BDG} \cdot \left(\cC^2_{R,T,\psi} \wedge \cC^3_{R,T,\psi}\right)  \bE\left[ \left(\int_0^{\tau \wedge T} \|h(t,\cdot)\|^2_{L_2(\fR^d)} \mathrm{d}t\right)^{1/2} \right],
\end{align*}
which appears to be a more conventional form in estimates for SPDEs. However, it does not suffice to show 
$$
u \in \cF^{-1}\bL_{1,loc}\bC L_{1,\ell oc}\left( \clbrk 0,\tau \cbrk,\sigma\left(   \rG \times \cB([0,\infty))  \cup \rH\cup \cP   \right)\times \cB(\fR^d)\right).
$$
\end{rem}

\begin{rem}
									\label{main cons rem}
Generally, it is hard to determine which constant within $\left(\cC^2_{R,T,\psi} \wedge \cC^3_{R,T,\psi}\right)$ is smaller since they are derived from different methods.
Additionally, the constant $\cC^2_{R,T,\psi}$ can be slightly smaller, which will be detailed in Remark \ref{20240730 rem}.
However, if our symbol $\psi$ is non-random, then the constants $C_{R,T}^{\mathrm{e}\int\sup|\Re[\psi]|}$ and $C_{R,T}^{\sup|\psi|}$
can be replaced by some smaller constants in Assumptions \ref{weaker as} and \ref{weaker as 2} (or \eqref{main deter as} and \eqref{main deter as 2}).
This makes it possible for $\cC^1_{R,T,\psi}$ and $\left(\cC^2_{R,T,\psi} \wedge \cC^3_{R,T,\psi}\right)$ to be substituted with significantly simpler constants $C_{R,T}^{|\psi|}$ and $C_{R,T}^{\mathrm{e}|\int\Re[\psi]|}$, respectively.
Additionally, for the deterministic symbol, we have
\begin{align*}
 \bE\left[ \int_0^{\tau \wedge T} \int_{B_R} |\cF[u(t,\cdot)](\xi)|\mathrm{d}\xi\right] \mathrm{d}t
										\notag
&\leq   T \cdot C_{R,T}^{\mathrm{e}\int\Re[\tilde\psi]} \cdot \bE \left[\int_{B_R} |\cF[u_0](\xi)|\mathrm{d}\xi 
+ \int_{B_R} \int_0^{\tau \wedge T}\left|\cF[f(s,\cdot)](\xi)\right|  \mathrm{d}s  \mathrm{d}\xi \right]
  \\
& \quad+ T \cdot C_{BDG} \cdot C_{R,T}^{\mathrm{e}\int\Re[\tilde\psi]} \bE\left[\int_{B_R} \left(\int_0^{\tau \wedge T}\left|\cF[g(s,\cdot)](\xi)\right|^2_{l_2}  \mathrm{d}s \right)^{1/2} \mathrm{d}\xi \right],
\end{align*}
where 
\begin{align*}
C^{\mathrm{e}\int\Re[\tilde \psi]}_{R,T}:=\esssup_{ 0\leq s \leq t \leq T, \xi \in   B_R} \left| \exp\left( \int_s^t\Re[\tilde \psi(r,\xi)]\mathrm{d}r  \right)  \right|   
< \infty.
\end{align*}
These substitute constants seem to be optimal since they are directly obtained from the kernels consisting of $\exp\left(\int_s^t\psi(r,\xi)\mathrm{d}r \right)$.
More details will be discussed in Corollary \ref{deter exist cor 2}.
\end{rem}

\begin{rem}
Due to Plancherel's theorem,  it is evident that
\begin{align*}
\|h(t,\cdot)\|^2_{L_2(\fR^d)}
=\|\cF[h(t,\cdot)]\|^2_{L_2(\fR^d)}.
\end{align*}
Thus \eqref{main a priori} and \eqref{main a priori 2} may indicate that the spatial Fourier transform of the solution $u$ to \eqref{main eqn} is governed by those of the data $u_0$, $f$, and $h$.
In other words, we could also say that 
the frequency function (with respect to the space variable) of the solution $u$ to \eqref{main eqn} is stable under the perturbations of the data $u_0$, $f$, and $h$ with respect to the spatial frequencies since our equation is linear.
\end{rem}

We also present the result without the random noise component to compare assumptions on $\psi$.
Specifically, we consider the deterministic version of \eqref{main eqn} as follows
\begin{align}
								\notag
&\mathrm{d}u=\left(\psi(t,-\mathrm{i}\nabla)u(t,x) + f(t,x)\right) \mathrm{d}t,\quad 
&(t,x) \in \Omega \times (0,\tau)\times \mathbf{R}^d,\\
&u(0,x)=u_0,\quad & x\in\mathbf{R}^d.
								\label{deter eqn}
\end{align}
Equation \eqref{deter eqn} remains random because the symbol $\psi$, the inhomogeneous data $f$, and the initial data $u_0$ could generally be random.
However, it is typically considered as a deterministic PDE since a solution can be easily obtained by solving the corresponding deterministic equation for each fixed sample point $\omega \in \Omega$.

This approach might raise concerns about whether the solution $u$, obtained by solving PDE for each $\omega \in \Omega$, is jointly measurable. Fortunately, when there is no stochastic (integral) term, joint measurability issues are minimal, making all mathematical conditions comparatively weaker. In particular, the conditions on $\psi$ can be relaxed in the absence of a stochastic term.

Therefore, we also examine the well-posedness of \eqref{deter eqn} under weaker conditions in this section, which is presented as an independent theorem in Theorem \ref{time deter thm}. We begin by proposing weaker assumptions on $\psi$.

\begin{assumption}
									\label{weaker as}
For all $R,T \in (0,\infty)$,
\begin{align*}
C^{\mathrm{e}\int\Re[\psi]}_{R,T}:=\esssup_{0\leq s \leq t \leq T, \xi \in B_R} \left| \exp\left( \int_s^t\Re[\psi(r,\xi)]\mathrm{d}r  \right)  \right|   
< \infty \quad (a.s.).
\end{align*}
\end{assumption}

\begin{assumption}
									\label{weaker as 2}

For all $R,T \in (0,\infty)$,
\begin{align*}
C^{|\psi|}_{R,T}:=\esssup_{ \xi \in  B_R} \left[ \int_0^T|\psi(t,\xi)|\sup_{ 0\leq s \leq t} \exp\left( \left|\int_s^t \Re[\psi(r,\xi)] \mathrm{d}r \right| \right)    \mathrm{d}t  \right]
< \infty \quad (a.s.).
\end{align*}
\end{assumption}
\begin{rem}
If the sign of $\Re[\psi(t,\xi)]$ is preserved, \textit{i.e.}
$$
\Re[\psi(t,\xi)] \geq 0 \quad \text{or} \quad \Re[\psi(t,\xi)] \leq 0 \quad \forall t,\xi,
$$
then the constant $C^{\mathrm{e}\int\Re[\psi]}_{R,T}$ has both clear upper and lower bounds.
For instance, if $\Re[\psi(t,\xi)] \leq 0 $ for all $t$ and $\xi$, then
\begin{align*}
\esssup_{\xi \in B_R} \left| \exp\left( \int_0^T\Re[\psi(r,\xi)]\mathrm{d}r  \right)  \right|   
\leq 
C^{\mathrm{e}\int\Re[\psi]}_{R,T}
\leq 1.
\end{align*}
This demonstrates that all constants associated with the symbol can be significantly simplified if the symbol meets an ellipticity condition.
\end{rem}

\begin{rem}
										\label{rem 20240730 10}
It is clear that Assumptions \ref{weaker as} and \ref{weaker as 2} are less stringent than Assumptions \ref{main as} and \ref{main as 2} as indicated by the inequalities
\begin{align}
										\label{20240721 50}
C^{\mathrm{e}\int\Re[\psi]}_{R,T} \leq C^{\mathrm{e}\int\sup|\Re[\psi]|}_{R,T} \quad (a.s.)
\end{align}
and
\begin{align*}
C^{|\psi|}_{R,T} \leq C^{\sup|\psi|}_{R,T} \quad (a.s.).
\end{align*}
In other words, Assumptions \ref{main as} and \ref{main as 2} are sufficient to ensure that Assumptions \ref{weaker as} and \ref{weaker as 2} are satisfied. 
In particular, 
$C^{\mathrm{e}\int\Re[\psi]}_{R,T} \leq C^{\mathrm{e}\int\sup|\Re[\psi]|}_{R,T}$ and $C^{|\psi|}_{R,T} \leq C^{\sup|\psi|}_{R,T}$
if the symbol is deterministic.
Note that the first equality is not generally true even though the symbol is deterministic.
Additionally, Assumption \ref{weaker as} is significant because this allows our main operators to encompass logarithmic operators, such as the logarithmic Laplacian (cf. \cite[Theorem 2.27]{Choi Kim 2024}).
Specifically, if $\psi(t,\xi) = \log |\xi|^2$, then $C^{\mathrm{e}\int\Re[\psi]}_{R,T} <\infty$ but $C^{\mathrm{e}\int\sup|\Re[\psi]|}_{R,T}=\infty$.

Moreover, the requirement in Assumption \ref{weaker as 2} can be somewhat relaxed to include logarithmic operators. 
This will be further elaborated in Remark \ref{20240803 rem 30}. 
In spite of this shortcoming, we opted for Assumption \ref{weaker as 2} over the weaker condition to maintain consistency with Assumption \ref{main as 2} since our primary focus is on SPDEs.  
Notably, our results do not encompass \eqref{main eqn} with logarithmic operators, and we firmly believe that theories involving logarithmic operators for SPDEs are not feasible due to intricate joint measurability issues, even though they are feasible for PDEs.

Lastly, if $\psi(t,\xi)$ is deterministic and $\Re[\psi(t,\xi)]$ is non-negative, then
$C^{\mathrm{e}\int\Re[\psi]}_{R,T}=C^{\mathrm{e}\int\sup|\Re[\psi]|}_{R,T}$.
This is why we use similar notations $C^{\mathrm{e}\int\Re[\psi]}_{R,T}$ and $C^{\mathrm{e}\int\sup|\Re[\psi]|}_{R,T}$ even though the supremum with respect to the time variable is taken in different ways. 
\end{rem}

\begin{thm}
							\label{time deter thm}
Let
\begin{align*}
u_0 \in \cF^{-1}\bL_{0,1,\ell oc}\left( \Omega \times \fR^d , \rG \times \cB(\fR^d)\right)
\quad \text{and} \quad
f \in \cF^{-1}\bL_{0,1,1,\ell oc}\left( \opar 0, \tau \cbrk \times \fR^d, \rH \times \cB(\fR^d) \right).
\end{align*}
Suppose that $\psi$ satisfies  Assumption \ref{weaker as} and Assumption \ref{weaker as 2}.
Then there exists a unique Fourier-space weak solution  $u$ to \eqref{deter eqn} in the intersection of the two classes 
$$
\cF^{-1}\bC L_{1,\ell oc}\left( \clbrk 0,\tau \cbrk,\sigma\left(   \rG \times \cB([0,\infty))  \cup \rH\cup \cP   \right)\times \cB(\fR^d)\right)
$$
and
$$
 \cF^{-1}\bL_{0,1,1, \ell oc}\left( \opar 0,\tau \cbrk \times \fR^d, \sigma\left( \rG \times \cB([0,\infty))  \cup \rH  \cup \cP   \right) \times \cB(\fR^d), |\psi(t,\xi)|\mathrm{d}t \mathrm{d}\xi\right).
$$
Moreover, the solution $u$ satisfies 
\begin{align}
										\label{202040524 20}
 \int_{B_R} \sup_{t \in [0,\tau \wedge T]}  |\cF[u(t,\cdot)](\xi)| \mathrm{d}\xi  
\leq   C_{R,T}^{\mathrm{e}\int\Re[\psi]} \left[\int_{B_R} |\cF[u_0](\xi)|\mathrm{d}\xi + \int_{B_R} \int_0^{\tau \wedge T}\left|\cF[f(s,\cdot)](\xi)\right|  \mathrm{d}s  \mathrm{d}\xi \right]
\end{align}
and
\begin{align}
										\label{202040524 21}
 \int_0^{\tau \wedge T} \int_{B_R}|\psi(t,\xi)||\cF[u(t,\cdot)](\xi)| \mathrm{d}\xi  \mathrm{d}t 
\leq
C_{R,T}^{|\psi|}
\left[\int_{B_R} |\cF[u_0](\xi)|\mathrm{d}\xi + \int_{B_R} \int_0^{\tau \wedge T}\left|\cF[f(s,\cdot)](\xi)\right|  \mathrm{d}s  \mathrm{d}\xi \right]
\end{align}
with probability one.
\end{thm}
The proof of Theorem \ref{time deter thm} is given in Section \ref{pf time thm}.

\begin{rem}
										\label{rem define psi}
Formally, the kernel to \eqref{deter eqn} (or \eqref{main eqn}) is given by
\begin{align*}
\cF_{\xi}^{-1}\left[ \exp \left( \int_s^t \psi(r,\xi) \mathrm{d}r \right)\right](x), \quad (0\leq s \leq t~\text{and}~ x,\xi \in \fR^d)
\end{align*}
which serves as a solution to
\begin{align*}
&\mathrm{d}u=\psi(t,-\mathrm{i}\nabla)u(t,x) \mathrm{d}t,\quad &(t,x) \in \Omega \times (s,\infty)\times \mathbf{R}^d,\\
&u(s,x)=\delta_0,\quad & x\in\mathbf{R}^d,
\end{align*}
where $\delta_0$ is the Dirac delta centered at zero in $\fR^d$.
Thus it is important to control the functions $\exp \left( \int_s^t \psi(r,\xi) \mathrm{d}r \right)$. 
It seems that the imaginary part of these functions could be arbitrary since
\begin{align*}
\left|\exp \left( \int_s^t \psi(r,\xi) \mathrm{d}r \right) \right|
=\left|\exp \left( \int_s^t \Re[\psi(r,\xi)] \mathrm{d}r \right) \right|.
\end{align*}
However, the imaginary part of $\psi(r,\xi)$ must be integrable on any interval $(s,t)$ to make the functions $\exp \left( \int_s^t \psi(r,\xi) \mathrm{d}r \right)$ well-defined even though these imaginary values do not affect estimates of the exponential functions.  A sufficient condition to ensure the local integrability on $\Im[\psi(t,\xi)]$ is provided by \eqref{20240801 01}.

Recall that we assumed the weaker condition, Assumption \ref{weaker as 2}, than Assumption \ref{main as 2} in Theorem \ref{time deter thm}.
However, \eqref{20240801 01} still can be derived from  Assumption \ref{weaker as 2}.
Thus we did not independently add the condition \eqref{20240801 01} in Theorem \ref{time deter thm}.
\end{rem}

\begin{rem}
If the spatial Fourier transforms of our data $u_0$  and $f$ have finite moments, then that of the solution 
$u$ (including those amplified by $\psi$) also has a finite moment, as both \eqref{202040524 20} and \eqref{202040524 21} hold almost surely. Additionally, \eqref{202040524 20} and \eqref{202040524 21} can be used to obtain various moment estimates of the solutions. For further details on these moment estimates, we refer to \cite{KID 2022}.
\end{rem}

\mysection{A well-posedness theory to SPDEs driven by white noises in time}

In this section, we examine SPDEs driven by white noises in time as follows:
\begin{align}
								\notag
&\mathrm{d}u=\left(\psi(t,-\mathrm{i}\nabla)u(t,x) + f(t,x)\right) dt + g^k(t,x)\mathrm{d}B^k_t,\quad &(t,x) \in \Omega \times (0,\tau)\times \mathbf{R}^d,\\
&u(0,x)=u_0,\quad & x\in\mathbf{R}^d.
								\label{time eqn}
\end{align}
Here $B_t^k$ $(k=1,2,\ldots)$ are independent one-dimensional Brownian motions (Wiener processes) on $\Omega$, $\psi(t,\xi)$ is a (complex-valued) $\cP \times \cB(\fR^d)$ function defined on $\Omega \times [0,\infty) \times \fR^d$ and $\psi(t,-\mathrm{i}\nabla)$ is a (time-measurable) pseudo-differential operator with the symbol $\psi(t,\xi)$, i.e.
\begin{align}
								\label{ab op}
\psi(t,-\mathrm{i}\nabla)u(t,x):=\cF^{-1}\left[\psi(t,\cdot)\cF[u](t,\cdot)\right](x).
\end{align}
It is well-known that \eqref{main eqn} is a special case of \eqref{time eqn}.
Consequently, there have been many attempts to study stochastic partial differential equations driven by space-time white noises 
by using theories of corresponding SPDEs driven by a sequence of independent one-dimensional Brownian motions,
which is already mentioned before Definition \ref{space weak solution}.
Therefore we introduce the definition of our weak solutions to \eqref{time eqn}, which includes Definition \ref{space weak solution}.

\begin{defn}[Fourier-space weak solution to \eqref{time eqn}]
									\label{space weak solution 2}
Let $u_0 \in \bL_0\left(\Omega,\rF;\cF^{-1}\cD'(\fR^d) \right)$, 
\begin{align*}
f \in \bL_{0,1, loc}\left( \opar 0,\tau \cbrk, \rF \times \cB([0,\infty));\cF^{-1}\cD'(\fR^d)\right), 
\quad g\in \bL_{0,2,loc}\left( \opar 0,\tau \cbrk, \cP;\cF^{-1}\cD'(\fR^d;l_2)\right),
\end{align*}
and 
$$
u \in  \bL_{0}\left( \clbrk 0,\tau \cbrk, \rF \times \cB( [0,\infty) );\cF^{-1}\cD'(\fR^d)\right).
$$
We say that $u$ is a {\bf Fourier-space weak solution} to equation \eqref{time eqn} if for any $\varphi \in \cF^{-1}\cD(\fR^d)$,
\begin{align}
									\notag
\left\langle u(t,\cdot),\varphi \right\rangle 
&= \langle u_0, \varphi \rangle +  \int_0^t  \left\langle  \psi(s,-\mathrm{i} \nabla)u(s,\cdot) , \varphi \right\rangle\mathrm{d}s 
+ \int_0^t \left\langle f(s,\cdot),\varphi\right\rangle \mathrm{d}s \\
									\label{20240703 20}
& \quad +\int_0^t \left\langle g^k(s,\cdot),\varphi\right\rangle \mathrm{d}B^k_s
\quad  (a.e.) \quad (\omega,t) \in \clbrk 0,\tau \cbrk.
\end{align}
\end{defn}
\begin{rem}
										\label{20240718 10}
Let $g\in \bL_{0,2,loc}\left( \opar 0,\tau \cbrk, \cP;\cF^{-1}\cD'(\fR^d;l_2)\right)$
and $\varphi \in \cF^{-1}\cD(\fR^d)$.
Then for all $T \in (0,\infty)$,
\begin{align*} 
\int_0^T 1_{\opar 0 , \tau \cbrk}(s) \left|\langle g,\varphi \rangle  \right|^2_{l_2} <\infty \quad (a.s.).
\end{align*}
Thus the stochastic integral term
\begin{align*}
\int_0^t 1_{\opar 0 , \tau \cbrk}(s) \left\langle g^k(s,\cdot),\varphi\right\rangle \mathrm{d}B^k_s
\end{align*}
is well-defined for all $(\omega,t) \in \Omega \times [0,\infty)$.
In particular, for any $(\omega,t) \in \clbrk 0, \tau \cbrk$, the term
\begin{align*}
\int_0^t \left\langle g^k(s,\cdot),\varphi\right\rangle \mathrm{d}B^k_s
:= \int_0^t  1_{\opar 0 , \tau \cbrk}(s)\left\langle g^k(s,\cdot),\varphi\right\rangle \mathrm{d}B^k_s
\end{align*}
in \eqref{20240703 20} is well-defined.
Moreover, if 
$$
g \in \cF^{-1}\bL^{\omega,t,\xi}_{0,2,1,\ell oc}\left( \opar 0,  \tau \cbrk \times \fR^d, \cP \times \cB(\fR^d) ; l_2\right),
$$
then the well-defined stochastic term is given by
\begin{align*}
\int_0^t \left\langle g^k(s,\cdot),\varphi\right\rangle \mathrm{d}B^k_s
=\int_0^t \left(\int_{\fR^d} \cF[g^k(s,\cdot)](\xi) \overline{\cF[\varphi](\xi)} \mathrm{d}\xi\right) \mathrm{d}B^k_s
\end{align*}
since for any $T \in (0,\infty)$
\begin{align*}
&\int_0^T 1_{\opar 0 , \tau \cbrk}(s)\left|\int_{\fR^d} \cF[g(s,\cdot)](\xi) \overline{\cF[\varphi](\xi)} \mathrm{d}\xi\right|_{l_2}^2 \mathrm{d}s \\
&\leq  
\left(\sup_{\xi \in \fR^d} \left|\cF[\varphi]\right|\right)\int_0^T 1_{\opar 0 , \tau \cbrk}(s)\left(\int_{B_{R_\varphi}} |\cF[g(s,\cdot)](\xi)|_{l_2}  \mathrm{d}\xi\right)^2 \mathrm{d}s
< \infty \quad (a.s.),
\end{align*}
where ${R_\varphi}$ is a positive constant so that $\supp \cF[\varphi] \subset B_{R_\varphi}$.
In particular, for any 
$$
g \in \cF^{-1}\bL^{\omega,\xi,t}_{0,1,2,\ell oc}\left( \opar 0,  \tau \cbrk \times \fR^d, \cP \times \cB(\fR^d) ; l_2\right),
$$
the stochastic term is well-defined due to the generalized Minkowski inequality.
Indeed, for any $R \in (0,\infty)$,
\begin{align*}
\left(\int_0^\tau \left|\int_{B_R} |\cF[g(t,\cdot)] (\xi)| \mathrm{d}\xi \right|_{l_2}^2\mathrm{d}t\right)^{1/2}
\leq \int_{B_R} \left[\int_0^\tau \left|\cF[g(t,\cdot)] (\xi)  \right|_{l_2}^2\mathrm{d}t \right]^{1/2} \mathrm{d}\xi.
\end{align*}

\end{rem}

\begin{thm}
							\label{time thm}
Let 
$$
u_0 \in \cF^{-1}\bL_{0,1,\ell oc}\left( \Omega \times \fR^d , \rG \times \cB(\fR^d)\right),  
\quad f \in \cF^{-1}\bL_{0,1,1,\ell oc}\left( \opar 0, \tau \cbrk \times \fR^d, \rH \times \cB(\fR^d) \right),
$$
and 
$$
g \in \cF^{-1}\bL^{\omega,\xi,t}_{0,1,2,\ell oc}\left( \opar 0,  \tau \cbrk \times \fR^d, \cP \times \cB(\fR^d) ; l_2\right).
$$
Suppose that $\psi$ satisfies Assumptions \ref{main as} and \ref{main as 2}.
Then there exists a unique Fourier-space weak solution  $u$ to \eqref{time eqn} in the intersection of the two classes 
$$
\cF^{-1}\bC L_{1,\ell oc}\left( \clbrk 0,\tau \cbrk,\sigma\left(   \rG \times \cB([0,\infty))  \cup \rH\cup \cP   \right)\times \cB(\fR^d)\right)
$$
and
\begin{align*}
 \cF^{-1}\bL_{0,1,1, \ell oc}\left( \opar 0,\tau \cbrk \times \fR^d, \sigma\left( \rG \times \cB([0,\infty))  \cup \rH  \cup \cP  \right) \times \cB(\fR^d) , |\psi(t,\xi)|\mathrm{d}t \mathrm{d}\xi \right).
\end{align*}
\end{thm}

\begin{rem}
If $u_0$ is $\rF_0 \times \cB(\fR^d)$-measurable and $f$ is $\cP \times \cB(\fR^d)$-measurable then
$$
u \in \cF^{-1}\bC L_{1,\ell oc}\left( \clbrk 0,\tau \cbrk,\cP\times \cB(\fR^d)\right)
$$
and
$$
u \in \cF^{-1}\bL_{0,1,1,\ell oc}\left( \opar 0,\tau \cbrk \times \fR^d, \cP, |\psi(t,\xi)|\mathrm{d}t \mathrm{d}\xi\right).
$$
This property is crucial for extending our results to treat multiplicative noise, which is an interesting direction for future research.
\end{rem}

\begin{rem}
The class handling stochastic inhomogeneous data $g$ is 
$$
\cF^{-1}\bL^{\omega,\xi,t}_{0,1,2,\ell oc}\left( \opar 0,  \tau \cbrk \times \fR^d, \cP \times \cB(\fR^d) ; l_2\right)
$$
in Theorem \ref{time thm}.
This class is smaller than the class
$$
\cF^{-1}\bL^{\omega,t,\xi}_{0,2,1,\ell oc}\left( \opar 0,  \tau \cbrk \times \fR^d, \cP \times \cB(\fR^d) ; l_2\right)
$$
mentioned in Remark \ref{20240718 10} for defining the stochastic term in \eqref{time eqn}.
This restriction on the stochastic inhomogeneous data is important to applying stochastic Fubini theorems.
More specifics will be discussed later in upcoming sections.
\end{rem}

\begin{rem}
								\label{u meaning}
Remark \ref{formal solution} can be reiterated in relation to Theorem \ref{time thm}. The solution $u$ in Theorem \ref{time thm} is generally not a mild solution. However, in cases where $\psi$ is deterministic, it can be considered as a mild solution within our new weak formulation. 
In order to demonstrate this, assume that the symbol $\psi(t,\xi)$ is non-random.
Then the solution $u$ in Theorem \ref{time thm} is  an $\cF^{-1}\cD'(\fR^d)$-valued stochastic process defined on $\clbrk 0, \tau \cbrk$ as follows: for any $(\omega, t) \in \clbrk 0, \tau \cbrk$,
\begin{align}
										\notag
\varphi \in \cF^{-1}\cD(\fR^d) \mapsto \langle u(t,\cdot), \varphi \rangle
&= \int_{\fR^d} \left(\exp\left(\int_0^t\psi(r,\cdot)\mathrm{d}r \right) \cF[u_0](\xi) \right) \overline{\cF[\varphi](\xi)} \mathrm{d}\xi \\
										\notag
&\quad + \int_{\fR^d} \left(\int_0^t  \exp\left(\int_s^t\psi(r,\xi)\mathrm{d}r \right) 1_{\opar 0,\tau \cbrk}(s)\cF[f(s,\cdot)](\xi) \mathrm{d}s  \right) \overline{\cF[\varphi](\xi)}  \mathrm{d}\xi \\
										\label{solution candidate}
&\quad + \int_{\fR^d} \left(\int_0^t  \exp\left(\int_s^t\psi(r,\xi)\mathrm{d}r \right) 1_{\opar 0,\tau \cbrk}(s)\cF[g^k(s,\cdot)](\xi) \mathrm{d}B^k_s \right) \overline{\cF[\varphi](\xi)} \mathrm{d}\xi.
\end{align}
It is important to note that the function $u$ in \eqref{solution candidate} is not defined as a complex-valued function pointwisely.
However, the Fourier transform of $u$ with respect to the space variable is a complex-valued function $(a.e.)$ defined on $\clbrk 0,\tau \cbrk \times \fR^d$
based on the canonical identification with a locally integrable function and a distribution. In other words,
\begin{align}
										\notag
\cF[u(t,\cdot)](\xi)
&= \exp\left(\int_0^t\psi(r,\cdot)\mathrm{d}r \right) \cF[u_0](\xi)
+\int_0^t  \exp\left(\int_s^t\psi(r,\xi)\mathrm{d}r \right) 1_{\opar 0,\tau \cbrk}(s)\cF[f(s,\cdot)](\xi) \mathrm{d}s  \\
										\label{20240803 40}
&\quad +\int_0^t  \exp\left(\int_s^t\psi(r,\xi)\mathrm{d}r \right) 1_{\opar 0,\tau \cbrk}(s)\cF[g^k(s,\cdot)](\xi) \mathrm{d}B^k_s, 
\end{align}
where all terms in the right-hand side are obviously complex-valued functions $(a.e.)$ defined on $\clbrk 0,\tau \cbrk \times \fR^d$ due to the assumptions in Theorem \ref{time thm}. 
Thus we may say that our solution has a strong or realizable frequency function (with respect to the space variable).

Moreover, observe that all functions appearing in \eqref{20240803 40} are locally integrable, which means they can be regarded as distributions.
However, they generally do not qualify as tempered distributions because of insufficient conditions on the data $u_0$, $f$, and $g$. 
Hence $u$ cannot be derived from \eqref{20240803 40} by taking the classical inverse Fourier transform applicable for tempered distributions. 
However, our new inverse Fourier transform, as defined in Definition \ref{fourier defn 2}, enables us to recover $u$ from \eqref{20240803 40} as $\cF^{-1}\cD'(\fR^d)$-valued stochastic processes by applying it with respect to the spatial frequencies.
This idea is the starting point to find a unique solution $u$ to \eqref{time eqn} rigorously.
\end{rem}

\begin{rem}
It is readily observable that the right-hand side of \eqref{solution candidate} is defined for all $t \in (0,\infty)$, which implies that the solution $u$ can nicely be extended to $\Omega \times [0,\infty) \times \fR^d$.
More precisely, there exists a $\cF^{-1}\cD'(\fR^d)$-valued stochastic process $\tilde u$  on $\Omega \times [0,\infty)$ such that for all $(\omega,t) \in \opar 0, \tau \cbrk$,
$\tilde u(\omega,t) = u(\omega,t) $ as $\cF^{-1}\cD'(\fR^d)$-valued functions.
Additionally, $\tilde u$ belongs to  the intersection of the classes
$$
\cF^{-1}\bC_{loc} L_{1,\ell oc}\left(  \Omega \times [0,\infty),\sigma\left(   \rG \times \cB([0,\infty))  \cup \rH\cup \cP   \right)\times \cB(\fR^d)\right)
$$
and
$$
 \cF^{-1}\bL_{0,1,1,loc,\ell oc}\left(\Omega \times (0,\infty) \times \fR^d, \sigma\left( \rG \times \cB([0,\infty)) \cup \rH  \cup \cP \right) \times \cB(\fR^d) , |\psi(t,\xi)|\mathrm{d}t \mathrm{d}\xi\right),
$$
where
\begin{align*}
&u \in \cF^{-1}\bC_{loc} L_{1,\ell oc}\left(  \Omega \times [0,\infty),\sigma\left(   \rG \times \cB([0,\infty))  \cup \rH\cup \cP   \right)\times \cB(\fR^d)\right) \\
&\iff u \in \cF^{-1}\bC L_{1,\ell oc}\left(  \Omega \times [0,T],\sigma\left(   \rG \times \cB([0,\infty))  \cup \rH\cup \cP   \right)\times \cB(\fR^d)\right) \quad \forall T \in (0,\infty) 
\end{align*}
and
\begin{align*}
&u \in \cF^{-1}\bL_{0,1,1,loc,\ell oc}\left(\Omega \times [0,\infty) \times \fR^d, \sigma\left( \rG \times \cB([0,\infty)) \cup \rH  \cup \cP \right) \times \cB(\fR^d) , |\psi(t,\xi)|\mathrm{d}t \mathrm{d}\xi\right) \\
&\iff \cF^{-1}\bL_{0,1,1,loc,\ell oc}\left(\Omega \times (0,T] \times \fR^d, \sigma\left( \rG \times \cB([0,\infty)) \cup \rH  \cup \cP \right) \times \cB(\fR^d), |\psi(t,\xi)|\mathrm{d}t \mathrm{d}\xi\right) \quad \forall T \in (0,\infty).
\end{align*}

These extensions will be revisited in later sections to demonstrate the existence of solutions with deterministic symbols.
\end{rem}

It is naturally anticipated that the solution $u$ has a finite expectation if the data have finite expectations. This rule also applies to their spatial Fourier transforms, as previously shown in Corollary \ref{main cor}. A similar result can be established for the relationship between the solution and the data in \eqref{time eqn}.
Moreover, the solution mapping $(u_0,f,g) \mapsto u$ is continuous in terms of the spatial Fourier transforms.
However, for the stability result to encompass random symbols, $\psi$ should satisfy a stronger condition, as emphasized in Corollary \ref{main cor}.

\begin{corollary}
									\label{time corollary}
Let
\begin{align*}
u_0 \in \cF^{-1}\bL_{1,1,\ell oc}\left( \Omega \times \fR^d , \rG \times \cB(\fR^d)\right),  
\quad f \in \cF^{-1}\bL_{1,1,1,loc,\ell oc}\left( \opar 0, \tau \cbrk \times \fR^d, \rH \times \cB(\fR^d) \right),
\end{align*}
and 
$$
g \in \cF^{-1}\bL^{\omega,\xi,t}_{1,1,2,loc,\ell oc}\left( \opar 0,  \tau \cbrk \times \fR^d, \cP \times \cB(\fR^d) ; l_2\right).
$$
Suppose that $\psi$ satisfies Assumption \ref{main as sta}.
Then there exists a unique Fourier-space weak solution  $u$ to \eqref{time eqn} in the intersection of the two classes 
$$
\cF^{-1}\bL_{1,loc}\bC L_{1,\ell oc}\left( \clbrk 0,\tau \cbrk \times \fR^d, \sigma\left( \rG \times \cB([0,\infty))  \cup \rH\cup \cP \right)\times \cB(\fR^d) \right)
$$
and
$$
 \cF^{-1}\bL_{1,1,1, loc,\ell oc}\left( \opar 0,\tau \cbrk \times \fR^d, \sigma\left( \rG \times \cB([0,\infty))  \cup \rH  \cup \cP  \right) \times \cB(\fR^d) , |\psi(t,\xi)|\mathrm{d}t \mathrm{d}\xi\right).
$$
Moreover, the solution $u$ satisfies
\begin{align}
										\notag
\bE\left[ \int_0^{\tau \wedge T} \int_{B_R}|\psi(t,\xi)| |\cF[u(t,\cdot)](\xi)| \mathrm{d}\xi  \mathrm{d}t \right]  
										\notag
&\leq  \cC^1_{R,T,\psi} \bE \left[\int_{B_R} \left|\cF[u_0](\xi)\right|    \mathrm{d}\xi 
+\int_{B_R} \int_0^{ \tau \wedge T}\left|\cF[f(s,\cdot)](\xi)\right|  \mathrm{d}s  \mathrm{d}\xi \right] \\
									\label{linear a priori}
&\quad+ C_{BDG} \cdot \cC^1_{R,T,\psi} \bE\left[\int_{B_R} \left(\int_0^{\tau \wedge T}\left|\cF[g(s,\cdot)](\xi)\right|^2_{l_2}  \mathrm{d}s \right)^{1/2} \mathrm{d}\xi \right]
\end{align}
and
\begin{align}
\bE\left[ \int_{B_R} \sup_{t \in [0,\tau \wedge T]} |\cF[u(t,\cdot)](\xi)|\mathrm{d}\xi\right] 
										\notag
&\leq \left(\cC^2_{R,T,\psi} \wedge \cC^3_{R,T,\psi}\right)\bE\left[\int_{B_R} |\cF[u_0](\xi)|\mathrm{d}\xi + \int_{B_R} \int_0^{\tau \wedge T}\left|\cF[f(s,\cdot)](\xi)\right|  \mathrm{d}s  \mathrm{d}\xi \right] \\
									\label{linear a priori 2}
&\quad+ C_{BDG} \cdot \left(\cC^2_{R,T,\psi} \wedge \cC^3_{R,T,\psi}\right)\bE\left[\int_{B_R} \left(\int_0^{\tau \wedge T}\left|\cF[g(s,\cdot)](\xi)\right|^2_{l_2}  \mathrm{d}s \right)^{1/2} \mathrm{d}\xi \right],
\end{align}
for all $R, T \in (0,\infty)$, where $\cC^i_{R,T,\psi}$ $(i=1,2,3)$ are positive constants which are explicitly given by
\begin{align*}
\cC^1_{R,T,\psi}=\left[ C_{R,T}^{\sup|\psi|}+2T\|\psi\|_{L_\infty\left( \Omega \times (0,T) \times B_R \right)} \cdot C_{R,T}^{\mathrm{e}\int\sup|\Re[\psi]|}  \right],
\end{align*}
\begin{align*}
\cC^2_{R,T,\psi}=C_{R,T}^{\mathrm{e}\int\sup|\Re[\psi]|}  \left[ 1+ 2T\|\psi\|_{L_\infty\left( \Omega \times (0,T) \times B_R \right)}C_{R,T}^{\mathrm{e}\int\sup|\Re[\psi]|}  \right],
\end{align*}
and
\begin{align*}
\cC^3_{R,T,\psi} =  1+\cC^1_{R,T,\psi}.
\end{align*}
\end{corollary}
The proofs of Theorem \ref{time thm} and Corollary \ref{time corollary} are given in Section \ref{pf time thm}.

\begin{rem}
$\cC^1_{R,T,\psi}$ and $\left(\cC^2_{R,T,\psi} \wedge \cC^3_{R,T,\psi}\right)$ can be optimized to $C_{R,T}^{|\psi|}$ and $C_{R,T}^{\mathrm{e}\int\Re[\psi]}$ if the symbol $\psi$ is non-random as discussed in Remark \ref{main cons rem}.
\end{rem}
\begin{rem}
Observe that the constants in \eqref{linear a priori} and \eqref{linear a priori 2} strongly depend on each $T$.
Thus there is no guarantee that the solution $u$ belongs to the smaller spaces such as
$$
\cF^{-1}\bL_1\bC L_{1,\ell oc}\left( \clbrk 0,\tau \cbrk \times \fR^d, \sigma\left( \rG \times \cB([0,\infty))  \cup \rH\cup \cP \right)\times \cB(\fR^d) \right)
$$
and
$$
 \cF^{-1}\bL_{1,1,1, \ell oc}\left( \opar 0,\tau \cbrk \times \fR^d, \sigma\left( \rG \times \cB([0,\infty))  \cup \rH  \cup \cP  \right) \times \cB(\fR^d) , |\psi(t,\xi)|\mathrm{d}t \mathrm{d}\xi\right)
$$
even if our data satisfy the better conditions that
$$
f \in \cF^{-1}\bL_{1,1,1,\ell oc}\left( \opar 0, \tau \cbrk \times \fR^d, \rH \times \cB(\fR^d)|_{\opar 0, \tau \cbrk \times \fR^d} \right)
$$
and 
$$
g \in \cF^{-1}\bL^{\omega,\xi,t}_{1,1,2,\ell oc}\left( \opar 0,  \tau \cbrk \times \fR^d, \cP \times \cB(\fR^d)|_{\opar 0,  \tau \cbrk \times \fR^d} ; l_2\right).
$$
It is because the BDG inequality does not perfectly work for our solution $u$, which is discussed further in Remark \ref{BDG fail}.

\end{rem}

\mysection{Stochastic Fubini theorems}
											\label{section stochastic Fubini}

Our estimates can be derived using approximations by simple stochastic processes, such as step processes. However, obtaining rigorous theories through these approximations is challenging since our solutions and data generally lack finite moments. Instead, it is more effective to directly estimate many stochastic terms using the stochastic Fubini theorems.

Recent advancements in the stochastic Fubini theorems (\textit{cf.} \cite{NV 2005, Krylov 2011, Veraar 2012}) have been made. 
We rigorously develop our theories with the help of these theorems. Here, we introduce a version of the stochastic Fubini theorem that is well-suited to our applications.

\begin{thm}[Stochastic Fubini theorem]
									\label{stochastic Fubini}
Let $T\in (0,\infty)$, $A \in \cB(\fR^d)$,  
$$
F \in \bL_{0}\left( \opar 0,\tau \cbrk \times \fR^d, \rH \times \cB(\fR^d) \right),
$$
and $G \in\bL_{0}\left( \opar 0,\tau  \cbrk \times \fR^d, \cP \times \cB(\fR^d); l_2\right)$.
Assume that $|A| < \infty$,
\begin{align*}
 \int_{A} \int_0^{T} 1_{\opar 0,\tau  \cbrk}(t)\left|F(\omega,t,x) \right|\mathrm{d}t \mathrm{d}x
< \infty \quad (a.s.)
\end{align*}
and
\begin{align*}
 \int_{A} \left(\int_0^{T} 1_{\opar 0,\tau  \cbrk}(t) \left|G(\omega,t,x) \right|^2_{l_2}\mathrm{d}t \right)^{1/2} \mathrm{d}x
< \infty \quad (a.s.),
\end{align*}
where $|A|$ denotes the Lebesgue measure of $A$.
Then there exists a $\sigma\left( \rH  \cup \cP  \right) \times \cB(\fR^d)|_{\Omega \times [0,T] \times A}$-measurable function $H(\omega,t,x)$ on $\Omega \times [0,T] \times A$ such that
\begin{enumerate}[(i)]
\item
\begin{align*}
\int_{A} \sup_{t \in [0,T]}|H(\omega,t,x)| \mathrm{d}x  <\infty \quad (a.s.) \quad \forall t \in [0,T]
\end{align*}
\item 
\begin{align*}
&\int_{A} H(\omega,t,x) \mathrm{d}x \\
&=\int_0^{t} \int_{A}1_{\opar 0,\tau \cbrk}(s) F(\omega,s,x) \mathrm{d} x \mathrm{d}s  + \int_0^{t} \int_{A} 1_{\opar 0,\tau  \cbrk}(s) G^k(\omega,s,x)  \mathrm{d}x \mathrm{d}B^k_s   \quad (a.s.) \quad \forall t \in [0,T],
\end{align*}
 where the series $ \int_0^t \int_{A} 1_{\opar 0,\tau  \cbrk}(s) G^k(\omega,s,x)\mathrm{d}x \mathrm{d}B^k_s  $ converges in probability  uniformly on $[0,T]$.
\item for almost every $x \in A$, 
$$
H(\omega,t,x)= \int_0^{t } 1_{\opar 0,\tau \cbrk}(s) F(\omega,s,x) \mathrm{d}s +\int_0^{t} 1_{\opar 0,\tau  \cbrk}(s) G^k(\omega,s,x) \mathrm{d}B^k_s
\quad (a.s.) \quad \forall t \in [0,T],
$$
 where the series $\int_0^t 1_{\opar 0,\tau  \cbrk}(s) G^k(\omega,s,x) \mathrm{d}B^k_s$ and $\int_A \int_0^t 1_{\opar 0,\tau  \cbrk}(s) G^k(\omega,s,x) \mathrm{d}B^k_s \mathrm{d}x$ converge  in probability uniformly on $[0,T]$.
\end{enumerate}
In particular,
\begin{align*}
&\int_{A}\int_0^{t} 1_{\opar 0,\tau \cbrk}(s) F(\omega,s,x)  \mathrm{d}s \mathrm{d} x  + \int_{A} \int_0^{t}  1_{\opar 0,\tau  \cbrk}(s) G^k(\omega,s,x)   \mathrm{d}B^k_s \mathrm{d}x \\
&=\int_0^{t} \int_{A}1_{\opar 0,\tau \cbrk}(s) F(\omega,s,x) \mathrm{d} x \mathrm{d}s  + \int_0^{t} \int_{A} 1_{\opar 0,\tau  \cbrk}(s) G^k(\omega,s,x)  \mathrm{d}x \mathrm{d}B^k_s   \quad (a.s.) \quad \forall t \in [0,T]
\end{align*}
by identifying $H(\omega,t,x)$  and $\int_0^{t } 1_{\opar 0,\tau \cbrk}(s) F(\omega,s,x) \mathrm{d}s +\int_0^{t} 1_{\opar 0,\tau  \cbrk}(s) G^k(\omega,s,x) \mathrm{d}B^k_s$ based on the equivalent relation that is equal $(a.e.)$ on $\Omega \times [0,T] \times A$.
Moreover, if $F$ is $\cP \times \cB(\fR^d)$-measurable, then $(\omega,t,x) \mapsto 1_{[0,T]}(t)1_{A}(x)H(\omega,t,x)$ is $\cP \times \cB(\fR^d)$-measurable. 
\end{thm}
\begin{proof}
The assumptions 
$$
F \in \bL_{0,0,1,\ell oc}\left( \opar 0,\tau  \cbrk \times A , \rH \times \cB(\fR^d) \right),
$$
and
$$
G \in\bL_{0,0,1,\ell oc}\left( \opar 0,\tau   \cbrk \times A, \cP \times \cB(\fR^d); l_2\right)
$$
imply that the extensions 
\begin{align*}
1_{(0,T]}(t)1_{\opar 0,\tau   \cbrk} \left( \omega,t \right)  F(\omega,t,x) 
:= 
\begin{cases}
& F(\omega,t,x)  \quad \text{if~ $(\omega,t) \in \opar 0,\tau  \wedge T \cbrk$ } \\
& 0 \quad \text{otherwise}
\end{cases}
\end{align*}
and
\begin{align*}
1_{(0,T]}(t)1_{\opar 0,\tau   \cbrk} \left( \omega,t \right)  G(\omega,t,x) 
:= 
\begin{cases}
& G(\omega,t,x)  \quad \text{if~ $(\omega,t) \in \opar 0,\tau  \wedge T \cbrk$ } \\
& 0 \quad \text{otherwise}
\end{cases}
\end{align*}
are  $\sigma(\rH \cup \cP)\times \cB(\fR^d) |_{ \Omega \times [0,T] \times A}$-measurable and
$\cP \times \cB(\fR^d)|_{ \Omega \times [0,T]\times A}$-measurable, respectively.

Then, the theorem can be derived straightforwardly from a version of the stochastic Fubini theorem. For further details, refer to \cite[Theorem 3.5]{NV 2005}, \cite[Lemmas 2.6 and 2.7]{Krylov 2011}, and \cite[Theorem 2.2]{Veraar 2012}, which discuss stochastic Fubini theorems in a general context sufficient to encompass our theorem.
However, we acknowledge that our theorem cannot be directly obtained from the aforementioned results. 
Some suitable adjustments need to be made by inspecting their proofs.
\end{proof}

Put $A=B_R$ with $R \in (0,\infty)$ in Theorem \ref{stochastic Fubini} and recall the condition that 
\begin{align}
									\label{20240403 01}
 \int_{B_R} \left(\int_0^{T} 1_{\opar 0,\tau  \cbrk}(t) \left|G(\omega,t,x) \right|^2_{l_2}\mathrm{d}t \right)^{1/2} \mathrm{d}x
< \infty \quad (a.s.).
\end{align}
Then this condition ensures that the stochastic term 
\begin{align*}
 \int_0^{t}  \left(\int_{B_R} 1_{\opar 0,\tau  \wedge T\cbrk}(s) G^k(\omega,s,x)  \mathrm{d}x \right) \mathrm{d}B^k_s
 \end{align*}
 is well-defined for all $(\omega,t) \in \Omega \times [0,T]$ due to the generalized Minkowski inequality:
 \begin{align*}
 \left(\int_0^{t} \left|\int_{B_R} 1_{\opar 0,\tau   \cbrk}(s) G^k(\omega,s,x)  \mathrm{d}x\right|^2_{l_2} \mathrm{d}s \right)^{1/2}
 &\leq  \int_{B_R} \left(\int_0^{t} \left|1_{\opar 0,\tau \wedge T \cbrk}(s) G^k(\omega,s,x) \right|^2_{l_2} \mathrm{d}s\right)^{1/2}  \mathrm{d}x \\
 &< \infty \quad (a.s.) \quad \forall t \in (0,\infty).
 \end{align*}
Additionally, \eqref{20240403 01} implies that for almost every $x \in B_R$,
 \begin{align*}
\left(\int_0^{T} 1_{\opar 0,\tau  \cbrk}(t) \left|G(\omega,t,x) \right|^2_{l_2}\mathrm{d}t \right)^{1/2} 
< \infty \quad (a.s.).
\end{align*}
Thus for almost every $x \in B_R$, 
\begin{align*}
 \int_0^{t}  1_{\opar 0,\tau   \cbrk}(s) G^k(\omega,s,x) \mathrm{d}B^k_s
\end{align*}
is well-defined for all $(\omega,t) \in  [0, T]$ and is continuous with respect to $t$.
In other words, the stochastic integral
\begin{align*}
 \int_0^{t}  1_{\opar 0,\tau   \cbrk}(s) G^k(\omega,s,x) \mathrm{d}B^k_s
\end{align*}
is a complex-valued function defined $(a.e.)$ on $\Omega \times [0,T] \times \fR^d$ so that for almost every $x \in B_R$ and $\omega \in \Omega$, 
the mapping
\begin{align*}
t \in [0,T] \mapsto \int_0^{t}  1_{\opar 0,\tau   \cbrk}(s) G^k(\omega,s,x) \mathrm{d}B^k_s
\end{align*}
is continuous.
Due to this construction, for almost every $x \in B_R$, the mapping
\begin{align*}
(\omega,t) \in \Omega \times [0,T] \mapsto  1_{ [0,T]}(t) \int_0^{t}  1_{\opar 0,\tau   \cbrk}(s) G^k(\omega,s,x) \mathrm{d}B^k_s
\end{align*}
is $\cP$-measurable. 
However, the joint measurability of 
\begin{align*}
 \int_0^{t}  1_{\opar 0,\tau   \cbrk}(s) G^k(\omega,s,x) \mathrm{d}B^k_s
\end{align*}
cannot be established directly in this way.
Fortunately, the stochastic Fubini theorem ensures that there exists a modification of 
\begin{align*}
\int_0^{t}  1_{\opar 0,\tau   \cbrk}(s) G^k(\omega,s,x) \mathrm{d}B^k_s
\end{align*}
which is joint measurable on $\Omega \times [0,T] \times A$.

Next consider another  version of the stochastic Fubini theorem which is uniformly given for all $T \in (0,\infty)$
by summing the localizations.
\begin{corollary}
									\label{Fubini corollary}
Let $F \in \bL_{0,1,1,\ell oc}\left( \opar 0,\tau  \cbrk \times \fR^d, \rH \times \cB(\fR^d) \right)$ and 
$$
G \in\bL^{\omega,x,t}_{0,1,2,\ell oc}\left( \opar 0,\tau \cbrk \times \fR^d, \cP \times \cB(\fR^d); l_2\right).
$$
Then there exists a $\sigma(\rH \times \cP) \times \cB(\fR^d)$-measurable function $H(\omega,t,x)$ on $\Omega \times [0,\infty) \times \fR^d$ such that 
\begin{enumerate}[(i)]
\item for each $R \in (0,\infty)$ and $T \in (0,\infty)$, 
\begin{align*}
\int_{B_R} \sup_{ t \in [0,T]}|H(\omega,t,x)| \mathrm{d}x  <\infty \quad (a.s.) ,
\end{align*}
\item for each $R \in (0,\infty)$, 
\begin{align*}
&\int_{B_R} H(\omega,t,x) \mathrm{d}x \\
&=\int_0^{t} \int_{B_R}1_{\opar 0,\tau \cbrk}(s) F(\omega,s,x) \mathrm{d} x \mathrm{d}t  
+ \int_0^{t} \int_{B_R} 1_{\opar 0,\tau  \cbrk}(s) G^k(\omega,s,x)  \mathrm{d}x \mathrm{d}B^k_s   \quad (a.s.) \quad \forall t \in [0,\infty),
\end{align*}
the series $\int_0^{t} \int_{B_R} 1_{\opar 0,\tau  \cbrk}(s) G^k(\omega,s,x)  \mathrm{d}x \mathrm{d}B^k_s$ converges in probability uniformly on $[0,T]$  for all $T \in (0,\infty)$.
\item for almost every $x \in \fR^d$, 
\begin{align*}
H(\omega,t,x)= \int_0^{t } 1_{\opar 0,\tau \cbrk}(s) F(\omega,s,x) \mathrm{d}t +\int_0^{t} 1_{\opar 0,\tau  \cbrk}(s) G^k(\omega,s,x) \mathrm{d}B^k_s
\quad (a.s.) \quad \forall t \in [0,\infty),
\end{align*}
\end{enumerate}
 where the series $\int_0^t 1_{\opar 0,\tau  \cbrk}(s) G^k(\omega,s,x) \mathrm{d}B^k_s$ converges  in probability uniformly on $[0,T]$ for all $T \in (0,\infty)$ and additionally for any $R \in (0,\infty)$, the series $\int_{B_R}\int_0^t 1_{\opar 0,\tau  \cbrk}(s) G^k(\omega,s,x) \mathrm{d}B^k_s$ converges  in probability uniformly on $[0,T]$ for all $T \in (0,\infty)$.
\end{corollary}
\begin{proof}
It may seem that this theorem can also be obtained from a general version of the stochastic Fubini theorem since the Lebesgue measure on $[0,\infty) \times \fR^d$ is $\sigma$-finite. However, we could not find an appropriate reference and thus suggest a proof with detail by applying  Theorem \ref{stochastic Fubini}.

It is sufficient to find a $\sigma(\rH \times \cP) \times \cB(\fR^d)$-measurable function $H(\omega,t,x)$ on $\Omega \times [0,\infty) \times \fR^d$ such that for all $R,T \in (0,\infty)$,
\begin{align}
									\label{20240309 30}		
\int_{B_R} \sup_{t \in [0,T]} |H(\omega,t,x)| \mathrm{d}x  <\infty \quad (a.s.),
\end{align}
\begin{align}
									\notag
&\int_{B_R} H(\omega,t,x) \mathrm{d}x \\
									\label{20240309 31}
&=\int_0^{t} \int_{B_R}1_{\opar 0,\tau \cbrk}(s) F(\omega,s,x) \mathrm{d} x \mathrm{d}s  
+ \int_0^{t} \int_{B_R} 1_{\opar 0,\tau  \cbrk}(s) G^k(\omega,s,x)  \mathrm{d}x \mathrm{d}B^k_s   \quad (a.s.) \quad \forall t \in [0,T],
\end{align}
and for almost every $x \in B_R$, 
\begin{align}
									\label{20240309 32}
H(\omega,t,x)= \int_0^{t } 1_{\opar 0,\tau \cbrk}(s) F(\omega,s,x) \mathrm{d}s +\int_0^{t} 1_{\opar 0,\tau  \cbrk}(s) G^k(\omega,s,x) \mathrm{d}B^k_s
\quad (a.s.) \quad \forall t \in [0,T],
\end{align}
where all the series (including the series $\int_{B_R}\int_0^t 1_{\opar 0,\tau  \cbrk}(s) G^k(\omega,s,x) \mathrm{d}B^k_s$) converge in probability uniformly on $[0,T]$ for any $T \in (0,\infty)$.
Recall the condition that $F \in \bL_{0,1,1,\ell oc}\left( \opar 0,\tau  \cbrk \times \fR^d,\cP \times \cB(\fR^d)\right)$ and 
$$
G \in\bL^{\omega,x,t}_{0,2,1,\ell oc}\left( \opar 0,\tau  \cbrk \times \fR^d, \cP \times \cB(\fR^d); l_2\right).
$$
Then for all $T,R \in (0,\infty)$, we have
\begin{align*}
 \int_{B_R} \int_0^{T} 1_{\opar 0,\tau  \cbrk}(t) \left|F(\omega,t,x) \right|\mathrm{d}t \mathrm{d}x
 +\int_{B_R} \left(\int_0^{T} 1_{\opar 0,\tau  \cbrk}(t) \left|G(\omega,t,x) \right|^2_{l_2}\mathrm{d}t \right)^{1/2} \mathrm{d}x
< \infty \quad (a.s.).
\end{align*}
Thus by Theorem \ref{stochastic Fubini}, for each $R,T \in (0,\infty)$, there exists a $\sigma(\cH \cup \cP) \times \cB(\fR^d)|_{\Omega \times [0,T] \times B_R}$-measurable function $H_{T,R}$ on $\Omega \times [0,T] \times B_R$ so that
\begin{align}
										\label{20240309 40}
\int_{B_R} \sup_{t \in T}|H_{T,R}(\omega,t,x)| \mathrm{d}x  <\infty \quad (a.s.),
\end{align}
\begin{align*}
&\int_{B_R} H_{T,R} (\omega,t,x) \mathrm{d}x \\
&=\int_0^{t} \int_{B_R}1_{\opar 0,\tau \cbrk}(s) F(\omega,s,x) \mathrm{d} x \mathrm{d}s  
+ \int_0^{t} \int_{B_R} 1_{\opar 0,\tau  \cbrk}(s) G^k(\omega,s,x)  \mathrm{d}x \mathrm{d}B^k_s   \quad (a.s.) \quad \forall t \in [0,T],
\end{align*}
and for almost every $x \in B_R$, 
\begin{align}
								\label{20240309 01}
H_{T,R}(\omega,t,x)= \int_0^{t } 1_{\opar 0,\tau \cbrk}(s) F(\omega,s,x) \mathrm{d}s +\int_0^{t} 1_{\opar 0,\tau  \cbrk}(s) G^k(\omega,s,x) \mathrm{d}B^k_s
\quad (a.s.) \quad \forall t \in [0,T],
\end{align}
 where all the series converge in probability uniformly on $[0,T]$.

Next we claim that for all $T \in (0,\infty)$ and $0<R_1 \leq R \leq R_2 <\infty$,
\begin{align}
									\label{20240309 02}
H_{T,R_1}= H_{T,R} = H_{T,R_2} \quad (a.e.)~ \text{on}~ \Omega \times [0,T] \times B_{R_1}
\end{align}
and
\begin{align}
									\notag
&\int_{B_{R} \setminus B_{R_1}} H_{T,R_2} (\omega,t,x) \mathrm{d}x \\
									\label{20240309 10}
&=\int_0^{t} \int_{B_{R} \setminus B_{R_1}} 1_{\opar 0,\tau \cbrk}(s) F(\omega,s,x) \mathrm{d} x \mathrm{d}s  
+ \int_0^{t} \int_{B_{R} \setminus B_{R_1}} 1_{\opar 0,\tau  \cbrk}(s) G^k(\omega,s,x)  \mathrm{d}x \mathrm{d}B^k_s   \quad (a.s.) \quad \forall t \in [0,T].
\end{align}
Due to \eqref{20240309 01} and the Fubini theorem, it is nearly evident that \eqref{20240309 02} holds.
To prove \eqref{20240309 10}, we use \eqref{20240309 02} and the linearity of the integrations.
Indeed,
\begin{align*}
&\int_{B_{R} \setminus B_{R_1}} H_{T,R_2} (\omega,t,x) \mathrm{d}x \\
&=\int_{B_{R}} H_{T,R_2} (\omega,t,x) \mathrm{d}x 
-\int_{ B_{R_1}} H_{T,R_2} (\omega,t,x) \mathrm{d}x\\
&=\int_{B_{R}} H_{T,R} (\omega,t,x) \mathrm{d}x 
-\int_{ B_{R_1}} H_{T,R_1} (\omega,t,x) \mathrm{d}x\\
&=\int_0^{t} \int_{B_{R} \setminus B_{R_1}} 1_{\opar 0,\tau \cbrk}(s) F(\omega,s,x) \mathrm{d} x \mathrm{d}s  
+ \int_0^{t} \int_{B_{R} \setminus B_{R_1}} 1_{\opar 0,\tau  \cbrk}(s) G^k(\omega,s,x)  \mathrm{d}x \mathrm{d}B^k_s   \quad (a.s.) \quad \forall t \in [0,T].
\end{align*}
Here we used the fact that all paths are continuous on $[0,T]$ almost surely.
From this, we can construct the function $H$ as follows:
\begin{align*}
H(\omega,t,x) :=
\begin{cases}
&0 \quad \text{if}~ t=0 \\
&H_{n,m}(\omega,t,x) \quad \text{if}~ (\omega,t,x) \in \Omega \times (n-1, n] \times B_m \setminus B_{m-1}.
\end{cases}
\end{align*}
Then it is obvious that $H$ is a $\sigma(\rH \cup \cP) \times \cB(\fR^d)$-measurable function function defined on $\Omega \times [0,\infty) \times \fR^d$.
We now show that $H$ satisfies \eqref{20240309 30}, \eqref{20240309 31}, and \eqref{20240309 32}.
Let $T, R \in (0,\infty)$. 
Denote by $\lceil T \rceil$  and $\lceil R \rceil$  the smallest integers such that
$T \leq \lceil T \rceil$ and $R \leq \lceil R \rceil$, respectively.
Then first by \eqref{20240309 40},
\begin{align*}
\int_{B_R} \sup_{t \in T}|H(\omega,t,x)| \mathrm{d}x  
&\leq \sum_{m=1}^{\lceil R \rceil} \sum_{n=1}^{\lceil T \rceil} \int_{B_R} 1_{B_m \setminus B_{m-1}}(x) \sup_{t \in T}|H_{n,m}(\omega,t,x)| \mathrm{d}x   \\
&\leq \sum_{m=1}^{\lceil R \rceil} \sum_{n=1}^{\lceil T \rceil}  \int_{B_m}  \sup_{t \in T}|H_{n,m}(\omega,t,x)| \mathrm{d}x   
<\infty \quad (a.s.).
\end{align*}
Next by \eqref{20240309 10},
\begin{align*}
&\int_{B_R} H(\omega,t,x) \mathrm{d}x \\
&=\sum_{m=1}^{\lceil R \rceil} \sum_{n=1}^{\lceil T \rceil} 1_{(n-1,n]}(t) \int_{B_R}1_{B_m \setminus B_{m-1}}(x) H_{n,m}(\omega,t,x) \mathrm{d}x \\
&=\int_0^{t} \int_{B_R}1_{\opar 0,\tau \cbrk}(s) F(\omega,s,x) \mathrm{d} x \mathrm{d}s  
+ \int_0^{t} \int_{B_R} 1_{\opar 0,\tau  \cbrk}(s) G^k(\omega,s,x)  \mathrm{d}x \mathrm{d}B^k_s   \quad (a.s.) \quad \forall t \in [0,T].
\end{align*}
Finally, by \eqref{20240309 02} and \eqref{20240309 01}, for almost every $x \in B_R \subset B_{ \lceil R \rceil}$, 
\begin{align*}
H(\omega,t,x)
&=\sum_{m=1}^{\lceil R \rceil} \sum_{n=1}^{\lceil T \rceil} 1_{(n-1,n]}(t) 1_{B_m \setminus B_{m-1}}(x) H_{n,m}(\omega,t,x) \\
&= \sum_{n=1}^{\lceil T \rceil} 1_{(n-1,n]}(t) H_{n,\lceil R \rceil}(\omega,t,x) \\
&= \int_0^{t } 1_{\opar 0,\tau \cbrk}(s) F(\omega,s,x) \mathrm{d}s +\int_0^{t} 1_{\opar 0,\tau  \cbrk}(s) G^k(\omega,s,x) \mathrm{d}B^k_s
\quad (a.s.) \quad \forall t \in [0,T].
\end{align*}
The corollary is proved. 
\end{proof}

\mysection{Stochastic Fubini theorems with deterministic symbols}
												\label{sto fubini deter}

In this section, we apply the stochastic Fubini theorems developed in the previous section to establish our existence result.
Recall that the solution $u$ to \eqref{time eqn} is given so that 
\begin{align}
										\notag
\cF[u(t,\cdot)](\xi)
&= \exp\left(\int_0^t\psi(r,\cdot)\mathrm{d}r \right) \cF[u_0](\xi)
+\int_0^t  \exp\left(\int_s^t\psi(r,\xi)\mathrm{d}r \right) 1_{\opar 0,\tau \cbrk}(s)\cF[f(s,\cdot)](\xi) \mathrm{d}s  \\
										\label{202040702 20}
&\quad +\int_0^t  \exp\left(\int_s^t\psi(r,\xi)\mathrm{d}r \right) 1_{\opar 0,\tau \cbrk}(s)\cF[g^k(s,\cdot)](\xi) \mathrm{d}B^k_s,
\end{align} 
if the symbol $\psi$ is deterministic, which is mentioned  in  Remark \ref{u meaning}.
Consequently, our solution $u$ will be derived from the above relation if the symbol is non-random.
Here, we provide more details explaining why this works only for deterministic symbols.

It is evident that the most challenging part to handle in \eqref{202040702 20} is the stochastic (integral) term
\begin{align}
										\label{20240724 30}
\int_0^t  \exp\left(\int_s^t\psi(r,\xi)\mathrm{d}r \right) 1_{\opar 0,\tau \cbrk}(s)\cF[g^k(s,\cdot)](\xi) \mathrm{d}B^k_s.
\end{align}
Firstly, it is essential to verify whether the stochastic term is well-defined based on our main assumptions and the stochastic Fubini theorems. Once this is established, a solution $u$ could be derived from the relation \eqref{202040702 20} with a certain mathematical meaning.
However, for each $t \in (0,\infty)$ and $\xi \in \fR^d$, the mapping
\begin{align}
										\label{20240702 10}
(\omega,s) \mapsto \int_s^t\psi(r,\xi)\mathrm{d}r \cdot 1_{\opar 0,\tau \wedge t\cbrk }(s)
\end{align}
is not predictable anymore since the term $\int_s^t\psi(r,\xi)\mathrm{d}r$ is always merely $\rF_t$-adapted for any $s \in (0,t)$ in general. In other words, the term in \eqref{20240702 10} is not predictable (even not progressive measurable) since it is not $\rF_s$-adapted for each $s$.
Therefore, \eqref{20240724 30} cannot be defined as an It\^o stochastic integral in general unless there is a very strong assumption regarding the randomness of $\psi$ to make the term in \eqref{20240702 10} predictable (or progressive measurable).

Nevertheless, the term in \eqref{20240702 10} naturally arises when seeking an appropriate candidate for a solution to \eqref{time eqn}, as demonstrated in \eqref{202040702 20}. This indicates that our existence theory cannot be directly derived from the stochastic Fubini theorems unless the symbol $\psi(t,\xi)$ meets a very strong assumption regarding the sample point $\omega$, due to the predictability issue mentioned earlier.

As a first step to ensure that \eqref{20240702 10} is predictable and employ \eqref{202040702 20}, we specifically assume that the symbol $\psi(t,\xi)$ is deterministic. 
This should be regarded as the simplest form of randomness.

Additionally, we use the notation $\tilde \psi(t,\xi)$ in order to emphasize that the symbol $\tilde \psi(t,\xi)$ is deterministic.
In other words, we assume that our symbol $\tilde \psi(t,\xi)$ is  non-random, \textit{i.e.} $\tilde \psi(t,\xi)$ is a complex-valued $\cB([0,\infty) \times \cB(\fR^d)$-measurable function defined on $[0,\infty) \times \fR^d$ throughout the section.

Moreover, suppose that our symbol $\tilde \psi$ satisfies weaker variants of Assumptions \ref{main as} and \ref{main as 2} so that for all  $R,T \in (0,\infty)$,
\begin{align}
									\label{main deter symbol as}
C^{\mathrm{e}|\int\Re[\tilde \psi]|}_{R,T}:=\esssup_{ 0\leq s \leq t \leq T, \xi \in   B_R} \left| \exp\left( \left|\int_s^t\Re[\tilde \psi(r,\xi)]\mathrm{d}r \right| \right)  \right|   
< \infty
\end{align}
and
\begin{align}
										\label{main deter symbol as 2}
C^{|\tilde \psi|}_{R,T}:=\esssup_{\xi \in  B_R} \left( \int_0^T|\tilde \psi(t,\xi)|\sup_{0\leq s \leq t}  \exp\left(\left|\int_s^t\Re[\tilde \psi(r,\xi)]\mathrm{d}r\right| \right)    \mathrm{d}t  \right)
< \infty.
\end{align}
One can view $\tilde \psi$ as a complex-valued $\cP \times \cB(\fR^d)$-measurable function on $\Omega \times [0,\infty) \times \fR^d$ by considering the canonical constant extension that $(\omega,t,\xi) \mapsto \tilde \psi(t,\xi)$.
Then the condition in \eqref{main deter symbol as} is equivalent to 
\begin{align*}
\esssup_{ (\omega,t, \xi) \in \Omega \times (0,T) \times B_R }\left|\int_0^t\Re[\psi(r,\xi)]\mathrm{d}r \right|  
< \infty,
\end{align*}
which shows that this condition is less restrictive than Assumption \ref{main as}. 
In particular, the constant $C^{\mathrm{e}|\int\Re[\tilde \psi]|}_{R,T}$ is located between $C^{\mathrm{e}\int\Re[\tilde \psi]}_{R,T}$ and $C^{\mathrm{e}\int\sup|\Re[\tilde \psi]|}_{R,T}$, \textit{i.e.} $C^{\mathrm{e}\int\Re[\tilde \psi]}_{R,T} \leq C^{\mathrm{e}|\int\Re[\tilde \psi]|}_{R,T}
\leq C^{\mathrm{e}\int\sup|\Re[\tilde \psi]|}_{R,T}$, where the other notations came from Assumption \ref{main as} and Assumption \ref{weaker as}. Moreover, by recalling Assumption \ref{main as 2}, it is obvious that
\begin{align*}
C^{|\tilde \psi|}_{R,T}
\leq C^{\sup|\tilde \psi|}_{R,T},
\end{align*}
which implies that \eqref{main deter symbol as 2} is slightly weaker.

Now we connect our main assumptions on the symbol with the stochastic Fubini theorems from the previous section. We present several corollaries of the stochastic Fubini theorems. 
The conditions \eqref{main deter symbol as} and \eqref{main deter symbol as 2} on the symbol $\tilde \psi$ are considered separately in the following corollaries except the last one, and even some weaker conditions on 
$\tilde \psi$ are introduced. 
Additionally, we explore natural extensions of the stochastic integrals to the entire space 
$\Omega \times [0,\infty) \times \fR^d$ with appropriate new classes.

\begin{corollary}
									\label{20240417 01}
Let $G \in \bL^{\omega,x,t}_{0,1,2,\ell oc}\left( \opar 0,  \tau \cbrk \times \fR^d, \cP \times \cB(\fR^d); l_2\right)$.
Suppose that \eqref{main deter symbol as} holds and for almost every $\xi \in \fR^d$,
\begin{align}
										\label{2024051801}
\int_0^t |\tilde\psi(s,\xi)| \mathrm{d}s < \infty \quad  \forall t \in (0,\infty).
\end{align}
Denote
\begin{align}
H(t,\xi) 
=\int_0^t  \exp\left(\int_s^t \tilde \psi(r,\xi)\mathrm{d}r \right) 1_{\opar 0,\tau \cbrk }(s) G^k(s,\xi)\mathrm{d}B^k_s.
\end{align}
Then 
\begin{align*}
H \in \bC_{loc}L_{1,\ell oc}\left(\Omega \times [0,\infty) \times \fR^d, \cP \times \cB(\fR^d)\right),
\end{align*}
where
\begin{align*}
&H \in  \bC_{loc} L_{1,\ell oc} \left(\Omega \times [0,\infty) \times \fR^d, \cP \times \cB(\fR^d)\right) \\
&\iff 
H \in  \bC L_{1,\ell oc}\left(\Omega \times [0,T]\times \fR^d, \cP \times \cB(\fR^d)\right) \quad \forall T \in (0,\infty).
\end{align*}
\end{corollary}
\begin{proof}
First we show that $H$ is a complex-valued $\cP \times \cB(\fR^d)$-measurable function defined on $\Omega \times [0,\infty) \times \fR^d$.
Put
\begin{align*}
H_1(t,\xi) 
=\exp\left(\int_0^t \tilde \psi(r,\xi)\mathrm{d}r \right)
\end{align*}
and
\begin{align*}
H_2(t,\xi) 
=\int_0^t  \exp\left(-\int_0^s\tilde\psi(r,\xi)\mathrm{d}r \right) 1_{\opar 0,\tau \cbrk }(s) G^k(s,\xi) \mathrm{d}B^k_s.
\end{align*}
Then obviously $H_1$ is well-defined due to \eqref{2024051801}.
Additionally, since the symbol $\tilde \psi(r,\xi)$ is non-random and $\cB([0,\infty) \times  \cB(\fR^d)$-measurable,
it is easy to show that the mapping
$$
(\omega,t,\xi) \in \Omega \times [0,\infty) \times \fR^d \mapsto \int_0^t \tilde \psi(r,\xi)\mathrm{d}r \in \fC
$$ 
is  $\cP \times \cB(\fR^d)$-measurable. 
On the other hand, applying \eqref{main deter symbol as} and recalling the condition 
$$
G \in \bL^{\omega,x,t}_{0,2,1,\ell oc}\left( \opar 0,  \tau \cbrk \times \fR^d, \cP \times \cB(\fR^d); l_2\right),
$$
for all $T, R \in (0,\infty)$, we have
\begin{align*}
&\int_{B_R}\left(\int_0^T \left|\exp\left(-\int_0^s\tilde \psi(r,\xi)\mathrm{d}r \right) 1_{\opar 0,\tau \cbrk }(s) G(s,\xi)\right|^2_{l_2} \mathrm{d}s \right)^{1/2} \mathrm{d}\xi \\
&\leq  \esssup_{s \in (0,T), \xi \in B_R} \exp\left( \left|\int_0^s\Re[\tilde \psi(r,\xi)] \mathrm{d}r \right|\right) 
\int_{B_R} \left(\int_0^T \left| 1_{\opar 0,\tau \cbrk }(s) G(s,\xi)\right|^2_{l_2} \mathrm{d}s \right)^{1/2} \mathrm{d}\xi 
< \infty \quad (a.s.).
\end{align*}
Thus by Corollary \ref{Fubini corollary}, 
$H_2(t,\xi)$ is also $\cP \times \cB(\fR^d)$-measurable  (by considering the modification in the corollary) and
\begin{align}
										\label{20240801 11}
\int_{B_R} \sup_{ t \in [0,T]}|H_2(t,\xi)| \mathrm{d}\xi  <\infty \quad (a.s.).
\end{align}
Therefore $H(t,\xi)=H_1(t,\xi)H_2(t,\xi)$ is  a $\cP \times \cB(\fR^d)$-measurable function defined on $\Omega \times [0,\infty) \times \fR^d$.

Next, we claim 
\begin{align*}
H \in \bC_{loc}L_{1,\ell oc}\left(\Omega \times [0,\infty) \times \fR^d, \cP \times \cB(\fR^d)\right)
\end{align*}
to complete the proof.
First, it is obvious that both $H_1(t,\xi)$ and $H_2(t,\xi)$ are continuous with respect to $t$ for almost every $\xi$ and $\omega$. 
Putting 
\begin{align*}
M=\esssup_{ (t,\xi) \in  [0,T] \times B_R} \left| \exp\left( \left|\int_0^t\Re[\tilde \psi(r,\xi)]\mathrm{d}r \right| \right)  \right|   
\end{align*}
and using \eqref{main deter symbol as} and \eqref{20240801 11}, we obtain 
\begin{align}
										\label{20240414 02}
 \int_{B_R} \sup_{t \in [0,T]} \left|H(t,\xi)\right|  \mathrm{d}\xi 
\leq M\int_{B_R} \sup_{ t \in [0,T]}|H_2(t,\xi)| \mathrm{d}\xi \quad (a.s.).
\end{align}
Additionally, applying the stochastic Fubini theorem, a well-known property of the It\^o stochastic integral (\textit{cf.\cite[Theorem 6.3.5]{Krylov 2002}} and \cite[(2.5)]{Krylov 2011}), and the generalized Minkowski inequality, for all  $\varepsilon, \delta, T,R \in (0,\infty)$,  we have
\begin{align}
										\notag
&P\left( \int_{B_R}\sup_{t \in [0,T]}\left| H_2(t,\xi) \right|\mathrm{d}\xi  \geq \varepsilon  \right) \\
										\label{20240414 01}
&\leq  P \left( \left[ \int_{B_R} \left|\int_0^T \left|\exp\left(-\int_0^s\tilde\psi(r,\xi)\mathrm{d}r \right) 1_{\opar 0,\tau \cbrk }(s) G(s,\xi)\right|^2_{l_2} \mathrm{d}s\right|^{1/2} \mathrm{d}\xi \right]^{2} \geq \delta \right) + \frac{\delta}{\varepsilon^2}.
\end{align}
By combining \eqref{20240414 01} and \eqref{20240414 02},
\begin{align*}
&P\left(\int_{B_R}\sup_{t \in [0,T]}\left| H(t,\xi)\right| \mathrm{d}\xi  \geq \varepsilon  \right) \\
&\leq P\left( \int_{B_R}\sup_{t \in [0,T]}\left| H_2(t,\xi)\right| \mathrm{d}\xi  \geq \frac{\varepsilon}{M}  \right) \\
&\leq  P \left( \left[ \int_{B_R} \left|\int_0^T \left|\exp\left(-\int_0^s\tilde\psi(r,\xi)\mathrm{d}r \right) 1_{\opar 0,\tau \cbrk }(s) G(s,\xi)\right|^2_{l_2} \mathrm{d}s\right|^{1/2} \mathrm{d}\xi \right]^{2} \geq \delta \right) + M^2\frac{\delta}{\varepsilon^2}.
\end{align*}
Taking $\varepsilon \to \infty$ and $\delta \to \infty$, we have
\begin{align*}
P\left( \int_{B_R}\sup_{t \in [0,T]}\left|H(t,\xi)\right| \mathrm{d}\xi  =\infty  \right)=0 \quad \forall T,R \in (0,\infty).
\end{align*}
Since the $\cP \times \cB(\fR^d)$-measurability had already been shown above, the corollary is proved.
\end{proof}

\begin{rem}
Generally, \eqref{main deter symbol as} does not imply \eqref{2024051801} because it only controls the real part of $\psi$. However, \eqref{main deter symbol as 2} does imply \eqref{2024051801}, which will be used later to apply Corollary \ref{20240417 01} with the symbol satisfying both \eqref{main deter symbol as} and \eqref{main deter symbol as 2}.
Moreover, it is possible to construct our theories with the symbol satisfying \eqref{main deter symbol as}, \eqref{2024051801}, and a slightly weaker condition than \eqref{main deter symbol as 2}, which will be revisited in Remark \ref{ensure conti rem}.
\end{rem}

Next, we impose a finite expected condition on the product of the $\exp\left(\tilde \psi \right)$ and data $G$.
We assume a slightly weaker condition on the symbol $\tilde\psi$ rather than \eqref{main deter symbol as}.
This weaker assumption still enables us to estimate $H$ in terms of $\tilde \psi$ and $G$.
Additionally, recall that \eqref{main deter symbol as} is used to show the joint measurability of $H$ in the proof of Corollary \ref{20240417 01}.
This demonstrates that \eqref{main deter symbol as} is sufficiently given not only for obtaining estimates but also for ensuring joint measurability. In other words, the condition on $\tilde \psi$ seems to be possible to be weakened if one just focuses on estimates.
However, without joint measurability, the rigor of all theories falls apart. 
In this context, we strongly believe that the condition in the following corollary cannot replace $\eqref{main deter symbol as}$ in our main theorems even with deterministic symbols.

\begin{corollary}
										\label{joint measurability issue 1}
Let $G \in \bL_{0}\left( \opar 0,  \tau \cbrk \times \fR^d, \cP \times \cB(\fR^d); l_2\right)$.
Suppose that \eqref{2024051801} holds and for almost every $\xi \in \fR^d$,
\begin{align}
										\label{20240505 10}
\bE\left[ \left(\int_0^T  \exp\left(\int_s^T \Re[2\tilde\psi(r,\xi)]\mathrm{d}r \right) 1_{\opar 0,\tau \cbrk}(s)\left|G(s,\xi)\right|^2_{l_2} \mathrm{d}s \right)^{1/2} \right] < \infty \quad  \forall T \in (0,\infty).
\end{align}
Denote
\begin{align*}
H(t,\xi) 
=\int_0^t  \exp\left(\int_s^t\tilde\psi(r,\xi)\mathrm{d}r \right) 1_{\opar 0,\tau \cbrk }(s) G^k(s,\xi)\mathrm{d}B^k_s.
\end{align*}
Then $H$ is defined (a.e.) on $\Omega \times [0,\infty) \times \fR^d$
and
\begin{align}
										\label{20240409 01}
\bE\left[ |H(t,\xi)|\right]
\leq C_{BDG}   \bE\left[  \left(\int_0^t \exp\left( \int_s^t 2\Re[\tilde\psi(r,\xi)]\mathrm{d}r \right)  \left| 1_{\opar 0,\tau \cbrk}(s) G(s,\xi)\right|^2_{l_2} \mathrm{d}s \right)^{1/2} \right]
\end{align}
for almost every $(t,\xi) \in [0,\infty) \times \fR^d$.
\end{corollary}
\begin{proof}
We show that for each $t$, $H(t,\xi)$ is well-defined as a random variable for almost every $\xi \in \fR^d$. 
Fix $t \in [0,\infty)$.
By applying \eqref{20240505 10}, we have
\begin{align}
										\notag
&\bE\left[  \left(\int_0^t  \left|\exp\left(\int_s^t\tilde\psi(r,\xi)\mathrm{d}r \right) 1_{\opar 0,\tau \cbrk}(s) G(s,\xi)\right|^2_{l_2} ds \right)^{1/2} \right] \\
										\label{20240730 90}
&=\bE\left[ \left(\int_0^t  \exp\left(\int_s^t  \Re[2\tilde\psi(r,\xi)]\mathrm{d}r \right) 1_{\opar 0,\tau \cbrk}(s) \left|G(s,\xi)\right|^2_{l_2} \mathrm{d}s \right)^{1/2}  \right]  < \infty  \quad  (a.e.) ~\xi \in \fR^d.
\end{align}
Additionally, by utilizing the Fubini theorem, it is possible to observe that for almost every $\xi \in \fR^d$, both the mappings $(\omega,s) \mapsto  1_{s \leq t} \int_s^t\tilde\psi(t,\xi) \mathrm{d}r$  and $(\omega,s) \mapsto 1_{\opar 0,\tau \cbrk}(s) G(s,\xi)$ are  $\cP$-measurable.
Fix $\xi \in \fR^d$ so that \eqref{20240730 90} and the above measurability conditions hold.
Then, the mapping
\begin{align*}
(\omega,s) \mapsto 
1_{s\leq t}\exp\left(\int_s^t\tilde\psi(r,\xi)\mathrm{d}r \right) 1_{\opar 0,\tau \cbrk}(s) G(s,\xi)
\end{align*}
is $\cP$-measurable.
Therefore, for almost every $\xi \in \fR^d$, the stochastic integral $H(t,\xi)$ is well-defined as a random variable (for each fixed $t$).
Additionally, by the BDG inequality, we have \eqref{20240409 01}.
The corollary is proved.
\end{proof}

\begin{rem}
								\label{measure issue rem}
\eqref{20240505 10} itself is not enough to guarantee that $H$ is joint measurable since the stochastic integral $H(t,\xi)$ in Corollary \ref{joint measurability issue 1} is defined in an iterated way. 
Even the stronger condition
\begin{align*}
\bE\left[\int_{B_R} \left(\int_0^T  \exp\left(\int_s^T \Re[2\tilde\psi(r,\xi)]\mathrm{d}r \right) 1_{\opar 0,\tau \cbrk}(s)\left|G(s,\xi)\right|^2_{l_2} \mathrm{d}s \right)^{1/2} \mathrm{d}\xi \right] < \infty \quad  \forall T,R \in (0,\infty)
\end{align*}
is not sufficient to ensure that $H$ is joint-measurable. 
We suggest an efficient condition to make $H$ joint-measurable in the next corollary.
\end{rem}

\begin{corollary}
										\label{joint measurability issue 2}
Let $G \in \bL_{0}\left( \opar 0,  \tau \cbrk \times \fR^d, \cP \times \cB(\fR^d); l_2\right)$.
Suppose that  \eqref{2024051801} holds and for any  $R \in (0,\infty)$,
\begin{align}
								\label{20240302 01}
\bE\left[\int_{B_R} \left(\int_0^T  \exp\left(-\int_0^s \Re[2\tilde\psi(r,\xi)]\mathrm{d}r \right) 1_{\opar 0,\tau \cbrk}(s)\left|G(s,\xi)\right|^2_{l_2} \mathrm{d}s \right)^{1/2} \mathrm{d}\xi \right] < \infty \quad  \forall T \in (0,\infty).
\end{align}
Denote
\begin{align*}
H(t,\xi) 
=\int_0^t  \exp\left(\int_s^t\tilde\psi(r,\xi)\mathrm{d}r \right) 1_{\opar 0,\tau \cbrk }(s) G^k(s,\xi)\mathrm{d}B^k_s.
\end{align*}
Then $H$ is $\cP \times \cB(\fR^d)$-measurable, \eqref{20240409 01} holds, and for any $T \in (0,\infty)$ and almost every $\xi \in \fR^d$, $H$ satisfies 
\begin{align}
										\notag
&\bE\left[ \sup_{t \in [0,T]} |H(t,\xi)|\right] \\
										\notag
&\leq C_{BDG}  \sup_{t \in [0,T]} \left[\exp\left(\int_0^t\Re[\tilde\psi(r,\xi)]\mathrm{d}r \right)\right] \bE\left[ \left(\int_0^{T}  \exp\left(-\int_0^s \Re[2\tilde\psi(r,\xi)]\mathrm{d}r \right) 1_{\opar 0,\tau \cbrk}(s)\left|G(s,\xi)\right|^2_{l_2} \mathrm{d}s \right)^{1/2} \right]\\
										\label{20240506 50}
&= C_{BDG}  \sup_{t \in [0,T]} \left( \bE\left[ \left(\int_0^{T}  \exp\left(\int_s^t \Re[2\tilde\psi(r,\xi)]\mathrm{d}r \right) 1_{\opar 0,\tau \cbrk}(s)\left|G(s,\xi)\right|^2_{l_2} \mathrm{d}s \right)^{1/2} \right] \right).
\end{align}
\end{corollary}
\begin{proof}
It is obvious that  \eqref{2024051801} and \eqref{20240302 01} imply \eqref{20240505 10} by elementary measure theories with the Fubini theorem.
Thus due to Corollary \ref{joint measurability issue 1}, 
it suffices to show that $H(t,\xi)$ is $\cP \times \cB(\fR^d)$-measurable and \eqref{20240506 50} holds. 
We split $H$ into two parts
\begin{align*}
H_1(t,\xi) 
=\exp\left(\int_0^t\tilde\psi(r,\xi)\mathrm{d}r \right)
\end{align*}
and
\begin{align*}
H_2(t,\xi) 
=\int_0^t  \exp\left(-\int_0^s\tilde\psi(r,\xi)\mathrm{d}r \right) 1_{\opar 0,\tau \cbrk }(s) G(s,\xi) \mathrm{d}B^k_s
\end{align*}
so that $H(t,\xi)= H_1(t,\xi) H_2(t,\xi)$.
Then by \eqref{2024051801}, $H_1(t,\xi)$ is well-defined for all $t \in [0,\infty)$ and almost every $\xi \in \fR^d$.
On the other hand, \eqref{20240302 01} with the Fubini theorem implies that for any $t \in (0,\infty)$,
\begin{align}
\bE\left[ \left(\int_0^t  \exp\left(-\int_0^s \Re[2\tilde\psi(r,\xi)]\mathrm{d}r \right) 1_{\opar 0,\tau \cbrk}(s)\left|G(s,\xi)\right|^2_{l_2} \mathrm{d}s \right)^{1/2} \right] < \infty \quad  (a.e.) \quad \xi \in \fR^d.
\end{align}
Thus for any $t \in (0,\infty)$, $H_2(t,\xi)$ is well-defined as a random variable for almost every $\xi \in \fR^d$
since the predictability of the integrand can be easily shown as in the proof of Corollary \ref{20240417 01}.
Moreover, it is straightforward to demonstrate that both functions $(\omega,t,\xi) \mapsto H_1(t,\xi)$ and $(\omega,t,\xi) \mapsto H_2(t,\xi)$
are $\cP \times \cB(\fR^d)$-measurable (by considering a modification) due to the stochastic Fubini theorem and continuity of paths. 
This directly implies that $H$ is $\cP \times \cB(\fR^d)$-measurable.

Next we show \eqref{20240506 50}. We estimate $H_2$ first.
By the BDG inequality, for each $T \in (0,\infty)$ and almost every $\xi \in \fR^d$,
\begin{align*}
&\bE\left[\sup_{t \in [0,T]}\left|H_2(t,\xi) \right|\right] \\
&\leq C_{BDG}  \bE\left[ \left(\int_0^{T}  \exp\left(-\int_0^s \Re[2\tilde\psi(r,\xi)]\mathrm{d}r \right) 1_{\opar 0,\tau \cbrk}(s)\left|G(s,\xi)\right|^2_{l_2} \mathrm{d}s \right)^{1/2} \right].
\end{align*}
Therefore,
\begin{align*}
&\bE\left[\sup_{t \in [0,T]}\left|H(t,\xi) \right|\right] \\
&\leq \sup_{t \in [0,T]} |H_1(t,\xi)| \bE\left[\sup_{t \in [0,T]}\left|H_2(t,\xi) \right|\right] \\
&\leq C_{BDG}  \sup_{t \in [0,T]} \left[\exp\left(\int_0^t\Re[\tilde\psi(r,\xi)]\mathrm{d}r \right)\right] \bE\left[ \left(\int_0^{T}  \exp\left(-\int_0^s \Re[2\tilde\psi(r,\xi)]\mathrm{d}r \right) 1_{\opar 0,\tau \cbrk}(s)\left|G(s,\xi)\right|^2_{l_2} \mathrm{d}s \right)^{1/2} \right],
\end{align*}
which completes \eqref{20240506 50}.
\end{proof}

\begin{rem}
Recall the term
\begin{align*}
\sup_{t \in [0,T]} \left( \bE\left[\left(\int_0^{T}  \exp\left(\int_s^t \Re[2\tilde\psi(r,\xi)]\mathrm{d}r \right) 1_{\opar 0,\tau \cbrk}(s)\left|G(s,\xi)\right|^2_{l_2} \mathrm{d}s \right)^{1/2}  \right]  \right), 
\end{align*}
from \eqref{20240506 50}. 
Here sup is taken for all $t \in [0,T]$, and the integration is also taken for all $s \in [0,T]$.
Thus the term $\int_s^t \Re[2\tilde\psi(r,\xi)]\mathrm{d}r$ is defined as
$$
\int_s^t \Re[2\tilde\psi(r,\xi)]\mathrm{d}r = 
-\int_t^s \Re[2\tilde\psi(r,\xi)]\mathrm{d}r \quad \text{for $t \leq s$}
$$
according to the convention of the notation for integrals.
Therefore, the inequality
\begin{align*}
&\sup_{t \in [0,T]} \left( \bE\left[\left(\int_0^{T}  \exp\left(\int_s^t \Re[2\tilde\psi(r,\xi)]\mathrm{d}r \right) 1_{\opar 0,\tau \cbrk}(s)\left|G(s,\xi)\right|^2_{l_2} \mathrm{d}s \right)^{1/2}  \right] \right) \\
&\leq \sup_{0\leq s\leq t \leq T} \left[\exp\left(\int_s^t \Re[\tilde\psi(r,\xi)]\mathrm{d}r \right)\right]
\left( \bE\left[\left(\int_0^{T}  1_{\opar 0,\tau \cbrk}(s)\left|G(s,\xi)\right|^2_{l_2} \mathrm{d}s \right)^{1/2}  \right] \right) 
\end{align*}
is not generally true.
Instead, the inequality 
\begin{align*}
&\sup_{t \in [0,T]} \left( \bE\left[\left(\int_0^{T}  \exp\left(\int_s^t \Re[\tilde\psi(r,\xi)]\mathrm{d}r \right) 1_{\opar 0,\tau \cbrk}(s)\left|G(s,\xi)\right|^2_{l_2} \mathrm{d}s \right)^{1/2}  \right] \right) \\
&\leq \sup_{0\leq s\leq t \leq T} \left[\exp\left(\left|\int_s^t \Re[\tilde\psi(r,\xi)]\mathrm{d}r \right| \right)\right]
\left( \bE\left[\left(\int_0^{T}  1_{\opar 0,\tau \cbrk}(s)\left|G(s,\xi)\right|^2_{l_2} \mathrm{d}s \right)^{1/2}  \right] \right) 
\end{align*}
seems to be optimal to include all sign-changing symbols.
\end{rem}

Next, we prepare to estimate the product of the symbol and the spatial Fourier transform of a solution.

\begin{corollary}
										\label{cor 20240506}
Let $G \in \bL_{0}\left( \opar 0,  \tau \cbrk \times \fR^d, \cP \times \cB(\fR^d); l_2\right)$.
Suppose that \eqref{2024051801} holds and for any  $R \in (0,\infty)$,
\begin{align}
										\label{20240505 20}
\bE\left[\int_{B_R} \left(\int_0^T  \exp\left(-\int_0^s \Re[2\tilde\psi(r,\xi)]\mathrm{d}r \right) 1_{\opar 0,\tau \cbrk}(s)\left|G(s,\xi)\right|^2_{l_2} \mathrm{d}s \right)^{1/2} \mathrm{d}\xi \right] < \infty \quad  \forall T \in (0,\infty),
\end{align}
and
\begin{align}
										\label{20240301 20}
\bE \left[\int_{B_R}\int_0^T |\tilde\psi(t,\xi)| \left(\int_0^t \exp\left(\int_s^t\Re[2\tilde\psi(r,\xi)]\mathrm{d}r \right) 1_{\opar 0,\tau \cbrk}(s)
\left|G(s,\xi)\right|^2_{l_2}  \mathrm{d}s \right)^{1/2} \mathrm{d}t \mathrm{d}\xi \right]  < \infty \quad \forall T \in (0,\infty).
\end{align}
Denote
\begin{align*}
H(t,\xi) 
=\int_0^t  \exp\left(\int_s^t\tilde\psi(r,\xi)\mathrm{d}r \right) 1_{\opar 0,\tau \cbrk }(s) G^k(s,\xi)\mathrm{d}B^k_s.
\end{align*}
Then $H$ satisfies \eqref{20240409 01}, \eqref{20240506 50},
\begin{align}
									\notag
&\bE\left[ \int_0^T \int_{B_R}|\tilde\psi(t,\xi)| |H(t,\xi)| \mathrm{d}\xi  \mathrm{d}t \right] \\
									\label{20240506 10}
&\leq C_{BDG}\bE \left[\int_{B_R}\int_0^T |\tilde\psi(t,\xi)| \left(\int_0^t \exp\left(\int_s^t\Re[2\tilde\psi(r,\xi)]\mathrm{d}r \right) 
1_{\opar 0,\tau \cbrk}(s)\left|G(s,\xi)\right|^2_{l_2}  \mathrm{d}s \right)^{1/2} \mathrm{d}t \mathrm{d}\xi \right]  
< \infty, 
\end{align}
\begin{align}
										\notag
&\bE\left[ \int_0^T \int_{B_R} \left|H(t,\xi) \right|\mathrm{d}\xi \mathrm{d}t\right]  \\
										\label{20240408 01-2}
&\leq T \cdot C_{BDG} 
\int_{B_R} \sup_{0\leq s \leq t \leq T} \left|\exp\left(\int_s^t \Re[\tilde\psi(r,\xi)]\mathrm{d}r \right)\right|   \left( \bE\left[\left(\int_0^{T}  1_{\opar 0,\tau \cbrk}(s)\left|G(s,\xi)\right|^2_{l_2} \mathrm{d}s \right)^{1/2}  \right] \mathrm{d}\xi \right),
\end{align}
and
\begin{align}
										\notag
& \int_{B_R}\bE\left[ \sup_{t \in [0,T]} |H(t,\xi)|\right] \mathrm{d}\xi \\
										\label{20240408 01}
&\leq C_{BDG}\int_{B_R}  \sup_{t \in [0,T]} \left( \bE\left[\left(\int_0^{T}  \exp\left(\int_s^t \Re[2\tilde\psi(r,\xi)]\mathrm{d}r \right) 1_{\opar 0,\tau \cbrk}(s)\left|G(s,\xi)\right|^2_{l_2} \mathrm{d}s \right)^{1/2}  \right] \right) \mathrm{d}\xi 
\end{align}
for all  positive constants $T$ and $R$.
Especially,
\begin{align}
									\label{20240311 01}
H \in   \bL_{1,1,1,loc,\ell oc}\left( \Omega \times (0,\infty) \times \fR^d, \cP \times \cB(\fR^d), |\tilde\psi(t,\xi)|\mathrm{d}t \mathrm{d}\xi\right),
\end{align}
where
\begin{align*}
&H \in  \bL_{1,1,1,loc,\ell oc}\left( \Omega \times (0,\infty) \times \fR^d, \cP \times \cB(\fR^d), |\tilde\psi(t,\xi)|\mathrm{d}t \mathrm{d}\xi\right) \\
&\iff 
H \in  
\bL_{1,1,1,loc,\ell oc}\left( \Omega \times (0,T] \times \fR^d, \cP \times \cB(\fR^d), |\tilde\psi(t,\xi)|\mathrm{d}t \mathrm{d}\xi\right) \quad \forall T \in (0,\infty).
\end{align*}
If additionally $\exp\left( \left| \int_0^t \Re[\tilde\psi(r,\xi)]\mathrm{d}r\right| \right)$ is locally bounded, \textit{i.e.}
\begin{align}
								\label{20240316 02}
\esssup_{t \in [0,T], \xi \in B_R} \left[ \exp\left(\left|\int_0^t\Re[\tilde\psi(r,\xi)]\mathrm{d}r \right| \right) \right]  
< \infty \quad \forall T,R \in (0,\infty),
\end{align}
then 
\begin{align}
								\label{20240316 01}
H \in  \bL_{1,loc}\bC L_{1,\ell oc}\left( \Omega \times [0,\infty) \times \fR^d, \cP \times \cB(\fR^d)\right),
\end{align}
where
\begin{align*}
&H \in   \bL_{1,loc}\bC L_{1,\ell oc}\left( \Omega \times [0,\infty) \times \fR^d, \cP \times \cB(\fR^d)\right) \\
&\iff 
H \in  \bL_1\bC L_{1,\ell oc}\left( \Omega \times [0,T]\times \fR^d, \cP \times \cB(\fR^d)\right) \quad \forall T \in (0,\infty).
\end{align*}
\end{corollary}
\begin{proof}
By Corollary \ref{joint measurability issue 2}, $H$ is $\cP \times \cB(\fR^d)$-measurable and satisfies \eqref{20240409 01} and \eqref{20240506 50}.
Thus it suffices to show \eqref{20240506 10}, \eqref{20240408 01-2}, \eqref{20240408 01}, and \eqref{20240316 01}.
Due to the joint measurability, we can apply the Fubini theorem with \eqref{20240409 01}.
Indeed, for  all $T \in (0,\infty)$ and $R \in (0,\infty)$, applying the Fubini theorem, we have
\begin{align}
									\notag
&\bE\left[ \int_0^T \int_{B_R}|\tilde\psi(t,\xi)| |H(t,\xi)| \mathrm{d}\xi  \mathrm{d}t \right] \\
									\notag
&= \int_0^T \int_{B_R}|\tilde\psi(t,\xi)| \bE\left[ |H(t,\xi)|\right] \mathrm{d}\xi  \mathrm{d}t \\
									\notag
&\leq C_{BDG}\int_0^T \int_{B_R}|\tilde\psi(t,\xi)| \bE\left[  \left(\int_0^t  \left|\exp\left(\int_s^t\tilde\psi(r,\xi)\mathrm{d}r \right) 1_{\opar 0,\tau \cbrk}(s) G(s,\xi)\right|^2_{l_2} ds \right)^{1/2} \right] \mathrm{d}\xi  \mathrm{d}t \\
									\label{20240302 40}
&=\bE \left[\int_{B_R}\int_0^T |\tilde\psi(t,\xi)| \left(\int_0^t \exp\left(\int_s^t\Re[2\tilde\psi(r,\xi)]\mathrm{d}r \right) 
1_{\opar 0,\tau \cbrk}(s)\left|G(s,\xi)\right|^2_{l_2}  \mathrm{d}s \right)^{1/2} \mathrm{d}t \mathrm{d}\xi \right]  < \infty, 
\end{align}
where the last finiteness comes from \eqref{20240301 20}. 
Similarly, 
\begin{align*}
&\bE\left[ \int_0^T \int_{B_R} \left|H(t,\xi) \right|\mathrm{d}\xi \mathrm{d}t\right]  \\
&\leq C_{BDG} 
 \int_0^T \int_{B_R}   \left( \bE\left[\left(\int_0^{t}   \exp\left(\int_s^t \Re[\tilde\psi(r,\xi)]\mathrm{d}r \right)   1_{\opar 0,\tau \cbrk}(s)\left|G(s,\xi)\right|^2_{l_2} \mathrm{d}s \right)^{1/2}  \right] \mathrm{d}\xi \right) \mathrm{d}t\\
&\leq T \cdot C_{BDG} 
\int_{B_R} \sup_{0\leq s \leq t \leq T} \left|\exp\left(\int_s^t \Re[\tilde\psi(r,\xi)]\mathrm{d}r \right)\right|   \left( \bE\left[\left(\int_0^{T}  1_{\opar 0,\tau \cbrk}(s)\left|G(s,\xi)\right|^2_{l_2} \mathrm{d}s \right)^{1/2}  \right] \mathrm{d}\xi \right).
\end{align*}

Therefore, we have \eqref{20240506 10}, \eqref{20240408 01-2}, and \eqref{20240311 01}.

Next we show \eqref{20240408 01}.
Let $R,T \in (0,\infty)$.
Taking the integration $\int_{B_R} \cdot \mathrm{d}\xi$ to the both sides of \eqref{20240506 50}, we have
\begin{align}
										\notag
&\int_{B_R}\bE\left[ \sup_{t \in [0,T]} |H(t,\xi)|\right] \mathrm{d}\xi \\
											\label{20240728 01}
&\leq C_{BDG}\int_{B_R}  \sup_{t \in [0,T]} \left( \bE\left[\left(\int_0^{T}  \exp\left(\int_s^t \Re[2\tilde\psi(r,\xi)]\mathrm{d}r \right) 1_{\opar 0,\tau \cbrk}(s)\left|G(s,\xi)\right|^2_{l_2} \mathrm{d}s \right)^{1/2}  \right] \right) \mathrm{d}\xi , 
\end{align}
which directly implies  \eqref{20240408 01}.

At last, we prove \eqref{20240316 01}.
Since the $\cP \times \cB(\fR^d)$-measurability  had already been obtained, and
the additional condition in \eqref{20240316 02} and the Fubini theorem with \eqref{20240505 20} and \eqref{20240408 01} imply
\begin{align*}
\bE\left[\int_{B_R} \sup_{t \in [0,T]} |H(t,\xi)| \mathrm{d}\xi \right] < \infty,
\end{align*} 
it is sufficient to to show that the continuity of the paths.
It is straightforward that for almost every $\omega \in \Omega$ and $\xi \in \fR^d$, 
the both mappings
$$
t \in [0,\infty) \mapsto H_1(t,\xi) \quad \text{and} \quad t \in [0,\infty) \mapsto H_2(t,\xi)
$$
are continuous with respect to $t$ due to the properties of the Lebesgue integral and the It\^o stochastic integral, respectively. 
Therefore, the product $H$ of $H_1$ and $H_2$ is also continuous with respect to $t$ for almost every $(\omega,\xi) \in \Omega \times \fR^d$.
The corollary is proved. 
\end{proof}

\begin{rem}
										\label{measurability emph}
We reexamine the joint measurability issues to assess the validity of our stronger assumptions.
The stronger conditions \eqref{20240505 20} and \eqref{20240316 02} obviously imply the previous weaker condition
\begin{align}
									\label{20240424 01}
\bE\left[\int_{B_R} \left(\int_0^T  \exp\left(\int_s^T \Re[2\tilde\psi(r,\xi)]\mathrm{d}r \right) 1_{\opar 0,\tau \cbrk}(s)\left|G(s,\xi)\right|^2_{l_2} \mathrm{d}s \right)^{1/2} \mathrm{d}\xi \right] < \infty \quad  \forall T \in [0,\infty).
\end{align}
Additionally, recalling
\begin{align*}
H(t,\xi) 
=\int_0^t  \exp\left(\int_s^t\tilde\psi(r,\xi)\mathrm{d}r \right) 1_{\opar 0,\tau \cbrk }(s) G^k(s,\xi)\mathrm{d}B^k_s
\end{align*}
and applying \eqref{20240409 01}, for almost every $\xi$ and $t$, we have
\begin{align}
										\label{20240719 30}
\bE\left[ |H(t,\xi)| \right]
\lesssim
\bE\left[\left(\int_0^t  \exp\left(\int_s^t \Re[2\tilde\psi(r,\xi)]\mathrm{d}r \right) 1_{\opar 0,\tau \cbrk}(s)\left|G(s,\xi)\right|^2_{l_2} \mathrm{d}s \right)^{1/2} \right].
\end{align}
However, \eqref{20240424 01} alone was not sufficient to ensure that  $H(t,\xi)$  is $\rF \times \cB([0,\infty)) \times \cB(\fR^d)$-measurable as mentioned in Remark \ref{measure issue rem}.
Therefore \eqref{20240719 30}  does not generally imply 
\begin{align*}
\bE\left[ \int_0^T \int_{B_R}|H(t,\xi)|\mathrm{d}\xi \mathrm{d}t \right]
\lesssim
\bE\left[ \int_0^T \int_{B_R} \left(\int_0^t   \exp\left(\int_s^t \Re[2\tilde\psi(r,\xi)]\mathrm{d}r \right) 1_{\opar 0,\tau \cbrk}(s)\left|G(s,\xi)\right|^2_{l_2} \mathrm{d}s \right)^{1/2} \mathrm{d}\xi \mathrm{d}t \right]
\end{align*}
 since the Fubini theorem is not applicable in this case.
For this reason, we imposed stronger conditions \eqref{20240505 20} and \eqref{20240316 02} to guarantee that $H(t,\xi)$ is $\rF \times \cB([0,\infty)) \times \cB(\fR^d)$-measurable (indeed $\cP \times \cB(\fR^d)$-measurable).
The joint measurability was importantly used in \eqref{20240302 40} to apply the Fubini theorem.

Additionally, \eqref{20240316 02} is not necessary if we merely have an interest in the joint measurability as shown in the previous corollaries.
 \eqref{2024051801} and \eqref{20240505 20} are enough to show the joint measurability.
Furthermore, \eqref{2024051801} and \eqref{20240505 20} are also sufficient to ensure the path continuity so that for almost every $\xi$ and $\omega$, $t \mapsto H(t,\xi)$ is continuous.
 However, continuity of its integration such as $t \mapsto \int_{B_R}H(t,\xi)\mathrm{d}\xi$ and $t \mapsto \bE\left[\int_{B_R}H(t,\xi)\mathrm{d}\xi\right]$ is not guaranteed without \eqref{20240316 02}.
This is another important property for $H$ to be a $\cD'(\fR^d)$-valued continuous processes as in Definition \ref{space conti}.
Therefore both \eqref{20240316 02} and \eqref{2024051801} are considered to preserve both the joint measurability and the $\cD'(\fR^d)$-valued continuity of paths simultaneously. 
\end{rem}

\begin{rem}
								\label{BDG fail}
Suppose that all assumptions in Corollary \ref{cor 20240506} are satisfied.
Then for any bounded stopping time $\tilde \tau$, it obviously holds that
$$
H \in \bL_{1,1,1,\ell oc}\left( \opar 0, \tilde \tau \cbrk \times \fR^d, \cP \times \cB(\fR^d),|\tilde\psi(t,\xi)|\mathrm{d}t \mathrm{d}\xi\right)
$$
However, for a finite stopping time $\tilde \tau$, it is generally not true that
$$
H \in \bL_{1,1,1,\ell oc}\left( \opar 0, \tilde \tau \cbrk \times \fR^d, \cP \times \cB(\fR^d),|\tilde\psi(t,\xi)|\mathrm{d}t \mathrm{d}\xi\right).
$$
This is because the BDG inequality does not ensure that
\begin{align*}
\bE\left[ 1_{t \leq \tilde \tau}|H_2(t,\xi)|\right]
\lesssim \bE\left[ 1_{t \leq \tilde \tau} \left(\int_0^t  \left|\exp\left(-\int_0^s\tilde\psi(r,\xi)\mathrm{d}r \right) 1_{\opar 0,\tau \cbrk}(s) G(s,\xi) \right|_{l_2} ds \right)^{1/2} \right]
\end{align*}
due to the additional random term $1_{t \leq \tilde \tau}$.
Specifically, 
\begin{align*}
&\bE\left[ 1_{t \leq \tilde \tau} \left(\int_0^t  \left|\exp\left(-\int_0^s\tilde\psi(r,\xi)\mathrm{d}r \right) 1_{\opar 0,\tau \cbrk}(s) G(s,\xi) \right|_{l_2} ds \right)^{1/2} \right] \\
&\leq \bE\left[  \left(\int_0^{t \wedge \tau}  \left|\exp\left(-\int_0^s\tilde\psi(r,\xi)\mathrm{d}r \right) 1_{\opar 0,\tau \cbrk}(s) G(s,\xi) \right|_{l_2} ds \right)^{1/2} \right]
\end{align*}
and the equality does not generally hold.
\end{rem}

Now we compile all the corollaries to prepare the stability results for stochastic inhomogeneous data whose spatial Fourier transform has a finite expectation.
\begin{corollary}
										\label{stochastic solution part}
Let $G \in \bL^{\omega,x,t}_{1,1,2,loc,\ell oc}\left( \opar 0,  \tau \cbrk \times \fR^d, \cP \times \cB(\fR^d); l_2\right)$.
Suppose that $\tilde\psi$ satisfies \eqref{main deter symbol as} and \eqref{main deter symbol as 2}.
Denote
\begin{align*}
H(t,\xi) 
=\int_0^t  \exp\left(\int_s^t\tilde\psi(r,\xi)\mathrm{d}r \right) 1_{\opar 0,\tau \cbrk }(s) G^k(s,\xi)\mathrm{d}B^k_s.
\end{align*}
Then $H$ belongs to the intersection of the classes
\begin{align*}
\bL_{1,loc}\bC L_{1,\ell oc}\left( \Omega \times [0,\infty) \times \fR^d, \cP \times \cB(\fR^d)\right),
\end{align*}
and
\begin{align*}
\bL_{1,1,1,loc,\ell oc}\left( \Omega \times (0,\infty) \times \fR^d, \cP \times \cB(\fR^d), |\tilde\psi(t,\xi)|\mathrm{d}t \mathrm{d}\xi\right).
\end{align*}
Moreover, for all positive constants $T$ and $R$,
\begin{align}
										\notag
&\bE\left[ \int_0^T \int_{B_R}|H(t,\xi)|\mathrm{d}\xi \mathrm{d}t \right] \\
										\label{20240720 00}
&\leq T \cdot C_{BDG}\esssup_{0\leq s\leq t \leq T, \xi \in B_R}\left[  \exp\left(\int_s^t \Re[\tilde\psi(r,\xi)]\mathrm{d}r \right) \right]
\bE\left[  \int_{B_R} \left( \int_0^T  1_{\opar 0,\tau \cbrk}(s)\left|G(s,\xi)\right|^2_{l_2} \mathrm{d}s \right)^{1/2} \mathrm{d}\xi  \right],
\end{align}
\begin{align}
										\label{20240720 02}
\bE\left[ \int_{B_R} \sup_{t \in [0,T]} \left|H(t,\xi) \right|\mathrm{d}\xi \right]  
\leq \esssup_{0 \leq s \leq t \leq T, \xi \in B_R} \exp\left( \left|\int_s^t\Re[\tilde\psi(r,\xi)]\mathrm{d}r \right| \right)   \bE\left[\int_{B_R} \left(\int_0^{\tau \wedge T } \left|G(s,\xi)\right|^2_{l_2} \mathrm{d}s \right)^{1/2} \mathrm{d}\xi \right],
\end{align}
and
\begin{align}
								\notag
&\bE\left[ \int_0^T \int_{B_R}|\tilde\psi(t,\xi)| |H(t,\xi)| \mathrm{d}\xi  \mathrm{d}t \right] \\
									\label{20240512 11}
&\leq C_{BDG}
\esssup_{\xi \in B_R} \left(\int_0^T |\tilde\psi(t,\xi)| \sup_{0 \leq s \leq t}\left|\exp\left(\int_s^t\Re[\tilde\psi(r,\xi)]\mathrm{d}r \right) \right|  \mathrm{d}t \right)
 \bE \left[\int_{B_R} 
\left(\int_0^{\tau \wedge T}\left|G(s,\xi)\right|^2_{l_2}  \mathrm{d}s \right)^{1/2} \mathrm{d}\xi \right].
\end{align}

\end{corollary}
\begin{proof}
Since 
$$
G \in \bL^{\omega,x,t}_{1,1,2,loc,\ell oc}\left( \opar 0,  \tau \cbrk \times \fR^d, \cP \times \cB(\fR^d); l_2\right)
$$
and $\tilde\psi$ satisfies \eqref{main deter symbol as} and \eqref{main deter symbol as 2}, for all $R,T \in (0,\infty)$, we have
\begin{align*}
&\bE\left[\int_{B_R} \left(\int_0^T  \exp\left(-\int_0^s \Re[2\tilde\psi(r,\xi)]\mathrm{d}r \right) 1_{\opar 0,\tau \cbrk}(s)\left|G(s,\xi)\right|^2_{l_2} \mathrm{d}s \right)^{1/2} \mathrm{d}\xi \right]  \\
&\leq \esssup_{0 \leq s \leq T, \xi \in B_R}\exp\left( \left|\int_0^s \Re[\tilde\psi(r,\xi)]\mathrm{d}r \right| \right) \bE\left[\int_{B_R} \left(\int_0^{\tau \wedge T}   \left|G(s,\xi)\right|^2_{l_2} \mathrm{d}s \right)^{1/2} \mathrm{d}\xi \right]  < \infty,
\end{align*}
\begin{align*}
&\bE \left[\int_{B_R}\int_0^T |\tilde\psi(t,\xi)| \left(\int_0^t \exp\left(\int_s^t\Re[2\tilde\psi(r,\xi)]\mathrm{d}r \right) 1_{\opar 0,\tau \cbrk}(s)
\left|G(s,\xi)\right|^2_{l_2}  \mathrm{d}s \right)^{1/2} \mathrm{d}t \mathrm{d}\xi \right]   \\
&\leq \bE \left[\int_{B_R}\int_0^T |\tilde\psi(t,\xi)| \sup_{0 \leq s \leq t}\left|\exp\left(\int_s^t\Re[\tilde\psi(r,\xi)]\mathrm{d}r \right) \right|  \mathrm{d}t   \left(\int_0^{\tau \wedge T}
\left|G(s,\xi)\right|^2_{l_2}  \mathrm{d}s \right)^{1/2} \mathrm{d}\xi \right]  \\
&\leq
\esssup_{\xi \in B_R} \left(\int_0^T |\tilde\psi(t,\xi)| \sup_{0 \leq s \leq t}\left|\exp\left(\int_s^t\Re[\tilde\psi(r,\xi)]\mathrm{d}r \right) \right|  \mathrm{d}t \right)
 \bE \left[\int_{B_R} 
\left(\int_0^{\tau \wedge T}\left|G(s,\xi)\right|^2_{l_2}  \mathrm{d}s \right)^{1/2} \mathrm{d}\xi \right]  \\
&< \infty,
\end{align*}
\begin{align*}
\esssup_{t \in [0,T], \xi \in B_R} \left[ \exp\left(\left|\int_0^t\Re[\tilde\psi(r,\xi)]\mathrm{d}r\right| \right) \right]  
< \infty,
\end{align*}
and
\begin{align*}
\esssup_{\xi \in  B_R} \left( \int_0^T|\tilde \psi(t,\xi)| \mathrm{d}t  \right)
\leq
\esssup_{\xi \in  B_R} \left( \int_0^T|\tilde \psi(t,\xi)|\sup_{0\leq s \leq t}  \exp\left(\left|\int_s^t\Re[\tilde \psi(r,\xi)]\mathrm{d}r \right|\right) \mathrm{d}t  \right)
< \infty.
\end{align*}
Thus by Corollary \ref{cor 20240506}, we have
\begin{align*}
H \in \bL_{1,loc}\bC L_{1,\ell oc}\left( \Omega \times [0,\infty)\times \fR^d, \cP \times \cB(\fR^d)\right)
\end{align*}
and
\begin{align*}
H \in \bL_{1,1,1,loc,\ell oc}\left( \Omega \times (0,\infty) \times \fR^d, \cP \times \cB(\fR^d), |\tilde\psi(t,\xi)|\mathrm{d}t \mathrm{d}\xi\right).
\end{align*}
Moreover, we apply \eqref{20240408 01-2}, \eqref{20240506 10}, and \eqref{20240408 01} from Corollary \ref{cor 20240506} to obtain \eqref{20240720 00}, \eqref{20240720 02}, and \eqref{20240512 11}.
Indeed, for all $T,R \in (0,\infty)$, applying \eqref{20240408 01-2}, we have
\begin{align*}
&\bE\left[ \int_0^T \int_{B_R}|H(t,\xi)|\mathrm{d}\xi \mathrm{d}t \right] \\
&\leq C_{BDG}\bE\left[ \int_0^T \int_{B_R} \left(\int_0^t   \exp\left(\int_s^t \Re[2\tilde\psi(r,\xi)]\mathrm{d}r \right) 1_{\opar 0,\tau \cbrk}(s)\left|G(s,\xi)\right|^2_{l_2} \mathrm{d}s \right)^{1/2} \mathrm{d}\xi \mathrm{d}t \right]\\
&\leq T \cdot C_{BDG}\esssup_{0\leq s\leq t \leq T, \xi \in B_R}\left[  \exp\left(\int_s^t \Re[\tilde\psi(r,\xi)]\mathrm{d}r \right) \right]
\bE\left[  \int_{B_R} \left( \int_0^T  1_{\opar 0,\tau \cbrk}(s)\left|G(s,\xi)\right|^2_{l_2} \mathrm{d}s \right)^{1/2} \mathrm{d}\xi  \right].
\end{align*}
Additionally, by  \eqref{20240506 10} and \eqref{20240408 01} with the Fubini theorem, 
\begin{align*}
&\bE\left[ \int_0^T \int_{B_R}|\tilde\psi(t,\xi)| |H(t,\xi)| \mathrm{d}\xi  \mathrm{d}t \right] \\
&\leq C_{BDG}\bE \left[\int_{B_R}\int_0^T |\tilde\psi(t,\xi)| \left(\int_0^t \exp\left(\int_s^t\Re[2\tilde\psi(r,\xi)]\mathrm{d}r \right) 
1_{\opar 0,\tau \cbrk}(s)\left|G(s,\xi)\right|^2_{l_2}  \mathrm{d}s \right)^{1/2} \mathrm{d}t \mathrm{d}\xi \right]  \\
&\leq C_{BDG}
\esssup_{\xi \in B_R} \left(\int_0^T |\tilde\psi(t,\xi)| \sup_{0 \leq s \leq t}\left|\exp\left(\int_s^t\Re[\tilde\psi(r,\xi)]\mathrm{d}r \right) \right|  \mathrm{d}t \right)
 \bE \left[\int_{B_R} 
\left(\int_0^{\tau \wedge T}\left|G(s,\xi)\right|^2_{l_2}  \mathrm{d}s \right)^{1/2} \mathrm{d}\xi \right]
\end{align*}
and
\begin{align*}
&\bE\left[ \int_{B_R} \sup_{t \in [0,T]}\left|H(t,\xi) \right|\mathrm{d}\xi \right]  \\
&\leq C_{BDG}\int_{B_R}  \sup_{t \in [0,T]} \left( \bE\left[\left(\int_0^{T}  \exp\left(\int_s^t \Re[2\tilde\psi(r,\xi)]\mathrm{d}r \right) 1_{\opar 0,\tau \cbrk}(s)\left|G(s,\xi)\right|^2_{l_2} \mathrm{d}s \right)^{1/2}  \right] \mathrm{d}\xi \right)\\
&\leq C_{BDG}\esssup_{0 \leq s \leq t \leq T, \xi \in B_R} \left( \exp\left( \left|\int_s^t\Re[\tilde\psi(r,\xi)]\mathrm{d}r \right| \right) \right)   \bE\left[\int_{B_R} \left(\int_0^{\tau \wedge T } \left|G(s,\xi)\right|^2_{l_2} \mathrm{d}s \right)^{1/2} \mathrm{d}\xi \right]  
\end{align*}
for all $R, T \in (0,\infty)$. The corollary is proved.
\end{proof}
\begin{rem}
										\label{20240803 rem 10}
It may seem that \eqref{20240720 02} is superior to \eqref{20240720 00}. 
However, \eqref{20240720 02} does not imply \eqref{20240720 00} since 
\begin{align*}
\sup_{0\leq s\leq t \leq T} \left[\exp\left(\int_s^t \Re[\tilde\psi(r,\xi)]\mathrm{d}r \right)\right]
\leq
\sup_{0\leq s\leq t \leq T} \left[\exp\left(\left|\int_s^t \Re[\tilde\psi(r,\xi)]\mathrm{d}r \right| \right)\right].
\end{align*}
\end{rem}
\begin{rem}
If $\tilde \psi$ satisfies \eqref{main deter symbol as}, then the two conditions \eqref{20240505 20} and \eqref{20240424 01} become equivalent.
\end{rem}

\begin{rem}
										\label{ensure conti rem}
The condition in \eqref{main deter symbol as 2} is sufficiently given to ensure that \eqref{2024051801} holds, which importantly used to show the joint measurability and continuity of paths. 
If one merely wants to obtain \eqref{20240512 11}, then \eqref{main deter symbol as 2} can be slightly relaxed to 
\begin{align}
									\label{20240730 80}
\esssup_{\xi \in  B_R} \left( \int_0^T|\tilde \psi(t,\xi)|\sup_{0\leq s \leq t}  \left|\exp\left(\int_s^t\Re[\tilde \psi(r,\xi)]\mathrm{d}r\right)\right|     \mathrm{d}t  \right)
< \infty.
\end{align}
Additionally, \eqref{main deter symbol as 2} can be perfectly replaced with \eqref{2024051801} and \eqref{20240730 80} if one also wishes to maintain the continuity of paths.
This generalization could be meaningful due to logarithmic operators as mentioned in Remark \ref{rem 20240730 10}.
\end{rem}

\mysection{Uniqueness of a solution}
										\label{unique section}

In this section, we examine the uniqueness of a Fourier-space weak solution to \eqref{time eqn}. We demonstrate that the uniqueness holds in a broader class than the one in which a solution exists as stated in Theorem \ref{time thm}. We assume throughout this section that a Fourier-space weak solution exists. However, understanding the strength of this assumption is challenging.
At a minimum, it ensures that the function $\psi(t,\xi) \cF[u(t,\cdot)](\xi)$ is locally integrable, based on the definition of the solution if $u$ is a Fourier-space weak solution.

Surprisingly, this existence assumption alone is sufficient to guarantee uniqueness without requiring additional assumptions on the symbol $\psi$. 
In other words, we only assume that our symbol $\psi$ is a complex-valued 
$\cP  \times \cB(\fR^d)$-measurable function defined on $\Omega \times [0,\infty) \times \fR^d$.
In particular, our symbol is random in this section.

Our approach is straightforward. First, we show that our solution has a strong form of representation in terms of data concerning their spatial frequencies, which is derived from the equation. Then, we apply the classical Gr\"onwall inequality to prove the uniqueness. This is feasible because our class of test functions $\cF^{-1}\cD(\fR^d)$ is sufficient to find a good approximation of the identity (Sobolev's mollifier). Here is our representation lemma.

\begin{lem}[A representation of a solution from the equation]
							\label{fourier repre}
Let $u_0 \in \cF^{-1}\bL_{0,1,\ell oc}\left( \Omega \times \fR^d, \rF \times \cB(\fR^d) \right)$, 
\begin{align*}
f \in \cF^{-1}\bL_{0,1,1,\ell oc}\left( \opar 0,\tau \cbrk \times \fR^d, \rF \times \cB\left([0,\infty)\right) \times \cB(\fR^d)\right), 
~g \in \cF^{-1}\bL^{\omega,\xi,t}_{0,1,2,\ell oc}\left( \opar 0,\tau \cbrk \times \fR^d, \cP \times \cB(\fR^d) ; l_2\right),
\end{align*}
and $u$ be a Fourier-space weak solution to \eqref{time eqn}.
Assume that
\begin{align*}
u \in \cF^{-1}\bL_{0,1,1, \ell oc}\left( \opar 0,\tau \cbrk \times \fR^d, \rF \times \cB([0,\infty)) \times \cB(\fR^d),|\psi(t,\xi)|\mathrm{d}t \mathrm{d}\xi\right).
\end{align*}
Then
\begin{align}
								\notag
\cF[u(t,\cdot)](\xi)
&= \cF[u_0](\xi)+\int_0^t \psi(s,\xi) \cF[u(s,\cdot)](\xi) \mathrm{d}s 
+\int_0^t \cF[f(s,\cdot)](\xi)  \mathrm{d}s  \\
								\label{20230617 02}
&\quad +\int_0^t \cF[g^k(s,\cdot)](\xi)  \mathrm{d}B^k_s\quad (a.e.)~(\omega,t,\xi) \in \opar 0,\tau \cbrk \times \fR^d.
\end{align}
In particular, \eqref{20230617 02} shows that 
\begin{align*}
u \in \cF^{-1}\bC L_{1,\ell oc}\left( \clbrk 0,\tau \cbrk,   \rF \times \cB([0,\infty)) \times \cB(\fR^d)\right)
\end{align*}
due to the stochastic Fubini theorem and there exists a continuous (with respect to $t$) modification of $\cF[u(t,\cdot)](\xi)$.
\end{lem}
\begin{proof}
Let $\varphi \in \cF^{-1}\cD(\fR^d)$.
Then by using  the definition of a Fourier-space solution (Definition \ref{space weak solution 2}), the definition of the Fourier transform, and assumptions on $u$, $u_0$, $f$, and $g$,
we have
\begin{align}
									\notag
&\left\langle \cF[u(t,\cdot)],\cF[\varphi] \right\rangle  \\
									\notag
&= \left( \cF[u_0], \cF[\varphi]  \right)_{L_2(\fR^d)} +  \int_0^t  \left(  \psi(s,\xi)\cF[u(s,\cdot)](\xi) , \cF[\varphi](\xi) \right)_{L_2(\fR^d)}\mathrm{d}s 
+ \int_0^t \left( \cF[f(s,\cdot)],\cF[\varphi] \right)_{L_2(\fR^d)} \mathrm{d}s \\
									\label{20240229 10}
& \quad +\int_0^t 1_{\opar 0,\tau \cbrk}(s)\left( \cF[g^k(s,\cdot)],\cF[\varphi] \right)_{L_2(\fR^d)} \mathrm{d}B^k_s
\quad  (a.e.) \quad (\omega,t) \in \clbrk 0,\tau \cbrk,
\end{align}
where $(\cdot,\cdot)_{L_2(\fR^d)}$ denotes the $L_2(\fR^d)$-product, for instance, 
\begin{align*}
\left(  \psi(s,\xi)\cF[u(s,\cdot)](\xi) , \cF[\varphi](\xi) \right)_{L_2(\fR^d)}
=\int_{\fR^d} \psi(s,\xi)\cF[u(s,\cdot)](\xi) \overline{\cF[\varphi](\xi)} \mathrm{d}\xi.
\end{align*}
Here the generalized Minkowski inequality
\begin{align*}
\left( \int_0^{T} 1_{\opar 0,\tau  \cbrk}(t) \left(\int_{B_R} \left|\cF[g(\omega,t,\cdot)](\xi) \right|_{l_2}\mathrm{d}\xi\right)^2 \mathrm{d}t  \right)^{1/2}
\leq  \int_{B_R} \left(\int_0^{T} 1_{\opar 0,\tau  \cbrk}(t) \left|\cF[g(\omega,t,\cdot)](\xi) \right|^2_{l_2}\mathrm{d}t \right)^{1/2} \mathrm{d}\xi
\end{align*}
is applied to show 
\begin{align*}
\left\langle \cF[g^k(s,\cdot)], \cF[\varphi] \right\rangle
=\left( \cF[g^k(s,\cdot)], \cF[\varphi] \right)_{L_2(\fR^d)} \quad \forall k \in \fN.
\end{align*}
Additionally, by the (stochastic) Fubini theorem, \eqref{20240229 10} implies
\begin{align}
									\notag
&\left\langle \cF[u(t,\cdot)],\cF[\varphi] \right\rangle  \\
									\notag
&= \left( \cF[u_0], \cF[\varphi]  \right)_{L_2(\fR^d)} +   \left(  \int_0^t  \psi(s,\xi)\cF[u(s,\cdot)](\xi)\mathrm{d}s, \cF[\varphi](\xi) \right)_{L_2(\fR^d)} 
+ \left( \int_0^t \cF[f(s,\cdot)]\mathrm{d}s,\cF[\varphi] \right)_{L_2(\fR^d)}  \\
									\notag
& \quad +\left( \int_0^t 1_{\opar 0,\tau \cbrk}(s) \cF[g^k(s,\cdot)] \mathrm{d}B^k_s,\cF[\varphi] \right)_{L_2(\fR^d)}
\quad  (a.e.) \quad (\omega,t) \in \clbrk 0,\tau \cbrk.
\end{align}
Moreover by using the separability of $\cD(\fR^d)$, we have
\begin{align}
									\notag
&\left( \cF[u(t,\cdot)],\cF[\varphi] \right)_{L_2(\fR^d)}  \\
									\notag
&= \left( \cF[u_0], \cF[\varphi]  \right)_{L_2(\fR^d)} +   \left(  \int_0^t  \psi(s,\xi)\cF[u(s,\cdot)](\xi)\mathrm{d}s, \cF[\varphi](\xi) \right)_{L_2(\fR^d)} 
+ \left( \int_0^t \cF[f(s,\cdot)]\mathrm{d}s,\cF[\varphi] \right)_{L_2(\fR^d)}  \\
									\label{20240229 20}
& \quad +\left( \int_0^t 1_{\opar 0,\tau \cbrk}(s) \cF[g^k(s,\cdot)] \mathrm{d}B^k_s,\cF[\varphi] \right)_{L_2(\fR^d)}
\quad  (a.e.) \quad (\omega,t) \in \clbrk 0,\tau \cbrk \quad \forall \varphi \in \cF^{-1}\cD(\fR^d).
\end{align}
Next, we use a Sobolev mollifier with additional specific properties. 
Let $\chi$ be a function in $\cS(\fR^d)$ 
so that $\cF[\chi]$ is non-negative and symmetric, \textit{i.e.}
$\cF[\chi](\xi) =\cF[\chi](-\xi) \geq 0$ for all $\xi \in \fR^d$. 
Additionally, assume that $\cF[\chi]$ has a compact support  and 
$$
\int_{\fR^d} \chi(x)\mathrm{d}x=(2\pi)^{d/2}.
$$
For $\varepsilon \in (0,1)$, denote
\begin{align*}
\chi^\varepsilon(x) := \frac{1}{\varepsilon^d}\chi\left( \frac{x}{\varepsilon}\right).
\end{align*}
Fix $x \in \fR^d$ and put $\chi^\varepsilon(x-\cdot)$ in \eqref{20240229 20} instead of $\varphi$.
Then
\begin{align*}
&\int_{\fR^d} \cF[u(t,\cdot)](\xi) \overline{\cF[\chi^\varepsilon(x-\cdot)](\xi)}\mathrm{d}\xi \\
&=  \int_{\fR^d} \cF[u_0](\xi) \overline{\cF[\chi^\varepsilon(x-\cdot)](\xi)}\mathrm{d}\xi
+\int_{\fR^d} \int_0^t  \psi(s,\xi)\cF[u(s,\cdot)](\xi)  \overline{\cF[\chi^\varepsilon(x-\cdot)](\xi)}  \mathrm{d}s \mathrm{d}\xi \\
&\quad +\int_{\fR^d} \int_0^t  \cF[f(s,\cdot)](\xi)  \overline{\cF[\chi^\varepsilon(x-\cdot)](\xi)} \mathrm{d}s \mathrm{d}\xi 
+ \int_{\fR^d} \int_0^t  \cF[g^k(s,\cdot)](\xi)  \overline{\cF[\chi^\varepsilon(x-\cdot)](\xi)}  \mathrm{d}B^k_s \mathrm{d}\xi
\end{align*}
for almost every $(\omega,t) \in \clbrk 0,\tau \cbrk$.
Thus recalling properties of $\cF[\chi]$ and the Fourier transform, we have
\begin{align}
									\notag
&\int_{\fR^d} \mathrm{e}^{\mathrm{i}x \cdot \xi} \cF[u(t,\cdot)](\xi) \cF[\chi^\varepsilon](\xi) \mathrm{d}\xi \\
									\notag
&= \int_{\fR^d} \mathrm{e}^{\mathrm{i}x \cdot \xi} \cF[u_0](\xi) \cF[\chi^\varepsilon](\xi) \mathrm{d}\xi
+\int_{\fR^d} \mathrm{e}^{\mathrm{i}x \cdot \xi} \left(\int_0^t \psi(s,\xi) \cF[u(s,\cdot)](\xi)   \cF[\chi^\varepsilon](\xi) \mathrm{d}s\right)  \mathrm{d}\xi \\
									\label{20240229 30}
&\quad + \int_{\fR^d} \mathrm{e}^{\mathrm{i}x \cdot \xi} \left(\int_0^t \cF[f(s,\cdot)](\xi)   \cF[\chi^\varepsilon](\xi) \mathrm{d}s\right)  \mathrm{d}\xi
+ \int_{\fR^d}  \mathrm{e}^{\mathrm{i}x \cdot \xi} \left(\int_0^t  \cF[g^k(s,\cdot)](\xi)   \cF[\chi^\varepsilon](\xi)  \mathrm{d}B^k_s\right) \mathrm{d}\xi
\end{align}
for almost every $(\omega,t) \in \clbrk 0,\tau \cbrk$.
Note that both sides of \eqref{20240229 30} are continuous with respect to $x$, which implies that \eqref{20240229 30} holds for all $x$.
Therefore, due to  the Fourier inversion theorem,
\begin{align}
									\notag
\cF[u(t,\cdot)](\xi) \cF[\chi^\varepsilon](\xi)
&=\cF[u_0](\xi) \cF[\chi^\varepsilon](\xi)
+\int_0^t \psi(s,\xi) \cF[u(s,\cdot)](\xi)   \cF[\chi^\varepsilon](\xi) \mathrm{d}s
+\int_0^t \cF[f(s,\cdot)](\xi)   \cF[\chi^\varepsilon](\xi) \mathrm{d}s \\
										\label{20230617 01}
&\quad +\int_0^t  \cF[g^k(s,\cdot)](\xi)   \cF[\chi^\varepsilon](\xi)  \mathrm{d}B^k_s
\end{align}
for almost every $(\omega,t,\xi) \in \clbrk 0,\tau  \cbrk \times \fR^d$.
Observe that
\begin{align*}
\cF[\chi^\varepsilon](\xi)=\cF[\chi](\varepsilon\xi)
\end{align*}
and
\begin{align*}
\cF[\chi](0)= (2\pi)^{-d/2} \int_{\fR^d} \chi(y) \mathrm{d}y = 1.
\end{align*}
Finally, taking $\varepsilon \downarrow 0$ in \eqref{20230617 01}, we have
\eqref{20230617 02}. The lemma is proved.
\end{proof}
From the representation above, we obtain important two theorems.
The first one is {\em a prior estimate} and the second one is {\em uniqueness of a solution}.

We present a deterministic version of {\em a prior estimate} first.
Notably, it is feasible to determine the explicit constant in {\em a prior estimate} if there is no noise term.

\begin{thm}[(Deterministic) A priori estimate]
								\label{thm deter a priori}
Let $u_0 \in \cF^{-1}\bL_{0,1,\ell oc}\left( \Omega \times \fR^d, \rF \times \cB(\fR^d) \right)$, 
$$
f \in \cF^{-1}\bL_{0,1,1,\ell oc}\left( \opar 0,\tau \cbrk \times \fR^d, \rF \times \cB\left([0,\infty)\right) \times \cB(\fR^d)\right), 
$$
and $u$ be a Fourier-space weak solution to \eqref{deter eqn}.
Assume that for all $R \in (0,\infty)$,
\begin{align}
										\label{2024051901}
 \int_0^{\tau } \int_{B_R}|\psi(t,\xi)| |\cF[u(t,\cdot)](\xi)| \mathrm{d}\xi  \mathrm{d}t
&\leq C_0\int_{B_R} \left|\cF[u_0](\xi)\right|    \mathrm{d}\xi 
+C_1\int_{B_R} \int_0^{\tau }\left|\cF[f(s,\cdot)](\xi)\right|  \mathrm{d}s  \mathrm{d}\xi
\end{align}
with probability one, where $C_0$ and $C_1$ are positive constants.
Then 
$$
u \in \cF^{-1}\bC L_{1,\ell oc}\left( \clbrk 0,\tau \cbrk \times \fR^d, \rF \times \cB([0,\infty))\times \cB(\fR^d)\right)
$$
and for all $R \in (0,\infty)$,
\begin{align}
									\notag
&\int_{B_R} \sup_{t \in [0,\tau]} |\cF[u(t,\cdot)](\xi)| \mathrm{d}\xi   \\
									\notag
&\leq (1+C_0 )\int_{B_R} \left|\cF[u_0](\xi)\right|    \mathrm{d}\xi 
+ (1+C_1)\int_{B_R} \int_0^{\tau }\left|\cF[f(s,\cdot)](\xi)\right|  \mathrm{d}s  \mathrm{d}\xi \\
									\label{20240719 01}
&\leq (1+ C_0 \vee C_1 ) \left(\int_{B_R} \left|\cF[u_0](\xi)\right|    \mathrm{d}\xi 
+ \int_{B_R} \int_0^{\tau }\left|\cF[f(s,\cdot)](\xi)\right|  \mathrm{d}s  \mathrm{d}\xi\right)
\end{align}
with probability one.
\end{thm}
\begin{proof}
\eqref{2024051901} obviously implies 
\begin{align*}
u \in \cF^{-1}\bL_{0,1,1, \ell oc}\left( \opar 0,\tau \cbrk \times \fR^d, \rF \times \cB([0,\infty)) \times \cB(\fR^d),|\psi(t,\xi)|\mathrm{d}t \mathrm{d}\xi\right).
\end{align*}
Thus by applying Lemma \ref{fourier repre}, it suffices to show \eqref{20240719 01}.
By \eqref{20230617 02}, there exists a continuous modification of $\cF[u(t,\cdot)](\xi)$ so that
\begin{align*}
\cF[u(t,\cdot)](\xi)
&= \cF[u_0](\xi)+\int_0^t \psi(s,\xi) \cF[u(s,\cdot)](\xi) \mathrm{d}s 
+\int_0^t \cF[f(s,\cdot)](\xi)  \mathrm{d}s \quad (a.e.)~(\omega,t,\xi) \in \opar 0,\tau \cbrk \times \fR^d
\end{align*}
and it implies
\begin{align}
										\label{20240706 01-2}
|\cF[u(t,\cdot)](\xi)|
&\leq |\cF[u_0](\xi)|+\int_0^t |\psi(s,\xi) \cF[u(s,\cdot)](\xi)| \mathrm{d}s 
+\int_0^t |\cF[f(s,\cdot)](\xi)|  \mathrm{d}s \quad (a.e.)~(\omega,t,\xi) \in \opar 0,\tau \cbrk \times \fR^d.
\end{align}
Taking the integral and the essential supremum to both sides of \eqref{20240706 01-2}, and using \eqref{2024051901}, we have
\begin{align*}
&\left[   \int_{B_R} \esssup_{t \in [0,\tau]} |\cF[u(t,\cdot)](\xi)| \mathrm{d}\xi  \right]  \\
&\leq
\Bigg(
\left[\int_{B_R} \left|\cF[u_0](\xi)\right|    \mathrm{d}\xi \right] 
+\left[\int_{B_R} \int_0^{\tau }\left|\psi(s,\xi)\cF[u(s,\cdot)](\xi)\right|  \mathrm{d}s  \mathrm{d}\xi \right] \Bigg)
+\left[\int_{B_R} \int_0^{\tau }\left|\cF[f(s,\cdot)](\xi)\right|  \mathrm{d}s  \mathrm{d}\xi \right] \Bigg)\\
&\leq (1+C_0)
 \left[\int_{B_R} \left|\cF[u_0](\xi)\right|    \mathrm{d}\xi \right] 
+(1+C_1)\left[\int_{B_R} \int_0^{\tau }\left|\cF[f(s,\cdot)](\xi)\right|  \mathrm{d}s  \mathrm{d}\xi \right].
\end{align*}
Finally, using the continuity of the modification of $\cF[u(t,\cdot)](\xi)$, we have \eqref{20240719 01}.
The theorem is proved.
\end{proof}

Continuously, we present a stochastic version of {\em a prior estimate}.

\begin{thm}[(Stochastic) A priori estimate]
								\label{thm a priori}
Let $u_0 \in \cF^{-1}\bL_{1,1,\ell oc}\left( \Omega \times \fR^d, \rF \times \cB(\fR^d) \right)$, 
$$
f \in \cF^{-1}\bL_{1,1,1,loc, \ell oc}\left( \opar 0,\tau \cbrk \times \fR^d, \rF \times \cB\left([0,\infty)\right) \times \cB(\fR^d)\right), 
$$
$$
g \in \cF^{-1}\bL^{\omega,\xi,t}_{1,1,2,loc, \ell oc}\left( \opar 0,\tau \cbrk \times \fR^d, \cP \times \cB(\fR^d) ; l_2\right),
$$ and $u$ be a Fourier-space weak solution to \eqref{time eqn}.
Assume that for each $R,T \in (0,\infty)$,
\begin{align}
									\notag
\bE\left[ \int_0^{\tau \wedge T} \int_{B_R}|\psi(t,\xi)| |\cF[u(t,\cdot)](\xi)| \mathrm{d}\xi  \mathrm{d}t \right] 
&\leq C_0\bE \left[\int_{B_R} \left|\cF[u_0](\xi)\right|    \mathrm{d}\xi \right] 
+C_1\bE \left[\int_{B_R} \int_0^{\tau \wedge T}\left|\cF[f(s,\cdot)](\xi)\right|  \mathrm{d}s  \mathrm{d}\xi \right]  \\
										\label{20240509 02}
&\quad +C_2\bE \left[\int_{B_R} \left(\int_0^{\tau \wedge T}\left|\cF[g(s,\cdot)](\xi)\right|^2_{l_2}  \mathrm{d}s \right)^{1/2} \mathrm{d}\xi \right],
\end{align}
where $C_0$, $C_1$, and $C_2$ are positive constants.
Then 
$$
u \in \cF^{-1}\bL_{1,loc}\bC L_{1,\ell oc}\left( \clbrk 0,\tau \cbrk \times \fR^d, \rF \times \cB([0,\infty))\times \cB(\fR^d)\right)
$$
and for all $R, T \in (0,\infty)$, 
\begin{align}
									\notag
&\bE\left[   \int_{B_R} \sup_{t \in [0,\tau \wedge T]} |\cF[u(t,\cdot)](\xi)| \mathrm{d}\xi  \right]  \\
									\notag
&\leq (1+C_0)\bE \left[\int_{B_R} \left|\cF[u_0](\xi)\right|    \mathrm{d}\xi \right] 
+(1+C_1)\bE \left[\int_{B_R} \int_0^{ \tau \wedge T}\left|\cF[f(s,\cdot)](\xi)\right|  \mathrm{d}s  \mathrm{d}\xi \right]  \\ 
									\notag
&\qquad \qquad \qquad \qquad \qquad +(C_{BDG}+C_2)\bE \left[\int_{B_R} \left(\int_0^{ \tau \wedge T}\left|\cF[g(s,\cdot)](\xi)\right|^2_{l_2}  \mathrm{d}s \right)^{1/2} \mathrm{d}\xi \right] \\
									\notag
&\leq (C_{BDG}+\max\{C_0,C_1,C_2\})
\Bigg(\bE \left[\int_{B_R} \left|\cF[u_0](\xi)\right|    \mathrm{d}\xi \right] 
+\bE \left[\int_{B_R} \int_0^{ \tau \wedge T}\left|\cF[f(s,\cdot)](\xi)\right|  \mathrm{d}s  \mathrm{d}\xi \right]  \\ 
									\label{20240509 03}
&\qquad \qquad \qquad \qquad \qquad \qquad +\bE \left[\int_{B_R} \left(\int_0^{ \tau \wedge T}\left|\cF[g(s,\cdot)](\xi)\right|^2_{l_2}  \mathrm{d}s \right)^{1/2} \mathrm{d}\xi \right] \Bigg).
\end{align}
\end{thm}

\begin{proof}
Due to Lemma \ref{fourier repre}, it is sufficient to show \eqref{20240509 03}.
By \eqref{20230617 02},
\begin{align*}
\cF[u(t,\cdot)](\xi)
&= \cF[u_0](\xi)+\int_0^t \psi(s,\xi) \cF[u(s,\cdot)](\xi) \mathrm{d}s 
+\int_0^t \cF[f(s,\cdot)](\xi)  \mathrm{d}s  \\
&\quad +\int_0^t \cF[g^k(s,\cdot)](\xi)  \mathrm{d}B^k_s\quad (a.e.)~(\omega,t,\xi) \in \opar 0,\tau \cbrk \times \fR^d
\end{align*}
and the right-hand side becomes a continuous (with respect to $t$) modification of $\cF[u(t,\cdot)]$.
For each $T \in (0,\infty)$, considering the modification, applying the above equality and  the BDG inequality,  taking the integration with respect to $\xi$ on each $B_R$, and using the Fubini theorem with \eqref{20240509 02}, we have
\begin{align}
									\notag
&\bE\left[   \int_{B_R} \sup_{t \in [0,\tau \wedge T]} |\cF[u(t,\cdot)](\xi)| \mathrm{d}\xi  \right]  \\
									\notag
&\leq \bE\left[\int_{B_R}|\cF[u_0](\xi)|\mathrm{d}\xi\right] + \bE\left[\int_0^{\tau \wedge T} \int_{B_R}|\psi(s,\xi) \cF[u(s,\cdot)](\xi)| \mathrm{d}s \mathrm{d}\xi\right]
\\
									\notag
&\quad +\bE\left[\int_0^{\tau \wedge T} \int_{B_R}|\cF[f(s,\cdot)](\xi)|  \mathrm{d}\xi \mathrm{d}s \right]  
+C_{BDG}\bE\left[ \int_{B_R}\left(\int_0^{\tau \wedge T} |\cF[g(s,\cdot)](\xi)|^2_{l_2}  \mathrm{d}s \right)^{1/2}\mathrm{d}\xi\right]\\
&\leq (1+C_0)\bE\left[\int_{B_R}|\cF[u_0](\xi)|\mathrm{d}\xi\right] +(1+C_1)\bE\left[\int_0^{\tau \wedge T} \int_{B_R}|\cF[f(s,\cdot)](\xi)|  \mathrm{d}\xi \mathrm{d}s \right]  
\\
									\notag
&\quad +(C_{BDG}+C_2)\bE\left[ \int_{B_R}\left(\int_0^{\tau \wedge T} |\cF[g(s,\cdot)](\xi)|^2_{l_2}  \mathrm{d}s \right)^{1/2}\mathrm{d}\xi\right].
\end{align}
Therefore, \eqref{bdg lower} completes \eqref{20240509 03}.
The theorem is proved.
\end{proof}
\begin{rem}
The last inequality in \eqref{20240509 03} is not utilized in the paper and does not seem to be optimal, since
it is believed that $C_{BDG} >1$. 
Instead, the second inequality in \eqref{20240509 03} is used to estimate a solution.
In particular, \eqref{20240509 03} directly implies
\begin{align}
									\notag
&\bE\left[  \int_0^{\tau \wedge T} \int_{B_R} |\cF[u(t,\cdot)](\xi)| \mathrm{d}\xi \mathrm{d}t \right]  \\
									\notag
&\leq T(1+C_0)\bE \left[\int_{B_R} \left|\cF[u_0](\xi)\right|    \mathrm{d}\xi \right] 
+T(1+C_1)\bE \left[\int_{B_R} \int_0^{ \tau \wedge T}\left|\cF[f(s,\cdot)](\xi)\right|  \mathrm{d}s  \mathrm{d}\xi \right]  \\ 
									\notag
&\qquad \qquad \qquad \qquad \qquad +T(C_{BDG}+C_2)\bE \left[\int_{B_R} \left(\int_0^{ \tau \wedge T}\left|\cF[g(s,\cdot)](\xi)\right|^2_{l_2}  \mathrm{d}s \right)^{1/2} \mathrm{d}\xi \right].
\end{align}
\end{rem}

\begin{rem}
Recall that a solution $u$ to \eqref{time eqn} can be alternatively represented by using kernels related to 
$\int_s^t\psi(r,\xi)\mathrm{d}r$, as shown in \eqref{202040702 20} if the symbol is non-random. 
It is also feasible to estimate 
$u$ based on this representation of $u$ with the kernels, as demonstrated in Corollary \ref{stochastic solution part}. This type of estimate is even simpler to obtain for deterministic terms.
An example appears in Lemma \ref{deterministic solution part}.

On the other hand, note that \eqref{20230617 02} is used to obtain Theorem \ref{thm a priori}.
As can be easily seen, \eqref{20230617 02} is rather straightforwardly linked with equation \eqref{time eqn}.
Additionally, they do not involve kernels directly, unlike  \eqref{202040702 20}.
Thus these-type estimates are referred to as ``kernel-free estimates".
Nowadays, there are many well-known kernel-free estimates to obtain {\em a priori estimate} of a solution.

Especially, there are numerous well-established kernel-free estimates that provide {\em a priori} estimates for solutions of PDEs (\textit{cf.} \cite{DK2010,DK2018,D2020,Krylov 2008}). These estimates are particularly beneficial for complex equations where explicit kernel forms are unavailable. 

Additionally, simple kernel-free estimates for SPDEs can be derived by applying It\^o's formula to compositions of smooth functions and solutions. However, recent kernel-free estimates developed for PDEs cannot be applied to SPDEs, as they often depend on local estimates that are not valid for SPDEs.

Overall, it remains uncertain whether "kernel-based" or "kernel-free" methods yield better estimates with respect to the size of their constants. 
Specifically, we are unsure which constant derived from Corollary \ref{stochastic solution part} or Theorem \ref{thm a priori} is smaller at first glance. 
To determine the most optimal constants, we compare all constants from our main theorems.
In particular, we utilize the straightforward relationship between constants provided in Remark \ref{simple constant} and \eqref{20240721 50} to simplify the constants in our estimates. 
Generally, the constants from kernel-based estimates tend to be smaller.
However, this is not universally true for all estimates, as some conditions on $\psi$ are refined due to joint measurability rather than the estimates themselves.
Additionally, if the symbol is random, a solution $u$ is not expected to have a representation with the kernels, as emphasized several times.
In such cases, kernel-free estimates could be crucial for obtaining better estimates of the solution.
\end{rem}

Now we show the uniqueness of the Fourier-space weak solution.
All data in the uniqueness theorem can be more easily generalized than those in the representation theorem due to the linearity of the equations.
\begin{thm}[Uniqueness of a Fourier-space weak solution]
									\label{unique weak sol}
Let $u_0 \in \bL_{0}\left( \Omega, \rF;\cF^{-1}\cD'(\fR^d)\right)$, 
$$
f \in \bL_{0,1,loc}\left( \opar 0,\tau \cbrk, \rF \times \cB\left( (0,\infty) \right);\cF^{-1}\cD'(\fR^d)\right)
\text{,~and}~
 g \in \bL_{0,1,loc}\left( \opar 0,\tau \cbrk, \cP;\cF^{-1}\cD'(\fR^d;l_2)\right).
$$
Then a Fourier-space weak solution to \eqref{time eqn} is unique in the intersection of the spaces
\begin{align*}
\cF^{-1}\bL_{0,1,1,\ell oc}\left( \opar 0,\tau \cbrk \times \fR^d, \rF \times \cB([0,\infty)) \times \cB(\fR^d), |\psi(t,\xi)|\mathrm{d}t \mathrm{d}\xi\right)
\end{align*}
and
\begin{align*}
\cF^{-1}\bL_{0,0,1,\ell oc}\left( \opar 0,\tau \cbrk \times \fR^d, \rF \times \cB([0,\infty)) \times \cB(\fR^d) \right).
\end{align*}
\end{thm}
\begin{proof}
Let $u_1$ and $u_2$ be Fourier-space weak solutions to \eqref{time eqn} such that both $u_1$ and $u_2$ belong to 
\begin{align*}
 \cF^{-1}\bL_{0,1,1,\ell oc}\left( \opar 0,\tau \cbrk \times \fR^d, \rF \times \cB([0,\infty)) \times \cB(\fR^d), |\psi(t,\xi)|\mathrm{d}t \mathrm{d}\xi\right)
\end{align*}
and
\begin{align*}
 \cF^{-1}\bL_{0,0,1,\ell oc}\left( \opar 0,\tau \cbrk \times \fR^d, \rF \times \cB([0,\infty)) \times \cB(\fR^d) \right).
\end{align*}
Put
$$
u = u_1 - u_2.
$$
Then by Lemma \ref{fourier repre}, we have
\begin{align*}
\cF[u(t,\cdot)](\xi)
= \int_0^t \psi(s,\xi) \cF[u(s,\cdot)](\xi) \mathrm{d}s \quad (a.e.)~(\omega,t,\xi) \in \opar 0,\tau \cbrk \times \fR^d.
\end{align*}
Due to Gr\"onwall's inequality, we obtain
\begin{align}
							\label{20240201 10}
\cF[u(t,\cdot)](\xi) = 0 \quad (a.e.)~ (\omega,t,\xi) \in \opar 0,\tau \cbrk \times \fR^d.
\end{align}
Thus 
\begin{align}
									\label{20240724 70}
\cF[u_1(t,\cdot)](\xi)= \cF[u_2(t,\cdot)](\xi) \quad (a.e.) ~ (\omega,t,\xi) \in \opar 0,\tau \cbrk \times \fR^d,
\end{align}
which obviously implies that $u_1=u_2$ as an element in both spaces
\begin{align*}
\cF^{-1}\bL_{0,1,1,\ell oc}\left( \opar 0,\tau \cbrk \times \fR^d, \rF \times \cB([0,\infty)) \times \cB(\fR^d), |\psi(t,\xi)|\mathrm{d}t \mathrm{d}\xi\right)
\end{align*}
and
\begin{align*}
\cF^{-1}\bL_{0,0,1,\ell oc}\left( \opar 0,\tau \cbrk \times \fR^d, \rF \times \cB([0,\infty)) \times \cB(\fR^d)\right).
\end{align*}
The theorem is proved.
\end{proof}
\begin{rem}
It may appear that \eqref{20240201 10} trivially implies \eqref{20240724 70}.
However, if both solutions $u_1$ and $u_2$ are merely in 
\begin{align*}
\cF^{-1}\bL_{0,1,1,\ell oc}\left( \opar 0,\tau \cbrk \times \fR^d, \rF \times \cB([0,\infty)) \times \cB(\fR^d), |\psi(t,\xi)|\mathrm{d}t \mathrm{d}\xi\right),
\end{align*}
then this implication does not generally hold.
It is because we do not know if both $u_1$ and $u_2$ have realizable spatial frequency functions on the whole set  $\opar 0,\tau \cbrk \times \fR^d$ due to the possibility that $\psi(t,\xi)=0$ for some points as partially noted in Remark \ref{unique rem}.
Thus the other class
\begin{align*}
\cF^{-1}\bL_{0,0,1,\ell oc}\left( \opar 0,\tau \cbrk \times \fR^d, \rF \times \cB([0,\infty)) \times \cB(\fR^d)\right).
\end{align*}
is crucially used to ensure both $u_1$ and $u_2$ have realizable spatial frequency functions on the whole set. 
Hence the implication is obtained from the fact that both $u_1$ and $u_2$ are elements in the above class.

Moreover, it is clear that
\begin{align*}
\cF^{-1}\bC L_{1,\ell oc}\left( \clbrk 0,\tau \cbrk \times \fR^d, \rF \times \cB([0,\infty))\times \cB(\fR^d)\right) \subset \cF^{-1}\bL_{0,0,1,\ell oc}\left( \opar 0,\tau \cbrk \times \fR^d, \rF \times \cB([0,\infty)) \times \cB(\fR^d)\right)
\end{align*}
as discussed in Remark \ref{20240803 rem 1}.
Therefore, the solution $u$ in Lemma \ref{fourier repre} is unique in the class, and the condition
\begin{align*}
u_1, u_2 \in \cF^{-1}\bL_{0,0,1,\ell oc}\left( \opar 0,\tau \cbrk \times \fR^d, \rF \times \cB([0,\infty)) \times \cB(\fR^d) \right)
\end{align*}
in Theorem \ref{unique weak sol} can be omitted if the data $u_0$, $f$, and $g$ are in smaller spaces as provided  in Lemma \ref{fourier repre}, \textit{i.e.}
\begin{align*}
u_0 \in \cF^{-1}\bL_{0,1,\ell oc}\left( \Omega \times \fR^d, \rF \times \cB(\fR^d) \right),
\quad f \in \cF^{-1}\bL_{0,1,1,\ell oc}\left( \opar 0,\tau \cbrk \times \fR^d, \rF \times \cB\left([0,\infty)\right) \times \cB(\fR^d)\right), 
\end{align*}
and
\begin{align*}
g \in \cF^{-1}\bL^{\omega,\xi,t}_{0,1,2,\ell oc}\left( \opar 0,\tau \cbrk \times \fR^d, \cP \times \cB(\fR^d) ; l_2\right).
\end{align*}
\end{rem}

\mysection{Existence of a solution with a deterministic symbol}
												\label{exist deter}

For functions to be It\^o integrable, they must exhibit good measurability properties, such as predictability or progressive measurability. However, many integrands of the stochastic integral terms presented in Section \ref{sto fubini deter} lose their predictability when the symbols are allowed to be random. In other words, Itô's calculus cannot directly construct solutions unless the symbol is deterministic as discussed in Section \ref{sto fubini deter}. Therefore, we initially focus on obtaining the existence of a solution with a deterministic symbol. Nonetheless, this assumption can be removed in the next section due to the linearity of the equations and the flexibility of the data.

We recall the conditions on the deterministic symbol $\tilde \psi(t,\xi)$:
\begin{align}
									\label{main deter as}
C^{\mathrm{e}|\int\Re[\tilde \psi]|}_{R,T}
:=\esssup_{ 0\leq s \leq t \leq T, \xi \in   B_R} \left| \exp\left( \left|\int_s^t\Re[\tilde \psi(r,\xi)]\mathrm{d}r \right| \right)  \right|   
< \infty
\end{align}
and
\begin{align}
										\label{main deter as 2}
C^{| \tilde \psi|}_{R,T}:=\esssup_{\xi \in  B_R} \left( \int_0^T|\tilde \psi(t,\xi)|\sup_{0\leq s \leq t} \exp\left(  \left| \int_s^t\Re[\tilde \psi(r,\xi)]\mathrm{d}r \right| \right)    \mathrm{d}t  \right)
< \infty.
\end{align}
In particular, \eqref{main deter as 2} implies that for almost every $\xi \in \fR^d$,
\begin{align}
										\label{20240729 10}
\int_0^t |\tilde\psi(s,\xi)| \mathrm{d}s < \infty \quad  \forall t \in (0,\infty).
\end{align}
Due to these assumptions on the symbol $\tilde \psi(t,\xi)$,  
all the stochastic integral terms can be rigorously handled through corollaries of the stochastic Fubini theorems developed in Section \ref{sto fubini deter}. 
Furthermore, recall that the deterministic term 
$$
\int_0^t  \exp\left(\int_s^t\tilde\psi(r,\xi)\mathrm{d}r \right) 1_{\opar 0,\tau \cbrk}(s)\cF[f(s,\cdot)](\xi) \mathrm{d}s,
$$
which appeared in \eqref{202040702 20}.
We also need to control this deterministic term to find a solution $u$ since we hope that our solution $u$ satisfies \eqref{202040702 20}.
 Deterministic terms behave relatively well and can be easily treated by using the classical Fubini theorem.
Thus one can consider the weaker condition 
\begin{align}
									\label{main deter as 3}
C^{\mathrm{e}\int\Re[\tilde \psi]}_{R,T}:=\esssup_{ 0\leq s \leq t \leq T, \xi \in   B_R} \left| \exp\left( \int_s^t\Re[\tilde \psi(r,\xi)]\mathrm{d}r  \right)  \right|   
< \infty
\end{align}
than \eqref{main deter as}.
Here is a prerequisite lemma for the deterministic terms to show the existence of a solution. 
It also plays an important role in obtaining Theorem \ref{time deter thm}.
A continuous extension to $\Omega \times [0,\infty) \times \fR^d$ which is easily derived from integrals is also considered 
as in Corollary \ref{20240417 01}.

\begin{lem}
							\label{deterministic solution part}
Let $U_0 \in \bL_{0}\left( \Omega \times \fR^d, \rG \times \cB(\fR^d)\right)$ 
and $F \in \bL_{0}\left( \opar 0, \tau \cbrk \times \fR^d, \rH \times \cB(\fR^d) \right)$.
Suppose that \eqref{20240729 10} holds and for any $R \in (0,\infty)$,
\begin{align}
											\notag
& \int_{B_R}\exp\left(\int_0^T \Re[\tilde\psi(r,\xi)]\mathrm{d}r \right) |U_0(\xi)| \mathrm{d}\xi \\
										\label{20240316 70}
&\quad +\int_{B_R} \int_0^T  \exp\left(\int_s^T \Re[\tilde\psi(r,\xi)]\mathrm{d}r \right)  1_{\opar 0,\tau \cbrk}(s) \left|F(s,\xi) \right| \mathrm{d}s \mathrm{d}\xi < \infty \quad (a.s.) \quad \forall T \in (0,\infty).
\end{align}
Additionally, assume that for any $R \in (0,\infty)$,
\begin{align}
									\notag
& \int_{B_R}\int_0^T |\tilde\psi(\rho,\xi)| \exp\left(\int_0^\rho \Re[\tilde\psi(r,\xi)]\mathrm{d}r \right) |U_0(\xi)| \mathrm{d}\rho \mathrm{d}\xi \\
									\label{20240313 01}
&+\int_{B_R}\int_0^T |\tilde\psi(\rho,\xi)| \int_0^\rho \exp\left(\int_s^\rho\Re[\tilde\psi(r,\xi)]\mathrm{d}r \right)  1_{\opar 0,\tau \cbrk}(s) \left|F(s,\xi)\right| \mathrm{d}s \mathrm{d}\rho \mathrm{d}\xi  < \infty \quad (a.s.) \quad \forall T \in (0,\infty).
\end{align}
Denote
\begin{align*}
H(t,\xi) = \exp\left(\int_0^t \tilde\psi(r,\xi)\mathrm{d}r \right)U_0(\xi)
+\int_0^t  1_{\opar 0, \tau \cbrk}(s)\exp\left(\int_s^t \tilde\psi(r,\xi)\mathrm{d}r \right) F(s,\xi)  \mathrm{d}s.
\end{align*}
Then $H$ becomes an element in the intersection of the following two spaces
\begin{align*}
 \bL_{0,0,1,\ell oc}\left( \Omega \times (0,\infty) \times \fR^d, 
 \sigma\left( \rG  \times \cB\left([0,\infty)\right)  \cup  \rH \right)\times \cB(\fR^d) 
 \right) 
 \end{align*}
and
\begin{align*}
\bL_{0,1,1,loc, \ell oc}\left( \Omega \times (0,\infty) \times \fR^d,
 \sigma\left( \rG  \times \cB\left([0,\infty) \right) \cup  \rH \right) \times \cB(\fR^d)
, |\tilde\psi(t,\xi)|\mathrm{d}t \mathrm{d}\xi\right).
\end{align*}
In particular, \eqref{20240316 70} and \eqref{20240313 01} are satisfied if \eqref{20240729 10} and \eqref{main deter as 3} hold with the conditions 
\begin{align*}
U_0 \in \bL_{0,1,\ell oc}\left( \Omega \times \fR^d, \rG \times \cB(\fR^d)\right) \quad \text{and} \quad F \in \bL_{0,1,1,\ell oc}\left( \opar 0, \tau \cbrk \times \fR^d, \rH \times \cB(\fR^d) \right).
\end{align*}
For this case, we additionally have
\begin{align*}
H \in \bC_{loc}L_{1,\ell oc}\left( \Omega \times [0,\infty)\times \fR^d, 
 \sigma\left( \rG \ \times \cB\left([0,\infty)\right) \cup  \rH  \right)\times \cB(\fR^d)\right)
\end{align*}
and the estimates that for all  positive constants $T,R \in (0,\infty)$
\begin{align}
										\notag
&\int_{B_R} \sup_{t \in [0,T]}  |H(t,\xi)| \mathrm{d}\xi  \\
										\label{20240725 01}
&\leq \left(\sup_{0 \leq s \leq t \leq T,\xi \in B_R}\left| \exp \left( \int_s^t \Re[\tilde\psi(r,\xi)] \mathrm{d}r \right)  \right| \right)
\left[\int_{B_R} |U_0(\xi)|\mathrm{d}\xi + \int_{B_R} \int_0^{\tau \wedge T}\left|F(s,\xi)\right|  \mathrm{d}s  \mathrm{d}\xi \right]~(a.s.)
\end{align}
and
\begin{align}
									\notag
& \int_0^T \int_{B_R}|\tilde\psi(t,\xi)| |H(t,\xi)| \mathrm{d}\xi  \mathrm{d}t \\
									\label{20240512 10}
&\leq
\sup_{\xi \in B_R} \left(\int_0^T |\tilde\psi(t,\xi)| \sup_{0 \leq s \leq t}\left|\exp\left(\int_s^t\Re[\tilde\psi(r,\xi)]\mathrm{d}r \right) \right|  \mathrm{d}t \right)
\left[\int_{B_R} |U_0(\xi)|\mathrm{d}\xi + \int_{B_R} \int_0^{\tau \wedge T}\left|F(s,\xi)\right|  \mathrm{d}s  \mathrm{d}\xi \right] ~(a.s.).
\end{align}
\end{lem}

\begin{proof}
Since there is no stochastic integral term in the definition of $H$, most parts of the proof are easy consequences of elementary measure theories. Therefore we skip the proof.
We only mention that \eqref{20240729 10} is not abundant since it is utilized to define $H$ and show continuity of paths as discussed in Remark \ref{rem define psi} and Remark \ref{ensure conti rem}.
\end{proof}
\begin{rem}
A very similar definition of the space
\begin{align*}
\bL_{0,1,1,loc, \ell oc}\left( \Omega \times (0,\infty) \times \fR^d,
 \sigma\left( \rG  \times \cB\left([0,\infty) \right) \cup  \rH \right) \times \cB(\fR^d), |\tilde\psi(t,\xi)|\mathrm{d}t \mathrm{d}\xi\right).
\end{align*}
was mentioned in Corollary \ref{cor 20240506}.
However, the definition of the space
\begin{align*}
 \bL_{0,0,1,\ell oc}\left( \Omega \times (0,\infty) \times \fR^d, 
 \sigma\left( \rG  \times \cB\left([0,\infty)\right)  \cup  \rH \right)\times \cB(\fR^d) 
 \right) 
 \end{align*}
had never been mentioned. 
Nonetheless, we omitted the precise definitions of these spaces in Lemma \ref{deterministic solution part} since they can be easily inferred from previous definitions.
\end{rem}

\begin{rem}
								\label{additional deter rem}
If conditions \eqref{20240316 70} and \eqref{20240313 01} in Lemma \ref{deterministic solution part} are enhanced by
\begin{align*}
& \bE\left[\int_0^T \int_{B_R}\exp\left(\int_0^t \Re[\tilde\psi(r,\xi)]\mathrm{d}r \right) |U_0(\xi)| \mathrm{d}\xi \mathrm{d}t \right] \\
&\quad +\bE\left[ \int_0^T \int_{B_R} \int_0^t  \exp\left(\int_s^t \Re[\tilde\psi(r,\xi)]\mathrm{d}r \right)  1_{\opar 0,\tau \cbrk}(s) \left|F(s,\xi) \right| \mathrm{d}s \mathrm{d}\xi \mathrm{d}t \right] < \infty  \quad \forall T \in (0,\infty).
\end{align*}
and
\begin{align*}
& \bE\left[\int_{B_R}\int_0^T |\tilde\psi(\rho,\xi)| \exp\left(\int_0^\rho \Re[\tilde\psi(r,\xi)]\mathrm{d}r \right) |U_0(\xi)| \mathrm{d}\rho \mathrm{d}\xi\right] \\
&+\bE\left[\int_{B_R}\int_0^T |\tilde\psi(\rho,\xi)| \int_0^\rho \exp\left(\int_s^\rho\Re[\tilde\psi(r,\xi)]\mathrm{d}r \right)  1_{\opar 0,\tau \cbrk}(s) \left|F(s,\xi)\right| \mathrm{d}s \mathrm{d}\rho \mathrm{d}\xi \right]  < \infty  \quad \forall T \in (0,\infty),
\end{align*}
then $H$ is in the intersection of the spaces
\begin{align*}
 \bL_{1,1,1,loc,\ell oc}\left( \Omega \times (0,\infty) \times \fR^d, 
 \sigma\left( \rG \ \times \cB\left([0,\infty)\right)  \cup  \rH  \right)\times \cB(\fR^d)
 \right),
 \end{align*}
and
\begin{align*}
\bL_{1,1,1,loc, \ell oc}\left( \Omega \times (0,\infty) \times \fR^d,
 \sigma\left( \rG \ \times \cB\left([0,\infty)\right)  \cup  \rH \right)\times \cB(\fR^d) 
, |\tilde\psi(t,\xi)|\mathrm{d}t \mathrm{d}\xi\right).
\end{align*}
Additionally, if \eqref{20240729 10} and \eqref{main deter as 3} are satisfied with the conditions 
\begin{align*}
U_0 \in \bL_{1,1,\ell oc}\left( \Omega \times \fR^d, \rG \times \cB(\fR^d)\right) \quad \text{and} \quad F \in \bL_{1,1,1,loc,\ell oc}\left( \opar 0, \tau \cbrk \times \fR^d, \rH \times \cB(\fR^d) \right),
\end{align*}
then
\begin{align*}
H \in \bL_{1,loc}\bC L_{1,\ell oc}\left( \Omega \times [0,\infty)\times \fR^d, 
 \sigma\left( \rG \ \times \cB\left([0,\infty)\right) \cup  \rH  \right)\times \cB(\fR^d)\right).
\end{align*}
Finally,  if $U_0$ is $\rF_0 \times \cB(\fR^d)$-measurable  and $F$ is  $\cP \times \cB(\fR^d)$-measurable, then $H$ becomes $\cP \times \cB(\fR^d)$-measurable.
\end{rem}

\begin{rem}
If the data $U_0$ and $F$ in Lemma \ref{deterministic solution part} are also  non-random, then 
\eqref{20240725 01} and \eqref{20240512 10} can be considered as deterministic estimates.
Consequently, this lemma is applicable even to a random symbol by utilizing these estimates for each $\omega$, 
which is a crucial aspect in proving Corollary \ref{exist cor 1} in the subsequent section.
\end{rem}

Since we solve equations on random time intervals $\opar 0, \tau \cbrk$, 
it needs to consider restrictions of classes consisting of functions on $\Omega \times [0,\infty) \times \fR^d$ or $\Omega \times (0,\infty) \times \fR^d$.
These restrictions can be understood straightforwardly.
However, we provide more detail to demonstrate the rigor of our theories.
The following embedding properties are easily proved as in Proposition \ref{temporal local prop} since our $\tau$ is a finite stopping time.
\begin{lem}
									\label{restriction lem}
\begin{enumerate}[(i)]

\item 
\begin{align*}
 \bL_{0,1,1,loc,\ell oc}\left( \Omega \times (0,\infty) \times \fR^d, 
 \sigma\left( \rG \ \times \cB\left([0,\infty)\right)  \cup  \rH  \right) \times \cB(\fR^d)
 \right) 
 \end{align*}
is a subspace of
\begin{align*}
 \bL_{0,1,1,\ell oc}\left( \opar 0,\tau \cbrk \times \fR^d, 
 \sigma\left( \rG \ \times \cB\left([0,\infty)\right) \cup  \rH\right) \times \cB(\fR^d) 
 \right) 
 \end{align*}

\item
\begin{align*}
\bL_{0,1,1,loc, \ell oc}\left( \Omega \times (0,\infty) \times \fR^d,
 \sigma\left( \rG \ \times \cB\left([0,\infty)\right)  \cup  \rH \right) \times \cB(\fR^d) 
, |\tilde\psi(t,\xi)|\mathrm{d}t \mathrm{d}\xi\right)
\end{align*}
is a subspace of
\begin{align*}
\bL_{0,1,1,\ell oc}\left(\opar 0,\tau \cbrk \times \fR^d,
 \sigma\left( \rG \ \times \cB\left([0,\infty)\right) \cup  \rH \right) \times \cB(\fR^d) 
 , |\tilde\psi(t,\xi)|\mathrm{d}t \mathrm{d}\xi\right).
\end{align*}

\item
\begin{align*}
\bC_{loc}L_{1,\ell oc}\left( \Omega \times [0,\infty)\times \fR^d, 
 \sigma\left( \rG  \times \cB\left([0,\infty)\right) \cup  \rH  \right)\times(\fR^d)\right)
\end{align*}
is a subspace of
\begin{align*}
 \bC L_{1,\ell oc}\left(\clbrk 0,\tau \cbrk \times \fR^d, \sigma\left( \rG  \times \cB([0,\infty)) \cup  \rH  \right) \times \cB(\fR^d)\right)
\end{align*}

\item 
\begin{align*}
 \bL_{1,1,1,loc,\ell oc}\left( \Omega \times (0,\infty) \times \fR^d, 
 \sigma\left( \rG \ \times \cB\left([0,\infty)\right)  \cup  \rH \right) \times \cB(\fR^d) 
 \right) 
 \end{align*}
is a subspace of
\begin{align*}
 \bL_{1,1,1,loc,\ell oc}\left( \opar 0,\tau \cbrk \times \fR^d, 
 \sigma\left( \rG \ \times \cB\left([0,\infty)\right)  \cup  \rH  \right) \times \cB(\fR^d)
 \right) 
 \end{align*}

\item
\begin{align*}
\bL_{1,1,1,loc, \ell oc}\left( \Omega \times (0,\infty) \times \fR^d,
 \sigma\left( \rG \ \times \cB\left([0,\infty)\right)  \cup  \rH  \right) \times \cB(\fR^d)
, |\tilde\psi(t,\xi)|\mathrm{d}t \mathrm{d}\xi\right)
\end{align*}
is a subspace of
\begin{align*}
\bL_{1,1,1,loc,\ell oc}\left(\opar 0,\tau \cbrk \times \fR^d,
 \sigma\left( \rG \ \times \cB\left([0,\infty)\right)  \cup  \rH \right) \times \cB(\fR^d) 
, |\tilde\psi(t,\xi)|\mathrm{d}t \mathrm{d}\xi\right).
\end{align*}

\item
\begin{align*}
\bL_{1,loc}\bC L_{1,\ell oc}\left( \Omega \times [0,\infty)\times \fR^d, 
 \sigma\left( \rG  \times \cB\left([0,\infty)\right) \cup  \rH  \right)\times(\fR^d)\right)
\end{align*}
is a subspace of
\begin{align*}
\bL_{1,loc} \bC L_{1,\ell oc}\left(\clbrk 0,\tau \cbrk \times \fR^d, \sigma\left( \rG  \times \cB([0,\infty)) \cup  \rH  \right) \times \cB(\fR^d)\right).
\end{align*}

\end{enumerate}
\end{lem}
\begin{proof}
The statements (i)-(iii) might appear trivial given that $\tau$ is a stopping time. The other statements (iv)-(vi) are derived from (i)-(iii) by simply taking expectations. 
Nonetheless, we provide additional details to ensure clarity for the readers. Due to the similarity in the proofs, we only prove (ii) and (v).

Let
\begin{align}
										\label{20240430 01}
H \in \bL_{0,1,1,loc, \ell oc}\left( \Omega \times (0,\infty) \times \fR^d,
 \sigma\left( \rG \ \times \cB\left([0,\infty)\right)  \cup  \rH \right)\times \cB(\fR^d) 
, |\tilde\psi(t,\xi)|\mathrm{d}t \mathrm{d}\xi\right).
\end{align}
It suffices to show
\begin{align*}
H \in  \bL_{0,1,1,\ell oc}\left( \opar 0,\tau \cbrk \times \fR^d, 
 \sigma\left( \rG \ \times \cB\left([0,\infty)\right)  \cup  \rH \right) \times \cB(\fR^d) 
, |\tilde\psi(t,\xi)|\mathrm{d}t \mathrm{d}\xi\right).
\end{align*}
 \eqref{20240430 01} implies that there exists a $\Omega'$ such that $P(\Omega')=1$ and for any $\omega \in \Omega'$,
\begin{align*}
\int_{B_R}\int_0^T |H(t,\xi)| |\tilde\psi(t,\xi)|  \mathrm{d}t \mathrm{d}\xi  < \infty \ \quad \forall R,T \in (0,\infty).
\end{align*}
Since $\tau$ is a finite stopping time, for each $\omega \in \Omega'$, there exists a $T_\omega \in (0,\infty)$ so that
\begin{align*}
\tau(\omega) \leq T_\omega.
\end{align*}
Therefore, for all $\omega \in \Omega'$ and $R \in (0,\infty)$, we have
\begin{align}
										\label{20240501 02}
\int_0^\tau \int_{\fR^d} |H(t,\xi)| |\tilde\psi(t,\xi)| \mathrm{d}t \mathrm{d}\xi 
\leq 
\int_0^{T_\omega} \int_{\fR^d} |H(t,\xi)| |\tilde\psi(t,\xi)| \mathrm{d}t \mathrm{d}\xi < \infty
\end{align}
and (ii) is proved.

To prove (v), consider a particular case of \eqref{20240501 02}.
For any $\omega \in \Omega'$ and $R,T \in (0,\infty)$, 
we have
\begin{align}
										\label{20240501 01}
\int_0^{\tau \wedge T} \int_{\fR^d} |H(t,\xi)| |\tilde\psi(t,\xi)| \mathrm{d}t \mathrm{d}\xi 
\leq 
\int_0^{T} \int_{\fR^d} |H(t,\xi)| |\tilde\psi(t,\xi)| \mathrm{d}t \mathrm{d}\xi < \infty.
\end{align}
Thus, if 
\begin{align*}
H \in \bL_{1,1,1,loc, \ell oc}\left( \Omega \times (0,\infty) \times \fR^d,
 \sigma\left( \rG \ \times \cB\left([0,\infty)\right)  \cup  \rH \right) \times \cB(\fR^d) 
, |\tilde\psi(t,\xi)|\mathrm{d}t \mathrm{d}\xi\right),
\end{align*}
then we have
\begin{align*}
H \in \bL_{1,1,1,loc,\ell oc}\left(\opar 0,\tau \cbrk \times \fR^d,
 \sigma\left( \rG \ \times \cB\left([0,\infty)\right)  \cup  \rH \right) \times \cB(\fR^d) 
, |\tilde\psi(t,\xi)|\mathrm{d}t \mathrm{d}\xi\right)
\end{align*}
by taking the expectations to both sides of \eqref{20240501 01}.
The lemma is proved. 
\end{proof}
\begin{rem}
Take the expectations to both sides of \eqref{20240501 02} and then we have
\begin{align}
										\label{20240501 03}
\bE\left[\int_0^\tau \int_{\fR^d} |H(t,\xi)| |\tilde\psi(t,\xi)| \mathrm{d}t \mathrm{d}\xi \right]
\leq 
\bE\left[\int_0^{T_\omega} \int_{\fR^d} |H(t,\xi)| |\tilde\psi(t,\xi)| \mathrm{d}t \mathrm{d}\xi \right].
\end{align}
Generally, there is no guarantee that the right-hand side of \eqref{20240501 03} is finite even though
\begin{align*}
H \in \bL_{1,1,1,loc, \ell oc}\left( \Omega \times (0,\infty) \times \fR^d,
 \sigma\left( \rG \ \times \cB\left([0,\infty)\right)  \cup  \rH  \right) \times \cB(\fR^d) 
, |\tilde\psi(t,\xi)|\mathrm{d}t \mathrm{d}\xi\right).
\end{align*}
It is because the choice $T_w$ is determined according to each sample point $\omega$.
Thus it turns out that the space
\begin{align*}
\bL_{1,1,1,loc, \ell oc}\left( \Omega \times (0,\infty) \times \fR^d,
 \sigma\left( \rG \ \times \cB\left([0,\infty)\right) \cup  \rH \right) \times \cB(\fR^d) 
, |\tilde\psi(t,\xi)|\mathrm{d}t \mathrm{d}\xi\right).
\end{align*}
is not embedded into the space
\begin{align*}
\bL_{1,1,1,\ell oc}\left(\opar 0,\tau \cbrk \times \fR^d,
 \sigma\left( \rG \ \times \cB\left([0,\infty)\right)  \cup  \rH  \right) \times \cB(\fR^d) 
, |\tilde\psi(t,\xi)|\mathrm{d}t \mathrm{d}\xi\right)
\end{align*}
unless the stopping time $\tau$ is bounded.
Nevertheless, it is embedded into a slightly larger local space as mentioned in Lemma \ref{restriction lem}.
\end{rem}
\begin{rem}
Due to Theorem \ref{homeo thm}, all results in Lemma \ref{restriction lem} can be translated to their inverse transforms.
For instance, 
\begin{align*}
\cF^{-1}\bC_{loc}L_{1,\ell oc}\left( \Omega \times [0,\infty)\times \fR^d, 
 \sigma\left( \rG  \times \cB\left([0,\infty)\right) \cup  \rH  \right)\times(\fR^d)\right)
\end{align*}
is a subspace of
\begin{align*}
\cF^{-1} \bC L_{1,\ell oc}\left(\clbrk 0,\tau \cbrk \times \fR^d, \sigma\left( \rG  \times \cB([0,\infty)) \cup  \rH  \right) \times \cB(\fR^d)\right).
\end{align*}
\end{rem}

\begin{thm}[Existence of a Fourier-space weak solution with a deterministic symbol]
									\label{deter exist weak sol}
Let 
\begin{align*}
u_0 \in \cF^{-1}\bL_{0,1,\ell oc}\left( \Omega \times \fR^d , \rG \times \cB(\fR^d)\right),  
\quad f \in \cF^{-1}\bL_{0,1,1,\ell oc}\left( \opar 0, \tau \cbrk \times \fR^d, \rH \times \cB(\fR^d) \right),
\end{align*}
and 
$$
g \in \cF^{-1}\bL^{\omega,\xi,t}_{0,1,2,\ell oc}\left( \opar 0,  \tau \cbrk \times \fR^d, \cP \times \cB(\fR^d) ; l_2\right).
$$
Suppose that $\tilde\psi$ satisfies  \eqref{main deter as} and \eqref{main deter as 2}.
For each $(\omega,t) \in \opar 0, \tau \cbrk$,
define 
\begin{align}
\varphi \in \cF^{-1}\cD(\fR^d) \mapsto \langle u(t,\cdot), \varphi \rangle
										\notag
&:= \int_{\fR^d}\left( \exp\left(\int_0^t\tilde\psi(r,\xi)\mathrm{d}r \right) \cF[u_0](\xi) \right) \overline{\cF[\varphi](\xi)} \mathrm{d}\xi \\
										\notag
&\quad +\int_{\fR^d} \left(\int_0^t  \exp\left(\int_s^t\tilde\psi(r,\xi)\mathrm{d}r \right) 1_{\opar 0, \tau \cbrk}(s)\cF[f(s,\cdot)](\xi)\mathrm{d}s \right) \overline{\cF[\varphi](\xi)} \mathrm{d}\xi  \\
										\label{20240314 10}
&\quad + \int_{\fR^d} \left(\int_0^t  \exp\left(\int_s^t\tilde\psi(r,\xi)\mathrm{d}r \right) 1_{\opar 0, \tau \cbrk}(s)\cF[g^k(s,\cdot)](\xi)\mathrm{d}B^k_s\right) \overline{\cF[\varphi](\xi)} \mathrm{d}\xi
\end{align}
as a $\cF^{-1}\cD'(\fR^d)$-valued stochastic processes.
Then $u$ becomes a Fourier-space weak solution to \eqref{time eqn}, which belongs to the intersection of the classes
$$
\cF^{-1}\bC L_{1,\ell oc}\left( \clbrk 0,\tau \cbrk \times \fR^d,\sigma\left( \rG \times \cB([0,\infty))  \cup \rH\cup \cP   \right)\times \cB(\fR^d)\right)
$$
and
\begin{align*}
\cF^{-1}\bL_{0,1,1,\ell oc}\left( \opar 0,\tau \cbrk \times \fR^d, 
\sigma\left( \rG \times \cB([0,\infty)) \cup \rH  \cup \cP \right) \times \cB(\fR^d), |\tilde\psi(t,\xi)|\mathrm{d}t \mathrm{d}\xi\right).
\end{align*}
\end{thm}
\begin{proof}
Denote
$$
U_0(\xi) := \cF[u_0](\xi),
~F(t,\xi) := \cF[f(t,\cdot)](\xi),
~G^k(t,\xi) := \cF[g^k(t,\cdot)](\xi),
$$
\begin{align*}
H_1(t,\xi)  :=  \exp\left(\int_0^t\tilde\psi(r,\xi)\mathrm{d}r \right) U_0(\xi)
+ \int_0^t  \exp\left(\int_s^t\tilde\psi(r,\xi)\mathrm{d}r \right) 1_{\opar 0, \tau \cbrk}(s)F(s,\xi)\mathrm{d}s,
\end{align*}
and
\begin{align*}
H_2(t,\xi) :=  \int_0^t  \exp\left(\int_s^t\tilde\psi(r,\xi)\mathrm{d}r \right) 1_{\opar 0, \tau \cbrk}(s)G^k(s,\xi)\mathrm{d}B^k_s.
\end{align*}
First we show that $U_0$ and $F$ satisfy \eqref{20240316 70} and \eqref{20240313 01} in Lemma \ref{deterministic solution part}
by reminding all assumptions in the theorem.
However, it is very elementary that  \eqref{main deter as} and 
the conditions
$$
u_0 \in \cF^{-1}\bL_{0,1,\ell oc}\left( \Omega \times \fR^d , \rG \times \cB(\fR^d)\right)  
$$
and
$$
f \in \cF^{-1}\bL_{0,1,1,\ell oc}\left( \opar 0, \tau \cbrk \times \fR^d, \rH \times \cB(\fR^d) \right)
$$
imply \eqref{20240316 70}. Thus we only show \eqref{20240313 01}.
Reminding  \eqref{main deter as 2}, the conditions on $u_0$ and $f$, and the Fubini theorem, we have
\begin{align*}
& \int_{B_R}\int_0^T |\tilde\psi(\rho,\xi)| \exp\left(\int_0^\rho \Re[\tilde\psi(r,\xi)]\mathrm{d}r \right) |U_0(\xi)| \mathrm{d}\rho \mathrm{d}\xi \\
&\quad +\int_{B_R}\int_0^T |\tilde\psi(\rho,\xi)| \int_0^\rho \exp\left(\int_s^\rho\Re[\tilde\psi(r,\xi)]\mathrm{d}r \right)  1_{\opar 0,\tau \cbrk}(s) \left|F(s,\xi)\right| \mathrm{d}s \mathrm{d}\rho \mathrm{d}\xi   \\
&\leq  \sup_{\xi \in B_R}\left(\int_0^T |\tilde\psi(\rho,\xi)| \exp\left(\int_0^\rho \Re[\tilde\psi(r,\xi)]\mathrm{d}r \right)\mathrm{d}\rho \right)
\int_{B_R} |\cF[u_0](\xi)|  \mathrm{d}\xi \\
&\quad +\int_{B_R}\int_0^T  \int_s^T|\tilde\psi(\rho,\xi)| \exp\left(\int_s^\rho\Re[\tilde\psi(r,\xi)]\mathrm{d}r \right)   \mathrm{d}\rho ~1_{\opar 0,\tau \cbrk}(s)  \left|\cF[f(s,\cdot)](\xi)\right| \mathrm{d}s \mathrm{d}\xi   \\
&\leq  \sup_{\xi \in B_R}\left(\int_0^T |\tilde\psi(\rho,\xi)| \exp\left(\int_0^\rho \Re[\tilde\psi(r,\xi)]\mathrm{d}r \right)\mathrm{d}\rho \right)
\int_{B_R} |\cF[u_0](\xi)|  \mathrm{d}\xi \\
&\quad +\sup_{\xi \in \fR^d}\left(\int_0^T|\tilde\psi(\rho,\xi)| \sup_{0 \leq s \leq \rho} \left[\exp\left(\int_s^\rho\Re[\tilde\psi(r,\xi)]\mathrm{d}r \right) \right]   \mathrm{d}\rho \right)
\int_{B_R} \int_0^T   ~1_{\opar 0,\tau \cbrk}(s)  \left|\cF[f(s,\cdot)](\xi)\right| \mathrm{d}s \mathrm{d}\xi  < \infty 
\end{align*}
with probability one for all $R,T \in (0,\infty)$.

Next, we show that \eqref{20240729 10} and \eqref{main deter as 3} hold.
However, they are obviously valid as discussed at the beginning of the section.
Therefore by Lemmas \ref{deterministic solution part} and \ref{restriction lem}, 
$H_1$ is in the intersection of two spaces
$$
\bC L_{1,\ell oc}\left( \clbrk 0,\tau \cbrk \times \fR^d,\sigma\left( \rG \times \cB([0,\infty))  \cup \rH\cup \cP   \right) \times \cB(\fR^d)\right)
$$
and
\begin{align*}
\bL_{0,1,1,\ell oc}\left( \opar 0,\tau \cbrk \times \fR^d, 
\sigma\left( \rG \times \cB([0,\infty)) \cup \rH  \cup \cP \right) \times \cB(\fR^d)  
, |\tilde\psi(t,\xi)|\mathrm{d}t \mathrm{d}\xi\right).
\end{align*}
Similarly, by Corollary \ref{20240417 01} and Lemma  \ref{restriction lem},
$H_2$ is in the intersection of the spaces
$$
\bC L_{1,\ell oc}\left( \clbrk 0,\tau \cbrk \times \fR^d,\sigma\left( \rG \times \cB([0,\infty))  \cup \rH\cup \cP   \right) \times \cB(\fR^d)\right)
$$
and
\begin{align*}
\bL_{0,1,1,\ell oc}\left( \opar 0,\tau \cbrk \times \fR^d, 
\sigma\left( \rG \times \cB([0,\infty))  \cup \rH  \cup \cP \right) \times \cB(\fR^d)  
, |\tilde\psi(t,\xi)|\mathrm{d}t \mathrm{d}\xi\right).
\end{align*}
Therefore, the sum
\begin{align*}
H:=H_1+H_2 
\end{align*}
is in the intersection
$$
\bC L_{1,\ell oc}\left( \clbrk 0,\tau \cbrk \times \fR^d,\sigma\left( \rG \times \cB([0,\infty))  \cup \rH\cup \cP   \right) \times \cB(\fR^d)\right)
$$
and
\begin{align*}
\bL_{0,1,1,\ell oc}\left( \opar 0,\tau \cbrk \times \fR^d, 
\sigma\left( \rG \times \cB([0,\infty)) \times \cB(\fR^d) \cup \rH \cup \cP \right) \times \cB(\fR^d)  
, |\tilde\psi(t,\xi)|\mathrm{d}t \mathrm{d}\xi\right).
\end{align*}
Here we identify $H$ with a $\cD'(\fR^d)$-valued stochastic process (defined on  $\clbrk 0, \tau \cbrk$) on the basis of an $L_2(\fR^d)$-inner product as follows: for each $(\omega,t) \in \clbrk 0, \tau \cbrk$,
\begin{align*}
\varphi \in \cD(\fR^d) \mapsto \langle H(t,\cdot), \varphi \rangle
&= \int_{\fR^d}\left( \exp\left(\int_0^t\tilde\psi(r,\xi)\mathrm{d}r \right) \cF[u_0](\xi) \right) \overline{\varphi(\xi)} \mathrm{d}\xi \\
&\quad +\int_{\fR^d} \left(\int_0^t  \exp\left(\int_s^t\tilde\psi(r,\xi)\mathrm{d}r \right) 1_{\opar 0, \tau \cbrk}(s)\cF[f(s,\cdot)](\xi)\mathrm{d}s \right) \overline{\varphi(\xi)} \mathrm{d}\xi  \\
&\quad + \int_{\fR^d} \left(\int_0^t  \exp\left(\int_s^t\tilde\psi(r,\cdot)\mathrm{d}r \right) 1_{\opar 0, \tau \cbrk}(s)\cF[g^k(s,\cdot) ]\mathrm{d}B^k_s\right) \overline{\varphi(\xi)} \mathrm{d}\xi.
\end{align*}
Thus by the definitions of $u$ in \eqref{20240314 10} and the Fourier transform on $\cF^{-1}\cD'(\fR^d)$, for any $(\omega,t) \in \clbrk 0,\tau \cbrk$, we have
\begin{align*}
\langle \cF[u(t,\cdot)] , \varphi \rangle = \langle H(t,\cdot), \varphi \rangle, \quad \forall \varphi \in \cD(\fR^d).
\end{align*}
In particular, for any $(\omega,t) \in \clbrk 0,\tau \cbrk$, 
$\xi \mapsto \cF[u(t,\cdot)](\xi)$ is a locally integrable function, which is given by
\begin{align*}
\cF[u(t,\cdot)](\xi)
=H(t,\xi)
&= \exp\left(\int_0^t\tilde\psi(r,\xi)\mathrm{d}r \right) \cF[u_0](\xi) 
+ \int_0^t  \exp\left(\int_s^t\tilde\psi(r,\xi)\mathrm{d}r \right) \cF[f(s,\cdot)](\xi)\mathrm{d}s \\
&\quad +\int_0^t  \exp\left(\int_s^t\tilde\psi(r,\xi)\mathrm{d}r \right) \cF[g^k(s,\cdot)](\xi)\mathrm{d}B^k_s.
\end{align*}
Additionally, since 
\begin{align*}
H \in \bL_{0,1,1,\ell oc}\left( \opar 0,\tau \cbrk \times \fR^d, \rF \times \cB([0,\infty)), |\tilde\psi(t,\xi)|\mathrm{d}t \mathrm{d}\xi\right),
\end{align*}
we have
\begin{align}
								\notag
&\int_0^t \tilde\psi(\rho,\xi) \cF[u(\rho,\cdot)](\xi) \mathrm{d}\rho \\
								\notag
&=\int_0^t \tilde\psi(\rho,\xi) H(\rho,\xi) \mathrm{d}\rho \\
								\notag
&=\int_0^t \tilde\psi(\rho,\xi) \exp\left(\int_0^\rho\tilde\psi(r,\xi)\mathrm{d}r \right) d\rho \cF[u_0](\xi)
+\int_0^t  \int_0^\rho \left[\tilde\psi(\rho,\xi)\exp\left(\int_s^\rho\tilde\psi(r,\xi)\mathrm{d}r \right)\right]   \cF[f(s,\cdot)](\xi)\mathrm{d}s \mathrm{d}\rho \\
								\label{20240315 20}
&\quad +\int_0^t  \int_0^\rho \left[\tilde\psi(\rho,\xi)\exp\left(\int_s^\rho\tilde\psi(r,\xi)\mathrm{d}r \right)\right]  \cF[g^k(s,\cdot)](\xi)\mathrm{d}B^k_s   \mathrm{d}\rho  \quad (a.e.)~(\omega,t,\xi) \in \clbrk 0,\tau \cbrk \times \fR^d.
\end{align}
Then by the fundamental theorem of calculus and the Fubini theorem,
\begin{align}
								\notag
&\cF[u(t,\cdot)](\xi) \\
								\notag
&= \exp\left(\int_0^t\tilde\psi(r,\xi)\mathrm{d}r \right) \cF[u_0](\xi) 
+ \int_0^t  \exp\left(\int_s^t\tilde\psi(r,\xi)\mathrm{d}r \right) \cF[f(s,\cdot)](\xi)\mathrm{d}s \\
								\notag
&\quad +\int_0^t  \exp\left(\int_s^t\tilde\psi(r,\xi)\mathrm{d}r \right) \cF[g^k(s,\cdot)](\xi)\mathrm{d}B^k_s \\
								\notag
&= \exp\left(\int_0^t\tilde\psi(r,\xi)\mathrm{d}r \right) \cF[u_0](\xi) + \int_0^t  \int_s^t \frac{\mathrm{d}}{\mathrm{d}\rho}\left[\exp\left(\int_s^\rho\tilde\psi(r,\xi)\mathrm{d}r \right)\right]  \mathrm{d}\rho \cF[f(s,\cdot)](\xi)\mathrm{d}s 
+\int_0^t  \cF[f(s,\cdot)](\xi)\mathrm{d}s \\
								\label{20240315 10}
& \quad + \int_0^t  \int_s^t \frac{\mathrm{d}}{\mathrm{d}\rho}\left[\exp\left(\int_s^\rho\tilde\psi(r,\xi)\mathrm{d}r \right)\right]  \mathrm{d}\rho \cF[g^k(s,\cdot)](\xi)\mathrm{d}B^k_s 
+\int_0^t  \cF[g^k(s,\cdot)](\xi)\mathrm{d}B_s^k 
\end{align}
for almost every $(\omega,t,\xi) \in \clbrk 0,\tau \cbrk \times \fR^d$, where two terms 
$\int_0^t  \cF[f(s,\cdot)](\xi)\mathrm{d}s$ and $\int_0^t  \cF[g^k(s,\cdot)](\xi)\mathrm{d}B_s^k$ are legitimate due to the assumptions
$$
f \in \cF^{-1}\bL_{0,1,1,\ell oc}\left( \opar 0, \tau \cbrk \times \fR^d, \rH \times \cB(\fR^d) \right)
$$
and 
$$
g \in \cF^{-1}\bL^{\omega,\xi,t}_{0,1,2,\ell oc}\left( \opar 0,  \tau \cbrk \times \fR^d, \cP \times \cB(\fR^d) ; l_2\right).
$$
By calculating the derivative of the compositions in \eqref{20240315 10}, we have
\begin{align}
									\notag
&\cF[u(t,\cdot)](\xi) \\
									\notag
&=\exp\left(\int_0^t\tilde\psi(r,\xi)\mathrm{d}r \right) \cF[u_0](\xi) 
+\int_0^t  \int_s^t \left[\tilde\psi(\rho,\xi)\exp\left(\int_s^\rho\tilde\psi(r,\xi)\mathrm{d}r \right)\right]  \mathrm{d}\rho \cF[f(s,\cdot)](\xi)\mathrm{d}s  \\
									\notag
&\quad +\int_0^t  \int_s^t \left[\tilde\psi(\rho,\xi)\exp\left(\int_s^\rho\tilde\psi(r,\xi)\mathrm{d}r \right)\right]  \mathrm{d}\rho \cF[g^k(s,\cdot)](\xi)\mathrm{d}B^k_s  \\
									\label{20240222 01}
&\quad +\int_0^t \cF[f(s,\cdot)](\xi)\mathrm{d}s +\int_0^t \cF[g^k(s,\cdot)](\xi)\mathrm{d}B^k_s
\end{align}
for almost every $(\omega,t,\xi) \in \clbrk 0,\tau \cbrk \times \fR^d$.
Additionally, applying the stochastic Fubini theorem and \eqref{20240315 20}, one can easily check that the sum of terms in \eqref{20240222 01} is equal to
\begin{align}
										\notag
& \exp\left(\int_0^t\tilde\psi(r,\xi)\mathrm{d}r \right) \cF[u_0](\xi) 
+\int_0^t  \int_0^\rho \left[\tilde\psi(\rho,\xi)\exp\left(\int_s^\rho\tilde\psi(r,\xi)\mathrm{d}r \right)\right]   \cF[f(s,\cdot)](\xi)\mathrm{d}s \mathrm{d}\rho \\
										\notag
&\quad +\int_0^t  \int_0^\rho \left[\tilde\psi(\rho,\xi)\exp\left(\int_s^\rho\tilde\psi(r,\xi)\mathrm{d}r \right)\right]  \cF[g^k(s,\cdot)](\xi)\mathrm{d}B^k_s   \mathrm{d}\rho\\
										\notag
&\quad +\int_0^t \cF[f(s,\cdot)](\xi)\mathrm{d}s +\int_0^t \cF[g^k(s,\cdot)](\xi)\mathrm{d}B^k_s \\
										\notag
&=\exp\left(\int_0^t\tilde\psi(r,\xi)\mathrm{d}r \right) \cF[u_0](\xi) 
-\int_0^t \tilde\psi(\rho,\xi) \exp\left(\int_0^\rho\tilde\psi(r,\xi)\mathrm{d}r \right) d\rho \cF[u_0](\xi) \\
										\notag
&\quad + \int_0^t \tilde\psi(\rho,\xi) \cF[u(\rho,\cdot)](\xi) \mathrm{d}\rho 
+\int_0^t \cF[f(s,\cdot)](\xi)\mathrm{d}s 
+\int_0^t \cF[g^k(s,\cdot)](\xi)\mathrm{d}B^k_s \\
									\label{20230213 20}
&= \cF[u_0](\xi)
+\int_0^t \tilde\psi(\rho,\xi) \cF[u(\rho,\cdot)](\xi) \mathrm{d}\rho 
+\int_0^t \cF[f(s,\cdot)](\xi)\mathrm{d}s
+\int_0^t \cF[g^k(s,\cdot)](\xi)\mathrm{d}B^k_s 
\end{align}
for almost every $(\omega,t,\xi) \in \clbrk 0,\tau \cbrk \times \fR^d$.
Finally, for any $\varphi \in \cF^{-1}\cD(\fR^d)$, taking the integration with $\overline{\cF[\varphi]}$ in \eqref{20230213 20} and applying the definitions of actions on the elements of the class $\cF^{-1}\cD'(\fR^d)$ and the operator $\tilde \psi (s,-\mathrm{i} \nabla)$ (Definition \ref{defn psi operator}) with the (stochastic) Fubini theorem,   we have
\begin{align*}
\left\langle u(t,\cdot),\varphi \right\rangle 
&= \langle u_0, \varphi \rangle +  \int_0^t \left\langle \tilde\psi(s,-\mathrm{i}\nabla)u(s,\cdot) , \varphi \right\rangle \mathrm{d}s 
+ \int_0^t \left\langle f(s,\cdot),\varphi\right\rangle \mathrm{d}s \\
&\quad + \int_0^t \left\langle g^k(s,\cdot),\varphi\right\rangle \mathrm{d}B^k_s 
\quad (a.e.)~(\omega,t) \in \clbrk 0,\tau \cbrk.
\end{align*}
Therefore $u$ is a Fourier-space weak solution to \eqref{time eqn}
by Definition \ref{space weak solution 2}.
The theorem is proved.
\end{proof}

\begin{rem}
									\label{g zero rem}
If $g=0$ in Theorem \ref{deter exist weak sol}, then the condition \eqref{main deter as} can be substituted with the less stringent condition \eqref{main deter as 3}. This is because the stronger assumption \eqref{main deter as} is solely employed to establish the joint measurability of $H_2$ in the proof of Theorem \ref{deter exist weak sol}.
Additionally, for this case $g=0$, the function $u$ without the stochastic integral term becomes a solution to \eqref{time eqn} even though the symbol is random.
\end{rem}

Next, we show a stability result of the solution $u$ obtained in the previous theorem.
Roughly saying, the spatial Fourier transform of the solution $u$ is stable in terms of that of data. 
The estimates for this stability result may depend on the constants
$C^{\mathrm{e}|\int\Re[\tilde \psi]|}_{R,T}$ and $C^{|\tilde \psi|}_{R,T}$
in \eqref{main deter as} and \eqref{main deter as 2}.
In the following estimates, the smaller constant $C^{\mathrm{e}\int\Re[\tilde \psi]}_{R,T}$ sometimes appears in stead of the constant $C^{\mathrm{e}|\int\Re[\tilde \psi]|}_{R,T}$ to optimize constants.
Additionally, the constants satisfy the inequalities
\begin{align*}
C^{\mathrm{e}\int\Re[\tilde \psi]}_{R,T} 
\leq C^{\mathrm{e}|\int\Re[\tilde \psi]|}_{R,T} 
\leq C^{\mathrm{e}\int \sup|\Re[\tilde \psi]|}_{R,T} 
\leq 1+C^{|\tilde \psi|}_{R,T}
\end{align*}
as explained in Remark \ref{simple constant}.
These inequalities help to simplify constants in the following Corollary.

Moreover, note that Theorem \ref{thm a priori} does not contribute to optimizing constants in estimates, as kernel-based estimates are superior for the deterministic symbols.
However, if the symbol is random, then a representation by the kernel is not to be expected.
In this case, Theorem \ref{thm a priori} can be used to optimize constants, as demonstrated in the proof of Corollary \ref{exist cor 2}. Nevertheless, Theorem \ref{thm deter a priori} does not aid in improving estimates for the deterministic terms, even if the symbol is random, which will be revisited in Remark \ref{deter est rem}.

\begin{corollary}
									\label{deter exist cor 2}
Let
\begin{align*}
u_0 \in \cF^{-1}\bL_{1,1,\ell oc}\left( \Omega \times \fR^d , \rG \times \cB(\fR^d)\right),  
\quad f \in \cF^{-1}\bL_{1,1,1,loc,\ell oc}\left( \opar 0, \tau \cbrk \times \fR^d, \rH \times \cB(\fR^d) \right),
\end{align*}
and 
$$
g \in \cF^{-1}\bL^{\omega,\xi,t}_{1,1,2,loc,\ell oc}\left( \opar 0,  \tau \cbrk \times \fR^d, \cP \times \cB(\fR^d) ; l_2\right).
$$
Suppose that $\tilde\psi$ satisfies  \eqref{main deter as} and \eqref{main deter as 2}.
Then the solution $u$ defined as \eqref{20240314 10} is in the intersection of the classes
$$
\cF^{-1}\bL_{1,loc}\bC L_{1,\ell oc}\left( \clbrk 0,\tau \cbrk \times \fR^d, 
\sigma\left( \rG \times \cB([0,\infty)) \cup \rH  \cup \cP  \right)\times \cB(\fR^d)\right)
$$
and
$$
\cF^{-1}\bL_{1,1,1,loc,\ell oc}\left( \opar 0,\tau \cbrk \times \fR^d, 
\sigma\left( \rG \times \cB([0,\infty))  \cup \rH \cup \cP   \right)\times \cB(\fR^d), |\tilde\psi(t,\xi)|\mathrm{d}t \mathrm{d}\xi\right).
$$
In particular, 
$$
u \in \cF^{-1}\bL_{1,loc}\bC L_{1,\ell oc}\left( \clbrk 0,\tau \cbrk \times \fR^d, 
\cP\times \cB(\fR^d)\right)
$$and
$$
u \in \cF^{-1}\bL_{1,1,1,loc,\ell oc}\left( \opar 0,\tau \cbrk \times \fR^d, \cP \times \cB(\fR^d), |\tilde\psi(t,\xi)|\mathrm{d}t \mathrm{d}\xi\right)
$$
if $\cF[u_0]$ is $\rF_0 \times \cB(\fR^d)$-measurable  and $\cF[f]$ is  $\cP \times \cB(\fR^d)$-measurable.
Moreover, $u$ satisfies
\begin{align}
										\notag
\bE\left[ \int_0^{\tau \wedge T} \int_{B_R}|\tilde\psi(t,\xi)| |\cF[u(t,\cdot)](\xi)| \mathrm{d}\xi  \mathrm{d}t \right]  
&\leq C_{R,T}^{|\tilde\psi|} \bE \left[\int_{B_R} \left|\cF[u_0](\xi)\right|    \mathrm{d}\xi 
+\int_{B_R} \int_0^{ \tau \wedge T}\left|\cF[f(s,\cdot)](\xi)\right|  \mathrm{d}s  \mathrm{d}\xi \right]\\
										\label{20240512 01}
& \quad +C_{BDG} \cdot C_{R,T}^{|\tilde\psi|} \bE\left[\int_{B_R} \left(\int_0^{\tau \wedge T}\left|\cF[g(s,\cdot)](\xi)\right|^2_{l_2}  \mathrm{d}s \right)^{1/2} \mathrm{d}\xi\right],
\end{align}
\begin{align}
 \bE\left[ \int_0^{\tau \wedge T} \int_{B_R} |\cF[u(t,\cdot)](\xi)|\mathrm{d}\xi\right] \mathrm{d}t
										\notag
&\leq   T \cdot C_{R,T}^{\mathrm{e}\int\Re[\tilde\psi]} \cdot \bE \left[\int_{B_R} |\cF[u_0](\xi)|\mathrm{d}\xi 
+ \int_{B_R} \int_0^{\tau \wedge T}\left|\cF[f(s,\cdot)](\xi)\right|  \mathrm{d}s  \mathrm{d}\xi \right]
  \\
									\label{2024052110-2}
& \quad+ T \cdot C_{BDG} \cdot C_{R,T}^{\mathrm{e}\int\Re[\tilde\psi]} \bE\left[\int_{B_R} \left(\int_0^{\tau \wedge T}\left|\cF[g(s,\cdot)](\xi)\right|^2_{l_2}  \mathrm{d}s \right)^{1/2} \mathrm{d}\xi \right],
\end{align}

and
\begin{align}
 \bE\left[ \int_{B_R} \sup_{t \in [0,\tau \wedge T]} |\cF[u(t,\cdot)](\xi)|\mathrm{d}\xi\right]  
										\notag
&\leq   C_{R,T}^{\mathrm{e}\int\Re[\tilde\psi]} \cdot \bE \left[\int_{B_R} |\cF[u_0](\xi)|\mathrm{d}\xi 
+ \int_{B_R} \int_0^{\tau \wedge T}\left|\cF[f(s,\cdot)](\xi)\right|  \mathrm{d}s  \mathrm{d}\xi \right]
  \\
									\label{2024052110}
& \quad+ C_{BDG} \cdot C_{R,T}^{\mathrm{e}|\int\Re[\tilde\psi]|} \bE\left[\int_{B_R} \left(\int_0^{\tau \wedge T}\left|\cF[g(s,\cdot)](\xi)\right|^2_{l_2}  \mathrm{d}s \right)^{1/2} \mathrm{d}\xi \right],
\end{align}
where $C_{R,T}^{|\tilde\psi|}$ and $C_{R,T}^{\mathrm{e}\int\Re[\tilde \psi]}$ are given in \eqref{main deter as 2} and \eqref{main deter as 3}.
\end{corollary}
\begin{proof}
Due to Theorem \ref{deter exist weak sol}, it suffices to show \eqref{20240512 01}, \eqref{2024052110-2}, and \eqref{2024052110} since all the right hand sides of \eqref{20240512 01}, \eqref{2024052110-2}, and \eqref{2024052110} are finite.
We split $u$ into two parts $u_1$ and $u_2$.
Put
\begin{align*}
\cF[u_1(t,\cdot)](\xi)
= \exp\left(\int_0^t\tilde\psi(r,\xi)\mathrm{d}r \right) \cF[u_0](\xi) 
+ \int_0^t  \exp\left(\int_s^t\tilde\psi(r,\xi)\mathrm{d}r \right) \cF[f(s,\cdot)](\xi)\mathrm{d}s
\end{align*}
and
\begin{align*}
\cF[u_2(t,\cdot)](\xi)
=\int_0^t  \exp\left(\int_s^t\tilde\psi(r,\xi)\mathrm{d}r \right) \cF[g^k(s,\cdot)](\xi)\mathrm{d}B^k_s.
\end{align*}
Then it is obvious that
\begin{align*}
\cF[u(t,\cdot)](\xi) = \cF[u_1(t,\cdot)](\xi) + \cF[u_2(t,\cdot)](\xi) \quad \forall (\omega,t,\xi) \in \opar 0, \tau \cbrk \times \fR^d
\end{align*}
and both $\cF[u_1(t,\cdot)]$, and $\cF[u_2(t,\cdot)]$ can be defined on $\Omega \times [0,\infty) \times \fR^d$.
Firstly, $\cF[u_1]$ can be easily estimated (for the details, see the proof of Corollary \ref{exist cor 1} in the next section).
Indeed, by Lemma \ref{deterministic solution part},
\begin{align*}
&\bE\left[\int_{B_R} \sup_{t \in [0,\tau \wedge T]}  |\cF[u_1(t,\cdot)](\xi)| \mathrm{d}\xi \right]  \\
&\leq \bE\left[\int_{B_R} \sup_{t \in [0,T]}  |\cF[u_1(t,\cdot)](\xi)| \mathrm{d}\xi \right]  \\
&\leq \left(\sup_{0 \leq s \leq t \leq T,\xi \in B_R}\left| \exp \left( \int_s^t \Re[\tilde\psi(r,\xi)] \mathrm{d}r \right)  \right| \right)
\bE\left[\int_{B_R} |\cF[u_0](\xi)|\mathrm{d}\xi + \int_{B_R} \int_0^{\tau \wedge T}\left|\cF[f(s,\cdot)](\xi)\right|  \mathrm{d}s  \mathrm{d}\xi \right]
\end{align*}
and
\begin{align*}
\bE\left[ \int_0^{\tau \wedge T} \int_{B_R}|\tilde\psi(t,\xi)| |\cF[u_1(t,\cdot)](\xi)| \mathrm{d}\xi  \mathrm{d}t \right]  
&\leq \bE\left[ \int_0^T \int_{B_R}|\tilde\psi(t,\xi)| |\cF[u_1(t,\cdot)](\xi)| \mathrm{d}\xi  \mathrm{d}t \right]  \\
&\leq
\sup_{\xi \in B_R} \left(\int_0^T |\tilde\psi(t,\xi)| \sup_{0 \leq s \leq t}\left|\exp\left(\int_s^t\Re[\tilde\psi(r,\xi)]\mathrm{d}r \right) \right|  \mathrm{d}t \right) \\
&\quad \times \bE\left[\int_{B_R} |\cF[u_0](\xi)|\mathrm{d}\xi + \int_{B_R} \int_0^{\tau \wedge T}\left|\cF[f(s,\cdot)](\xi)\right|  \mathrm{d}s  \mathrm{d}\xi \right].
\end{align*}
Next, we focus on estimating $\cF[u_2]$.
It is also easily obtained from  Corollary \ref{stochastic solution part} with the Fubini theorem.
Indeed, 
\begin{align*}
&\bE\left[ \int_0^{\tau \wedge T} \int_{B_R}|\tilde\psi(t,\xi)| |\cF[u_2(t,\cdot)](\xi)| \mathrm{d}\xi  \mathrm{d}t \right] \\
&\leq \bE\left[ \int_0^T \int_{B_R}|\tilde\psi(t,\xi)| |\cF[u_2(t,\cdot)](\xi)| \mathrm{d}\xi  \mathrm{d}t \right] \\
&\leq C_{BDG} \sup_{\xi \in B_R} \left(\int_0^T |\tilde\psi(t,\xi)| \sup_{0 \leq s \leq t}\left|\exp\left(\int_s^t\Re[\tilde\psi(r,\xi)]\mathrm{d}r \right) \right|  \mathrm{d}t \right)
\left[ \int_{B_R} \int_0^{\tau \wedge T}\left|\cF[g(s,\cdot)](\xi)\right|  \mathrm{d}s  \mathrm{d}\xi \right],
\end{align*}
\begin{align*}
&\bE\left[ \int_0^{\tau \wedge T} \int_{B_R}|\cF[u_2(t,\cdot)](\xi)|\mathrm{d}\xi \mathrm{d}t \right] \\
&\leq \bE\left[ \int_0^T \int_{B_R}|\cF[u_2(t,\cdot)](\xi)|\mathrm{d}\xi \mathrm{d}t \right] \\
&\leq T \cdot C_{BDG}\esssup_{0\leq s\leq t \leq T, \xi \in B_R}\left[  \exp\left(\int_s^t \Re[\tilde\psi(r,\xi)]\mathrm{d}r \right) \right]
\bE\left[  \int_{B_R} \left( \int_0^{\tau \wedge T}\left|\cF[g(s,\cdot)](\xi)\right|  \mathrm{d}s  \right)^{1/2} \mathrm{d}\xi  \right],
\end{align*}
and
\begin{align*}
&\bE\left[\int_{B_R}  \sup_{t \in [0,\tau \wedge T]}  \left|\cF[u_2(t,\cdot)](\xi) \right|\mathrm{d}\xi \right]  \\
&\leq C_{BDG} \left(\sup_{0 \leq s \leq t \leq T,\xi \in B_R}\exp \left( \left| \int_s^t \Re[\tilde\psi(r,\xi)] \mathrm{d}r \right|\right)   \right) \bE\left[\int_{B_R}  \left(\int_0^{\tau \wedge T} \left|\cF[g(s,\cdot)](\xi)\right|^2_{l_2} \mathrm{d}s \right)^{1/2} \mathrm{d}\xi   \right].
\end{align*}
Combining all estimates above and recalling the definitions of the notations $C_{R,T}^{\mathrm{e}\int\Re[\tilde \psi]}$, $C_{R,T}^{\mathrm{e}|\int \Re[\tilde \psi|]}$ and $C_{R,T}^{|\tilde\psi|}$, we have
\begin{align*}
&\bE\left[ \int_0^{\tau \wedge T} \int_{B_R}|\tilde\psi(t,\xi)| |\cF[u(t,\cdot)](\xi)| \mathrm{d}\xi  \mathrm{d}t \right] \\
&\leq\bE\left[ \int_0^{\tau \wedge T} \int_{B_R}|\tilde\psi(t,\xi)| |\cF[u_1(t,\cdot)](\xi)| \mathrm{d}\xi  \mathrm{d}t \right] 
+\bE\left[ \int_0^{\tau \wedge T} \int_{B_R}|\tilde\psi(t,\xi)| |\cF[u_2(t,\cdot)](\xi)| \mathrm{d}\xi  \mathrm{d}t \right] \\
&\leq C_{R,T}^{|\tilde\psi|} \bE\left[\int_{B_R} |\cF[u_0](\xi)|\mathrm{d}\xi + \int_{B_R} \int_0^{\tau \wedge T}\left|\cF[f(s,\cdot)](\xi)\right|  \mathrm{d}s  \mathrm{d}\xi +C_{BDG}\int_{B_R} \int_0^{\tau \wedge T}\left|\cF[g(s,\cdot)](\xi)\right|  \mathrm{d}s  \mathrm{d}\xi \right]
\end{align*}
and
\begin{align*}
\bE\left[\int_{B_R} \sup_{t \in [0,\tau \wedge T]}  \left|\cF[u(t,\cdot)](\xi) \right|\mathrm{d}\xi \right]  
&\leq \bE\left[\int_{B_R} \sup_{t \in [0,\tau \wedge T]}  \left|\cF[u_1(t,\cdot)](\xi) \right|\mathrm{d}\xi \right]  
+\bE\left[\int_{B_R} \sup_{t \in [0,\tau \wedge T]}  \left|\cF[u_2(t,\cdot)](\xi) \right|\mathrm{d}\xi \right]   \\
&\leq C_{R,T}^{\mathrm{e}\int\Re[\tilde\psi]}  \bE\left[\int_{B_R} |\cF[u_0](\xi)|\mathrm{d}\xi + \int_{B_R} \int_0^{\tau \wedge T}\left|\cF[f(s,\cdot)](\xi)\right|  \mathrm{d}s  \mathrm{d}\xi \right] \\
&\quad + C_{BDG}  \cdot C_{R,T}^{\mathrm{e}\int\Re[\tilde\psi]}
   \bE\left[\int_{B_R} \left(\int_0^{\tau \wedge T}\left|\cF[g(s,\cdot)](\xi)\right|^2_{l_2} \mathrm{d}s \right)^{1/2} \mathrm{d}\xi \right]. 
\end{align*}
The corollary is proved. 
\end{proof}

\begin{rem}
It might seem odd that  the constant $C^{\mathrm{e}|\int\Re[\tilde \psi]|}_{R,T}$ does not appear in \eqref{2024052110-2} since the finiteness of $C^{\mathrm{e}|\int\Re[\tilde \psi]|}_{R,T}$ is a key assumption in the corollary.
However, this condition was utilized to show the joint measurability of $\cF[u(t,\cdot)](\xi)$ in order to apply the Fubini theorem as emphasized in Remarks \ref{measure issue rem} and  \ref{measurability emph}.
This difficulty of the joint measurability arose due to the stochastic integral terms.
Therefore for the case that $g=0$, it may be possible that the symbol is random and Assumption \ref{main as} is replaced with Assumption \ref{weaker as},
which leads to the proof of Theorem \ref{time deter thm}.
We postpone proving this special case until the next section where random symbols are addressed.
\end{rem}

\mysection{Existence of a solution with a random symbol}
										\label{existence section}

In this section, we show that there exists a Fourier-space weak solution $u$ to \eqref{time eqn} even though the symbol $\psi$ 
is random. 
The function  $u$ in Theorem \ref{deter exist weak sol} which was a solution to \eqref{time eqn} with a deterministic symbol cannot be defined in a legitimate way anymore.
This is because the stochastic integral term
\begin{align*}
\int_0^t  \exp\left(\int_s^t\psi(r,\cdot)\mathrm{d}r \right) 1_{\opar 0, \tau \cbrk}(s)\cF[g^k(s,\cdot) ]\mathrm{d}B^k_s
\end{align*}
is not defined generally since the term $\int_s^t\psi(r,\cdot)\mathrm{d}r$ related to the symbol is not $\rF_s$-adapted and thus is not predictable as emphasized in Section \ref{sto fubini deter}. 
However, the other terms in \eqref{20240314 10} are still valid 
and fortunately, the function $u$ still becomes a solution to \eqref{time eqn} in spite of the randomness on the symbol $\psi$ if $g=0$. Therefore,  we start showing the existence of our solution with this simple case.

\begin{corollary}
									\label{exist cor 1}
Suppose that all the assumptions in Theorem \ref{time deter thm} hold and
define $u$ as  a $\cF^{-1}\cD'(\fR^d)$-valued stochastic process so that for each $(\omega,t) \in \clbrk 0, \tau \cbrk$ and  $\varphi \in \cF^{-1} \cD(\fR^d)$,
\begin{align}
									\notag
\langle u(t,\cdot), \varphi \rangle
&:= \int_{\fR^d}\left( \exp\left(\int_0^t\psi(r,\xi)\mathrm{d}r \right) \cF[u_0](\xi) \right) \overline{\cF[\varphi](\xi)} \mathrm{d}\xi \\
									\label{deter sol rep}
&\quad +\int_{\fR^d} \left(\int_0^t  \exp\left(\int_s^t\psi(r,\xi)\mathrm{d}r \right) 1_{\opar 0, \tau \cbrk}(s)\cF[f(s,\cdot)](\xi)\mathrm{d}s \right) \overline{\cF[\varphi](\xi)} \mathrm{d}\xi.
\end{align}
Then the function $u$ becomes a Fourier-space weak solution to \eqref{deter eqn}, which is in the intersection of the classes
$$
\cF^{-1}\bC L_{1,\ell oc}\left( \clbrk 0,\tau \cbrk \times \fR^d,\sigma\left( \rG \times \cB([0,\infty))  \cup \rH\cup \cP   \right)\times \cB(\fR^d)\right)
$$
and
$$
 \cF^{-1}\bL_{0,1,1, \ell oc}\left( \opar 0,\tau \cbrk \times \fR^d, \sigma\left( \rG \times \cB([0,\infty))  \cup \rH  \cup \cP \right) \times \cB(\fR^d)  , |\psi(t,\xi)|\mathrm{d}t \mathrm{d}\xi\right).
$$
Moreover, for all  $R,T \in (0,\infty)$, $u$ satisfies
\begin{align}
\int_{B_R} \sup_{t \in [0,\tau \wedge T]}  \left(|\cF[u(t,\cdot)](\xi)|\right) \mathrm{d}\xi   
									\label{20240523 10}
&\leq    C_{R,T}^{\mathrm{e}\int\Re[\psi]} \left(\int_{B_R} |\cF[u_0](\xi)|\mathrm{d}\xi + \int_{B_R} \int_0^{\tau \wedge T}\left|\cF[f(s,\cdot)](\xi)\right|  \mathrm{d}s  \mathrm{d}\xi \right)
\end{align}
and
\begin{align}
									\label{20240523 11}
 \int_0^{\tau \wedge T} \int_{B_R}|\psi(t,\xi)||\cF[u(t,\cdot)](\xi)| \mathrm{d}\xi  \mathrm{d}t 
\leq
C_{R,T}^{|\psi|}
\left(\int_{B_R} |\cF[u_0](\xi)|\mathrm{d}\xi + \int_{B_R} \int_0^{\tau \wedge T}\left|\cF[f(s,\cdot)](\xi)\right|  \mathrm{d}s  \mathrm{d}\xi \right)
\end{align}
with probability one.
\end{corollary}
\begin{proof}
Note that \eqref{deter sol rep} is well-defined even though the symbol $\psi$ is random since there are no stochastic integral terms.
Moreover, it is easy to show that the function $u$ given by \eqref{deter sol rep} is a Fourier space weak solution to \eqref{deter eqn} according to Theorem \ref{deter exist weak sol} (\textit{cf}. Remark \ref{g zero rem}).
Therefore, it suffices to show \eqref{20240523 10} and \eqref{20240523 11}.

First, due to the definition of $u$, it is obvious that
\begin{align*}
\cF[u(t,\cdot)](\xi)
&= \exp\left(\int_0^t\psi(r,\xi)\mathrm{d}r \right) \cF[u_0](\xi) 
+ \int_0^t  \exp\left(\int_s^t\psi(r,\xi)\mathrm{d}r \right) \cF[f(s,\cdot)](\xi)\mathrm{d}s \\
&= \exp\left(\int_0^t\psi(r,\xi)\mathrm{d}r \right) \cF[u_0](\xi) 
+ \int_0^t  1_{\opar 0, \tau \cbrk}(s)\exp\left(\int_s^t\psi(r,\xi)\mathrm{d}r \right) \cF[f(s,\cdot)](\xi)\mathrm{d}s \\
\end{align*}
for all  $(\omega,t,\xi) \in \opar 0, \tau \cbrk \times \fR^d$.
Observe that the domains of all functions above can be naturally extended to $\Omega \times [0,\infty) \times \fR^d$.
Put
\begin{align*}
H(t,\xi)=\cF[u(t,\cdot)](\xi), \quad U_0(\xi)= \cF[u_0](\xi), \quad \text{and} \quad F(t,\xi)=\cF[f(t,\cdot)](\xi).
\end{align*}
Then, applying Lemma \ref{deterministic solution part} for each $\omega \in \Omega$, we have
\begin{align}
										\notag
&\int_{B_R} \left(\sup_{t \in [0,\tau \wedge T]}|\cF[u(t,\cdot)](\xi)| \right) \mathrm{d}\xi  \\
										\notag
&\leq \int_{B_R} \left( \sup_{t \in [0,T]}|H(t,\xi)| \right) \mathrm{d}\xi  \\
										\notag
&\leq \left(\sup_{0 \leq s \leq t \leq T,\xi \in B_R}\left| \exp \left( \int_s^t \Re[\psi(r,\xi)] \mathrm{d}r \right)  \right| \right)
\left[\int_{B_R} |U_0(\xi)|\mathrm{d}\xi + \int_{B_R} \int_0^{\tau \wedge T}\left|F(s,\xi)\right|  \mathrm{d}s  \mathrm{d}\xi \right]\\
										\notag
&= C_{R,T}^{\mathrm{e}\int\Re[\psi]}
\left[\int_{B_R} |\cF[u_0](\xi)|\mathrm{d}\xi + \int_{B_R} \int_0^{\tau \wedge T}\left|\cF[f(s,\cdot)](\xi)\right|  \mathrm{d}s  \mathrm{d}\xi \right]
\end{align}
and
\begin{align}
								\notag
& \int_0^{\tau \wedge T} \int_{B_R}|\psi(t,\xi)| |\cF[u(t,\cdot)](\xi)| \mathrm{d}\xi  \mathrm{d}t \\
								\notag
&\leq \int_0^T \int_{B_R}|\psi(t,\xi)| |H(t,\xi)| \mathrm{d}\xi  \mathrm{d}t \\
								\notag
&\leq
\sup_{\xi \in B_R} \left(\int_0^T |\psi(t,\xi)| \sup_{0 \leq s \leq t}\left|\exp\left(\int_s^t\Re[\psi(r,\xi)]\mathrm{d}r \right) \right|  \mathrm{d}t \right)
\left[\int_{B_R} |U_0(\xi)|\mathrm{d}\xi + \int_{B_R} \int_0^{\tau \wedge T}\left|F(s,\xi)\right|  \mathrm{d}s  \mathrm{d}\xi \right] \\
								\label{20240523 20}
&= C_{R,T}^{|\psi|}
\left[\int_{B_R} |\cF[u_0](\xi)|\mathrm{d}\xi + \int_{B_R} \int_0^{\tau \wedge T}\left|\cF[f(s,\cdot)](\xi)\right|  \mathrm{d}s  \mathrm{d}\xi \right].
\end{align}
Therefore,  \eqref{20240523 10} and \eqref{20240523 11} are obtained.
The corollary is proved.
\end{proof}
\begin{rem}
										\label{deter est rem}
Applying Theorem \ref{thm deter a priori} with \eqref{20240523 11}, we have
\begin{align}
										\label{20240523 31}
\int_{B_R} \left(\sup_{t \in [0,\tau \wedge T]}  |\cF[u(t,\cdot)](\xi)| \right) \mathrm{d}\xi  
\leq \left(1+C_{R,T}^{|\psi|}  \right)
\left[\int_{B_R} |\cF[u_0](\xi)|\mathrm{d}\xi + \int_{B_R} \int_0^{\tau \wedge T}\left|\cF[f(s,\cdot)](\xi)\right|  \mathrm{d}s  \mathrm{d}\xi \right]
\end{align}
with probability one.
However, \eqref{20240523 31} does not give any better information than \eqref{20240523 10}
since 
\begin{align*}
C_{R,T}^{\mathrm{e}\int\Re[\psi]} \leq \left(1+C_{R,T}^{|\psi|}  \right)
\end{align*}
as discussed in Remark \ref{simple constant} and \eqref{20240721 50}. 
\end{rem}
\begin{rem}
										\label{20240803 rem 30}
Recall that the logarithmic Laplacian operator $\log \Delta$ whose symbol is given by $-\log |\xi|^2$ satisfies Assumption \ref{weaker as} but it does not satisfy Assumption \ref{weaker as 2}.
Thus Theorem \ref{time deter thm} is not applicable to \eqref{deter eqn} with logarithmic operators.
However, Assumption \ref{weaker as 2} in Theorem \ref{time deter thm} also can be replaced with a weaker one to include logarithmic operators as mentioned in Remark \ref{20240803 rem 10}.
Specifically, the weaker condition consists of two parts as follows:
\begin{enumerate}[(i)]
\item for almost every $\xi \in \fR^d$,
\begin{align*}
\int_0^t |\psi(s,\xi)| \mathrm{d}s < \infty \quad  (a.s.) \quad \forall t \in (0,\infty).
\end{align*}
\item for all $R,T \in (0,\infty)$.
\begin{align*}
\esssup_{\xi \in  B_R} \left( \int_0^T|\psi(t,\xi)|\sup_{0\leq s \leq t}  \left|\exp\left(\int_s^t\Re[\psi(r,\xi)]\mathrm{d}r\right)\right|     \mathrm{d}t  \right) < \infty \quad (a.s.).
\end{align*}
\end{enumerate}
Therefore, we can achieve the weak well-posedness of \eqref{deter eqn} with logarithmic operators, even though such results might be impossible for SPDEs.
\end{rem}

Before going further to handle non-zero $g$, we consider a simple measurability issue for a predictable process.
Let $X_t$ be a predictable process on $[0,\infty)$ and assume that
\begin{align}
										\label{20240711 01}
\esssup_{\omega \in \Omega} |X_t|  < \infty.
\end{align}
Then the mapping
$$
t \in [0,\infty) \mapsto \esssup_{\omega \in \Omega} |X_t|
$$
is $\cB([0,\infty))$-measurable.
One can easily show it as least in two folds.
First, our probability space $\Omega$ include Brownian motions and thus it could be derived from the Wiener space.
Thus we may assume that $\Omega$ is separable.
On the other hand, any predictable process could be generated from an approximation of step processes and any step process possessing  \eqref{20240711 01} satisfies the measurability.
Here $\rF_t$-adapted process $X_t$ is called a step process if
\begin{align*}
X_t = \sum_{k=1}^n X_k 1_{(t_{k-1},t_k]}(t),
\end{align*}
for some  $0\leq t_0 \leq t_1 \leq \cdots \leq t_n<\infty$.

\begin{thm}[Existence of a Fourier-space weak solution with a random symbol]
									\label{exist weak sol}
Suppose that all assumptions in Theorem \ref{time thm} hold.
Then there exists a Fourier-space weak solution $u$ to \eqref{time eqn} in the intersection of the classes
\begin{align}
									\label{2024061610}
\cF^{-1}\bC L_{1,\ell oc}\left( \clbrk 0,\tau \cbrk \times \fR^d,\sigma\left(\rG \times \cB([0,\infty))  \cup \rH\cup \cP   \right)\times \cB(\fR^d)\right)
\end{align}
and
\begin{align}
									\label{2024061611}
\cF^{-1}\bL_{0,1,1,\ell oc}\left( \opar 0,\tau \cbrk \times \fR^d, 
\sigma\left(  \rG \times \cB([0,\infty))  \cup \rH  \cup \cP   \right)\times \cB(\fR^d), |\psi(t,\xi)|\mathrm{d}t \mathrm{d}\xi\right).
\end{align}

\end{thm}

\begin{proof}
Denote
\begin{align*}
\tilde \psi(t,\xi) := \esssup_{\omega \in \Omega} |\Re[\psi(t,\xi)]|+ \mathrm{i}\esssup_{\omega \in \Omega} |\Im[\psi(t,\xi)]|.
\end{align*}
Then $\tilde \psi(t,\xi)$ becomes a deterministic symbol satisfying \eqref{main deter as} and \eqref{main deter as 2} due to Assumptions \ref{main as} and \ref{main as 2}.
Thus by Theorem \ref{deter exist weak sol}, there exists a solution $\tilde u$ in the intersection of the classes
\begin{align*}
\cF^{-1}\bC L_{1,\ell oc}\left( \clbrk 0,\tau \cbrk \times \fR^d,\sigma\left( \rG \times \cB([0,\infty))  \cup \rH\cup \cP   \right)\times \cB(\fR^d)\right)
\end{align*}
and
\begin{align*}
\cF^{-1}\bL_{0,1,1,\ell oc}\left( \opar 0,\tau \cbrk \times \fR^d, 
\sigma\left( \rG \times \cB([0,\infty))  \cup \rH  \cup \cP  \right) \times \cB(\fR^d), |\psi(t,\xi)|\mathrm{d}t \mathrm{d}\xi\right).
\end{align*}
so that
\begin{align}
								\notag
&\mathrm{d}\tilde u=\left(\tilde \psi(t,-\mathrm{i}\nabla)\tilde u(t,x) + f(t,x)\right) \mathrm{d}t + g^k(t,x)\mathrm{d}B^k_t,\quad &(t,x) \in \Omega \times (0,\tau)\times \mathbf{R}^d,\\
&\tilde u(0,x)=u_0,\quad & x\in\mathbf{R}^d
								\label{2024061601}
\end{align}
and
\begin{align}
\cF[\tilde u(t,\cdot)](\xi)
										\notag
&= \exp\left(\int_0^t \tilde \psi(r,\xi)\mathrm{d}r \right) \cF[u_0](\xi) 
 +\int_0^t  \exp\left(\int_s^t \tilde \psi(r,\xi)\mathrm{d}r \right) 1_{\opar 0, \tau \cbrk}(s)\cF[f(s,\cdot)](\xi)\mathrm{d}s   \\
										\notag
&\quad + \int_0^t  \exp\left(\int_s^t \tilde \psi(r,\cdot)\mathrm{d}r \right) 1_{\opar 0, \tau \cbrk}(s)\cF[g^k(s,\cdot) ]\mathrm{d}B^k_s.
\end{align}
Put
\begin{align*}
\tilde f :=\psi(t,-\mathrm{i}\nabla)\tilde u - \tilde \psi(t,-\mathrm{i}\nabla)\tilde u,
\end{align*}
\begin{align*}
H_1(t,\xi) 
=\exp\left(\int_0^t \tilde \psi(r,\xi)\mathrm{d}r \right) \cF[u_0](\xi) 
 +\int_0^t  \exp\left(\int_s^t \tilde \psi(r,\xi)\mathrm{d}r \right) 1_{\opar 0, \tau \cbrk}(s)\cF[f(s,\cdot)](\xi)\mathrm{d}s,
\end{align*}
and
\begin{align*}
H_2(t,\xi) 
=\int_0^t  \exp\left(\int_s^t \tilde \psi(r,\cdot)\mathrm{d}r \right) 1_{\opar 0, \tau \cbrk}(s)\cF[g^k(s,\cdot) ]\mathrm{d}B^k_s.
\end{align*}
Then applying  Assumptions \ref{main as} and \ref{main as 2} with Lemma \ref{deterministic solution part} and Corollary \ref{20240417 01}, we have
\begin{align*}
H_1, H_2 \in   \bL_{0,1,1,loc,\ell oc}\Big( \Omega \times [0,\infty) \times \fR^d, \sigma\left( \rG \ \times \cB\left([0,\infty)\right)  \cup  \rH  \cup \cP \right) \times \cB(\fR^d), |\tilde \psi(t,\xi)|\mathrm{d}t \mathrm{d}\xi\Big).
\end{align*}
Additionally, since $\psi(t,\xi)$ is $\cP \times \cB(\fR^d)$-measurable and
\begin{align*}
|\psi(t,\xi)| \leq |\tilde \psi(t,\xi) | \quad \forall (\omega,t,\xi) \in \Omega \times [0,\infty) \times \fR^d,
\end{align*}
we have
\begin{align*}
\tilde f  
\in  \cF^{-1}\bL_{0,1,1,loc,\ell oc}\Big( \Omega \times [0,\infty) \times \fR^d, \sigma\left( \rG \ \times \cB\left([0,\infty)\right)  \cup  \rH  \cup \cP \right) \times \cB(\fR^d)\Big).
\end{align*}
In particular, by Lemma \ref{restriction lem},
\begin{align*}
\tilde f  
\in  \cF^{-1}\bL_{0,1,1,\ell oc}\Big( \opar 0, \tau \cbrk \times \fR^d, \sigma\left( \rG  \times \cB\left([0,\infty)\right)  \cup  \rH  \cup \cP \right) \times \cB(\fR^d)\Big).
\end{align*}
Thus, by Corollary \ref{exist cor 1}, there exists a solution $v$
in the intersection of the classes
$$
\cF^{-1}\bC L_{1,\ell oc}\left( \clbrk 0,\tau \cbrk \times \fR^d,\sigma\left(  \rG \times \cB([0,\infty))  \cup \rH\cup \cP   \right)\times \cB(\fR^d)\right)
$$
and
\begin{align*}
\cF^{-1}\bL_{0,1,1,\ell oc}\Big(\opar 0, \tau \cbrk \times \fR^d, \sigma\left( \rG  \times \cB\left([0,\infty)\right)  \cup  \rH  \cup \cP \right) \times \cB(\fR^d),|\tilde \psi(t,\xi)|\mathrm{d}t \mathrm{d}\xi\Big)
\end{align*}
so that
\begin{align}
								\notag
&\mathrm{d}v=\left(\psi(t,-\mathrm{i}\nabla)v(t,x) + \tilde f (t,x) \right) \mathrm{d}t,\quad &(t,x) \in \Omega \times (0,\tau)\times \mathbf{R}^d,\\
&v(0,x)=0,\quad & x\in\mathbf{R}^d.
								\label{2024061602}
\end{align}
Set
\begin{align*}
u(t,\xi) = \tilde u(t,\xi) + v(t,\xi).
\end{align*}
Then, leveraging the linearity of equations \eqref{2024061601} and \eqref{2024061602}, it is straightforward to demonstrate that $u$ serves as a Fourier weak solution to \eqref{time eqn}. 
Additionally, $u$ belongs to the classes given in  
\eqref{2024061610} and \eqref{2024061611} since both $\tilde u$ and $v$ are elements in these classes.
This completes the proof of the theorem.
\end{proof}

Recall that our solution $u$ in Theorem \ref{exist weak sol} cannot be represented by using kernels directly due to the randomness of the symbol. Consequently, obtaining optimal estimates from kernels, as shown in Corollary \ref{deter exist cor 2}, is not feasible. Finally, we believe that Theorem \ref{thm a priori} may play a crucial role in achieving better estimates in the following corollary.
\begin{corollary}
									\label{exist cor 2}
Suppose that all the assumptions in Corollary \ref{time corollary} hold.
Then the solution $u$ in Theorem \ref{exist weak sol} becomes an element in the intersection of the classes
$$
\cF^{-1}\bL_{1,loc}\bC L_{1,\ell oc}\left( \clbrk 0,\tau \cbrk \times \fR^d, 
\sigma\left( \rG \times \cB([0,\infty)) \cup \rH  \cup \cP  \right)\times \cB(\fR^d)\right)
$$
and
$$
\cF^{-1}\bL_{1,1,1,loc,\ell oc}\left( \opar 0,\tau \cbrk \times \fR^d, 
\sigma\left( \rG \times \cB([0,\infty)) \cup \rH  \cup \cP \right)\times \cB(\fR^d), |\psi(t,\xi)|\mathrm{d}t \mathrm{d}\xi\right).
$$
In particular, 
$$
u \in \cF^{-1}\bL_{1,loc}\bC L_{1,\ell oc}\left( \clbrk 0,\tau \cbrk \times \fR^d, 
\cP \times \cB(\fR^d)\right)
$$
and
$$
u \in \cF^{-1}\bL_{1,1,1,loc,\ell oc}\left( \opar 0,\tau \cbrk \times \fR^d, \cP \times \cB(\fR^d), |\psi(t,\xi)|\mathrm{d}t \mathrm{d}\xi\right)
$$
if $\cF[u_0]$ is $\rF_0 \times \cB(\fR^d)$-measurable  and $\cF[f]$ is  $\cP \times \cB(\fR^d)$-measurable.
Moreover, $u$ satisfies
\begin{align}
										\notag
\bE\left[ \int_0^{\tau \wedge T} \int_{B_R}|\psi(t,\xi)| |\cF[u(t,\cdot)](\xi)| \mathrm{d}\xi  \mathrm{d}t \right]  
										\notag
&\leq  \cC^1_{R,T,\psi} \bE \left[\int_{B_R} \left|\cF[u_0](\xi)\right|    \mathrm{d}\xi 
+\int_{B_R} \int_0^{ \tau \wedge T}\left|\cF[f(s,\cdot)](\xi)\right|  \mathrm{d}s  \mathrm{d}\xi \right] \\
										\label{20240721 10}
&\quad+ C_{BDG} \cdot \cC^1_{R,T,\psi} \bE\left[\int_{B_R} \left(\int_0^{\tau \wedge T}\left|\cF[g(s,\cdot)](\xi)\right|^2_{l_2}  \mathrm{d}s \right)^{1/2} \mathrm{d}\xi \right]
\end{align}
and
\begin{align}
\bE\left[ \int_{B_R} \sup_{t \in [0,\tau \wedge T]} |\cF[u(t,\cdot)](\xi)|\mathrm{d}\xi\right] 
										\notag
&\leq \left(\cC^2_{R,T,\psi} \wedge \cC^3_{R,T,\psi}\right)\bE\left[\int_{B_R} |\cF[u_0](\xi)|\mathrm{d}\xi + \int_{B_R} \int_0^{\tau \wedge T}\left|\cF[f(s,\cdot)](\xi)\right|  \mathrm{d}s  \mathrm{d}\xi \right] \\
										\label{20240721 11}
&\quad+ C_{BDG} \cdot \left(\cC^2_{R,T,\psi} \wedge \cC^3_{R,T,\psi}\right)\bE\left[\int_{B_R} \left(\int_0^{\tau \wedge T}\left|\cF[g(s,\cdot)](\xi)\right|^2_{l_2}  \mathrm{d}s \right)^{1/2} \mathrm{d}\xi \right]
\end{align}
for all $R, T \in (0,\infty)$, where $\cC^i_{R,T,\psi}$ $(i=1,2,3)$ are positive constants which are explicitly given by
\begin{align*}
\cC^1_{R,T,\psi}=\left[ C_{R,T}^{\sup|\psi|}+2T\|\psi\|_{L_\infty\left( \Omega \times (0,T) \times B_R \right)} \cdot C_{R,T}^{\mathrm{e}\int\sup|\Re[\psi]|}  \right],
\end{align*}
\begin{align*}
\cC^2_{R,T,\psi}=C_{R,T}^{\mathrm{e}\int\sup|\Re[\psi]|}  \left[ 1+ 2T\|\psi\|_{L_\infty\left( \Omega \times (0,T) \times B_R \right)}C_{R,T}^{\mathrm{e}\int\sup|\Re[\psi]|}  \right],
\end{align*}
and
\begin{align*}
\cC^3_{R,T,\psi} =  1+\cC^1_{R,T,\psi}.
\end{align*}
\end{corollary}
\begin{proof}
Let $R$ and $T$ be positive constants.
By Theorem \ref{exist weak sol}, there exists a Fourier-space weak solution $u$ to \eqref{time eqn} in a larger class. 
As shown in the proof of Theorem \ref{exist weak sol}, this solution $u$ can be split into two parts 
$u = \tilde u + v$, where $\tilde u$ satisfies \eqref{2024061601} and $v$ satisfies \eqref{2024061602}.
Since our data $u_0$ and $f$ are enhanced with finite expectations, by Corollary \ref{deter exist cor 2}, the solution $\tilde u$ satisfies
\begin{align}
\bE\left[ \int_0^{\tau \wedge T} \int_{B_R}|\tilde \psi(t,\xi)| |\cF[\tilde u(t,\cdot)](\xi)| \mathrm{d}\xi  \mathrm{d}t \right]  
										\notag
&\leq 
 C_{R,T}^{| \tilde \psi|} \bE \left[\int_{B_R} \left|\cF[u_0](\xi)\right|    \mathrm{d}\xi 
+\int_{B_R} \int_0^{ \tau \wedge T}\left|\cF[f(s,\cdot)](\xi)\right|  \mathrm{d}s  \mathrm{d}\xi  \right] \\
										\label{20240720 81}
& \quad +C_{BDG} \cdot C_{R,T}^{| \tilde \psi|}\bE\left[\int_{B_R} \left(\int_0^{\tau \wedge T}\left|\cF[g(s,\cdot)](\xi)\right|^2_{l_2}  \mathrm{d}s \right)^{1/2} \mathrm{d}\xi \right]
\end{align}
and
\begin{align}
\bE\left[ \int_{B_R} \sup_{t \in [0,\tau \wedge T]} |\cF[\tilde u(t,\cdot)](\xi)| \mathrm{d}\xi \right]
										\notag
&\leq  C_{R,T}^{\mathrm{e}\int\Re[\tilde\psi]} \cdot \bE \left[\int_{B_R} |\cF[u_0](\xi)|\mathrm{d}\xi 
+ \int_{B_R} \int_0^{\tau \wedge T}\left|\cF[f(s,\cdot)](\xi)\right|  \mathrm{d}s  \mathrm{d}\xi \right]
  \\
										\label{2024061801}
& \quad+ C_{BDG} \cdot C_{R,T}^{\mathrm{e}|\int\Re[\tilde\psi]|} \bE\left[\int_{B_R} \left(\int_0^{\tau \wedge T}\left|\cF[g(s,\cdot)](\xi)\right|^2_{l_2}  \mathrm{d}s \right)^{1/2} \mathrm{d}\xi \right].
\end{align}
On the other hand, by Corollary \ref{exist cor 1}, $v$ satisfies
\begin{align}
										\label{20240721 20}
\int_{B_R} \left(\sup_{t \in [0,\tau \wedge T]}  |\cF[v(t,\cdot)](\xi)| \right) \mathrm{d}\xi   
\leq  C_{R,T}^{\mathrm{e}\int\Re[\psi]} \int_{B_R} \int_0^{\tau \wedge T}\left|\cF[\tilde f(s,\cdot)](\xi)\right|  \mathrm{d}s  \mathrm{d}\xi \quad (a.s.)
\end{align}
and
\begin{align}
										\label{20240721 21}
 \int_0^T \int_{B_R}|\psi(t,\xi)||\cF[v(t,\cdot)](\xi)| \mathrm{d}\xi  \mathrm{d}t 
\leq
C_{R,T}^{|\psi|}
\int_{B_R} \int_0^{\tau \wedge T}\left|\cF[\tilde f(s,\cdot)](\xi)\right|  \mathrm{d}s  \mathrm{d}\xi \quad (a.s.),
\end{align}
where
\begin{align}
										\label{20240720 50}
\tilde f :=\psi(t,-\mathrm{i}\nabla)\tilde u - \tilde \psi(t,-\mathrm{i}\nabla)\tilde u.
\end{align}
Observe that
\begin{align*}
C^{\mathrm{e}\int\Re[\tilde \psi]}_{R,T}
\leq C^{\mathrm{e}|\int\Re[\tilde \psi]|}_{R,T}
&=\esssup_{ 0 \leq s \leq t \leq T, \xi \in  B_R}  
\left[\exp\left( \left|\int_s^t\Re[\tilde \psi(r,\xi)]  \mathrm{d}r \right| \right)\right]    \\
&=\esssup_{ 0 \leq s \leq t \leq T, \xi \in  B_R}  
\left[\exp\left( \int_s^t \esssup_{\omega \in \Omega}\left|\Re[\psi(r,\xi)] \right| \mathrm{d}r  \right)\right] 
= C^{\mathrm{e}\int\sup|\Re[ \psi]|}_{R,T},
\end{align*}
\begin{align*}
C^{| \tilde \psi|}_{R,T}
&=\esssup_{\xi \in  B_R} \left[ \int_0^T|\tilde \psi(t,\xi)|\sup_{0\leq s \leq t} \left|\exp\left( \int_s^t \Re[\tilde \psi(r,\xi)] \mathrm{d}r \right)\right|    \mathrm{d}t  \right] \\
&=\esssup_{ \xi \in B_R} \left[ \int_0^T\esssup_{\omega \in \Omega}|\psi(t,\xi)|
\left[\sup_{ 0\leq s \leq t} \exp\left( \int_s^t\esssup_{\omega \in \Omega}\left|\Re[\psi(r,\xi)]\right| \mathrm{d}r \right)\right] \mathrm{d}t  \right] 
=C^{\sup|\psi|}_{R,T},
\end{align*}
\begin{align*}
C^{\mathrm{e}\int\Re[\psi]}_{R,T}
&=\esssup_{ 0 \leq s \leq t \leq T,  \xi \in  B_R}  
\left[\exp\left( \int_s^t\Re[\psi(r,\xi)]  \mathrm{d}r  \right)\right]    \\
&\leq \esssup_{ (t,\xi) \in  (0,T) \times B_R}  
\left[\exp\left( \int_0^t\esssup_{\omega \in \Omega}\left|\Re[\psi(r,\xi)] \right| \mathrm{d}r  \right)\right] 
=C^{\mathrm{e}\int\sup|\Re[ \psi]|}_{R,T} \quad (a.s.),
\end{align*}

and
\begin{align*}
C^{|   \psi|}_{R,T}
&=\esssup_{\xi \in  B_R} \left[ \int_0^T| \psi(t,\xi)|\sup_{0\leq s \leq t} \exp\left( \left|\int_s^t \Re[\psi(r,\xi)] \mathrm{d}r \right| \right) \mathrm{d}t  \right] \\
&\leq \esssup_{(\omega,\xi) \in \Omega \times B_R} \left[ \int_0^T|\psi(t,\xi)|
\left[\esssup_{\omega \in \Omega, 0\leq s \leq t} \exp\left( \int_s^t\left| \Re[\psi(r,\xi)]\right| \mathrm{d}r \right)\right] \mathrm{d}t  \right] 
=C^{\sup|\psi|}_{R,T} \quad (a.s.).
\end{align*}

Thus by Fubini's theorem, \eqref{20240720 50}, and \eqref{2024061801}, 
\begin{align}
										\notag
&\bE\left[\int_0^{\tau \wedge T} \int_{B_R} |\cF[\tilde f(t,\cdot)](\xi)| \mathrm{d}\xi \mathrm{d}t \right] \\
										\notag
&\leq  2\|\psi\|_{L_\infty\left( \Omega \times (0,T) \times B_R \right)} \bE\left[ \int_0^{\tau \wedge T} \int_{B_R} |\cF[\tilde u(t,\cdot)](\xi)| \mathrm{d}\xi \mathrm{d}t \right] \\
										\notag
&\leq    2T\|\psi\|_{L_\infty\left( \Omega \times (0,T) \times B_R \right)} C_{R,T}^{\mathrm{e}\int\sup|\Re[\psi]|}\Bigg(  
\bE \left[\int_{B_R} |\cF[u_0](\xi)|\mathrm{d}\xi 
+ \int_{B_R} \int_0^{\tau \wedge T}\left|\cF[f(s,\cdot)](\xi)\right|  \mathrm{d}s  \mathrm{d}\xi \right]
  \\
										\label{20240720 82}
& \quad+ C_{BDG}\bE\left[\int_{B_R} \left(\int_0^{\tau \wedge T}\left|\cF[g(s,\cdot)](\xi)\right|^2_{l_2}  \mathrm{d}s \right)^{1/2} \mathrm{d}\xi \right] \Bigg).
\end{align}
Since $u =  \tilde u + v$, applying \eqref{20240720 81} and \eqref{20240721 21}, we have
\begin{align}
										\notag
&\bE\left[ \int_0^{\tau \wedge T} \int_{B_R}|\psi(t,\xi)| |\cF[u(t,\cdot)](\xi)| \mathrm{d}\xi  \mathrm{d}t \right]  \\
										\notag
&\leq \bE\left[ \int_0^{\tau \wedge T} \int_{B_R}|\psi(t,\xi)| |\cF[\tilde u(t,\cdot)](\xi)| \mathrm{d}\xi  \mathrm{d}t \right]  
+\bE\left[ \int_0^{\tau \wedge T} \int_{B_R}|\psi(t,\xi)| |\cF[v(t,\cdot)](\xi)| \mathrm{d}\xi  \mathrm{d}t \right]  
\\
										\notag
&\leq 
 C_{R,T}^{\sup|\psi|} \bE \left[\int_{B_R} \left|\cF[u_0](\xi)\right|    \mathrm{d}\xi 
+\int_{B_R} \int_0^{ \tau \wedge T}\left|\cF[f(s,\cdot)](\xi)\right|  \mathrm{d}s  \mathrm{d}\xi  \right] \\
										\notag
& \quad +C_{BDG} \cdot C_{R,T}^{\sup|\psi|}\bE\left[\int_{B_R} \left(\int_0^{\tau \wedge T}\left|\cF[g(s,\cdot)](\xi)\right|^2_{l_2}  \mathrm{d}s \right)^{1/2} \mathrm{d}\xi \right] \\
										\label{20240720 80}
&\quad +C_{R,T}^{\sup|\psi|}
\bE\left[\int_{B_R} \int_0^{\tau \wedge T}\left|\cF[\tilde f(s,\cdot)](\xi)\right|  \mathrm{d}s  \mathrm{d}\xi \right]. 
\end{align}
Additionally,  due to \eqref{20240720 82}, the sum of the last three terms in \eqref{20240720 80} is less than or equal to 
\begin{align}
								\notag
&\left[ C_{R,T}^{\sup|\psi|}+2T\|\psi\|_{L_\infty\left( \Omega \times (0,T) \times B_R \right)} \cdot C_{R,T}^{\mathrm{e}\int\sup|\Re[\psi]|}  \right] 
\bE \left[\int_{B_R} |\cF[u_0](\xi)|\mathrm{d}\xi  
+ \int_{B_R} \int_0^{\tau \wedge T}\left|\cF[f(s,\cdot)](\xi)\right|  \mathrm{d}s  \mathrm{d}\xi \right]
  \\
										\label{20240720 70}
& \quad+C_{BDG} \left[   C_{R,T}^{\sup|\psi|}
+2T \|\psi\|_{L_\infty\left( \Omega \times (0,T) \times B_R \right)}  \cdot C_{R,T}^{\mathrm{e}\int\sup|\Re[\psi]|}  \right]
\bE\left[\int_{B_R} \left(\int_0^{\tau \wedge T}\left|\cF[g(s,\cdot)](\xi)\right|^2_{l_2}  \mathrm{d}s \right)^{1/2} \mathrm{d}\xi \right].
\end{align}
Therefore \eqref{20240720 80} and \eqref{20240720 70} lead to \eqref {20240721 10}.
Next, we show \eqref{20240721 11} similarly.
We apply \eqref{2024061801} and \eqref{20240721 20} to obtain
\begin{align*}
& \bE\left[\int_{B_R} \sup_{t \in [0,\tau \wedge T]} |\cF[u(t,\cdot)](\xi)| \mathrm{d}\xi\right] \\
&\leq \bE\left[\int_{B_R} \sup_{t \in [0,\tau \wedge T]} |\cF[\tilde u(t,\cdot)](\xi)| \mathrm{d}\xi\right] 
+ \bE\left[\int_{B_R} \sup_{t \in [0,\tau \wedge T]} |\cF[v(t,\cdot)](\xi)| \mathrm{d}\xi \right]\\
&\leq  C_{R,T}^{\mathrm{e}\int\sup|\Re[\psi]|}\bE \left[\int_{B_R} |\cF[u_0](\xi)|\mathrm{d}\xi 
+ \int_{B_R} \int_0^{\tau \wedge T}\left|\cF[f(s,\cdot)](\xi)\right|  \mathrm{d}s  \mathrm{d}\xi \right]
  \\
& \quad+ C_{BDG} \cdot C_{R,T}^{\mathrm{e}\int\sup|\Re[\psi]|}\bE\left[\int_{B_R} \left(\int_0^{\tau \wedge T}\left|\cF[g(s,\cdot)](\xi)\right|^2_{l_2}  \mathrm{d}s \right)^{1/2} \mathrm{d}\xi \right] \\
&\quad + C_{R,T}^{\mathrm{e}\int\sup|\Re[\psi]|} 
\int_{B_R} \int_0^{\tau \wedge T}\left|\cF[\tilde f(s,\cdot)](\xi)\right|  \mathrm{d}s  \mathrm{d}\xi.
\end{align*}
Applying \eqref{20240720 82} to the last terms in the above inequality, we additionally obtain
\begin{align}
										\notag
 \bE\left[\int_{B_R} \sup_{t \in [0,\tau \wedge T]} |\cF[u(t,\cdot)](\xi)| \mathrm{d}\xi \right]
&\leq   C_{R,T}^{\mathrm{e}\int\sup|\Re[\psi]|}  \left[ 1+ 2T\|\psi\|_{L_\infty\left( \Omega \times (0,T) \times B_R \right)}C_{R,T}^{\mathrm{e}\int\sup|\Re[\psi]|}  \right] \\
										\notag
&\qquad \times \bE \left[\int_{B_R} |\cF[u_0](\xi)|\mathrm{d}\xi + \int_{B_R} \int_0^{\tau \wedge T}\left|\cF[f(s,\cdot)](\xi)\right|  \mathrm{d}s  \mathrm{d}\xi \right]  \\
										\notag
& \quad+ C_{BDG} \cdot C_{R,T}^{\mathrm{e}\int\sup|\Re[\psi]|}  \left[ 1+ 2T\|\psi\|_{L_\infty\left( \Omega \times (0,T) \times B_R \right)}C_{R,T}^{\mathrm{e}\int\sup|\Re[\psi]|}  \right]  \\
										\label{20240721 80}
&\quad \qquad \times \bE\left[\int_{B_R} \left(\int_0^{\tau \wedge T}\left|\cF[g(s,\cdot)](\xi)\right|^2_{l_2}  \mathrm{d}s \right)^{1/2} \mathrm{d}\xi \right].
\end{align}
On the other hand, we apply Theorem \ref{thm a priori} with \eqref{20240721 10}.
Then we have
\begin{align}
										\notag
 \bE\left[ \int_{B_R} \sup_{t \in [0,\tau \wedge T]} |\cF[u(t,\cdot)](\xi)| \mathrm{d}\xi \right]
&\leq   \left( 1+\left[ C_{R,T}^{\sup|\psi|}+2T\|\psi\|_{L_\infty\left( \Omega \times (0,T) \times B_R \right)} \cdot C_{R,T}^{\mathrm{e}\int\sup|\Re[\psi]|}  \right] \right) \\
										\notag
&\qquad \times \bE \left[\int_{B_R} |\cF[u_0](\xi)|\mathrm{d}\xi + \int_{B_R} \int_0^{\tau \wedge T}\left|\cF[f(s,\cdot)](\xi)\right|  \mathrm{d}s  \mathrm{d}\xi \right]  \\
										\notag
& \quad+ C_{BDG}\left(1 +  \left[   C_{R,T}^{\sup|\psi|}
+2T \|\psi\|_{L_\infty\left( \Omega \times (0,T) \times B_R \right)}  \cdot C_{R,T}^{\mathrm{e}\int\sup|\Re[\psi]|}  \right] \right) \\
										\label{20240721 81}
&\quad \qquad \times \bE\left[\int_{B_R} \left(\int_0^{\tau \wedge T}\left|\cF[g(s,\cdot)](\xi)\right|^2_{l_2}  \mathrm{d}s \right)^{1/2} \mathrm{d}\xi \right].
\end{align}
Finally, \eqref{20240721 80} and \eqref{20240721 81} complete \eqref{20240721 11}.
The corollary is proved. 
\end{proof}
\begin{rem}
										\label{20240730 rem}
The first $\cC^2_{R,T,\psi}$ in \eqref{20240721 11} can be slightly weaken by 
\begin{align}
										\label{20240730 30}
\esssup_{\omega \in \Omega} \left(C_{R,T}^{\mathrm{e} \int \Re[\psi]} \right)
+ 2T \cdot C_{R,T}^{\mathrm{e}\int\sup|\Re[\psi]|}  \|\psi\|_{L_\infty\left( \Omega \times (0,T) \times B_R \right)}C_{R,T}^{\mathrm{e}\int\sup|\Re[\psi]|}.
\end{align}
Moreover, \eqref{20240721 11} leads to the inequality
\begin{align*}
\bE\left[ \int_0^{\tau \wedge T} \int_{B_R} |\cF[u(t,\cdot)](\xi)|\mathrm{d}\xi\right] 
&\leq T\left(\cC^2_{R,T,\psi} \wedge \cC^3_{R,T,\psi}\right)\bE\left[\int_{B_R} |\cF[u_0](\xi)|\mathrm{d}\xi + \int_{B_R} \int_0^{\tau \wedge T}\left|\cF[f(s,\cdot)](\xi)\right|  \mathrm{d}s  \mathrm{d}\xi \right] \\
&\quad+ T\cdot C_{BDG} \cdot \left(\cC^2_{R,T,\psi} \wedge \cC^3_{R,T,\psi}\right)\bE\left[\int_{B_R} \left(\int_0^{\tau \wedge T}\left|\cF[g(s,\cdot)](\xi)\right|^2_{l_2}  \mathrm{d}s \right)^{1/2} \mathrm{d}\xi \right],
\end{align*}
for all $R, T \in (0,\infty)$.
Here even the second $\cC^2_{R,T,\psi}$ can be replaced with the constant from \eqref{20240730 30} by directly estimating
\begin{align*}
\bE\left[ \int_0^{\tau \wedge T} \int_{B_R} |\cF[u(t,\cdot)](\xi)|\mathrm{d}\xi\right] 
\end{align*}
based on \eqref{2024052110-2} in Corollary \ref{deter exist cor 2} instead of relying on \eqref{20240721 11}.
\end{rem}

\mysection{Proofs of Theorem \ref{time thm}, Corollary \ref{time corollary}, and  Theorem \ref{time deter thm} }
										\label{pf time thm}

We now proceed to prove Theorem \ref{time thm}, Corollary \ref{time corollary}, and Theorem \ref{time deter thm}. The necessary background for these proofs had been established in previous sections. Our task is to synthesize this information to reach the stated conclusions.

Specifically, we established the uniqueness and existence of a Fourier-space weak solution in Section \ref{unique section} and Section \ref{existence section}, respectively. Therefore, we need only verify that our solution classes align with both uniqueness and existence criteria. Fortunately, we had already demonstrated that our solution is unique within a large class, as shown in Theorem \ref{unique weak sol}, which encompasses all classes considered for the existence of a solution.

We conclude this section by succinctly connecting the prerequisites to reach our final conclusions.

\vspace{2mm}
{\bf Proof of Theorem \ref{time thm}}
\vspace{2mm}

Recall that we found a Fourier-space weak solution $u$  in Theorem \ref{exist weak sol} under all assumptions in Theorem \ref{time thm}.
We also proved the uniqueness of a Fourier-space weak solution in a larger class than the one containing $u$ in Theorem \ref{unique weak sol}. Therefore, the theorem is proved.
\qed

\vspace{2mm}
{\bf Proof Corollary \ref{time corollary}}
\vspace{2mm}

Recall that Assumption \ref{main as sta} obviously implies both  Assumptions \ref{main as} and \ref{main as 2}.
Therefore, based on Theorem \ref{time thm}, it is sufficient to demonstrate that the solution $u$ satisfies \eqref{linear a priori} and \eqref{linear a priori 2}. This had already been established in Corollary \ref{exist cor 2}.
\qed

\vspace{2mm}
{\bf Proof of Theorem \ref{time deter thm}}
\vspace{2mm}

Given the uniqueness of the solution in Theorem \ref{unique weak sol}, it is sufficient to show that the solution 
$u$ in Corollary \ref{exist cor 1} satisfies \eqref{202040524 20} and \eqref{202040524 21}.
However, it is also obtained in the same corollary.

\mysection{proof of Theorem \ref{main thm}}
											\label{pf main thm}

Recall that a space-time white noise can be constructed from a sequence of independent one-dimensional Brownian motions $B_t^k$ according to \eqref{20240525 01} with the help of an orthogonal basis $\eta_k$ on $L_2(\fR^d)$.
Also, remember that we additionally assumed that all elements in the orthogonal basis are in the Schwartz class $\cS(\fR^d)$.
Due to this additional assumption and Parseval's identity, we can explicitly calculate the $l_2$-values of the sequence formed by the Fourier transform of the product $\eta_k$ and an element $h$ in $L_2(\fR^d)$ in terms of the $L_2(\fR^d)$-norm of $h$. Here are the details.

\begin{lem}
									\label{l2 con lem}
Let $\{\eta_k : k \in \fN\}$ be an $L_2(\fR^d)$-orthonomal basis such that $\eta_k \in \cS(\fR^d)$ for all $k \in \fN$.
Then for any $h \in L_2(\fR^d)$,
\begin{align*}
\left|\cF[\eta h](\xi)\right|^2_{l_2} = (2\pi d)^{d} \|h\|^2_{L_2(\fR^d)} \quad (a.e.)~ \xi \in \fR^d,
\end{align*}
where 
$$
\cF[\eta h] = \left( \cF[\eta_1 h],\cF[\eta_2 h], \ldots \right).
$$

\end{lem}
\begin{proof}
For each $k \in \fN$, $\eta_k h \in L_2(\fR^d)$.
Thus the Fourier (Plancherel) transform of $\eta_k h$ is well-defined as an element in $L_2(\fR^d)$.
Additionally, by some properties of the Fourier and inverse Fourier transforms,
\begin{align*}
\cF[\eta_k h](\xi)
= (2\pi d)^{d/2} \cF[\eta_k] \ast \cF[h](\xi)
:= (2\pi d)^{d/2}\int_{\fR^d} \cF[\eta_k](\zeta) \cF[h](\xi - \zeta) \mathrm{d} \zeta.
\end{align*}
Observe that $\{\cF[\eta_k] :k \in \fN\}$ is another orthonomal basis on $L_2(\fR^d)$ due to the Plancherel theorem.
Thus for each $\xi$, applying Parseval's identity and the Plancherel theorem, we have
\begin{align*}
\left|\cF[\eta h](\xi)\right|^2_{l_2}
= (2\pi d)^{d} \|\cF[h](\xi - \cdot)\|^2_{L_2(\fR^d)}
= (2\pi d)^{d} \|h\|^2_{L_2(\fR^d)}.
\end{align*}
The lemma is proved.
\end{proof}
This identity allows us to regard our main equation \eqref{main eqn} as a specific case of \eqref{time eqn}.
In particular, the stochastic differential term $h(t,x) \mathrm{d}\fB_t$ in \eqref{main eqn} becomes well-defined if
$$
h \in \bL_{0,2,2}\left(\opar 0,\tau \cbrk \times \fR^d, \cP \times \cB(\fR^d)   \right).
$$
We show this in the following corollary.
\begin{corollary}
									\label{well define sto}
Let
$$
h \in \bL_{0,2,2}\left(\opar 0,\tau \cbrk \times \fR^d, \cP \times \cB(\fR^d)   \right).
$$
Then for any $\varphi \in \cF^{-1}\cD(\fR^d)$, the stochastic term 
$$
\int_0^t\left\langle h (s,\cdot) \eta^k(\cdot),\varphi\right\rangle \mathrm{d}B^k_s
:=\int_0^t 1_{\opar 0,\tau \cbrk}(s)\left\langle h (s,\cdot) \eta^k(\cdot),\varphi\right\rangle \mathrm{d}B^k_s
$$
in \eqref{solution meaning} is well-defined.
\end{corollary}
\begin{proof}
Let $\varphi \in \cF^{-1}\cD(\fR^d)$.
It suffices to show that for any $T \in (0,\infty)$,
\begin{align*}
\int_0^{T }  1_{\opar 0,\tau \cbrk}(t) \left|\left\langle h (t,\cdot) \eta^k(\cdot),\varphi\right\rangle \right|_{l_2}^2 \mathrm{d}t < \infty \quad (a.s.).
\end{align*}
For each $t \in (0,\infty)$, applying Lemma \ref{l2 con lem}, we have
\begin{align}
										\label{20240627 01}
\left|\cF[\eta h(t,\cdot)](\xi)\right|^2_{l_2} = (2\pi d)^{d} \|h(t,\cdot)\|^2_{L_2(\fR^d)} \quad (a.e.)~ \xi \in \fR^d,
\end{align}
In particular, $\left|\cF[\eta h(t,\cdot)](\xi)\right|_{l_2}$ is locally integrable on $\fR^d$.
Thus we have
\begin{align*}
\eta(\cdot) h(t,\cdot)  \in \cF^{-1}\cD'(\fR^d;l_2).
\end{align*}
and for each $\varphi \in \cF^{-1}\cD(\fR^d)$ and $k \in \fN$,
\begin{align}
										\label{20240627 02}
\left\langle h (t,\cdot) \eta^k(\cdot),\varphi\right\rangle
= \int_{\fR^d} \cF[h(t,\cdot) \eta^k](\xi) \overline{\cF[\varphi](\xi)} \mathrm{d}\xi
= \int_{\supp \cF[\varphi]} \cF[h(t,\cdot) \eta^k](\xi) \overline{\cF[\varphi](\xi)} \mathrm{d}\xi.
\end{align}
Therefore, combining \eqref{20240627 01} and \eqref{20240627 02} with the generalized Minkowski inequality, for each $T \in (0,\infty)$, we have
\begin{align*}
\int_0^{T }  1_{\opar 0,\tau \cbrk}(t) \left|\left\langle h (t,\cdot) \eta^k(\cdot),\varphi\right\rangle \right|_{l_2}^2 \mathrm{d}t 
\leq  (2\pi d)^d \left|\supp \cF[\varphi] \right| \sup_{\xi \in \fR^d} |\cF[\varphi](\xi)| \int_0^\tau \|h(t,\cdot)\|^2_{L_2(\fR^d)}
\mathrm{d}t < \infty \quad (a.s.),
\end{align*}
where $\left|\supp \cF[\varphi] \right|$ denotes the Lebesgue measure of $\supp \cF[\varphi]$ and is finite since $\cF[\varphi] \in \cD(\fR^d)$.
The corollary is proved.
\end{proof}

Now, Theorem \ref{main thm} follows as a corollary of Theorem \ref{time thm}.

\vspace{2mm}
{\bf Proof of Theorem \ref{main thm}}
\vspace{2mm}

Recall that $\{\eta_k : k \in \fN\}$ is an $L_2(\fR^d)$-orthonomal basis so that $\eta_k \in \cS(\fR^d)$ for all $k \in \fN$.
Put
$$
g=\left(g^1,g^2, \ldots \right) = \left(\eta_1 h, \eta_2 h, \ldots \right).
$$
Then for each $R \in (0,\infty)$, applying  H\"older's inequality, Fubini's theorem, and Lemma \ref{l2 con lem}, we have
\begin{align*}
 \int_{B_R} \left(\int_0^{\tau} |\cF[g(t,\cdot)](\xi)|^2_{l_2}\mathrm{d}t\right)^{1/2} \mathrm{d} \xi 
\leq  R \left(\int_0^\tau \|h(t,\cdot)\|^2_{L_2(\fR^d)} \mathrm{d}t\right)^{1/2} < \infty.
\end{align*}
Thus
$$
g \in \cF^{-1}\bL^{\omega,\xi,t}_{0,1,2,\ell oc}\left( \opar 0,  \tau \cbrk \times \fR^d, \cP \times \cB(\fR^d) ; l_2\right).
$$
Therefore,  Theorem \ref{main thm} is obtained from Theorem \ref{time thm}.

\vspace{2mm}
{\bf Proof of Corollary \ref{main cor}}
\vspace{2mm}

Put
$$
g=\left(g^1,g^2, \ldots \right) = \left(\eta_1 h, \eta_2 h, \ldots \right)
$$
as in the proof of Theorem \ref{main thm}.
Then for each $R,T \in (0,\infty)$, applying  H\"older's inequality, Fubini's theorem, and Lemma \ref{l2 con lem}, we have
\begin{align}
									\label{20240524 30}
 \bE\left[\int_{B_R} \left(\int_0^{\tau \wedge T} |\cF[g(t,\cdot)](\xi)|^2_{l_2}\mathrm{d}t\right)^{1/2} \mathrm{d} \xi \right] 
&\leq  R \cdot \bE\left[ \left(\int_0^{\tau \wedge T} \|h(t,\cdot)\|^2_{L_2(\fR^d)} \mathrm{d}t\right)^{1/2} \right]  
< \infty.
\end{align}
Thus
$$
g \in \cF^{-1}\bL^{\omega,\xi,t}_{1,1,2,\ell oc}\left( \opar 0,  \tau \cbrk \times \fR^d, \cP \times \cB(\fR^d)|_{\opar 0,  \tau \cbrk \times \fR^d} ; l_2\right).
$$
Therefore,  Corollary \ref{main cor} is obtained from Corollary \ref{time corollary} with \eqref{20240524 30}.

\bibliographystyle{plain}

\end{document}